\documentclass[a4paper,reqno]{amsart}

\usepackage{amsmath,amssymb,amsthm}

\usepackage[abbrev]{amsrefs}
\usepackage{graphicx}

\usepackage{mathtools}
\mathtoolsset{showonlyrefs=true}

\usepackage[top=25mm, bottom=25mm, left=25mm, right=25mm]{geometry}

\usepackage{cite,mathrsfs}
\usepackage{ulem}

\DeclareMathAlphabet\mathbfcal{OMS}{cmsy}{b}{n}

\usepackage{bm}
\newcommand{\mathbsf}[1]{\bm{\mathsf{#1}}} 

\usepackage{color}

\definecolor{blue}{cmyk}{0,0,0,1}
\definecolor{red}{cmyk}{0,0,0,1}
\definecolor{magenta}{cmyk}{0,0,0,1}

\newtheorem{thm}{Theorem}[section]
\newtheorem{lem}[thm]{Lemma}
\newtheorem{cor}[thm]{Corollary}
\newtheorem{prop}[thm]{Proposition}
\numberwithin{equation}{section}

\theoremstyle{definition}
\newtheorem{dfn}[thm]{Definition}
\newtheorem{assump}[thm]{Assumption}
\newtheorem{rem}[thm]{Remark}
\newtheorem{conv}[thm]{Convention}

\newcommand{\BA}{\mathbb{A}}
\newcommand{\BB}{\mathbb{B}}
\newcommand{\BC}{\mathbb{C}}
\newcommand{\BD}{\mathbb{D}}
\newcommand{\BH}{\mathbb{H}}
\newcommand{\BI}{\mathbb{I}}
\newcommand{\BK}{\mathbb{K}}
\newcommand{\BM}{\mathbb{M}}
\newcommand{\BN}{\mathbb{N}}
\newcommand{\BP}{\mathbb{P}}
\newcommand{\BQ}{\mathbb{Q}}
\newcommand{\BR}{\mathbb{R}}

\newcommand{\BZ}{\mathbb{Z}}

\newcommand{\CB}{\mathcal{B}}
\newcommand{\CD}{\mathcal{D}}
\newcommand{\CF}{\mathcal{F}}
\newcommand{\CH}{\mathcal{H}}

\newcommand{\CJ}{\mathcal{J}}
\newcommand{\CM}{\mathcal{M}}
\newcommand{\CN}{\mathcal{N}}
\newcommand{\CL}{\mathcal{L}}
\newcommand{\CP}{\mathcal{P}}
\newcommand{\CQ}{\mathcal{Q}}
\newcommand{\CR}{\mathcal{R}}
\newcommand{\CS}{\mathcal{S}}
\newcommand{\CT}{\mathcal{T}}

\newcommand{\CW}{\mathcal{W}}

\newcommand{\ba}{\boldsymbol{a}}
\newcommand{\bb}{\boldsymbol{b}}

\newcommand{\bd}{\mathbf{d}}
\newcommand{\bff}{\boldsymbol{f}}
\newcommand{\bF}{\boldsymbol{F}}
\newcommand{\bg}{\boldsymbol{g}}
\newcommand{\bG}{\boldsymbol{G}}
\newcommand{\bh}{\boldsymbol{h}}
\newcommand{\bH}{\boldsymbol{H}}
\newcommand{\bJ}{\boldsymbol{J}}
\newcommand{\bK}{\boldsymbol{K}}
\newcommand{\bL}{\boldsymbol{L}}
\newcommand{\bn}{\boldsymbol{n}}
\newcommand{\bM}{\boldsymbol{M}}
\newcommand{\bu}{\boldsymbol{u}}

\newcommand{\bU}{\boldsymbol{U}}
\newcommand{\bv}{\boldsymbol{v}}
\newcommand{\bV}{\boldsymbol{V}}
\newcommand{\bw}{\boldsymbol{w}}

\newcommand{\bX}{\boldsymbol{X}}

\newcommand{\bPhi}{\boldsymbol{\Phi}}

\newcommand{\bCF}{\mathbfcal{F}}
\newcommand{\bsF}{\mathbsf{F}}

\newcommand{\fA}{\mathfrak{A}}

\newcommand{\fp}{\mathfrak{p}}
\newcommand{\fq}{\mathfrak{q}}

\newcommand{\sA}{\mathscr{A}}
\newcommand{\sB}{\mathscr{B}}

\newcommand{\sfe}{\mathsf{e}}
\newcommand{\sfE}{\mathsf{E}}
\newcommand{\sfF}{\mathsf{F}}

\newcommand{\sfQ}{\mathsf{Q}}
\newcommand{\sfP}{\mathsf{P}}
\newcommand{\sfs}{\mathsf{s}}

\newcommand{\pd}{\partial}
\newcommand{\wh}{\widehat}
\newcommand{\wt}{\widetilde}
\newcommand{\HS}{\mathbb{R}^d_+}

\newcommand{\tdB}{\widetilde{\dot{B}}{}}
\newcommand{\tdH}{\widetilde{\dot{H}}{}}

\DeclareMathOperator{\dv}{{div}}
\DeclareMathOperator{\DV}{{Div}}
\DeclareMathOperator{\RE}{{Re}}
\renewcommand{\d}{\,\mathrm{d}}
\DeclareMathOperator{\supp}{{supp\,}}
\DeclareMathOperator{\Hol}{{Hol\,}}

\begin{document}
\title[]{Maximal $L_1$-regularity of the Navier--Stokes equations 
with free boundary conditions via a generalized semigroup theory}

\author[]{Yoshihiro Shibata}
\address{(Y. Shibata) 
Emeritus Professor of Waseda University, Waseda University,
3-4-1 Ohkubo Shinjuku-ku Tokyo, 169-8555, Japan. \\
Adjunct faculty member in the Department of Mechanical Engineering 
and	Materials Science, University of Pittsburgh}
\email{yshibata325@gmail.com}

\author[]{Keiichi Watanabe}
\address{(K. Watanabe) School of General and Management Studies,
Suwa University of Science, 5000-1, Toyohira, Chino, Nagano 391-0292, Japan}		
\email{\texttt{watanabe\_keiichi@rs.sus.ac.jp}}


\thanks{The first author was partially supported by JSPS KAKENHI 
Grant Number 23K22405,	
and the second author was partially supported by JSPS KAKENHI Grant Number 21K13826.}


\date{\today}   

\maketitle

\begin{abstract}
This paper develops a {\color{magenta}new} approach to show the maximal regularity
theorem of the Stokes equations with free boundary 
conditions in the half-space $\mathbb R^d_+$, $d \ge 2$, within 
the $L_1$-in-time and $\mathcal B^s_{q, 1}$-in-space framework with
$(q, s)$ satisfying $1 < q < \infty$ and $- 1 + 1 \slash q < s < 1 \slash q$,
where $\mathcal B^s_{q, 1}$ stands for either homogeneous or inhomogeneous
Besov spaces. 
In particular, we establish a generalized semigroup
theory within an $L_1$-in-time and $\mathcal B^s_{q,1}$-in-space
framework, which extends a classical $C_0$-analytic semigroup
theory to the case of inhomogeneous boundary conditions.
The maximal $L_1$-regularity theorem is proved by estimating the Fourier--Laplace
inverse transform of the solution to the generalized Stokes resolvent problem
with inhomogeneous boundary conditions, where density and interpolation arguments
are used. The maximal $L_1$-regularity theorem is applied to show the unique existence of a local 
strong solution to the Navier--Stokes equations with free boundary 
conditions for arbitrary initial data $\boldsymbol a$ 
in $B^s_{q, 1} (\mathbb R^d_+)^d$, where $q$ and $s$ satisfy 
${\color{black}d-1 < q \le d}$ and $-1+d/q < s < 1/q$, respectively.
If we assume that the initial data $\boldsymbol a$ are small in 
$\dot B^{- 1 + d \slash q}_{q, 1} (\mathbb R^d_+)^d$, 
$d - 1 < q < 2 d$, then the unique existence of a global 
strong solution to the system is proved.
\end{abstract}

\section{Introduction}
\subsection{Formulation of the problem}
We consider the Navier--Stokes equations,  
which describe the motion of incompressible viscous fluids,  at
time $t > 0$ in a moving domain $\Omega(t)$ in the $d$-dimensional 
Euclidean space $\BR^d$ with $d \ge 2$, where the initial domain
$\Omega (0)$ coincides with the $d$-dimensional half-space
$\HS$. Here, $\HS$ and its boundary are denoted by
$$
\HS := \{(x', x_d) \mid x' \in \BR^{d - 1}, \enskip x_d > 0 \}, \quad 
\pd \HS := \{(x', x_d) \mid x' \in \BR^{d - 1}, \enskip x_d = 0 \},
$$
respectively. For $t > 0$ the governing equations read
\begin{align}
\label{eq-original}
\left\{\begin{aligned}
\pd_t \bv + (\bv \cdot \nabla) \bv - \DV(\mu \BD(\bv) - P\BI) & = 0 & \quad & \text{in $\Omega (t)$}, \\
\dv \bv & = 0 & \quad & \text{in $\Omega (t)$}, \\
(\mu \BD(\bv) - P\BI) \bn & = 0 & \quad & \text{on $\pd \Omega (t)$}, \\
V_{\pd \Omega (t)} & = \bv \cdot \bn & \quad & \text{on $\pd \Omega (t)$}, \\
\bv \vert_{t = 0} & = \ba & \quad & \text{in $\HS$}, \\
\Omega(0) & = \HS.
\end{aligned}\right.		
\end{align}
System \eqref{eq-original} is called \textit{the free boundary problem for the Navier--Stokes equations}. 
Here, the unknowns are the velocity field $\bv = \bv (x, t)$, 
the pressure $P = P (x, t)$, and the domain $\Omega = \Omega (t)$,
whereas an initial velocity field $\ba = \ba (x)$ is given.
{\color{black} Furthermore}, the stress tensor is given by 
$\mu \BD (\bv) - P \BI$ with $\BD (\bv) = \nabla \bv + [\nabla \bv]^\top$, where $\mu > 0$ 
stands for the viscosity coefficient of the fluid and $\BI$ is the $d \times d$ identity matrix. 
In addition, we use the notation
\begin{equation}
\DV \bM = \bigg(\sum_{j = 1}^d \pd_j M_{1,j}, \ldots, \sum_{j = 1}^d \pd_j M_{d,j}\bigg)^\top
\end{equation}
for a $d\times d$ matrix-valued function $\bM = (M_{i, j})_{1 \le i, j \le d}$,
where we have used the abbreviation $\pd_j = \pd\slash \pd x_j$ for $j = 1, \ldots, d$.
In this paper, we assume that $\mu$ is a constant for simplicity and 
the atmospheric pressure $P_0$ is assumed to satisfy $P_0 \equiv 0$ 
without loss of generality (cf. \cite[$(3.2)$]{Shi20}).
The outward unit normal vector to 
$\pd \Omega (t)$ is denoted by $\bn$ and the normal velocity of 
$\pd \Omega (t)$ is denoted by $V_{\pd \Omega (t)}$.
Let us briefly explain the meaning of each equation in \eqref{eq-original}. 
The first and second equations are the incompressible Navier--Stokes equations, 
where the density of the fluid is assumed to be 1. 
The third equation in \eqref{eq-original} is called the dynamic boundary condition and, 
physically, this equation states that the normal
stress is continuous as one passes through $\pd \Omega (t)$. 
The fourth equation in \eqref{eq-original} implies that 
the free surface is advected with the fluid, 
which is called the kinematic boundary condition. 
In other words, this boundary condition ensures that fluid particles 
do not cross the free surface $\pd \Omega (t)$. 
\subsection{Problem reformulation}
It is well-known that the movement of the domain $\Omega (t)$ 
creates numerous mathematical difficulties.
Hence, as usual, we will transform the free boundary problem 
\eqref{eq-original} to a problem with a fixed interface
in order to avoid their difficulties, motivated by the 
pioneering {\color{black}works} due to Beale \cite{B80} and 
Solonnikov \cite{Sol77}.
Since we consider the system \textit{without} taking surface 
tension into account, following the idea due to Solonnikov \cite{Sol88} we may rely on a 
Lagrangian transformation to flatten the free boundary 
$\pd \Omega (t)$ due to $\pd \Omega (0) = \pd \HS$. 
More precisely, let $x = x_{\bv} (t) \in \Omega (t)$ be 
the solution to the Cauchy problem
\begin{equation}
\label{eq-Cauchy-Lagrange}
\frac{\d x}{\d t} = \bv (x (t), t) \quad \text{for $t > 0$}, \qquad x \vert_{t = 0} = y \in \HS.
\end{equation}
This means that $x = x_{\bv} (t)$ describes the position of 
fluid particle at time $t > 0$ which was located in 
$y \in \HS$ at $t = 0$. Then the change of coordinates 
$(x, t) \rightsquigarrow (y, t)$ is said to be the 
\textit{Lagrangian transformation}. 
We note that for each $t > 0$ the kinematic condition for 
$\pd \Omega (t)$ (i.e., the forth equation in \eqref{eq-original}) 
is automatically satisfied under the Lagrangian transformation. 
If $\bv (x, t)$ is  Lipschitz continuous with respect to $x$, 
then \eqref{eq-Cauchy-Lagrange} admits a unique solution given by
\begin{equation}
x_{\bv} (t) = y + \int_0^t \bv (x (y, \tau), \tau) \d \tau, \qquad y \in \HS, \; t > 0,
\end{equation}
which represents the particle trajectory, due to the classical Picard-Lindel\"of theorem. 
We now set $\bu (y, t) := \bv (x, t)$ and $P (x, t) = \sfQ (y, t)$,
where $\bu (y, t)$ is so-called the \textit{Lagrangian velocity filed}. 
Moreover, we  set
\begin{equation}
\label{Lagrangian-transformation}
\bX_{\bu} (y, t) := y + \int_0^t \bu (y, \tau) \d \tau, \qquad y \in \overline{\HS}, \; t > 0.
\end{equation}
Clearly, we have
$$
\Omega (t)  = \{x \in \BR^d \mid x = \bX_{\bu} (y, t), \enskip y \in \HS \}, \quad 
\pd \Omega (t)  = \{x \in \BR^d \mid x = \bX_{\bu} (y, t), \enskip y \in \pd \HS \}.
$$
Since the Jacobian matrix of the transformation $\bX_{\bu} (y, t)$ is given by
\begin{equation}
\nabla_y \bX_{\bu} (y, t) = \BI + \int_0^t \nabla_y \bu (y, \tau) \d \tau,
\end{equation}
the invertibility of $\bX_{\bu} (y, t)$ in \eqref{Lagrangian-transformation} 
is guaranteed for all $t \ge 0$ if $\bu$ satisfies
\begin{equation}
\label{integral-smallness}
\bigg\lVert \int_0^t \nabla \bu (\,\cdot\,, \tau) \d \tau \bigg\rVert_{L_\infty (\HS)} \ll 1,
\end{equation}
which may be achieved by a Neumann-series argument. 
By virtue of \eqref{integral-smallness}, we may write
\begin{equation}
\label{def-Au}
\BA_{\bu} (y, t) := (\nabla_y \bX_{\bu} (y, t))^{- 1} 
= \sum_{l = 0}^{\infty} \bigg(- \int_0^t \nabla_y \bu (y, \tau) \d \tau \bigg)^l.
\end{equation}
With the above notation, for $0 < T \le \infty$ Problem \eqref{eq-original} in Lagrangian
coordinates reads
\begin{align}
\label{eq-fixed}
\left\{\begin{aligned}
\pd_t \bu - \DV (\mu \BD(\bu) - \sfQ \BI) & = \bF (\bu) & \quad & \text{in $\HS \times {\color{black} (0, T)}$}, \\
\dv \bu & = G_\mathrm{div} (\bu) = \dv \bG (\bu) & \quad & \text{in $\HS \times {\color{black} (0, T)}$}, \\
(\mu \BD(\bu) - \sfQ \BI) \bn_0 & = \BH (\bu) \bn_0 & \quad & \text{on $\pd \HS \times {\color{black} (0, T)}$}, \\
\bu \vert_{t = 0} & = \ba & \quad & \text{in $\HS$},
\end{aligned}\right.		
\end{align}
where we have set $\bn_0 = (0, \ldots, 0, - 1)$. The right-hand members 
$\bF (\bu)$, $G_\mathrm{div} (\bu)$, $\bG (\bu)$, and $\BH (\bu)$
represent nonlinear terms given as follows:
\begin{equation}
\label{def-nonlinear-terms}
\begin{split}
\bF (\bu) & := 
\bigg(\int^t_0\nabla_y \bu\d\tau\bigg)^\top
\Big(\pd_t \bu - \mu \Delta_y \bu  \Big) 
+ \mu \bigg\{\BI+ \bigg(\int^t_0\nabla_y \bu\d\tau \bigg)^\top \bigg\}\dv_y 
\Big( (\BA_{\bu} \BA_{\bu}^\top - \BI) \nabla_y \bu \Big) \\
& \quad\; + \mu \nabla_y \Big((\BA_{\bu}^\top - \BI) \colon \nabla_y \bu \Big), \\
G_\mathrm{div} (\bu) & := ( \BI - \BA_{\bu}^\top) \colon \nabla_y \bu, \\
\bG (\bu) & := (\BI - \BA_{\bu}) \bu, \\
\BH (\bu) & := 
\mu \bigg[ \bigg\{\nabla_y \bu + \bigg(\BI+ \bigg(\int^t_0\nabla_y \bu\d\tau \bigg)^\top \bigg)
[\nabla_y \bu]^\top \BA_{\bu} \bigg\} (\BI-\BA_{\bu}^\top)
+ \bigg(\int_0^t \nabla_y \bu \d \tau \bigg)^\top [\nabla_y \bu]^\top \BA_{\bu} \\
& \qquad \;
+ [\nabla_y \bu]^\top (\BI-\BA_{\bu}) \bigg], \\
\BD_y (\bu) & := \nabla_y \bu + [\nabla_y \bu]^\top.	
\end{split}
\end{equation}
Recall that for $d \times d$ matrices 
$\BA = (A_{j, k})$ and $\BB = (B_{j, k})$, 
we write $\BA \colon \BB = \sum_{j, k}^d A_{j, k} B_{j ,k}$. 
The formulas \eqref{def-nonlinear-terms} have essentially been derived in 
\cite[Subsec.~3.3.2]{Shi20}, but we shall provide a precise derivation  of nonlinear terms in
\eqref{def-nonlinear-terms}, see Appendix A below.
It should be emphasized here that all nonlinear terms given in 
\eqref{def-nonlinear-terms} do not contain the pressure term $\sfQ$,
which is different from the formula derived by Solonnikov \cite{Sol88}. 
{\color{magenta}The advantage of our representation is that it is not necessary to use
the estimate of the pressure to construct solutions to \eqref{eq-fixed}, which means that
the estimate of the boundary trace of the pressure term is not necessary in our analysis.}
\par
\subsection{References overview}
The pioneer work for free boundary problems of the Navier--Stokes 
equations had been completed by Beale \cite{B80} and Solonnikov \cite{Sol77,Sol88}.
Here, Beale \cite{B80} considered the problem in the layer domain
$\{(x',x_3) \in \BR^2 \times \BR \mid -b (x') \le x_3 \le \eta (x',t)\}$, where $b (x')$ is a given function
and $\eta (x',t)$ is an unknown function describing the free surface of the fluid.
In particular, after a reformulation of the problem to a system with fixed and flat interfaces,
he obtained the solution formulas by using the partial Fourier
transform in the tangential direction of the boundary and, by solving the resultant ordinary differential 
equations, and he estimated the solution formulas in the $L_2$-space framework by employing
the Plancherel theorem. On the other hand, Solonnikov \cite{Sol77,Sol88} considered the problem in 
a bounded domain, and he used
potential theory to obtain the representation formulas of solutions to 
equations on the boundary and he used either $L_2$ or H\"older estimates for
the potentials. We refer to surveys \cite{SD,DSBook} for recent progress by Solonnikov's approach.
In the approach due to Beale \cite{B80} and Solonnikov \cite{Sol77,Sol88}, it was necessary 
to consider the problem within either the $W^{2+\ell, 1+\ell/2}_2$ ($1/2 < \ell < 1$) framework or the 
H\"older spaces $C^{2+\theta, 1+\theta}$ ($\theta>0$) framework, which requires regularities more 
than two with respect to the spacial variables
to enclose their iteration schemes for solving the nonlinear problems, 
although the equations are second-order parabolic equations.
Moreover, their analyses were far from being employed a $C_0$-analytic semigroup theory. \par
However, it seems to be reasonable to apply the theory of maximal regularity
to the free boundary problem of the Navier--Stokes equations as they are quasilinear parabolic equations.
From this perspective, following a suggestion of Prof.~Herbert Amann in 2000, 
the groups of the first author and Jan Pr\"uss started to study the free 
boundary problems for the Navier--Stokes equations
by means of maximal regularity theory. 
Pr\"uss and his coauthors used techniques of $\CH^\infty$-calculus to solve the equations
on the boundary and proved the maximal 
$L_p$-regularity theorem for the Stokes equations with a free surface subject to
inhomogeneous boundary conditions and applied the theory to show the local as
well as global well-posedness of the two-phase free boundary value problem for
the Navier--Stokes equations \cite{PS10,KPW13}. The details of the theory established by
Pr\"uss and his coauthors may be found in his monograph \cite{PSbook} and the references therein.
On the other hand, the idea of the first author is to construct the $\CR$-bounded solution 
operators to the generalized resolvent problem, 
where his strategy is different from the group of Pr\"uss. Here, 
the generalized resolvent problem means the resolvent problem
subject to inhomogeneous boundary conditions. Together with the operator-valued
Fourier multiplier theorem due to Weis \cite{W01}, the first author proved the 
maximal $L_p$-$L_q$-regularity (an $L_p$-in-time and $L_q$-in-space) estimates
for the Stokes equations with a free surface subject to inhomogeneous boundary conditions,
see, e.g., \cite{Shi14,Shi20,Shi23,SS12}. He also proved the local and global 
well-posedness of the free boundary problem of the Navier--Stokes equations not only in
a bounded domain \cite{Shi15} but also in unbounded domains \cite{OS22,Shi23,Shi14,SSpre,Shi20}.
See also \cite{SS20,SSZ20} for the case of two-phase flows.
The advantage of this method is that it is also possible to apply to a system of 
hyperbolic-parabolic mixed equations with inhomogeneous boundary conditions
and the periodic solutions. 
For instance, the global well-posedness of the free boundary problem of 
the compressible viscous fluid flows was proved in \cite{ShiZhang23}, the 
local well-posedness of the
compressible-incompressible two-phase flows was proved in \cite{KS22}, and the existence of 
time-periodic solutions to the (one-phase or two-phase) Navier--Stokes equations 
with the effect of surface tension were studied in \cite{EKS21}. 
In all of these studies, the strong solutions were constructed
in an $L_p$-in-time and $L_q$-in-space setting allowing the case $p \ne q$.
This framework plays a crucial role in proving the global well-posedness
results for the free boundary problem of the Navier--Stokes equations with,
in particular, unbounded reference domains.
Regarding the theory established by the first author, 
we refer to comprehensive surveys \cite{Shi20,Shi23} and the references therein. \par
{\color{magenta}
Recently, an $L_1$-in-time and $\dot B^s_{q,1}$-in-space setting
was independently investigated by Danchin, Hieber, Mucha, and Tolksdorf \cite{DHMTpre} 
and Ogawa and Shimizu \cite{OS24}. Here, $\dot B^s_{q,1}$ stands for homogeneous
Besov spaces. Danchin et al. \cite{DHMTpre} developed a \textit{homogeneous} version 
of the classical Da Prato--Grisvard theorem \cite{DG75} and applied it to show the global well-posedness 
of \eqref{eq-original} with $d \ge 3$ provided that $q$ satisfies $d-1 < q < d$.
Their proof relied on the Stokes resolvent estimate due to the first author and 
Shimizu \cite{SS12} with the aid of real interpolation as well as the homogeneous 
version of the Da Prato--Grisvard theorem. On the other hand, Ogawa and Shimizu
\cite{OS24} proved the global well-posedness of \eqref{eq-original} with 
$d \ge 2$ provided that $q$ satisfies $d-1 < q < 2d - 1$, where they used
the explicit expression of the fundamental integral kernel of the
linearized system and the Littlewood--Paley dyadic decompositions 
in the time and space variables. Namely, Ogawa and Shimizu \cite{OS24} slightly 
improved the result due to Danchin et al. \cite{DHMTpre} since $q$ is allowed to take 
$d-1 < q < 2d - 1$. It should be emphasized here that the approach due to 
\cite{DHMTpre,OS24} strongly relied on some reductions to parabolic systems 
(i.e., heat equation), and hence it seems to be difficult to extend their approach
to a system of hyperbolic-parabolic mixed equations with, in particular, 
inhomogeneous boundary conditions.}
\par
The aim of the present paper is to develop a new approach to show 
the maximal regularity theorem of the Stokes equations with free boundary 
conditions in the half-space $\mathbb R^d_+$, $d \ge 2$, within 
{\color{black}an} $L_1$-in-time and $\CB^s_{q, 1}$-in-space framework, where 
$q$ and $s$ satisfy $1 < q < \infty$
and $- 1 + 1 \slash q < s < 1 \slash q$, respectively. Here, 
$\CB^s_{q, 1}$ stands for either homogeneous or inhomogeneous Besov spaces.
In this paper, we establish a method to treat evolution equations with inhomogeneous
boundary conditions via its generalized resolvent problem.  
{\color{magenta}
Our approach is based on a completely different idea due to Danchin et al. \cite{DHMTpre} 
as well as Ogawa and Shimizu \cite{OS24}. To be precise, our fundamental 
tool is to use the $L_q$-boundedness of the solution operator of the generalized 
resolvent problem associated with the linearized system of \eqref{eq-fixed},
which was proved in \cite{Shi14,Shi20,SS12}. Then, by means of a density argument and
interpolation, we will extend this result to the $\CB^s_{q,1}(\HS)$-estimates with
$-1+1/q<s<1/q$. To show the maximal $L_1$-$\CB^s_{q,1}$-regularity theorem for
the Stokes equations with free boundary conditions, we estimate the Fourier--Laplace
transform of the solution operator to the generalized resolvent problem, i.e.,
we extend the classical theory of analytic semigroup theory \cite[Ch.~IX]{Yosida}
to the case of inhomogeneous boundary conditions, which may be called 
\textit{generalized semigroup theory}.}
Moreover, it should be emphasized here that, in contrast to the approach due 
to Danchin et al. \cite{DHMTpre}, our approach is applicable to systems of 
hyperbolic-parabolic mixed equations with inhomogeneous 
boundary conditions in \textit{unbounded} domains, see \cite{KSpreprint} 
for the application to the compressible Navier--Stokes equations.
As a by-product of our linear theory, we prove the local well-posedness of \eqref{eq-original} 
provided that the initial velocity $\ba$ 
belongs to $B^s_{q, 1} (\mathbb R^d_+)^d$ without any restriction on
the size of initial data, where $q$ and $s$ satisfy 
${\color{black} d - 1 < q \le d}$ and $-1+d/q < s < 1/q$. If we impose the smallness condition on
the initial data, then the restriction on the regularity index $s$ is 
relaxed and we may show the local well-posedness of \eqref{eq-original} 
for arbitrary given $T > 0$ provided that $(q, s)$ satisfies $-1+d\slash q \le s < 1 \slash q$. 
Furthermore, if we switch the function spaces from $B^s_{q, 1} (\HS)$ 
to $\dot B^{- 1 + d \slash q}_{q, 1} (\HS)$, then we obtain
the global well-posedness of \eqref{eq-original} provided that
the initial data are small in $\dot B^{- 1+d \slash q}_{q, 1} (\HS)$, $d-1<q<2d$.

\subsection{Main results for the free boundary problem for the Navier--Stokes equations.}
Throughout this paper, for $(q, r, s) \in (1, \infty) \times [1, \infty-] \times \BR$ 
the space $\CB^s_{q, r} (\Omega)$ stands for either an inhomogeneous Besov space
$B^s_{q, r} (\Omega)$ or a homogeneous Besov space $\dot B^s_{q, r} (\Omega)$
with $\Omega \in \{\BR^d, \HS\}$, where the precise definition will be given in the next section.
Let $\CS'(\BR^d)$ and $\CP(\BR^d)$ be the set of all 
tempered distributions and polynomials defined in $\BR^d$, respectively. Set
\begin{equation}\label{homo:0}\begin{aligned}
{\wh B}{}^{s+1}_{q,r}(\BR^d) &= \{f \in \CS'(\BR^d) \mid \nabla f \in B^{s}_{q,r}(\BR^d)^d\}, \\
{\wh {\dot B}}{}^{s+1}_{q,r}(\BR^d) & = \{f \in \CS'(\BR^d)\setminus \CP(\BR^d) \mid
\nabla f \in \dot B^s_{q,r}(\BR^d)^d\}.
\end{aligned}\end{equation}
In this paper, we write $\wh{\CB}{}^{s+1}_{q,r}(\BR^d) = {\wh B}{}^{s+1}_{q,r}(\BR^d)$ 
if $\CB^s_{q,r} = B^s_{q,r}$ and $\wh{\CB}{}^{s+1}_{q,r}(\BR^d) = {\wh {\dot B}}{}^{s+1}_{q,r}(\BR^d)$ 
if $\CB^s_{q,r} = \dot B^s_{q,r}$. In order to state our main results, we first introduce  a function 
space associated with the pressure term.
\begin{dfn} Let $1 < q < \infty$ and $1 \leq r \leq \infty-$. 
\label{def-hat-space}
Let $\wh \CB^{s + 1}_{q, r}(\HS)$ be 
the set of $f \in \CD' (\HS)$ such that there exists a $g \in \wh\CB^{s+1}_{q,r}(\BR^d)$ such that
$\supp (g) \subset \overline{\HS}$ and $g \vert_{\HS} = f$.
If $f$ satisfies additionally $f \vert_{\pd \HS} = 0$, 
then $\wh \CB^{s + 1}_{q, r} (\HS)$ is replaced by $\wh \CB^{s + 1}_{q, r, 0} (\HS)$.	 
\end{dfn}
\begin{rem}\label{rem:def-hat-g} 
If $-1+1/q < s < 1/q$ and $f \in \wh \CB^{1+ s}_{q,r}(\HS)$, 
the trace $f \vert_{\pd \HS}$ is well-defined due to $1+ s >1/q$
(cf. Proposition \ref{prop-trace} in Section \ref{sec-2} below), 
and hence we see that $f \in \wh \CB^{1+s}_{q, r, 0}(\HS)$. 
Thus, the zero extension $f_0$ of $f \in \wh \CB^{1+s}_{q, r} (\HS)$
satisfies
\begin{equation}
\label{def-hat-g} 
f_0 \in \wh\CB^{s+1}_{q,r}(\BR^d), \quad  f_0 = \nabla f_0 = 0\quad\text{ for $x_d < 0$}.
\end{equation}
 In addition, the condition $-1+1/q < s < 1/q$ implies
$1-s > 1 - 1 \slash q = 1 \slash q'$, and hence the zero extension $\varphi_0$ 
of $\varphi \in \wh \CB^{1- s}_{q', r', 0}(\HS)$ satisfies
\begin{equation}
\label{def-hat-g*} 
\varphi_0 \in \wh\CB^{1-s}_{q',r'}(\BR^d), \quad  
\varphi_0 \vert_{\pd \HS} = 0, \quad 
\varphi_0 = \nabla \varphi_0 = 0 \text{ for $x_d < 0$},
\end{equation} 
where we have set $q' = q \slash (q-1)$ and $r'=r \slash (r-1)$ with the conventions that
$1'=\infty-$ and $(\infty-)' = 1$.
\end{rem}
Next, we define a solenoidal vector field. 
\begin{dfn}
\label{def-solenoidal}
Let $1 < q < \infty$, $1 \leq r \leq \infty-$,  and $-1+1/q < s < 1/q$.  For $\ba \in \CB^s_{q, r}(\HS)^d$, 
we say that $\ba$ is solenoidal if $\ba$ satisfies
\begin{equation}\label{solenoidal}
(\ba, \nabla\varphi) = 0 \quad\text{for all $\varphi \in \wh \CB^{1-s}_{q', r', 0}(\HS)$}.
\end{equation}
The set of all {\color{black} $\ba \in \CB^s_{q, r}(\HS)^d$ satisfying \eqref{solenoidal} 
is denoted by $\CJ^s_{q, r} (\HS)$}. In particular, we write 
$\CJ^s_{q, r} (\HS) = J^s_{q, r} (\HS)$ if $\CB^s_{q, r} = B^s_{q, r}$ 
and $\CJ^s_{q, r} (\HS) = \dot J^s_{q, r} (\HS)$ if $\CB^s_{q, r} = \dot B^s_{q, r}$.
\end{dfn}
\begin{rem} \label{rem:solenoidal}
In Theorem \ref{thm-solenoidal-characterization} below, we will prove that 
for $\ba \in \CB^s_{q,1}(\HS)^d$, there holds $\ba \in \CJ^s_{q,1}(\HS)$ 
if and only if $\dv \ba = 0$ in the sense of distributions. 
\end{rem}
The local well-posedness of \eqref{eq-fixed} reads as follows.
\begin{thm}
\label{th-local-well-posedness-fixed}
Let {\color{black} $d-1 < q \leq d$} and let  
$s \in \BR$ satisfy 
\begin{equation}
\label{cond-qs-local1} 
-1+\dfrac{d}{q} < s  < \dfrac1q.
\end{equation}
Let $\bb \in J^s_{q, 1} (\HS)$ 
\textcolor{red}{be initial data for Problem \eqref{eq-fixed}}.  Then,  
there exist $T > 0$ and $\sigma > 0$ such that for 
any $\ba \in J^s_{q,1}(\HS)$ satisfying $\|\ba-\bb\|_{B^s_{q,1}(\HS)} 
< \sigma$, 
Problem \eqref{eq-fixed} \textcolor{red}{with initial data not only $\bb$ but also with $\ba$} admits a unique local strong solution $(\bu, \sfQ)$
such that
\begin{equation}\label{solution:1}\begin{aligned}
\bu & \in
W^1_1 ((0, T), B^s_{q, 1} (\HS)^d) \cap  L_1 ((0, T), B^{s + 2}_{q, 1} (\HS)^d), \\
\sfQ & \in L_1 ((0, T), B^{s + 1}_{q, 1} (\HS) + \wh B^{s + 1}_{q, 1, 0} (\HS))
\end{aligned}\end{equation}
satisfying
\begin{equation}\label{solution:2}
\lVert (\pd_t \bu, \nabla \sfQ) \rVert_{L_1 ((0, T), B^s_{q, 1} (\HS))}
+ \lVert \bu \rVert_{L_1 ((0, T), B^{s + 2}_{q, 1} (\HS))}
+ \sup_{t \in [0, T)} \lVert \bu (\,\cdot\, , t) 
\rVert_{B^s_{q, 1} (\HS)}
\le M\|\bb\|_{B^s_{q,1}(\HS)}
\end{equation}
for some constant $M>0$ independent of $\bb$ and  $T$.
\end{thm}
If we assume the smallness of the initial data, then the 
restriction on $s$ may be relaxed. In this case,
the smallness condition on $T$ is replaced by the smallness condition on the initial data.
\begin{thm}
\label{thm:lw.2} 
Let {\color{black}$d-1 < q < 2d$} and let  
$s \in \BR$ satisfy 
\begin{equation}
\label{cond-qs-local2} 
-1+\dfrac{d}{q} \leq s < \dfrac1q.
\end{equation}
Then, for every $T > 0$, there exist small constants
$c_0 > 0$ and $\omega>0$ depending on $T$ such that 
for every $\ba \in J^s_{q,1}(\HS)$ with $\|\ba\|_{B^s_{q,1}(\HS)} \le c_0$, 
Problem \eqref{eq-fixed} admits a unique local 
strong solution $(\bu, \sfQ)$ such that 
\begin{equation}\label{solution:3}\begin{aligned}
\bu & \in 
W^1_1 ((0, T), B^s_{q, 1} (\HS)^d) \cap  L_1 ((0, T), B^{s + 2}_{q, 1} (\HS)^d), \\
\sfQ & \in L_1 ((0, T), B^{s + 1}_{q, 1} (\HS) + \wh B^{s + 1}_{q, 1, 0} (\HS))
\end{aligned}\end{equation}
satisfying
\begin{equation}\label{solution:4}
\lVert (\pd_t \bu, \nabla \sfQ) \rVert_{L_1 ((0, T), B^s_{q, 1} (\HS))}
+ \lVert \bu \rVert_{L_1 ((0, T), B^{s + 2}_{q, 1} (\HS))} 
+ \sup_{t \in [0, T)} \lVert \bu (\,\cdot\, , t) 
\rVert_{B^s_{q, 1} (\HS)} \le \omega.
\end{equation}
\end{thm}
The  global well-posedness of  \eqref{eq-fixed} for small 
initial data reads as follows.
\begin{thm}
\label{th-global-well-posedness-fixed}
Let $d-1 < q < 2 d$.
Then there exists a small constant
$c_0 > 0$ such that 
for every $\ba \in \dot J^{- 1 + d \slash q}_{q,1}(\HS)$ with 
$\|\ba\|_{\dot B^{- 1 + d \slash q}_{q,1}(\HS)} \le c_0$, 
Problem \eqref{eq-fixed} admits a unique global 
strong solution $(\bu, \sfQ)$ such that 
\begin{equation}\label{solution:3*}
\pd_t \bu \in L_1 (\BR_+, \dot B^{- 1 + d \slash q}_{q, 1} (\HS)^d), \quad
\nabla^2 \bu \in L_1 (\BR_+, \dot B^{- 1 + d \slash q}_{q, 1} (\HS)^{d^3}), \quad
\nabla \sfQ \in L_1 (\BR_+, \dot B^{- 1 + d \slash q}_{q, 1} (\HS)^d)
\end{equation}
satisfying
\begin{equation}\label{solution:4*} \begin{aligned}
\lVert (\pd_t \bu, \nabla^2 \bu, \nabla \sfQ) 
\rVert_{L_1 (\BR_+, \dot B^{- 1 + d \slash q}_{q, 1} (\HS))} 
+ \sup_{t \in [0, \infty)} \lVert \bu (\,\cdot\, , t) 
\rVert_{\dot B^{- 1 + d \slash q}_{q, 1} (\HS)}
\le M \|\ba\|_{\dot B^{- 1 + d \slash q}_{q,1}(\HS)} 
\end{aligned}\end{equation}	
for some $M>0$ independent of $\ba$. Here and in the following, $\BR_+ = (0, \infty)$.
\end{thm}

Since there holds
\begin{equation}
\label{emb-BC1}
\left\{\begin{aligned}
B^{s + 2}_{q, 1} (\HS) \hookrightarrow \mathrm{BC}^1 (\overline{\HS}),
& \qquad s \ge - 1 + d \slash q, \\
\dot B^{1 + d \slash q}_{q, 1} (\HS) \hookrightarrow \mathrm{BC}^1 (\overline{\HS}), 
\end{aligned}\right.
\end{equation}
the classical Picard-Lindel\"of theorem implies that for $0 < T \le \infty$
there exists a unique $C^1$-flow $\bX_{\bu} (\,\cdot\, , t)$ satisfying
\begin{equation}
\label{representation-Xu}
\bX_{\bu} (y, t) := y + \int_0^t \bu (y, \tau) \d \tau, \qquad y \in \overline{\HS}, \; 0 \le t < T,
\end{equation}
which implies that $\bX_{\bu} (\,\cdot\, , t)$ is a $C^1$-diffeomorphism from 
$\HS$ onto $\Omega (t)$ for each $t \in [0, T)$, $0 < T \le \infty$, and measure preserving. 
Then the local well-posedness of Problem 
\eqref{eq-original} as a corollary of Theorems \ref{th-local-well-posedness-fixed}
and \ref{thm:lw.2}, whereas 
we obtain  the global well-posedness of  Problem 
\eqref{eq-original} as a corollary of Theorem \ref{th-global-well-posedness-fixed}.
\begin{cor}
\label{cor-local-well-posedness-original}
Let {\color{black}$d-1 < q  \leq d$} and $-1+d/q < s < 1/q$.
Let $\bb \in J^s_{q, 1} (\HS)$ 
\textcolor{red}{be initial data for Problem \eqref{eq-original}}.  Then,  
there exist $T > 0$ and $\sigma > 0$ such that for 
any $\ba \in J^s_{q,1}(\HS)$ satisfying $\|\ba-\bb\|_{B^s_{q,1}(\HS)} 
< \sigma$, 
Problem \eqref{eq-original} \textcolor{red}{with initial data not only $\bb$ but also with $\ba$}
admits a unique local strong solution 
$(\bv, P)$ such that
\begin{align}
\partial_t\bv(\,\cdot\, , t)
\in B^s_{q, 1} (\Omega (t))^d, \quad
\bv(\,\cdot\, , t) \in B^{s+2}_{q, 1} (\Omega (t))^d, \quad 
\nabla P(\,\cdot\, , t) \in B^s_{q, 1} (\Omega (t))^d
\end{align} 
for all $t \in (0, T)$ satisfying
\begin{align}
\int^T_0 \|(\pd_t \bv, \nabla P)(\,\cdot\, , t)\|_{B^s_{q, 1} (\Omega (t))}\d t 
+ \int^T_0 \|\bv(\,\cdot\, , t)\|_{B^{s + 2}_{q, 1} (\Omega (t))}\d t 
+ \sup_{0 \le t < T} \lVert \bv (\,\cdot\, , t) \rVert_{B^s_{q, 1} (\Omega (t))}
\le M\|\bb\|_{B^s_{q,1}(\HS)},	
\end{align}
where $M$ is independent of $T$.
\end{cor}

\begin{cor}
\label{cor:lw.2}
Let {\color{black}$d-1<q<2d$} and let $s\in\BR$ satisfy $- 1+ d\slash q \le s < 1 \slash q$.
Then, for every $T > 0$, there exist small constants
$c_0 > 0$ and $M>0$ depending on $T$ such that 
for every $\ba \in J^s_{q,1}(\HS)$ with $\|\ba\|_{B^s_{q,1}(\HS)} \le c_0$,
Problem \eqref{eq-original} admits a unique local strong solution 
$(\bv, P)$ such that
\begin{align}
\partial_t\bv(\,\cdot\, , t)
\in B^s_{q, 1} (\Omega (t))^d, \quad
\bv(\,\cdot\, , t) \in B^{s+2}_{q, 1} (\Omega (t))^d, \quad 
{\color{black} \nabla P(\,\cdot\, , t) \in B^s_{q, 1} (\Omega (t))^d}	
\end{align}
for all $t \in (0, T)$ satisfying
\begin{equation}
\int^T_0 \|(\pd_t \bv, \nabla P)(\,\cdot\, , t)\|_{B^s_{q, 1} (\Omega (t))}\d t 
+ \int^T_0 \|\bv(\,\cdot\, , t)\|_{B^{s + 2}_{q, 1} (\Omega (t))}\d t 
+ \sup_{0 \le t < T} \lVert \bv (\,\cdot\, , t) \rVert_{B^s_{q, 1} (\Omega (t))}
\le M\|\ba\|_{B^s_{q,1}(\HS)}.
\end{equation}
\end{cor}

\begin{cor}
\label{cor-global-well-posedness-original}
Let $d-1<q<2d$. Then there exists a small constant $c_0 > 0$ such that 
for every $\ba \in \dot J^{- 1 + d \slash q}_{q,1}(\HS)$ with 
$\|\ba\|_{\dot B^{- 1 + d \slash q}_{q,1}(\HS)} \le c_0$, 
Problem \eqref{eq-original} admits a unique global strong solution 
$(\bv, P)$ such that
\begin{align}
\partial_t\bv(\,\cdot\, , t)
\in \dot B^{- 1 + d \slash q}_{q, 1} (\Omega (t))^d, \quad
\nabla^2 \bv(\,\cdot\, , t) \in \dot B^{- 1 + d \slash q}_{q, 1} (\Omega (t))^{d^3}, \quad 
\nabla P(\,\cdot\, , t) \in \dot B^{- 1 + d \slash q}_{q, 1} (\Omega (t))^d
\end{align} 
for all $t \in \BR_+$ satisfying
\begin{align}
\int^\infty_0 \|(\pd_t \bv, \nabla^2 \bv, \nabla P)(\,\cdot\, , t)
\|_{\dot B^{- 1 + d \slash q}_{q, 1} (\Omega (t))}\d t 
+ \sup_{t \in [0, \infty)} \lVert \bv (\,\cdot\, , t) \rVert_{\dot B^{- 1 + d \slash q}_{q, 1} (\Omega (t))}
\le M\|\ba\|_{B^s_{q,1}(\HS)}	
\end{align}
with some $M > 0$ independent of $\ba$.
\end{cor}
\begin{rem}
Let us make some comments on our main results. Theorem \ref{th-global-well-posedness-fixed} 
was first proved by Danchin et al. \cite{DHMTpre} provided that $q$ satisfies $d < q < 2d$
with $d \ge 3$.
{\color{magenta}
Recently, Ogawa and Shimizu \cite{OS24} slightly improved the range of $q$ such that $d-1<q<2d-1$,
where the case $d=2$ is also allowed. Note that the upper range of $q < 2d - 1$ was
caused by the worst nonlinear estimate on the boundary (cf. \cite[Lem.~8.4]{OS24}). 
However, in our approach, it is not necessary to consider such estimates so that
we only impose the minimum assumption on $q$ in Theorem \ref{th-global-well-posedness-fixed}.
In fact, the lower condition $d-1< q$ ensures the density of $C^\infty_0 (\HS)$ in
$\CB^s_{q,1} (\HS)$, whereas the upper condition $q<2d$ ensures the nonlinear estimate
\begin{equation}
\|uv\|_{\CB^s_{q, 1}(\HS)} \le C\|u\|_{\CB^s_{q,1}(\HS)}
\|v\|_{\CB^{d/q}_{q,1}(\HS)}
\end{equation}
with $-1+d/q\le s< 1/q$, see \eqref{est-prop:8.2} below. Namely, the condition on 
$q \in (d-1,2d)$ in Theorem \ref{th-global-well-posedness-fixed} seems to be optimal
and improves the previous contributions due to \cite{DHMTpre} and \cite{OS24}.}
Notice that, in \cite{DHMTpre,OS24}, the local well-posedness
result was not proved, and actually, to show its result requires some technical restriction 
as we stated in Theorems \ref{th-local-well-posedness-fixed} and \ref{thm:lw.2}.
In contrast to the approach of \cite{DHMTpre,OS24}, as we explained before,
our approach, a generalized semigroup theory, may be applied to systems of 
hyperbolic-parabolic mixed equations with inhomogeneous boundary conditions. 
Indeed, Kuo and the first author \cite{KSpreprint} (cf. Kuo \cite{Kuopre})
have established the maximal $L_1$-regularity 
theory for the compressible Navier--Stokes equations in the case of the half-space. 
In the forthcoming paper, we will address the maximal $L_1$-regularity issue 
for the free boundary problem of the Navier--Stokes equations \textit{with} 
surface tension as well.
\end{rem}
\subsection{Results for the Stokes system with free boundary conditions}
We will prove Theorems \ref{th-local-well-posedness-fixed} -- \ref{th-global-well-posedness-fixed}
by the standard fixed point argument in a suitable solution space together 
with the maximal $L_1$-regularity estimate for the following Stokes system 
in the half-space $\HS$:
\begin{align}
\label{eq-Stokes}
\left\{\begin{aligned}
\pd_t \bV - \DV (\mu \BD(\bV) - \Pi \BI) & = \bF & \quad & \text{in $\HS \times \BR_+$}, \\
\dv \bV & = G_\mathrm{div}& \quad & \text{in $\HS \times \BR_+$}, \\
(\mu \BD(\bV) - \Pi \BI) \bn_0 & = \bH \vert_{\pd \HS} & \quad & \text{on $\pd \HS \times \BR_+$}, \\
\bV \vert_{t = 0} & = \bV_0 & \quad & \text{in $\HS$}.
\end{aligned}\right.
\end{align}
Here, $\bH = (H_1, \ldots, H_d)$ 
is a vector-valued function defined on $\HS$. \par
One of the main contributions of this paper is to prove the
maximal $L_1$-regularity estimates for System \eqref{eq-Stokes}.
Precisely speaking, the following theorems are crucial in our study.
\begin{thm}
\label{th-MR-inhomogeneous}
Let $d \ge 2$, $1 < q < \infty$, $- 1 + 1 \slash q < s < 1 \slash q$, and $\gamma > 0$.
Let $\bV_0 \in B^s_{q, 1} (\HS)^d$ and $\bF$, $G_\mathrm{div}$,
and $\bH$ satisfy the following regularity conditions:
\begin{align}
e^{- \gamma t} \bF & \in L_1 (\BR, B^s_{q, 1} (\HS)^d), \\
e^{- \gamma t} G_\mathrm{div} & \in  L_1 (\BR, B^{s + 1}_{q, 1} (\HS)), \\
e^{- \gamma t} \bH & \in W^{1/2}_1 (\BR, B^{s}_{q, 1} (\HS)^d)
\cap L_1 (\BR, B^{s + 1}_{q, 1} (\HS)^d).
\end{align}
Assume that there exists a function $\bG$ satisfying
$e^{- \gamma t} \bG  \in W^1_1 (\BR, B^s_{q, 1} (\HS)^d)$
such that 
$G_\mathrm{div}= \dv\bG$ holds in the sense of distributions
and that the compatibility condition $\bG\vert_{t = 0} - \bV_0
\in J^s_{q,1}(\HS)$ is satisfied.  Then, 
Problem \eqref{eq-Stokes} admits a unique solution $(\bV, \Pi)$ with
\begin{align}
e^{- \gamma t} \bV & \in W^1_1 (\BR_+, B^s_{q, 1} (\HS)^d) 
\cap L_1 (\BR_+, B^{s + 2}_{q, 1} (\HS)^d), \\
e^{- \gamma t} \Pi & \in L_1 (\BR_+, B^{s + 1}_{q, 1} (\HS) 
+  \wh B^{s + 1}_{q, 1, 0} (\HS))
\end{align}
possessing the estimate
\begin{align}
& \lVert e^{-\gamma t}(\pd_t \bV, \nabla \Pi) 
\rVert_{L_{1} (\BR_+, B^s_{q, 1} (\HS))} 
+ \lVert e^{-\gamma t}\bV \rVert_{L_{1} (\BR_+, B^{s + 2}_{q, 1} (\HS))}\\ 
&\quad \le C \Big(\lVert \bV_0 \rVert_{B^s_{q, 1} (\HS)}
+ \lVert e^{-\gamma t}(\bF, \nabla G_\mathrm{div}, \pd_t \bG, \nabla \bH) 
\rVert_{L_{1} (\BR, B^s_{q, 1} (\HS))} 
+ \lVert e^{-\gamma t}\bH 
\rVert_{W^{1/2}_{1} (\BR, B^s_{q, 1} (\HS))} \Big).
\end{align}
Here, for a Banach space $X$ and $I \in \{\BR, \BR_+\}$, we write 
\begin{gather*}
\lVert e^{-\gamma t}f \rVert_{L_{1} (I, X)} 
= \int_I e^{- \gamma t}\|f(\,\cdot\, ,t)\| _X \d t, \\
\|f\|_{W^{1/2}_{1}(\BR, X)}
= \lVert f \rVert_{L_{1} (\BR, X)} + \int_\BR \lvert h \rvert^{- 1 \slash 2}
\lVert f (\,\cdot\, , \,\cdot + h ) - f (\,\cdot\, , \,\cdot\,) 
\rVert_{L_{1} (\BR, X)} \d h.
\end{gather*}
\end{thm}
\begin{thm}
\label{th-MR-homogeneous}
Let $d \ge 2$, $1 < q < \infty$, and $- 1 + 1 \slash q < s < 1 \slash q$.
Let $\bV_0 \in \dot B^s_{q, 1} (\HS)^d$ and $\bF$, $G_\mathrm{div}$,
and $\bH$ satisfy the following regularity conditions:
\begin{alignat}4
\bF & \in L_1 (\BR, \dot B^s_{q, 1} (\HS)^d), & \quad
\nabla G_\mathrm{div} & \in  L_1 (\BR, \dot B^s_{q, 1} (\HS)^d), \\
\bH & \in \dot W^{1/2}_1 (\BR, \dot B^{s}_{q, 1} (\HS)^d) & \quad
\nabla \bH & \in L_1 (\BR, \dot B^s_{q, 1} (\HS)^{d^2}).	
\end{alignat}
Assume that there exists a function $\bG$ satisfying
$\pd_t \bG \in L_1 (\BR, \dot B^s_{q, 1} (\HS)^d)$
such that 
$G_\mathrm{div}= \dv\bG$ holds in the sense of distributions
and that the compatibility condition $\bG\vert_{t = 0} - \bV_0
\in \dot J^s_{q,1}(\HS)$ is satisfied.  Then, 
Problem \eqref{eq-Stokes} admits a unique solution $(\bV, \Pi)$ with
\begin{equation}
\pd_t \bV \in L_1 (\BR_+, \dot B^s_{q, 1} (\HS)^d), \quad
\nabla^2 \bV \in L_1 (\BR_+, \dot B^{s}_{q, 1} (\HS)^{d^3}), \quad
\nabla \Pi \in L_1 (\BR_+, \dot B^s_{q, 1} (\HS)^d)	
\end{equation}
possessing the estimate
\begin{align}
& \lVert(\pd_t \bV, \nabla^2 \bV, \nabla \Pi) 
\rVert_{L_{1} (\BR_+, \dot B^s_{q, 1} (\HS))} \\ 
&\quad \le C \Big(\lVert \bV_0 \rVert_{\dot B^s_{q, 1} (\HS)}
+ \lVert (\bF, \nabla G_\mathrm{div}, \pd_t \bG, \nabla \bH) 
\rVert_{L_1 (\BR, \dot B^s_{q, 1} (\HS))} 
+ \lVert \bH 
\rVert_{\dot W^{1/2}_1 (\BR, \dot B^s_{q, 1} (\HS))} \Big).	
\end{align}
Here, $\nabla f = \{\pd^\alpha_x f \mid \alpha \in \BN_0^d, \, |\alpha|=1\}$ 
and $\nabla^2f = \{\pd^\alpha_x f \mid \alpha \in \BN_0^d, \, |\alpha|=2\}$, and 
for a Banach space $X$ the seminorm 
of $\dot W^{1 \slash 2}_1 (I, X)$ is defined by
\begin{align}
\|f\|_{\dot W^{1/2}_1 (\BR, X)}
& := \int_\BR \lvert h \rvert^{- 1 \slash 2}
\lVert f (\,\cdot\, , \,\cdot + h ) - f (\,\cdot\, , \,\cdot\,) 
\rVert_{L_1 (\BR, X)} \d h.
\end{align}	
\end{thm}
\subsection{Generalized resolvent problem for the Stokes systems with free boundary conditions}
To prove Theorems \ref{th-MR-inhomogeneous} and \ref{th-MR-homogeneous}, we use solution operators
$\CS(\lambda)$ and $\CP(\lambda)$ of the generalized resolvent problem for the Stokes systems
with free boundary conditions, which
reads as 
\begin{equation}\label{resol:1.1} \left\{\begin{aligned}
\lambda \bu - \DV(\mu\BD(\bu) - \fq \BI) &= \bff&\quad&\text{in $\HS$}, \\
\dv \bu &=0&\quad&\text{in $\HS$}, \\
(\mu\BD(\bu) - \fq \BI)\bn_0 &=\bh|_{\pd\HS} &\quad&\text{on $\pd\HS$}
\end{aligned}\right.\end{equation}
subject to the assumptions that
\begin{equation}\label{resolassump}
\dv\bff=0 \quad\text{in  $\HS$} \quad \text{and} \quad
\langle \bh, \bn_0 \rangle = 0.
\end{equation}
Here and in the following, $\langle\ba, \bb\rangle = \sum_{j=1}^d a_jb_j$ 
for two $d$-vectors $\ba=(a_1, \ldots, a_d)$ and $\bb=(b_1, \ldots, b_d)$.
For $\Omega \in \{\BR^d, \HS\}$, we set 
$$\sB^{s+m}_{q,r} (\Omega) = \bigcap_{\ell=0}^m \CB^{s+\ell}_{q,r} (\Omega).$$ 
In particular, if $\CB^{s+\ell}_{q,r} (\Omega) = \dot B^{s+\ell}_{q,r} (\Omega)$ for all 
$\ell = 0, \ldots, m$, we write $\sB^{s+m}_{q,r} (\Omega) = \dot \sB^{s+m}_{q,r} (\Omega)$,
Clearly, if $\CB^{s+\ell}_{q,r} (\Omega) = B^{s+\ell}_{q,r} (\Omega)$ for all 
$\ell = 0, \ldots, m$, there holds $\sB^{s+m}_{q,r} (\Omega) = B^{s+m}_{q,r} (\Omega)$.
On the other hand, since $\dot B^{s+\ell}_{q,r}(\Omega) \supset \dot B^{s+m}_{q,r}(\Omega)
\cap \dot B^s_{q,r}(\Omega)$, we have $\dot\sB^{s+m}_{q,r}(\Omega) = \dot B^s_{q,r}(\Omega)
\cap \dot B^{s+m}_{q,r}(\Omega)$. \par
For the right members $(\bff, \bh)$ in Problem \eqref{resol:1.1}, 
we introduce the spaces
\begin{equation}\label{defD}\begin{aligned}
D^s_{q,r} (\HS) & = \{(\bff,  \bh') \mid \bff \in J^s_{q,r}(\HS), \quad 
\bh' = (h_1, \ldots, h_{d-1}) \in B^{s+1}_{q,r}(\HS)^{d-1} \}, \\
\dot D^s_{q,r} (\HS) &=\{(\bff, \bh') \mid \bff \in \dot J^s_{q,r}(\HS), \quad 
\bh'=(h_1, \ldots, h_{d-1}) \in \dot \sB^{s+1}_{q,r}(\HS)^{d-1}\}.
\end{aligned}\end{equation}
To describe the result for \eqref{resol:1.1} in an orderly way, 
we agree on the following convention on the symbols $\CB^s_{q,r}$, $\CH^\mu_{q,r}$, $\CJ^s_{q,r}$,
${\wh \CB}{}^{s+1}_{q,r}$, ${\wh \CB}{}^{s+1}_{q,r,0}$, and $\CD^s_{q,r}$ as well as
the symbols $\nabla_b$ and $\nabla_b^2$.
\begin{conv}
\label{convention}
Given $\gamma > 0$ let $\gamma_b$ be either $\gamma_b = \gamma$ or $\gamma_b = 0$. 
\begin{enumerate}
\item If $\gamma_b = \gamma$, we write
\begin{equation}\label{symbol:01}\CB^s_{q,r} = B^s_{q,r}, \quad
\CH^s_q= H^s_q, \quad 
\CJ^s_{q,r} = J^s_{q,r},\quad {\wh \CB}{}^{s+1}_{q,r} = {\wh B}{}^{s+1}_{q,r},
\quad  {\wh\CB}{}^{s+1}_{q,r,0} = {\wh B}{}^{s+1}_{q,r,0}, \quad 
\CD^s_{q,r} = D^s_{q,r}.
\end{equation}
If $\gamma_b=0$, we write
\begin{equation}\label{symbol:02}
\CB^s_{q,r}=\dot B^s_{q,r}, \quad \CH^s_q = \dot H^s_q, \quad  \CJ^s_{q,r} = \dot J^s_{q,r}, \quad  
{\wh \CB}{}^{s+1}_{q,r} = {\wh{\dot  B}}{}^{s+1}_{q,r}, \quad 
{\wh\CB}{}^{s+1}_{q,r,0} = {\wh {\dot B}}{}^{s+1}_{q,r}, \quad 
\CD^s_{q,r} = \dot D^s_{q,r}.
\end{equation}
Here, $H^s_q$ and $\dot H^s_q$ denote the inhomogeneous and 
homogeneous Sobolev spaces, respectively, where 
the definitions will be given in Section \ref{sec-2} below. 
\item If $\gamma_b = \gamma$, the symbols $\nabla_b$ and $\nabla_b^2$ stand for
$\nabla_bf = \{\pd^\alpha_x f\mid \alpha \in \BN_0^d, \, |\alpha| \leq 1\}$ and 
$\nabla^2_bf = \{\pd^\alpha_x f\mid \alpha \in \BN_0^d, \, |\alpha| \leq 2\}$, respectively.
If $\gamma_b = 0$, the symbols $\nabla_b$ and $\nabla_b^2$ stand for
$\nabla_bf = \{\pd^\alpha_x f \mid \alpha \in \BN_0^d, \, |\alpha| = 1\}$ and 
$\nabla^2_bf = \{\pd^\alpha_x f \mid \alpha \in \BN_0^d, \, |\alpha| = 2\}$, respectively.
\end{enumerate}
\end{conv}
To introduce the solution operators for equations \eqref{resol:1.1}, let $\bPhi_1
\in \CB^s_{q,r}(\HS)^d$, $\bPhi_2' \in \CB^s_{q,r}(\HS)^{d-1}$, $\bPhi_3' 
\in \CB^s_{q,r}(\HS)^{d(d-1)}$ and 
$\bPhi = (\bPhi_1, \bPhi_2', \bPhi_3') \in \CB^s_{q,r}(\HS)^{M_d}$ with $M_d = d -1 + d^2$.
Here, $\bPhi_1$, $\bPhi_2'$, and $\bPhi_3'$ are the corresponding
variables to $\bff$, $\lambda^{1/2}\bh'$, and $\nabla \bh'$, respectively. \par 
For equations \eqref{resol:1.1}, we shall prove the following theorem.
\begin{thm}\label{th-sth} Let $1 < q < \infty$, 
$-1+1/q < s < 1/q$, $\epsilon
\in (0, \pi/2)$, and $\gamma > 0$. Moreover, we assume that $1 \leq r \leq \infty-$ in the assertions
$(1) - (3)$ and that $1 \leq r \leq \infty$ 
in assertion $(4)$ below. 
Let
\begin{equation}\label{sect:1.1}
\Sigma_\epsilon = \{\lambda \in \BC\setminus\{0\} \mid |\arg\lambda| \leq \pi-\epsilon\}.
\end{equation}
Then, there exist two operators $\CS(\lambda)$ and  $\CP(\lambda)$  such that 
the following four assertions hold:
\begin{itemize}
\item[$(1)$]~ $\CS(\lambda) \in \Hol(\Sigma_\epsilon + \gamma_b, 
\CL(\CB^s_{q,r}(\HS)^{M_d}, \sB^{s+2}_{q,r}(\HS)^d))$, \\
\phantom{}\, $\CP(\lambda) \in \Hol(\Sigma_\epsilon + \gamma_b, 
\CL(\CB^s_{q,r}(\HS)^{M_d}, \sB^{s+1}_{q,r}(\HS)))$. 
\item[$(2)$]  
For every $\lambda \in \Sigma_\epsilon + \gamma_b$ and $(\bff, \bh') \in \CD^s_{q,r}(\HS)$,  
Problem \eqref{resol:1.1} admits solutions $\bu \in \sB^{s+2}_{q,r}(\HS)^d$ and $\fq
\in \sB^{s+1}_{q,r}(\HS)$ which are represented by 
$\bu = \CS(\lambda)\bsF$ and $\fq = \CP(\lambda)\bsF$ with
$\bsF = (\bff, \lambda^{1/2}\bh',  \nabla\bh')$. 
\item[$(3)$]  For any $\lambda\in \Sigma_\epsilon + \gamma_b$ and $\bPhi \in 
\CB^s_{q,r}(\HS)^{M_d}$, there holds
$$\|(\lambda, \lambda^{1/2}\nabla, \nabla_b^2)\CS(\lambda)\bPhi\|_{\CB^s_{q,r}(\HS)}
+\|(\lambda^{1/2}, \nabla_b)\CP(\lambda)\bPhi\|_{\CB^{s}_{q,r}(\HS)} 
\leq C\|\bPhi\|_{\CB^s_{q,r}(\HS)}.
$$
\item[$(4)$] Let $\sigma>0$ be a small number such that $-1+1/q < s-\sigma < s < s+\sigma
< 1/q$.  Then, for every  $\lambda \in \Sigma_\epsilon + \gamma_b$
and $\bPhi\in C^\infty_0(\HS)^{M_d}$,  there hold
\begin{align}
\|(\lambda, \lambda^{1/2}\nabla, \nabla_b^2)\CS(\lambda)\bPhi\|_{\CB^s_{q,r}(\HS)}
+\|(\lambda^{1/2}, \nabla_b)\CP(\lambda)\bPhi\|_{\CB^{s}_{q,r}(\HS)} 
&\leq C|\lambda|^{-\frac{\sigma}{2}}\|\bPhi\|_{\CB^{s+\sigma}_{q,r}(\HS)}, 
\label{(4)} \\
\|(\lambda, \lambda^{1/2}\nabla, \nabla_b^2)\pd_\lambda\CS(\lambda)\bPhi\|_{\CB^s_{q,r}(\HS)}
+\|(\lambda^{1/2}, \nabla_b)\pd_\lambda \CP(\lambda)\bPhi\|_{\CB^{s}_{q,r}(\HS)} 
&\leq C|\lambda|^{-(1-\frac{\sigma}{2})}\|\bPhi\|_{\CB^{s-\sigma}_{q,r}(\HS)}.
\label{(5)}
\end{align}
\end{itemize}
Here, 
$\Hol(U, X)$ denotes the set of all $X$ valued holomorphic functions defined on 
$U$, and $\CL(X, Y)$ denotes the set of all bounded linear operators from
$X$ into $Y$. 
\end{thm}
\begin{rem} From Theorem \ref{th-sth} the existence of solutions to equations \eqref{resol:1.1}
follows. 
The uniqueness of solutions to \eqref{resol:1.1} will be proved as Theorem \ref{thm:unique} in Subsec. 4.4 below.
\end{rem}

\subsection{Structure of the paper}
This paper is organized as follows. In the next section, 
we recall the notation of function spaces and describe 
several preliminary results that play crucial roles 
throughout this paper. We first describe the results for 
boundary value problems for the Laplacian, which is necessary
to use in our analysis in order to handle the pressure term
appearing in the system. Using the unique existence results
for the weak Dirichlet problems, we will derive the solution 
formula to \eqref{resol:1.1} with assumption \eqref{resolassump}.
Although the solution 
formula to \eqref{resol:1.1} was derived by the first
author in his previous paper (cf. \cite[Sec. 6]{SS12}),
we give a new formula which avoids some technical difficulties.
To be precise, our new solution formula does not
contain a singular integral operator such as
\begin{equation}
\int_0^\infty \CF^{- 1}_{\xi'} [m (\lambda, \xi') 
e^{- \lvert \xi' \rvert (x_d + y_d)} 
\CF_{x'} [\mathsf h] (\xi', y_d) ] (x') \d y_d,
\end{equation}
see Section \ref{sec-solution-formula} below.  In 
Sections \ref{sec-3} and \ref{sec-4}, we will study the generalized 
resolvent problem \eqref{resol:1.1} and prove Theorem \ref{th-sth}. 
Then, Section \ref{sec-5} is devoted to the proof of maximal 
$L_1$-regularity estimates for the solutions $(\bV_1, \Pi_1)$ 
to 
{\color{magenta}
\begin{align}
\label{eq-Stokes-inhomogeneous}
\left\{\begin{aligned}
\pd_t \bV_1 - \DV (\mu \BD (\bV_1) - \Pi_1 \BI) & = \bF & \quad & \text{in $\HS \times \BR$}, \\
\dv \bV_1 & =0 & \quad & \text{in $\HS \times \BR$}, \\
(\mu \BD (\bV_1) - \Pi_1 \BI) \bn_0 & = \bH' \vert_{\pd \HS} & \quad & \text{on $\pd \HS \times \BR$}
\end{aligned}\right.			
\end{align}	 
subject to $\dv \bF=0$ and $\bH' = (H_1, \ldots, H_{d-1}, 0)$,
where Theorem \ref{th-sth} is used.} 
We will see in Section~\ref{sec-6} that the resolvent 
estimates derived in Sections \ref{sec-3} and \ref{sec-4} also
induce the generation of a $C_0$-analytic semigroup associated with 
the linearized system 
{\color{magenta}
\begin{align}
\label{eq-Stokes-homogeneous}
\left\{\begin{aligned}
\pd_t \bV_2 - \DV (\mu\BD (\bV_2) - \Pi_2 \BI) & = 0 & \quad & \text{in $\HS \times \BR_+$}, \\
\dv \bV_2 & = 0 & \quad & \text{in $\HS \times \BR_+$}, \\
(\mu\BD (\bV_2) - \Pi_2 \BI) \bn_0 & = 0 & \quad & \text{on $\pd \HS \times \BR_+$}, \\
\bV_2 \vert_{t = 0} & = \bV_0 - \bV_1 \vert_{t = 0} & \quad & \text{in $\HS$}.	
\end{aligned}\right.			
\end{align}
Here, we remark that $\bV_0 - \bV_1 \vert_{t = 0}$ satisfies
$\bV_0 - \bV_1 \vert_{t = 0} \in \CJ^s_{q,1}(\HS)$, 
i.e., the compatibility condition is satisfied.}
Then, recalling that
$\bV = \bV_1 + \bV_2$ and $\Pi = \Pi_1 + \Pi_2$, the proof of
Theorems \ref{th-MR-inhomogeneous} and \ref{th-MR-homogeneous}
is complete. Together with the contraction mapping principle, 
in Section \ref{sec-8}, we will show the unique existence of 
local and global strong solutions to \eqref{eq-fixed}, i.e., Theorems 
\ref{th-local-well-posedness-fixed}, \ref{thm:lw.2}, and 
\ref{th-global-well-posedness-fixed}. To prove these well-posedness results, 
the most difficult part here arises in the estimate of the 
boundary data $\BH(\bu)$. To overcome this difficulty, we use an interpolation
inequality, which is stated in Proposition~\ref{prop-B7} below, to estimate 
the boundary data in $W^{1\slash2}_1 (\BR, B^{s + 1}_{q, 1} (\HS))$ 
as well as in $\dot W^{1\slash2}_1 (\BR, \dot B^{d \slash q}_{q, 1} (\HS))$.
Finally, by taking advantage of the fact that
$\bX_{\bu} (\,\cdot\, , t)$ is a $C^1$-diffeomorphism 
from $\HS$ onto $\Omega (t)$ for each $t \in [0, T)$, we will show 
the well-posedness result for \eqref{eq-original}, i.e., 
Corollaries  \ref{cor-local-well-posedness-original}, \ref{cor:lw.2},
and \ref{cor-global-well-posedness-original}. 
The precise form of nonlinearities 
is recorded in Appendix \ref{sec-A} and several technical tools used in 
the proofs are collected in Appendix \ref{ap.B} for the reader's convenience.

\section{Preliminaries}\label{sec-2}
\subsection{Notation}
Let us fix the symbols used in this paper. 
As usual, $\BR$, $\BN$, and $\BC$ stand for the set of all real, natural, 
complex numbers, respectively,
while $\BZ$ stands for the set of all integers. {\color{black} In addition}, 
$\BK$ means either $\BR$ or $\BC$. 
Set  $\BN_0 := \BN \cup \{0\}$. \par
For $d \in \BN$ and a Banach space $X$, let $\CS (\BR^d, X)$ be the 
Schwartz class of $X$-valued functions on $\BR^d$, while 
$\CS' (\BR^d, X)$ be the space of $X$-valued tempered distributions, i.e., 
the set of all continuous linear mappings from $\CS (\BR^d)$ to $X$.
If $X \in \BK$, we will write $\CS (\BR^d) = \CS (\BR^d, X)$ and 
$\CS' (\BR^d) = \CS' (\BR^d, X)$ for short.\par
For $d \in \BN$, we define the Fourier transform $f \mapsto \CF [f]$ from 
$\CS (\BR^d, X)$ onto itself and its inverse $\CF^{- 1}_\xi$ as
\begin{equation}
\CF[f] (\xi) := \int_{\BR^d} e^{- i x \cdot \xi} f (x) \d x, \qquad 
\CF^{- 1}_\xi [g] (x) := \frac{1}{(2 \pi)^d} \int_{\BR^d} e^{i x \cdot \xi} g (\xi) \d \xi,
\end{equation}
respectively. Notice that $\CF$ and $\CF^{- 1}_\xi$ may be extended to 
operators on $\CS' (\BR^d, X)$ in the usual way.
For  functions $f(x', {\color{black} x_d})$ and $g(\xi', {\color{black} x_d})$ defined for $x'=(x_1, \ldots, x_{d-1})$,
$\xi'=(\xi_1, \ldots, \xi_{d-1}) \in \BR^{d-1}$, and ${\color{black} x_d} > 0$, {\color{black}let}
$\CF'[f](\xi', x_d)$ and 
$\CF^{-1}_{\xi'}[g](x', x_d)$ be the partial Fourier transform of $f$ and the partial
inverse Fourier transform of $g$, respectively, which are defined by 
\begin{align}
\CF'[f](\xi', x_d) &= \int_{\BR^{d-1}}e^{-ix'\cdot\xi'} f(x', x_d)\d x',
\\
\CF^{-1}_{\xi'}[g](x', x_d) &= \CF^{-1}_{\xi'}[g(\xi', x_d)](x')
= \frac{1}{(2\pi)^{d-1}}\int_{\BR^{d-1}} e^{ix'\cdot \xi'}g(\xi', {\color{black} x_d})\d\xi'
\end{align}
with $x'\cdot\xi' = \sum_{j=1}^{d-1}x_j\xi_j$.  
For $X$-valued functions $f$ and $g$,
the Fourier--Laplace transform $\CL[f](\lambda)$ and its inverse transform 
$\CL^{-1}[g](t)$ are defined by 
\begin{equation}
\label{def-Laplace}
\begin{split}
\CL[f](\lambda) & = \int_{-\infty}^\infty e^{-\lambda t}f(t) \d t
= \CF_t [e^{-\gamma t}f (t)](\tau), \\
\CL^{-1}[g](t) & = \frac{1}{2\pi} \int_{-\infty}^\infty e^{\lambda t}g(\lambda) \d\tau
= e^{\gamma t}\CF^{-1}_\tau [g (\gamma + i\tau)](t),
\end{split}
\end{equation}
respectively. Here, it has been assumed that $\lambda = \gamma + i \tau$
with $\gamma > 0$ and $\tau \in \BR$, and $\CF_t$ and $\CF^{-1}_\tau$ denote 
the Fourier transform with respect to $t$ and the inverse Fourier transform with
respect to $\tau$, respectively. For any $\epsilon \in (0, \pi)$ and $\gamma > 0$, we set
\begin{equation}
\Sigma_\epsilon := \{z \in \BC \setminus \{0\} \mid \lvert \arg z \rvert 
< \pi - \epsilon \}, \qquad
\Sigma_{\epsilon} + \gamma := \{z + \gamma  \mid z \in \Sigma_\epsilon\}.
\end{equation}
For a domain $\Omega$ and scalar-valued functions $f(x)$, $g(x)$,  and 
$d$-vector of functions $\bff(x)$ and $\bg(x)$ defined on $\Omega$, 
we set $(f, g)_{\Omega} 
= \int_{\Omega} f (x) g (x) \d x$ and $(\bff, \bg)_\Omega 
= \int_\Omega \langle \bff(x), \bg(x) \rangle \d x$, where
for $d$-vector functions $\bff = (f_1, \ldots, f_d)$ and $\bg = (g_1, \ldots, g_d)$
we denote $\langle \bff, \bg\rangle  =\sum_{j=1}^d f_jg_j$.
Here, we will write $(f, g) = (f, g)_{\HS}$ and 
$(\bff, \bg) = (\bff, \bg)_{\HS}$ for short if there is no confusion.
To denote various positive constants, we use the same letter
$C$, and  $C_{a,b,\ldots}$ denote the constant $C_{a,b,\ldots}$ depends
on the quantities $a$, $b$, $\ldots$ The constants $C$ and 
$C_{a,b, \ldots}$ may change from line to line whenever there is no confusion.
\par
For a Banach space $X$,  exponents 
$p, q \in [1, \infty]$, and $m \in \BN$ and $s \in \BR$, 
let $L_p(\BR^d, X)$ and $W^m_p(\BR^d, X)$ denote
the standard $X$-valued Lebesgue and Sobolev spaces on $\BR^d$. 
For a domain $G \subset \BR^d$, the spaces $L_p (G, X)$ and $W^m_p (G, X)$
denote the restriction of $L_p(\BR^d, X)$ and $W^m_p(\BR^d, X)$ on 
$G$, respectively, which are equipped with the norms
$\lVert \enskip \cdot \enskip \rVert_{L_p (G, X)}$ 
and $\lVert \enskip \cdot \enskip \rVert_{W^m_p (G, X)}$, respectively. 
Moreover, $\dot W^1_p(\BR, X) = \{f \in L_{p, {\rm loc}}(\BR, X) \mid 
\pd_t f \in L_1(\BR, X)\}$. 
The H\"older conjugate of $p$ is
denoted by $p' = p \slash (p - 1)$.
If $X \in \BK$, we will write $L_p (G) = L_p (G, X)$ and
$W^m_p (G) = W^m_p (G, X)$ for short. 
For $\gamma > 0$ and $I = \{\BR, \BR_+\}$, we write
\begin{align}
\|e^{-\gamma t}f\|_{L_1(I, X)} = \int_Ie^{-\gamma t}\|f(t)\|_X\d t.
\end{align}
The space of $X$-valued  bounded  continuous functions 
on the interval $[0, T)$ is denoted by $\mathrm{BC} ([0, T), X)$.
The space of $X$-valued differentiable bounded continuous functions on $G$ is
denoted by $\mathrm{BC}^1 (G, X)$.
In addition, $C^\infty (G, X)$ is the set of all $X$-valued
smooth functions on $G$, and $C^\infty_0(G, X)$ denotes the set of all 
$C^\infty(G, X)$ functions which are compactly supported on $G$. If $X=\BK$,
we write $C^\infty(\Omega, \BK) = C^\infty(\Omega)$, 
and $C^\infty_0(\Omega, \BK) = C^\infty_0(\Omega)$.\par
For a domain $U$ in $\BC$ and a Banach space $X$, the set of all $X$-valued 
holomorphic functions defined on $U$ is denoted by $\Hol (U, X)$. 
For two Banach spaces, $X$ and $Y$, $\CL(X, Y)$ denotes the set of all 
bounded linear operators from $X$ into $Y$, and for the simplicity, we write
$\CL(X)=\CL(X,X)$. 
For each interpolation couple $(X_0, X_1)$ of Banach spaces,
each $0 < \theta < 1$ and $1 \le p \le \infty$, 
the real interpolation space is denoted by $(X_0, X_1)_{\theta, p}$,
whereas the complex interpolation space is denoted by $[X_0, X_1]_\theta$.
In addition, the operations  $(X_0, X_1) \mapsto (X_0, X_1)_{\theta, p}$
and $(X_0, X_1) \mapsto [X_0, X_1]_\theta$
are called the \textit{real interpolation functor} for each $\theta$ and $p$
and the \textit{complex interpolation functor} for each $\theta$, respectively. 
Finally, if $X_0$ and $X_1$ are normed vector spaces with $X_0$ 
continuously embedded in $X_1$, the notation $X_0 \hookrightarrow X_1$ is used.

\subsection{Sobolev and Besov spaces}
\label{sec-2.2}
Since the Sobolev and Besov spaces
play essential roles in our analysis, 
we shall recall their definitions.  
Throughout the paper, we assume that $d \ge 2$ is an integer
describing the space dimension. 
For $q \in (1, \infty)$ and $s \in \BR$, inhomogeneous
Sobolev spaces $H^s_q (\BR^d)$ are defined as the sets of all 
$f \in \CS' (\BR^d)$ such that $\lVert f \rVert_{H^s_q (\BR^d)} < \infty$,
where we have set
\begin{equation}
\lVert f \rVert_{H^s_q (\BR^d)} := \lVert \CF^{- 1}_\xi 
[(1 + \lvert \xi \rvert^2)^{s \slash 2} \CF [f] (\xi)] \rVert_{L_q (\BR^d)}.
\end{equation}
For $q \in (1, \infty)$ and $s \in \BR$,
homogeneous Sobolev spaces $\dot H^s_q (\BR^d)$ are defined as the sets of all 
$f \in \CS' (\BR^d) \slash \CP (\BR^d)$ such that 
$\lVert f \rVert_{\dot H^s_q (\BR^d)} < \infty$,
where we have set
\begin{equation}
\lVert f \rVert_{\dot H^s_q (\BR^d)} := \lVert \CF^{- 1}_\xi 
[\lvert \xi \rvert^s \CF [f] (\xi)] \rVert_{L_q (\BR^d)}.
\end{equation}
Here, $\CP (\BR^d)$ stands for the set of all polynomials.
\par
To introduce Besov spaces, 
let $\phi \in \CS (\BR^d)$ with $\supp (\phi) = \{ \xi \in \BR^d \mid 
1 \slash 2 \le \lvert \xi \rvert \le 2 \}$ such that
$\sum_{j \in \BZ} \phi (2^{- j} \xi) = 1$ for all $\xi \in \BR^d \setminus \{0\}$
and set $\phi_0(\xi) = 1 - \sum_{j=1}^\infty \phi(2^{-j}\xi)$. 
Let $\{\Delta_j\}_{j \in \BZ}$ be the nonhomogeneous dyadic blocks defined by
\begin{align}\label{dyadic:1}
\Delta_j f := 
\begin{cases}
\CF^{- 1}_{\xi} [\phi_0(\xi) \CF[f](\xi)], & \qquad j = 0, \\
\CF^{- 1}_{\xi} [\phi (2^{- j} \xi) \CF [f](\xi)], & \qquad j \ge 1,
\end{cases}
\end{align}
while let $\{\dot \Delta_j\}_{j \in \BZ}$ be the homogeneous 
dyadic blocks defined by
\begin{equation}
\dot \Delta_j f := 
\CF^{- 1}_{\xi} [\phi (2^{- j} \xi) \CF [f](\xi)], \qquad j \in \BZ.
\end{equation}
For $1 \le q \le \infty$ and $s \in \BR$ we denote
\begin{align}
\lVert f \rVert_{B^s_{q, r} (\BR^d)} & := 
\Big\lVert 2^{j s} \lVert \Delta_j f \rVert_{L_q (\BR^d)} 
\Big\rVert_{\ell^r (\BN_0)}, \\
\lVert f \rVert_{\dot B^s_{q, r} (\BR^d)} & := 
\Big\lVert 2^{j s} \lVert \dot \Delta_j f \rVert_{L_q (\BR^d)} 
\Big\rVert_{\ell^r (\BZ)},
\end{align}
where $\ell^r$ stands for sequence spaces. 
Then inhomogeneous Besov spaces $B^s_{q, r} (\BR^d)$ are defined as the sets of all 
$f \in \CS' (\BR^d)$ such that $\lVert f \rVert_{B^s_{q, r} (\BR^d)} < \infty$,
while homogeneous Besov spaces $\dot B^s_{q, r} (\BR^d)$ are defined as the sets of all 
$f \in \CS' (\BR^d) \slash \CP (\BR^d)$ such that 
$\lVert f \rVert_{\dot B^s_{q, r} (\BR^d)} < \infty$. In particular, we define 
$B^s_{q,\infty-}(\BR^d)$ and $\dot B^s_{q, \infty-}(\BR^d)$ by 
\begin{align*}
B^s_{q,\infty-}(\BR^d) &= \bigg\{ f \in B^s_{q,\infty}(\BR^d) \,\left\vert\, \lim_{j\to\infty} 2^{js}\|\Delta_jf\|_{L_q(\BR^d)} = 0 \bigg\} \right., \\
\dot B^s_{q,\infty-}(\BR^d) &= \bigg\{ f \in\dot B^s_{q,\infty}(\BR^d) \,\left\vert\, \lim_{j\to\pm\infty}
 2^{js}\|\dot\Delta_jf\|_{L_q(\BR^d)} = 0 \bigg\}\right..
\end{align*}
In this paper, we use the following conventions: $r < \infty- < \infty$ for $r \in \BR$ and 
$1 \slash \infty- = 1 \slash \infty = 0$. In particular, the H\"older conjugate of 
$r \in \{1,\infty-\}$ is read as $1' = \infty-$ and $(\infty-)' = 1$.
\begin{rem}
It is well-known that the aforementioned definition of inhomogeneous 
Besov spaces may be extended to vector-valued cases, see \cite{A97} (cf. \cite{M73}).
To be precise, for $\theta \in (0, 1)$ and a Banach space $X$,
we define $X$-valued inhomogeneous Besov spaces $B^\theta_{1,1} (\BR, X)$ by 
the set of all $f \in \CS' (\BR, X)$ such that 
$\lVert f \rVert_{B^\theta_{1,1} (\BR, X)} < \infty$, where we have set
\begin{equation}
\lVert f \rVert_{B^\theta_{1,1} (\BR, X)} := 
\Big\lVert 2^{j \theta} \lVert \Delta_j f \rVert_{L_1 (\BR, X)} 
\Big\rVert_{\ell^1 (\BN_0)}.	
\end{equation}
According to \cite[(5.8)]{A97}, there holds 
$B^\theta_{1,1} (\BR, X) = W^\theta_1 (\BR, X)$.
Here, $W^\theta_1 (\BR, X)$ stands for the $X$-valued Sobolev-Slobodeckij space
defined by the set of all $f \in \CS' (\BR, X)$ such that 
$\lVert f \rVert_{W^\theta_1 (\BR, X)} < \infty$, where we have set
\begin{equation}
\lVert f \rVert_{W^\theta_1 (\BR, X)}
= \lVert f \rVert_{L_1 (\BR, X)} +  [f]_{W^\theta_1(\BR,X)} \quad \text{with} \quad 
 [f]_{W^\theta_1(\BR,X)}=\int_\BR \lvert h \rvert^{- (\theta + 1)} \lVert f (\,\cdot\, , 
\,\cdot + h) - f (\,\cdot\, , \,\cdot\,) \rVert_{L_1 (\BR, X)} \d h.
\end{equation}	
It is also well-known (cf. \cite[Thm. 2.5.17]{HNVW}) that 
$W^\theta_1 (\BR, X)$ may be characterized via real interpolation:
\begin{equation}\label{inhomo-int.1}
W^\theta_1 (\BR, X) = (L_1 (\BR, X), W^1_1 (\BR, X))_{\theta, 1}.
\end{equation}
Similarly, for $\theta \in (0, 1)$ and a Banach space $X$,
we define $X$-valued homogeneous Besov spaces $\dot B^\theta_{1, 1} (\BR, X)$ by 
the set of all $f \in \CS' (\BR, X) \slash \CP (\BR, X)$ such that 
$\lVert f \rVert_{\dot B^\theta_{1, 1} (\BR, X)} < \infty$, where we have set
\begin{equation}
\lVert f \rVert_{\dot B^\theta_{1, 1} (\BR, X)} =\Big\lVert 2^{j \theta} \lVert \dot \Delta_j f \rVert_{L_1 (\BR, X)} 
\Big\rVert_{\ell^1 (\BZ)}.		
\end{equation}
Here, $\CP(\BR,X)$ stands for all polynomials taking the values in $X$.
According to the proof of
\cite[Thm. 2.36]{BCD}, we see that there holds $\dot B^\theta_{1, 1} (\BR, X) 
= \dot W^\theta_1 (\BR, X)$, where 
$\dot W^\theta_1 (\BR, X)$ stands for the $X$-valued Sobolev-Slobodeckij space
defined by the set of all $f \in \CS' (\BR, X) \slash \CP (\BR, X)$ such that 
$[ f ]_{W^\theta_1 (\BR, X)} < \infty$. We also know that the space
$\dot W^\theta_1 (\BR, X)$ may be characterized via real interpolation:
\begin{equation}\label{homo-int.1}
\dot W^\theta_1 (\BR, X) = (L_1 (\BR, X), \dot W^1_1 (\BR, X))_{\theta, 1},
\end{equation}
where we refer to \cite[Thm. 2.5.17]{HNVW} for the proof. \end{rem}
To simplify the notation, we write  $\CW^\theta_1(\BR, \CB^s_{q,r}(\HS)) = W^\theta_1(\BR, B^s_{q,r}(\HS))$ 
if $\CB^s_{q,r}(\HS)= B^s_{q,r}(\HS)$
and $\CW^\theta_1(\BR, \CB^s_{q,r}(\HS)) = \dot W^\theta_1(\BR, \dot B^s_{q,r}(\HS))$ 
if $\CB^s_{q,r}(\HS)=\dot B^s_{q,r}(\HS)$ for $0 < \theta \leq 1$. 
\begin{rem}
Sobolev spaces $\dot H^s_q (\BR^d)$ and Besov spaces $\dot B^s_{q, r} (\BR^d)$ of
the homogeneous type is often defined as the set of all $f \in \CS'_h (\BR^d)$ satisfying
$\lVert f \rVert_{\dot H^s_q (\BR^d)} < \infty$ and 
$\lVert f \rVert_{\dot B^s_{q, r} (\BR^d)} < \infty$, respectively, where
the space $\CS'_h (\BR^d)$ is defined by
\begin{equation}
\CS'_h (\BR^d) := \Big\{f \in \CS' (\BR^d) \,\left\vert\enskip 
\lim_{\nu \to \infty} \lVert \CF^{- 1} [\Theta (\nu \xi) \CF[f] (\xi)] \rVert_{L_\infty (\BR^d)} 
\enskip \text{for any $\Theta \in C^\infty_0 (\BR^d)$} \Big\}\right..
\end{equation}	
However, if $(q, r, s)$ satisfies the condition $s < d \slash q$ (or $s \le d \slash q$
if $r = 1$), then this definition coincides with our definition of Sobolev and Besov spaces
of the homogeneous type described above. Indeed, the restriction on $(q, r, s)$ implies that
the spaces $\dot H^s_q (\BR^d)$ and $\dot B^s_{q, r} (\BR^d)$ (whichever we define
as subsets of $\CS' (\BR^d) \slash \CP (\BR^d)$ or $\CS'_h (\BR^d)$)
may be identified as subspaces of $\CS' (\BR^d)$. See 
\cite[Rem. 2.24]{BCD} and \cite{B88} for the details.
\end{rem}
To consider Sobolev and Besov spaces on $\HS$, we rely on the following definition.
\begin{dfn}
\label{def-Function-sp}
Let $d \ge 2$, $1 \le q,r \le \infty$, and $s \in \BR$. 
Let $\CD' (\HS)$ be continuous linear functionals on $\CD (\HS) = C^\infty_0 (\HS)$.
\begin{enumerate}  
\item $\CB^s_{q, r} (\HS)$ is the collection of all $f \in \CD' (\HS)$ such 
that there exists a function $g \in \CB^s_{q, r} (\BR^d)$
with $g \vert_{\HS} = f$. Moreover, the norm of $f \in \CB^s_{q, r} (\HS)$ is defined by
\begin{equation}
\lVert f \rVert_{\CB^s_{q, r} (\HS)} = \inf \lVert g \rVert_{\CB^s_{q, r} (\BR^d)},
\end{equation}
where the infimum is taken over all $g \in \CB^s_{q, r} (\BR^d)$ such that  $g \vert_{\HS} = f$  in $\CD' (\HS)$.
\item $\wt\CB^s_{q, r} (\HS)$ is the collection of $f \in \CD' (\HS)$ 
such that there exists a function $g$ such that
\begin{equation}
\label{def-tilda}
g \in \CB^s_{q, r} (\BR^d) \quad \text{with} \quad g \vert_{\HS} = f \quad 
\text{and} \quad \supp (g) \subset \overline{\HS}.
\end{equation}
Furthermore, the norm of $f \in \widetilde \CB^s_{q, r} (\HS)$ 
is given by
\begin{equation}
\lVert f \rVert_{\widetilde \CB^s_{q, r} (\HS)} 
= \inf \lVert g \rVert_{\CB^s_{q, r} (\BR^d)},
\end{equation}
where the infimum is taken over all $g \in \CB^s_{q, r} (\BR^d)$ with \eqref{def-tilda}.
\end{enumerate}
In addition, $\CH^s_p (\HS)$ and $\widetilde \CH^s_p (\HS)$ are defined similarly.
\end{dfn}
\begin{rem}
Defining function spaces on $\HS$ by restriction 
has the advantage that estimates on $\BR^d$ carry over 
to the spaces on $\HS$ whenever there exists a bounded 
linear extension operator that maps measurable functions on $\HS$
to measurable functions on $\BR^d$, see Section \ref{sec-8.1}.
\end{rem}
As in the whole space case, the Besov spaces $\CB^s_{q, r} (\HS)$ defined 
above are Banach spaces with the Fatou property. The following proposition is
due to \cite[Thm. 2.25]{BCD} and \cite[Thm. 4.3/1]{F86}.
\begin{prop}
\label{prop-Fatou}
Let $1 \le q, r \le \infty$ and $s \in \BN$.
In the case $\CB^s_{q, r} (\HS) = \dot B^s_{q, r} (\HS)$
assume additionally that $s < d \slash q$ if $1 < q \le \infty$ and 
$s \le d \slash q$ if $q = 1$.
For every bounded sequence $(f_j)_{j \in \BN}$ of $\CB^s_{q, r} (\HS)$
converging to some $f$ in $\CD' (\HS)$, there holds
\begin{equation}
\lVert f \rVert_{\CB^s_{q, r} (\HS)} 
\le C \liminf_{j \to \infty} \lVert f_j \rVert_{\CB^s_{q, r} (\HS)}.
\end{equation}
\end{prop}
In the following, we collect useful tools for spaces introduced in 
Definition~\ref{def-Function-sp}. 
The next proposition is well-known, see, e.g., \cite[pp.368-369]{M73},
\cite[p. 132]{M76}, and
Theorems 2.9.3 and 2.10.3 in \cite{Tbook78}.
\begin{prop}
\label{prop-density}
For $d \ge 2$, $1 < q < \infty$, $1 \le r \leq \infty-$, and $s \in \BR$, 
the following assertions are valid.
\begin{enumerate}
\item  $C^\infty_0 (\BR^d)$ is dense in $H^s_q (\BR^d)$ as well as in 
$B^s_{q, r} (\BR^d)$.
\item $C^\infty_0 (\HS)$ is dense in $\widetilde H^s_q (\HS)$ as well as in 
$\widetilde B^s_{q, r} (\HS)$. 
\item If $- 1 + 1 \slash q < s < 1 \slash q$, then it follows that $H^s_q (\HS) 
= \widetilde H^s_q (\HS)$ and $B^s_{q, r} (\HS) = \widetilde B^s_{q, r} (\HS)$.
\end{enumerate}	
\end{prop}
\begin{rem}
Proposition \ref{prop-density} for the case $r = \infty-$ 
was addressed in \cite{M76}.
\end{rem}
Concerning the above result for homogeneous-type spaces, 
we refer to Propositions 2.24, Lemma 2.32, and Corollaries 2.26 and 2.34 in \cite{Gpre}.
\begin{prop}
\label{prop-density-homogeneous}
For $d \ge 2$, $1 < q < \infty$, $1 \le r \leq \infty-$, and $s \in \BR$, 
the following assertions are valid.
\begin{enumerate}
\item If $-d\slash q' < s < d \slash q$, then $C^\infty_0 (\BR^d)$ is dense
in $\dot B^s_{q,r} (\BR^d)$. In addition, $C^\infty_0 (\BR^d)$ is dense in $\dot B^{d \slash q}_{q, 1} (\BR^d)$.
\item If $- d \slash q' < s < d \slash q$, then $C^\infty_0 (\HS)$ is dense 
in $\tdH^s_p (\HS)$.
\item If $- 1/q' < s < d \slash q$, then $C^\infty_0 (\HS)$ is dense in $\tdB^s_{q, r} (\HS)$. 
In addition, $C^\infty_0 (\BR^d_+)$ is dense in $\tdB^{d \slash q}_{q, 1} (\HS)$.
\item If $- 1/q'< s < 1 \slash q$, then it follows that $\dot H^s_q (\HS) 
= \tdH^s_q (\HS)$ and $\dot B^s_{q, r} (\HS) = \tdB^s_{q, r} (\HS)$.		
\end{enumerate}		
\end{prop}
\begin{rem}
Proposition \ref{prop-density-homogeneous} for the case $r = \infty-$ 
was not addressed in \cite{Gpre}. However, if one modifies the proof of
Corollary 3.19 in \cite{Gpre}, we see that if $1 < q < \infty$ and $- 1/q' < s < d \slash q$,  
then $C^\infty_0 (\HS)$ is dense in $\tdB^s_{q, \infty-} (\HS)$.
\end{rem}
We will prove the first assertion of Proposition \ref{prop-density-homogeneous} $(3)$
in Proposition \ref{prop-B4} in Appendix \ref{ap.B} below. \par 
The analysis in this paper strongly relies on the following results,
see Theorems 2.10.2/2, 2.10.4/1 in \cite{Tbook78}, 
Theorems 2.7.1, 2.11.2, Section 5.2.5 in \cite{Tbook83},
Theorems 8, 9, 11 in \cite{M73},
and Propositions 3.11, 3.17, 3.23 in \cite{Gpre}.
\begin{prop}
\label{prop-dual-interpolation}
For $d \ge 2$, $1 < q < \infty$, and $\Omega \in \{\BR^d, \HS\}$, 
the following assertions are valid.
\begin{enumerate}
\item For $1 \le r \le \infty$ and $- \infty < s < \infty$, it follows that
$(\CH^s_q(\BR^d))' = \CH^{-s}_{q'}(\BR^d)$, $(\wt H^s_q (\HS))' = H^{- s}_{q'} (\HS)$,
and $(\CB^s_{q, r} (\BR^d))' = \CB^{- s}_{q', r'} (\BR^d)$. 
\item For $- d \slash q' < s < d \slash q$, it follows that $(\dot H^s_q (\HS))' 
= \tdH^{- s}_{q'} (\HS)$ and $(\tdH^s_q (\HS))' = \dot H^{- s}_{q'} (\HS)$.
\item For $1 \le q_0, q_1, r, r_0, r_1 \le \infty$, $- \infty < s_0, s_1 < \infty$, 
$s_0 \ne s_1$, and $0 < \theta < 1$, it follows that 
\begin{align}
(H^{s_0}_q (\Omega), H^{s_1}_q (\Omega))_{\theta, r} 
& = B^s_{q, r} (\Omega), \\
(B^{s_0}_{q, r_0} (\Omega), B^{s_1}_{q, r_1} (\Omega))_{\theta, r} 
& = B^s_{q, r} (\Omega), \\
[H^{s_0}_{q_0} (\Omega), H^{s_1}_{q_1} (\Omega)]_\theta
& = H^s_{q} (\Omega).
\end{align}
with $s := (1 - \theta) s_0 + \theta s_1$, $q = (1-\theta)/q_0 + \theta/q_1$, 
and $1/r = (1-\theta)/r_0 + \theta/r_1$. \par 
If $s_0$ and $s_1$ 
 satisfy additionally
$s_j > - 1 + 1 \slash q_j$, $j \in \{0, 1\}$,
it follows that 
\begin{align}
(\dot H^{s_0}_q (\Omega), \dot H^{s_1}_q (\Omega))_{\theta, r} 
& = \dot B^s_{q, r} (\Omega), \\
(\dot B^{s_0}_{q, r_0} (\Omega), \dot B^{s_1}_{q, r_1} (\Omega))_{\theta, r} 
& = \dot B^s_{q, r} (\Omega), \\
[\dot H^{s_0}_{q_0} (\Omega), \dot H^{s_1}_{q_1} (\Omega)]_\theta
& = \dot H^s_{q} (\Omega)
\end{align}
with $s := (1 - \theta) s_0 + \theta s_1$, $q = (1-\theta)/q_0 + \theta/q_1$, 
and $1/r = (1-\theta)/r_0 + \theta/r_1$.
\item For $1 \le q_1 \le q_0 \le \infty$ and $1 \le r_1 \le r_0 \le \infty$,
and $s \in \BR$, there holds $\CB^s_{q_0, r_0}(\Omega) \hookleftarrow 
\CB^{s + d (\frac1{q_1}-\frac1{q_0})}_{q_1, r_1} (\Omega)$.
\end{enumerate}	
\end{prop}
\begin{rem}
Under the assumptions of Proposition \ref{prop-dual-interpolation},
it was proved in Propositions 3.17 and 3.22 in \cite{Gpre} that
the real interpolation $(\dot H^{s_0}_q (\HS), \dot H^{s_1}_q (\HS))_{\theta, r}
= \dot B^s_{q, r} (\HS)$ and $(\dot B^{s_0}_{q, r_0} (\HS), 
\dot B^{s_1}_{q, r_1} (\HS))_{\theta, r} = \dot B^s_{q, r} (\HS)$ hold 
provided that $s$ satisfies additionally $s < d \slash q$ (or $s \le d \slash q$ if $r = 1$).
This additional restriction on $s$ stems from the density of 
$\CS_0 (\BR^d) := \{f \in \CS (\BR^d) \mid 0 \notin \supp (\CF [f]) \}$ 
in $\dot B^s_{q,r} (\BR^d)$, see Proposition 2.1 in \cite{Gpre}.
However, since we define $\dot H^s_q (\BR^d)$ as well as $\dot B^s_{q,r} (\BR^d)$
as subsets of $\CS' (\BR^d) \slash \CP (\BR^d)$, we observe that
$\CS_0 (\BR^d)$ is dense in $\dot B^s_{q,r} (\BR^d)$ for all $1 \le p,q < \infty$
and $s \in \BR$, see \cite[Thm. 5.1.5]{Tbook83} for the details.
Hence, if we follow the proof of Propositions 3.17 in \cite{Gpre},
we see that the aforementioned real interpolation results hold without
any additional restriction on $s$.
In addition, for $\Omega \in \{\BR^d, \HS\}$ the complex interpolation
$[\dot H^{s_0}_{q_0} (\Omega), \dot H^{s_1}_{q_1} (\Omega)]_\theta
= \dot H^s_{q} (\Omega)$ was proved in Theorem 2.6 and Proposition 3.17 in
\cite{Gpre} with additional assumptions $s_j < d \slash q_j$, $j \in \{0,1\}$, 
but these assumptions are not required here. In fact, these additional
restrictions stem from the completeness assumption in \cite{Gpre}, but
we do not need these assumptions here since we define homogeneous
Sobolev spaces as subspaces of $\CS' (\BR^d) \slash \CP (\BR^d)$.
\end{rem}
Since we consider the boundary condition on $\pd\HS$, we need the following
proposition (cf. \cite[Thms. 4.1 and 4.3]{Gpre}) concerning the boundary trace.
\begin{prop}\label{prop-trace} Let $1 < q < \infty$, $1 \leq r \leq \infty$ and $1/q < \nu <\infty$.
Then, the boundary trace: 
\begin{equation}
\mathrm{tr}_0 \colon
\left\{\begin{aligned}
\CB^\nu_{q,r}(\HS) & \to \CB^{\nu-1/q}_{q,r}(\BR^{d-1}), \\
u &\mapsto u \vert_{\pd\HS}
\end{aligned}\right.
\end{equation}
is a bounded surjection possessing the estimate:
$$\|u|_{\pd\HS}\|_{\CB^{\nu-1/q}_{q,r}(\BR^{d-1})} \leq C\|u\|_{\CB^\nu_{q,r}(\HS)}.$$
\end{prop}
\subsection{Weak Dirichlet problem and the second Helmholtz decomposition} \label{subsec.2.3}
This subsection concerns the boundary value problems for the Laplacian, 
which will play a fundamental role in our analysis. 
We first describe the following theorem.
\begin{thm}
\label{thm-weak-Dirichlet}
Let $1 < q < \infty$, $1 \leq r \leq \infty-$, and $-1 + 1 \slash q < s < 1 \slash q$. 
Then, for any $\bff \in \CB^s_{q, r} (\HS)^d$, there exists a unique solution 
$u \in \wh \CB^{s + 1}_{q, r, 0} (\HS)$ to the weak Dirichlet problem
\begin{equation}
\label{wd:1}
(\nabla u, \nabla \varphi) = (\bff, \nabla\varphi)
\end{equation}
for every $\varphi \in \wh \CB^{1 - s}_{q', r', 0} (\HS)$.
In addition, $u$ satisfies the estimate
\begin{equation}\label{wd:2.1}
\lVert \nabla u \rVert_{\CB^s_{q, r} (\HS)} 
\le C \lVert \bff \rVert_{\CB^s_{q, r} (\HS)}.
\end{equation}
For the later use, let $Q$ be an operator defined by $Q\bff = u$. 
\end{thm}
\begin{proof}  Since $C^\infty_0(\HS)$ is dense in $\CB^s_{q,r}(\HS)$ due to Proposition 
\ref{prop-density-homogeneous}, 
we may assume $\bff = (f_1, \ldots, f_d) \in C^\infty_0(\HS)^d$ by the density argument.  
For $j = 1, \ldots, d - 1$, let $f^o_j$ be the odd extension of $f_j$, respectively,
and let $f^e_d$ be the even extension of $f_d$, that is
\begin{equation}\label{extension:1}
f^o_j(x)  = \begin{cases} f_j(x), &\quad x_d > 0, \\ 
-f_j(x', -x_d), &\quad x_d < 0, 
\end{cases} \quad 
f^e_d(x)  = \begin{cases} f_d(x), &\quad x_d > 0, \\ 
f_d(x', -x_d), &\quad x_d < 0, 
\end{cases}
\end{equation}
where $x' = (x_1, \ldots, x_{d-1})$. 
Set $\sfe [\bff] := (f^o_1, \ldots, f^o_{d-1}, f^e_d)$.
Clearly, we see that $\sfe [\bff] \in C^\infty_0 (\BR^d)$ and $\dv (\sfe [\bff]) = (\dv \bff)^o$,
where $(\dv \bff)^o$ stands for the odd extension of $\dv \bff$. 
We define $u_0$ by
\begin{equation}
\label{eq:2.1}
u_0 = -\CF^{-1}_\xi \bigg[\frac{\CF[(\dv \bff)^o] (\xi)}{\lvert \xi \rvert^2}\bigg]
= - \sum_{j = 1}^{d - 1} \CF^{- 1}_\xi \bigg[\frac{i \xi_j \CF[f^o_j](\xi)}{\lvert \xi \rvert^2} \bigg]
- \CF^{- 1}_\xi \bigg[\frac{i \xi_d \CF[f^e_d] (\xi)}{\lvert \xi \rvert^2} \bigg].
\end{equation}
Since $i\xi_j|\xi|^{-2}\CF[f^o](\xi)$ and $i\xi_d|\xi|^{-2}\CF[f^e](\xi)$ belong to 
$\CS'(\BR^d)$, $u_0$ also belongs to $\CS'(\BR^d)$.  
By the construction of $u_0$, we see that $u_0 \vert_{\pd \HS} = 0$ and $\Delta u_0 =
\dv (\sfe [\bff])$ in $\BR^d$.  
In addition, it follows from the Fourier multiplier theorem 
(cf. Proposition \ref{prop-Fourier-multiplier}) that 
\begin{equation}
\label{est-u0}
\lVert \nabla u_0 \rVert_{\CB^s_{q, r} (\BR^d)} \le C \lVert \sfe [\bff] \rVert_{\CB^s_{q, r} (\BR^d)}
\le C \lVert \bff \rVert_{\CB^s_{q, r}(\HS)}.
\end{equation}
Let $u$ be the zero extension of $u_0$, that is $u=u_0$ for $x_d > 0$ and $u=0$ for $x_d < 0$,
and then $u \in \wh \CB^{s + 1}_{q, r, 0} (\HS)$.
Furthermore, let $\varphi \in \wh \CB^{1 - s}_{q', r', 0} (\HS)$ be arbitrary.
As was stated in Remark \ref{rem:def-hat-g}, the zero extension $\varphi_0$ of $\varphi$ to $x_d < 0$
 satisfies \eqref{def-hat-g*}, and hence
$$(\nabla u, \nabla \varphi)_{\HS} 
 = (\nabla u_0, \nabla \varphi_0)_{\BR^d}. $$
From Proposition \ref{prop:Bdense} it follows that there exists a sequence $\{\varphi_j\}_{j=1}^\infty
$ of $C^\infty_0(\BR^d)$ such that 
$$\lim_{j\to\infty} \|\nabla(\varphi_j - \varphi_0)\|_{\CB^{-s}_{q',r'}(\BR^d)} = 0.$$
Thus, we have 
\begin{equation}
\label{id-weak-Dirichlet}
\begin{aligned}
 & (\nabla u_0, \nabla \varphi_0)_{\BR^d} =
\lim_{j\to\infty}(\nabla u_0, \nabla\varphi_j)_{\BR^d} 
=- \lim_{j\to\infty} (\Delta u_0, \varphi_j)_{\BR^d}\\
&\quad = -\lim_{j\to\infty}(\dv(\sfe [\bff]), \varphi_j)_{\BR^d}
= \lim_{j\to\infty}(\sfe [\bff], \nabla\varphi_0)_{\BR^d}
=(\bff, \nabla\varphi)_{\HS}.
\end{aligned}
\end{equation}
Therefore, $u \in \wh \CB^{s + 1}_{q, r, 0} (\HS)$ is a solution to the 
weak Dirichlet problem \eqref{wd:1}. Moreover, from the definition of
the Besov spaces by restriction and \eqref{est-u0}, 
we see that $u$ satisfies the estimate \eqref{wd:2.1}. 
\par
It remains to prove the uniqueness part. 
Let $u \in \wh \CB^{s + 1}_{q, r, 0} (\HS)$ satisfy the homogeneous equation:
\begin{equation}
\label{wd:2.2}
(\nabla u, \nabla \varphi) = 0
\end{equation}	
for all $\varphi \in \wh \CB^{1 - s}_{q', r', 0} (\HS)$.
Now, let $\bg = (g_1, \ldots, g_d) \in C^\infty_0(\HS)^d$ be arbitrary.
Define $\varphi_0$ by
\begin{equation}
\varphi_0 = -\CF^{-1}_\xi \bigg[\frac{\CF[(\dv \bg)^o] (\xi)}{\lvert \xi \rvert^2}\bigg]
= - \sum_{j = 1}^{d - 1} \CF^{- 1}_\xi \bigg[\frac{i \xi_j \CF[g^o_j](\xi)}{\lvert \xi \rvert^2} \bigg]
- \CF^{- 1}_\xi \bigg[\frac{i \xi_d \CF[g^e_d] (\xi)}{\lvert \xi \rvert^2} \bigg],
\end{equation}
and then, we have $\varphi_0 \in \wh \CB^{1-s}_{q', r'}(\BR^d)$ and $\varphi_0|_{\pd\HS}=0$. 
Let $\sfe(\bg) = (g_1^o, \ldots, g_{d-1}^o, g_d^e)$, and then $\sfe(\bg) \in C^\infty_0(\BR^d)$,
 and $\Delta \varphi_0 =\dv\sfe(\bg)= (\dv\bg)^o$ in $\BR^d$.  
Let $\varphi$ be the zero extension of $\varphi_0$, that is $\varphi=\varphi_0$ for $x_d > 0$ and 
$\varphi=0$ for $x_d < 0$,  and 
let  $u_0 \in \CB^{1-s}_{q,r}(\BR^d)$ be the zero 
extension of $u$ to $x_d < 0$,  that is $u_0=u$ for $x_d > 0$ and $u_d=0$ for 
$x_d < 0$, and then  $u_0$ satisfies \eqref{def-hat-g} with $f=u$.  We have 
$$0 = (\nabla u, \nabla \varphi)_{\HS} = (\nabla u_0, \nabla\varphi_0)_{\BR^d}.$$
It follows from Proposition \ref{prop:Bdense} that there exists a sequence
$\{u_j\}_{j=1}^\infty$ of $C^\infty_0(\BR^d)$ such that
$$\lim_{j\to\infty}\|\nabla(u_j - u_0)\|_{\CB^s_{q,r}(\BR^d)} = 0.$$
Then, we have
\begin{align*}
&(\nabla u_0, \nabla\varphi_0)_{\BR^d} = \lim_{j\to\infty} (\nabla u_j, \nabla\varphi_0)_{\BR^d}
=-\lim_{j\to\infty}(u_j, \Delta \varphi_0)_{\BR^d} \\
&\quad = -\lim_{j\to\infty}(u_j, \dv\sfe[\bg])_{\BR^d}
= \lim_{j\to\infty} (\nabla u_j, \sfe[\bg])_{\BR^d} = (\nabla u_0, \sfe[\bg])_{\BR^d} 
= (\nabla u, \bg)_{\HS}.\end{align*}
Thus,  we have $(\nabla u, \bg)=0$. 
Recalling that $\bg \in C^\infty_0 (\HS)^d$ is taken arbitrarily,
we infer that $\nabla u = 0$ in $\HS$ in the sense of distributions 
and therefore, $u$ is a constant in $\HS$. By \eqref{def-hat-g}, 
{\color{black} we have} $u \vert_{\pd \HS} = 0$, which {\color{black} implies} that $u=0$.
\end{proof}
\begin{cor}\label{cor:2.1} Let $1 < q < \infty$, $1 \leq r \leq \infty-$,  and $-1 + 1/q < s < 1/q$. 
For any $\bff \in \CB^s_{q, r}(\HS)^{\color{black} d}$
and $f \in \CB^{s+1}_{q, r}(\HS)$,  let $u= f + Q(\bff-\nabla f)$.  Then,
$u \in \CB^{s+1}_{q, r}(\HS) 
+ \wh \CB^{s+1}_{q,r,0}(\HS)$
 satisfies equations \eqref{wd:2.1} subject to  $u |_{\pd \HS}= f$ on $\pd\HS$, 
and estimate: 
\begin{equation}
\|\nabla u\|_{\CB^s_{q, r}(\HS)} \le C\|(\nabla f, \bff)\|_{\CB^s_{q, r}(\HS)}
\end{equation}
for some constant $C > 0$. 
\end{cor}
\begin{proof} {\color{black} Set $v = u - f$.} Then $v$ satisfies
$$(\nabla v, \nabla\varphi) = (\bff - \nabla f, \nabla\varphi)
$$
for every $\varphi 
\in \wh \CB^{1-s}_{q', r', 0}(\HS)$. 
Thus, by Theorem \ref{thm-weak-Dirichlet}, 
{\color{black}the proof is complete.}
\end{proof}

\subsection*{The second Helmholtz decomposition}
Let $\bff \in \CB^s_{q, r}(\HS)^d$.
We define operators $\BP^s_{q, r} \colon \CB^s_{q, r}(\HS)^d \to \CB^s_{q, r}(\HS)^d$ and
$\BQ^s_{q, r} \colon \CB^s_{q, r}(\HS)^d \to \wh \CB^{s+1}_{q,r,0}(\HS)$ by setting
$\BP^s_{q, 1} \bff = \bff - \nabla\fp$ and $\BQ^s_{q, r} \bff = \fp$, {\color{black}
where $\fp \in \wh \CB^{s + 1}_{q, r, 0} (\HS)$ is a unique solution to 
\begin{equation}\label{wd:6.1}
(\nabla\fp, \nabla \varphi) = (\bff, \nabla\varphi)\quad
\text{for every $\varphi \in \wh \CB^{1-s}_{q', r', 0}(\HS)$}.
\end{equation}
Notice that the unique existence of $\fp \in \wh \CB^{s + 1}_{q, 1, 0} (\HS)$ is 
indeed guaranteed by Theorem \ref{thm-weak-Dirichlet}. Clearly, there holds
$\bff = \BP^s_{q, 1} \bff + \nabla \BQ^s_{q, r} \bff$. Here, by Theorem 
\ref{thm-weak-Dirichlet}, we see that $\BP^s_{q, r}$ and $\nabla \BQ^s_{q, r}$ 
are bounded linear operators such that
\begin{equation}\label{sol:6.1}
\|\BP^s_{q, r} \bff\|_{\CB^s_{q, r}(\HS)} \le C\|\bff\|_{\CB^s_{q, r}(\HS)},
\quad 
\|\nabla \BQ^s_{q, r} \bff\|_{\CB^s_{q, r} (\HS)} \le C\|\bff\|_{\CB^s_{q, r}(\HS)}
\end{equation}
for some constant $C > 0$ depending solely on $s$, $q$, $r$, and $d$. Furthermore, by setting
\begin{equation}
\mathcal{G}^s_{q, r}= \{\nabla \BQ^s_{q, r} \bff \mid \bff \in \CB^s_{q, r}(\HS)^d\},
\end{equation}
and noting that $\CJ^s_{q, r}(\HS) = \{\BP^s_{q, r} \bff \mid\bff \in \CB^s_{q, r}(\HS)^d\}$, 
we have the \textit{second Helmholtz decomposition}:
\begin{equation}
\label{helmholtz.1}
\CB^s_{q, r}(\HS) = \CJ^s_{q, r}(\HS) \oplus \mathcal{G}^s_{q, r}(\HS),
\end{equation} 
where the symbol $\oplus$ stands for direct sum operation.  
In fact, if we take $\bff \in \CB^s_{q, r} (\HS)^d$ such that 
$\bff \in \CJ^s_{q, r} (\HS) \cap \mathcal{G}^s_{q, r} (\HS)$, there holds
$\bff = \BP^s_{q, r} \bff = \nabla \BQ^s_{q, r} \bff$. Then for any 
$\varphi \in \wh \CB^{1-s}_{q', r', 0}(\HS)$,} we have
\begin{equation}
0 = (\BP^s_{q, r} \bff, \nabla \varphi) = (\nabla \BQ^s_{q, r} \bff, \nabla \varphi).
\end{equation}
Thus, {\color{black} we infer from Theorem \ref{thm-weak-Dirichlet} that} $\bff = 0$,
{\color{black}i.e., $\CJ^s_{q, r} (\HS) \cap \mathcal{G}^s_{q, r} (\HS) = \{0\}$}. 
We write   
$\mathcal G^s_{q, r} (\HS) = G^s_{q, r} (\HS)$ if
$\CB^s_{q, r} (\HS) = B^s_{q, r} (\HS)$ and $\mathcal G^s_{q, r} (\HS)
= \dot G^s_{q, r} (\HS)$ if $\CB^s_{q, r} (\HS) = \dot B^s_{q, r} (\HS)$.\vskip0.5pc

We shall give a characterization of {\color{black} solenoidal vector fields}. 
\begin{thm}
\label{thm-solenoidal-characterization}
Let $1 < q < \infty$, $1 \leq r \leq \infty-$,  and $- 1 + 1 \slash q < s < 1 \slash q$. Then, 
there holds $\bu_\sigma \in \CJ^s_{q, r}(\HS)$
if and only if $\bu_\sigma \in \CB^s_{q, r} (\HS)^d$ and $\bu_\sigma$ satisfies $\dv \bu_\sigma = 0$
in the distribution sense.
\end{thm}
\begin{proof}
Assume that $\bu_\sigma \in \CB^s_{q, r} (\HS)^d$ satisfies
$(\bu_\sigma, \nabla \varphi) = 0$ for every 
$\varphi \in \wh \CB^{1 - s}_{q', r', 0} (\HS)$. Since 
$C^\infty_0(\HS)$ is a subset of $\wh \CB^{1 - s}_{q', r', 0} (\HS)$,
we have $(\bu_\sigma, \nabla \varphi) = 0$ for every 
$\varphi \in C^\infty_0(\HS)$, which shows that 
$\dv \bu_\sigma = 0$ in the distribution sense. \par
Assume that $\bu_\sigma \in \CB^{s}_{q, r} (\HS)$ satisfies $\dv \bu_\sigma = 0$ in $\HS$ 
 in the sense of distributions. Let $\bu_{\sigma, e}=(u_1^o, \ldots, u_{d-1}^o, u_d^e)$, and then
we see that $\dv \bu_{\sigma, e} = (\dv \bu_\sigma)^o = 0$ in $\BR^d$.  Let $\varphi$ be any element of
$\wh B^{1-s}_{q',r',0}(\HS)$ and let $\varphi_0$ be the zero extension of $\varphi$ to $x_d < 0$, which has 
the properties \eqref{def-hat-g*}. From Proposition \ref{prop:Bdense} it follows that  there exists   
a sequence $\{\varphi_j\}_{j=1}^\infty$  in $C^\infty_0(\HS)$ such that 
$\lim_{j\to\infty}\|\nabla(\varphi_j-\varphi_0)\|_{\CB^{-s}_{q',r'}(\BR^d)} = 0$.  
 Noting \eqref{def-hat-g*},
we have
\begin{equation}
(\bu_\sigma, \nabla\varphi)_{\HS} = (\bu_{\sigma,e}, \nabla\varphi_0)_{\BR^d} 
= \lim_{j\to\infty}(\bu_{\sigma, e}, \nabla \varphi_j)_{\BR^d}
= -\lim_{j\to\infty} (\dv \bu_{\sigma, e}, \varphi_j)_{\BR^d} = 0,
\end{equation}
which shows that $\bu_\sigma \in \CJ^s_{q,r}(\HS)$.  This completes the proof of
Theorem \ref{thm-solenoidal-characterization}. 
\end{proof}
Finally, we shall give a uniqueness theorem for the resolvent problem of the weak 
Dirichlet problem. 
\begin{prop}
\label{prop:wd.1} 
Let $1 < q < \infty$, $1 \leq r \leq \infty-$,  and $-1 + 1 \slash q < s < 1 \slash q$.  
Moreover, let $0 < \epsilon < \pi/2$ and 
$\lambda \in \Sigma_{\epsilon}$. Set $\sB^{1-s}_{q', r', 0}(\HS) = \{\varphi 
\in \sB^{1-s}_{q', r'}(\HS) \mid \varphi|_{\pd\HS}=0\}$. 
If $u \in \CB^{s + 1}_{q, r} (\HS)$ satisfies 
\begin{equation}
\left\{\begin{aligned}
\label{dir:12}
(\lambda u, \varphi) + \mu (\nabla u, \nabla \varphi) & = 0
& \quad & \text{for all $\varphi \in  \sB^{1 - s}_{q', r', 0} (\HS)$}, \\
u \vert_{\pd \HS} & =0 & \quad & \text{on $\pd \HS$},
\end{aligned}\right.
\end{equation}
then it holds $u = 0$.  
\end{prop}
\begin{proof} 
Let $f \in C^\infty_0(\HS)$ be arbitrary.  We set
\begin{equation}
\varphi_0 = \CF^{-1}_\xi \bigg[
\frac{\CF[f^o] (\xi)}{ \lambda + \mu \lvert \xi \rvert^2} \bigg],
\end{equation}	
where $f^o$ stands for the odd extension of $f$. We see that 
$f^o \in C^\infty_0(\BR^d)$, and hence $f^o \in \CB^{-s}_{q', r'}(\BR^d)$ since
$-1+1/q' < -s < 1/q'$.  Notice that
there exists a positive constant $c_1$ depending on $\epsilon \in 
(0, \pi/2)$ and $\mu>0$ such that for any $\lambda \in \Sigma_\epsilon$ and 
$\xi \in \BR^d$ there holds
\begin{equation}
\label{resol:3.1}
c_1^{-1}(|\lambda|^{1/2}+|\xi|^2) \le |\lambda + \mu|\xi|^2| 
\le c_1(|\lambda|^{1/2} + |\xi|^2).
\end{equation}
Using \eqref{resol:3.1} and  the Fourier multiplier theorem 
(cf. Proposition \ref{prop-Fourier-multiplier}), we have 
\begin{equation}
\lVert \nabla^\ell \varphi_0 \rVert_{\CB^{- s}_{q', r'} (\BR^d)} 
\le C \lVert f^o \rVert_{\CB^{- s}_{q', r'} (\BR^d)} 
\le C \lVert f \rVert_{\CB^{- s}_{q', r'} (\HS)}, \qquad \ell = 0, 1, 2
\end{equation}
with the convention $\nabla^0 \varphi_0 = \varphi_0$.
Namely, $\varphi_0 \in \sB^{1-s}_{q',r'} (\HS)$.  
Moreover, $\varphi_0 \vert_{\pd \HS} = 0$ due to its definition, and hence
$\varphi = \varphi_0|_{\HS} \in \sB^{1-s}_{q', r',0}(\HS)$,  which implies that 
\eqref{dir:12}$_{1}$ holds for $\varphi$ constructed above.  
Moreover, there holds $(\lambda-\mu\Delta)\varphi = f$ in $\HS$.
Since $u$ vanishes on $\pd \HS$,  integration by parts implies that 
\begin{equation}
(u, f) = (u, (\lambda - \mu \Delta)\varphi) 
= (\lambda u,   \varphi) + \mu (\nabla u, \nabla \varphi) 
=0.
\end{equation}
Since $f \in C^\infty_0(\HS)$ is taken arbitrarily, we deduce that 
$u = 0$.
\end{proof}

\subsection{Solution formula}\label{subsec.2.4} 
\label{sec-solution-formula}
In this subsection, we derive the solution formula to \eqref{resol:1.1}.
Notice that the solution formula to \eqref{resol:1.1} gives the
solution formula to \eqref{eq-Stokes-inhomogeneous} via the 
inverse Fourier--Laplace transform with respect to a parameter $\lambda$.
In addition, by considering \eqref{resol:1.1}, the resolvent 
estimate for the linearized operator will be obtained, which
gives the solution formula to \eqref{eq-Stokes-homogeneous} with
the aid of the standard semigroup theory. \par
In the following, let $1 < q < \infty$, $1 \leq r \leq \infty-$,  and $- 1 + 1 \slash q < s < 1 \slash q$.
\subsection*{Step 1: Reductions} 
We shall use the symbols introduced in Convention \ref{convention} in Section 1.
First, according to Lemma 3.2 in \cite{Shi14}, we reduce Problem \eqref{eq-Stokes} to the divergence-free case. Let $G_\mathrm{div}= \dv\bG$ be the right member for the divergence condition in Problem 
\eqref{eq-Stokes}. Let $G_\mathrm{div}^e(x, t)$ be the even extension of $G_\mathrm{div}(x, t)$ with respect to
$x$-variables and $G^e_\mathrm{div} = \dv \bG_e$, where $\bG_e=(G_1^e, \ldots, G_{d-1}^e, G_d^o)$ for
$\bG=(G_1, \ldots, G_d)$ (cf. \eqref{extension:1}). 
We define a compensation function $\bV_3(x, t)$ for the divergence condition 
by
\begin{equation}\label{div:1}\begin{aligned}
&\bV_3(x, t) = \nabla\Delta^{-1}G_\mathrm{div}^e = \nabla\Delta^{-1}\dv \bG_e \\
&=- \CF^{-1}_\xi\bigg[\frac{i\xi\CF[G_\mathrm{div}^e](\xi, t)}{|\xi|^2}\biggr](x)
= -\CF^{-1}\bigg[\frac{i\xi(\sum_{k=1}^{d-1}i\xi_k \CF[G_k^e](\xi', t)+i\xi_d\CF[G_d^o](\xi, t))}{|\xi|^2}
\biggr](x).
\end{aligned}\end{equation}
By the Fourier multiplier theorem (cf. Proposition \ref{ap.B} in Appendix B below), we obtain
\begin{equation}\label{div:3}\begin{aligned}
\|\bV_3(\,\cdot\,, t)\|_{\CB^s_{q,r}(\BR^d)} &\leq C\|\bG(\,\cdot\,, t)\|_{\CB^s_{q,r}(\HS)}, \\
\|\nabla \bV_3(\,\cdot\,, t)\|_{\CB^s_{q,r}(\BR^d)} &\leq C\|G_\mathrm{div}(\,\cdot\,, t)\|_{\CB^s_{q,r}(\HS)}, \\
\|\nabla^2 \bV_3(\,\cdot\,, t)\|_{\CB^s_{q,r}(\BR^d)} & \leq C\|\nabla G_\mathrm{div}(\,\cdot\,, t)\|_{\CB^s_{q,r}(\HS)}.
\end{aligned}\end{equation}
and thus, we have 
\begin{equation}\label{div:4} \begin{aligned}
\|e^{-\gamma_bt}\pd_t\bV_3\|_{L_1(\BR, \CB^s_{q,1}(\HS))} &\leq 
C\|e^{-\gamma_bt} \pd_t \bG \|_{L_1(\BR, \CB^s_{q,1}(\HS))}, \\
\|e^{-\gamma_bt}\nabla \bV_3\|_{\CW^{1/2}_1(\BR, \CB^{s}_{q,1}(\HS))} 
&\leq C \|e^{-\gamma_bt} \nabla G_\mathrm{div}\|_{\CW^{1/2}_1(\BR, \CB^s_{q,1}(\HS))}, \\
\|e^{-\gamma_bt}\nabla^2 \bV_3\|_{L_1(\BR, \CB^{s}_{q,1}(\HS))}  
&\leq C\|e^{-\gamma_b t}\nabla G_\mathrm{div} \|_{L_1(\BR, \CB^s_{q,1}(\HS))}.
\end{aligned}\end{equation}
Moreover, by the definition of $\bV_3$, if we set $\bV_3=(V_{3,1}, \ldots, V_{3,d})$, we see that
\begin{equation}\label{g-vanish1}
\pd_d V_{3,j}|_{\pd\HS}=0 \quad (j=1, \ldots, d-1), \quad V_{3,d}|_{\pd\HS}=0.
\end{equation}
Thus, in particular, there holds
\begin{equation}\label{g-vanish2}
\mu\BD(\bV_3)\bn_0|_{\pd\HS} = (0, \ldots, 0, 2\mu \pd_d V_{3,d}|_{\pd\HS}).
\end{equation}
Setting $\bV = \bU + \bV_3$ in Problem \eqref{eq-Stokes}, we see that $\bU$ and $\Pi$ 
should satisfy the following equations:
\begin{equation}\label{eq-Stokes.1}\left\{\begin{aligned}
\pd_t \bU - \DV (\mu \BD(\bU) - \Pi \BI) & = \bF -(\pd_t\bV_3-\mu\dv\BD(\bV_3))
 & \quad & \text{in $\HS \times \BR_+$}, \\
\dv \bU & = 0& \quad & \text{in $\HS \times \BR_+$}, \\
(\mu \BD(\bU) - \Pi \BI) \bn_0 & = (\bH - 2\mu V_{3,d}\bn_0)\vert_{\pd \HS} & \quad & \text{on $\pd \HS \times \BR_+$}, \\
\bU \vert_{t = 0} & = \bV_0 - \bV_3|_{t=0}& \quad & \text{in $\HS$}.
\end{aligned}\right.
\end{equation}
Since $\bV_0$ satisfies the compatibility condition $\dv\bV_0=G_\mathrm{div}|_{t=0}$ in $\HS$, we have
$\dv(\bV_0-\bV_3|_{t=0}) = 0$ in $\HS$. 
Therefore, in the sequel, we consider the following evolution equations:
\begin{equation}\label{eq-Stokes.0}\left\{\begin{aligned}
\pd_t \bU - \DV (\mu \BD(\bU) - \Pi \BI) & = \bF 
 & \quad & \text{in $\HS \times \BR_+$}, \\
\dv \bU & = 0& \quad & \text{in $\HS \times \BR_+$}, \\
(\mu \BD(\bU) - \Pi \BI) \bn_0 & = \bH \vert_{\pd \HS} & \quad & \text{on $\pd \HS \times \BR_+$}, \\
\bU \vert_{t = 0} & = \bU_0 & \quad & \text{in $\HS$}.
\end{aligned}\right.
\end{equation}
with compatibility condition: $\bU_0 \in \CJ^s_{q,r}(\HS)$.  \par
Next, let the Stokes pressure $\CQ$ be defined by $\CQ=H_d + Q(\bF-\nabla H_d)$, where 
$\bH = (H_1, \ldots, H_d)$ and $Q$ is the operator given in Theorem \ref{thm-weak-Dirichlet}.
Set $\Pi = \CQ+\Pi'$, and then Problem \eqref{eq-Stokes.0} 
is reduced to 
\begin{equation}\label{eq-Stokes.2}\left\{\begin{aligned}
\pd_t \bU - \DV (\mu\BD (\bU) - \Pi' \BI) & = \bF'
 & \quad & \text{in $\HS \times \BR_+$}, \\
\dv \bU & = 0& \quad & \text{in $\HS \times \BR_+$}, \\
(\mu\BD (\bU) - \Pi' \BI) \bn_0 & = \bH'\vert_{\pd \HS} & \quad & \text{on $\pd \HS \times \BR_+$}, \\
\bU\vert_{t = 0} & = \bU_0 & \quad & \text{in $\HS$},
\end{aligned}\right.
\end{equation}
where $\bF' = \bF- \nabla \CQ$ and  $\bH'=(H_1, \ldots, H_{d-1}, 0)$. We observe that
\begin{equation}\label{solenoidal:2.1}
(\bF(\,\cdot\,, t), \nabla\varphi) = 0 \quad\text{for every $\varphi \in \wh \CB^{1-s}_{q', r', 0}(\HS)$}.
\end{equation}
 According to Corollary~\ref{cor:2.1}, we see that $\CQ$ satisfies the estimate
\begin{equation}\label{dir:6}
\lVert \nabla \CQ \rVert_{\CB^s_{q, r} (\HS)} 
\le C \lVert (\bF, H_d) \rVert_{\CB^s_{q, r} (\HS)}.
\end{equation}
The corresponding generalized resolvent problem reads as 
\begin{equation}\label{newresol:1.1} \left\{\begin{aligned}
\lambda \bu - \DV(\mu\BD(\bu) - \fq \BI) &= \bff&\quad&\text{in $\HS$}, \\
\dv \bu &=0 &\quad&\text{in $\HS$}, \\
(\mu\BD(\bu) - \fq \BI)\bn_0 &=\bh|_{\pd\HS} &\quad&\text{on $\pd\HS$}.
\end{aligned}\right.\end{equation}
with $\bff \in \CJ^s_{q,r}(\HS)$ and $\bh=(h_1, \ldots, h_{d-1},0)$. 
Let $\bu_1 = (u_{1, 1}, \ldots, u_{1, d}) \in \sB^{s + 2}_{q, r} (\HS)^d$ 
be a solution to the following parabolic system:
\begin{equation}
\label{whole:1}
\left\{\begin{aligned}
(\lambda - \mu \Delta) \bu_1 
&= \bff 
& \quad & \text{in $\HS$}, \\
\dv\bu_1&=0& \quad & \text{in $\HS$}, \\
u_{1, j} & = 0 & \quad &\text{on $\pd\HS$},  \quad \text{for $j = 1, \ldots, d - 1$}, \\
\pd_d u_{1, d} & = 0 & \quad &\text{on $\pd\HS$}.
\end{aligned}\right.
\end{equation}
Notice that $\dv\bu_1=0$ implies $\dv \BD(\bu_1) = \Delta\bu_1$. \par
Next, we set $\bu_2 := \bu - \bu_1$  with
$\bu_2 = (u_{2, 1}, \ldots, u_{2, d})$.  Then $(\bu_2, \fq)$ solves
\begin{equation}
\label{red:2}
\left\{\begin{aligned}
\lambda \bu_2 - \dv (\mu \BD (\bu_2) - \fq \BI ) 
& = 0 & \quad & \text{in $\HS$}, \\
\dv \bu_2 & = 0 & \quad & \text{in $\HS$}, \\
\mu (\pd_d u_{2, j} + \pd_j u_{2, d}) & = \Big( h_j 
- \mu(\pd_d u_{1, j} + \pd_j u_{1, d}) \Big) \Big|_{\pd \HS} & \quad &
\text{on $\pd\HS$ \quad  for $j = 1, \ldots, d - 1$}, \\
(2 \mu \pd_d u_{2, d} - \fq)  & = 0 & \quad &\text{on $\pd\HS$.}
\end{aligned}\right.
\end{equation}
Clearly, $\bu = \bu_1 + \bu_2$ and $\fq$ are solutions to \eqref{resol:1.1}.
\subsection*{Step 2: Solution formulas}
Concerning equations \eqref{whole:1},  it suffices to consider the following systems: 
\begin{equation}
\label{fund:1}
 \left\{\begin{aligned}
\lambda \bw_1 - \mu \Delta \bw_1 & = \bff 
& \quad &\text{in $\HS$}, \\
w_{1, j}  & =0
 & \quad &\text{on $\pd\HS$\quad  for $j = 1, \ldots, d - 1$}, \\
\pd_d w_{1, d}  & = 0  & \quad &\text{on $\pd\HS$},
\end{aligned} \right.\end{equation}
The functions $\bw_1 = (w_{1, 1}, \ldots, w_{1, d})$  
  given by
\begin{equation}
\label{sol-w1-w2}
w_{1, j}  = \CF^{-1}_\xi \bigg[
\frac{\CF[f_j^o] (\xi)}{\lambda + \mu \lvert \xi \rvert^2} \bigg] \bigg\vert_{\HS} 
\, (j  = 1, \ldots, d - 1),  \quad
w_{1, d}  = \CF^{-1}_\xi \bigg
[\frac{\CF[f_d^e] (\xi)}{\lambda + \mu \lvert \xi \rvert^2} \bigg] \bigg\vert_{\HS}
\end{equation}
solve \eqref{fund:1}.  
Moreover,  if $\bff \in \CJ^s_{q,r}(\HS)$ and $\bw_1 \in \sB^{s+2}_{q,r}(\HS)$, 
then $\bw_1$ satisfies 
\begin{equation}
\label{dir:7}
\dv \bw_1 =0 \qquad \text{in $\HS$}.
\end{equation}
Thus, 
$\bu_1=\bw_1$ satisfies \eqref{whole:1}. \par
In the sequel, we shall show \eqref{dir:7}. 
From $\bff \in \CJ^s_{q,r}(\HS)$ it follows that for any $\varphi \in \wh \CB^{1 - s}_{q', r', 0} (\HS)$, 
\begin{equation}
\label{dir:9}
0 = (\bff, \nabla\varphi) =(\lambda \bw_1 - \mu \Delta \bw_1, \nabla \varphi).
\end{equation}  
In addition, it holds 
\begin{equation}
\label{dir:8}
(\Delta \bw_1, \nabla \varphi)
= (\nabla \dv \bw_1, \nabla \varphi) \qquad 
\text{for all $\varphi \in \wh \CB^{1-s}_{q', r', 0} (\HS)$}.
\end{equation}
In fact, let $\wt \bw_1 \in \sB^{s + 2}_{q, r} (\BR^d)^d$ be an 
extension of $\bw_1$ such that $\wt\bw_1|_{\HS} = \bw_1$. 
For any $\varphi \in \wh \CB^{1 - s}_{q', r', 0} (\HS)$, let $\varphi_0$ be the zero extension of
$\varphi$ to $x_d < 0$, that is $\varphi_0= \varphi$ for $x_d > 0$ and $\varphi_0=0$ for 
$x_d < 0$.  Recalling that
$\varphi_0$ satisfies \eqref{def-hat-g*}, 
 from Proposition \ref{prop:Bdense} in Appendix B below we see that there exists
a sequence $\{\varphi_j\}_{j=1}^\infty \subset C^\infty_0(\BR^d)$ such that 
$\lim_{j\to\infty}\|\nabla(\varphi - \varphi_j)\|_{B^{-s}_{q', r'}(\BR^d)} = 0$. 
Hence, we have
\begin{equation}
\label{eq:dir1}
\begin{aligned}
(\Delta \bw_1, \nabla \varphi) 
& = (\Delta \wt \bw_1, \nabla \varphi_0)_{\BR^d} \\
& =  \lim_{j\to\infty}(\Delta \wt \bw_1, \nabla \varphi_j)_{\BR^d}\\
& =-\lim_{j\to\infty} (\dv \wt \bw_1, \Delta \varphi_j)_{\BR^d} \\
& = \lim_{j\to\infty}(\nabla \dv\wt\bw_1, \nabla\varphi_j)_{\BR^d} \\
& = (\nabla\dv\wt\bw_1, \nabla\varphi_0)_{\BR^d} \\
& = (\nabla\dv\bw, \nabla\varphi).
\end{aligned}
\end{equation}
Thus, combining \eqref{dir:9} and \eqref{dir:8} gives that 
\begin{equation}
\label{dir:10}
(\lambda \bw_1, \nabla \varphi) 
- \mu (\nabla \dv \bw_1, \nabla \varphi) = 0
\qquad \text{for all $\varphi \in \wh \CB^{1 - s}_{q', r', 0} (\HS)$}.
\end{equation}
Noting that $ \sB^{1 - s}_{q', r', 0}(\HS) 
\subset \wh \CB^{{\color{black} 1 - s}}_{q', r', 0}(\HS)$, by integration by parts we have 
\begin{equation}
\lambda ( \dv \bw_1, \varphi)+ 
\mu (\nabla \dv \bw_1, \nabla \varphi)
= 0
\end{equation}
for all $\varphi \in  \sB^{1 - s}_{q', r', 0} (\HS)$.
Moreover, it holds $\dv \bw_1= 0$ on $\pd\HS$
on account of the boundary conditions \eqref{fund:1}$_{2, 3}$.  
Thus, by Proposition \ref{prop:wd.1}, we deduce that 
$\dv\bw_1 = 0$.  Namely, we arrive at \eqref{dir:7}. \par
Concerning \eqref{red:2}, we consider the following boundary value problem:
\begin{equation}
\label{fund:3}
\left\{\begin{aligned}
\lambda \bw_2 - \dv (\mu \BD(\bw_2) - \fq_2 \BI) & = 0 
& \quad & \text{in $\HS$}, \\
\dv \bw_2 & = 0 & \quad & \text{in $\HS$}, \\ 
\mu(\pd_d w_{2, j} + \pd_j w_{2, d}) & = h_j {\color{black} \vert_{\pd \HS}} 
&\quad& \text{on $\pd\HS$ \quad 
for  $j = 1, \ldots, d - 1$}, \\		
2 \mu \pd_d w_{2, d} - \fq_2 & =0
 &\quad& \text{on $\pd\HS$ }
\end{aligned}\right. 
\end{equation}
for $\bh'=(h_1, \ldots, h_{d-1}) \in \CB^s_{q,r}(\HS)^{d-1}$. 
To give solution formulas of $\bw_2$ and  $\fq_2$,
we shall introduce symbols.  Let $\epsilon \in (0, \pi \slash 2)$
and  $\xi' = (\xi_1, \ldots, \xi_{d - 1}) \in \BR^{d - 1}$.   
We define $A = \lvert \xi' \rvert = (\sum_{j = 1}^{d - 1} \xi_j^2)^{1 \slash 2}$.
Let $B = B (\lambda, \xi')$ be the principal branch of the square root of 
$\mu^{- 1} \lambda + \lvert \xi' \rvert^2$, i.e., 
$B = \sqrt{\mu^{- 1} \lambda + \lvert \xi' \rvert^2}$
for $\lambda \in \Sigma_\epsilon$ with $\RE B > 0$;
and for $x_d > 0$ let $M_{x_d} = M (\lambda, \xi', x_d)$ be defined by
\begin{equation}
M_{x_d} = M (\lambda, \xi', x_d) = \frac{e^{- B x_d} - e^{- A x_d}}{B - A}.
\end{equation}
We also define 
\begin{equation}
D_{A, B} = B^3 + A B^2 + 3 A^2 B - A^3,
\end{equation}
which is the determinant of the Lopatinskii matrix. \par

To derive the formulas for $\bw_2$ and $\fq_2$, we take
the Fourier transform in the tangential direction $x'$
with covariable $\xi'$ and solve the transformed problem
(i.e., a boundary value problem for an ordinary differential
equation on $\BR_+$). Following the computations in 
\cite[p. 589--590]{SS12}, we observe that $(\bw_2, \fq_2)$ is given by
\begin{equation}
\label{hsol:4}
\begin{aligned}
w_{2, j} (x) & = \CF^{- 1}_{\xi'} \bigg[
\frac{e^{- B x_d}}{\mu B} \wh h_j - \frac{i \xi_j M_{x_d}}{\mu D_{A, B}}
(2 B i \xi' \cdot \wh h') 
+ \frac{i \xi_j e^{- B x_d}}{\mu B D_{A, B}}(3 B - A) (i \xi' \cdot \wh h') \bigg] (x'), \\
w_{2, d} (x) & = \CF^{- 1}_{\xi'} \bigg[\frac{A M_{x_d}}{\mu D_{A, B}} 
(2 B i\xi' \cdot \wh h')
+ \frac{e^{- B x_d}}{\mu D_{A, B}}(B - A) (i \xi' \cdot \wh h')\bigg] (x'), \\
\fq_2 (x) & = - \CF^{- 1}_{\xi'} \bigg[
\frac{(A + B) e^{- A x_d}}{D_{A, B}}
2 B i \xi' \cdot \wh h' \bigg] (x'),
\end{aligned}
\end{equation}
where $j = 1, \ldots, d - 1$ and we have set $i \xi' \cdot \wh h'
= \sum_{j = 1}^{d - 1} i \xi_j \wh h_j (\xi')$. 
Replacing  $h_j$ by 
$h_j - \mu(\pd_d u_{1, j} + \pd_j u_{1, d})$ for every $j = 1, \ldots, d - 1$
in \eqref{fund:3}, we have $\bu_2$ and $\fq$, which solve equations \eqref{red:2}. 
 \par
\subsection{Estimates of multipliers}
The proof of Theorems \ref{th-MR-inhomogeneous} and \ref{th-MR-homogeneous} 
are based on the explicit solution
formulas to \eqref{eq-Stokes}, see \eqref{sol-w1-w2} and \eqref{hsol:4}.
To estimate the explicit solution formulas, we introduce the class of multiplier symbols.
\begin{dfn}\label{def:2.15}
Let $0 < \epsilon < \pi \slash 2$.  For nonnegative integers $n$, 
define $\BM_{- n}$ as the set of all functions 
$m (\lambda, \xi') \in \Hol (\Sigma_\epsilon, 
C^\infty (\BR^{d - 1} \setminus \{0\}))$ such that for any multi-index $\alpha' \in \BN_0^{d-1}$, 
 there exists a constant $C_{\alpha'}$ such that
\begin{align}
\lvert \pd_{\xi'}^{\alpha'} m (\lambda, \xi') \rvert 
& \le C_{\alpha'} (\lvert \lambda \rvert^{1 \slash 2} 
+ \lvert \xi ' \rvert)^{- n} \lvert \xi' \rvert^{- \lvert \alpha' \rvert}, \\ 
\bigg\lvert \pd_{\xi'}^{\alpha'} 
\bigg(\frac{\pd m (\lambda, \xi')}{\pd \lambda} \bigg) \bigg\rvert 
& \le C_{\alpha'} (\lvert \lambda \rvert^{1 \slash 2} + \lvert \xi' \rvert)^{- (n + 2)}
\lvert \xi' \rvert^{- \lvert \alpha'\rvert}.
\end{align}
\end{dfn}
Notice that it follows from \cite[Lem. 5.2]{SS12} that 
$B^{- 1} \in \BM_{- 1}$ and $D_{A, B}^{- 1} \in \BM_{- 3}$, and 
$\xi_j/A \in \BM_0$.
The following technical proposition plays a crucial role in this paper.
\begin{prop}
\label{prop:3.1}
Let $1 < q < \infty$ and $0 < \epsilon < \pi \slash 2$.  
For $\lambda \in \Sigma_\epsilon$ let $\CL_k (x_d, y_d)$, $k = 1, 2, 3$, be defined by
$\CL_1 (x_d, y_d) = B e^{- B (x_d + y_d)}$,
$\CL_2 (x_d, y_d) = A B M_{x_d + y_d}$, and
$\CL_3(x_d, y_d) = A e^{- A (x_d + y_d)}$. 
For $m (\lambda, \xi') \in \BM_0$ and $f \in L_q (\HS)$, set
\begin{alignat}3
\CM_k (\lambda) f (x) & = \int^\infty_0 \CF^{- 1}_{\xi'}
[m (\lambda, \xi') \CL_k (x_d, y_d) \CF'[ f] (\xi', y_d)] (x') \d y_d, & \qquad k & = 1, 2, 3;\\
\CN_k (\lambda) f (x) & = \int^\infty_0 \CF^{- 1}_{\xi'}
[m (\lambda, \xi') B^2 \pd_\lambda \CL_k (x_d, y_d)
\CF'[f] (\xi', y_d)] (x') \d y_d & \qquad k & = 1, 2.
\end{alignat}
Then, there exists a constant $C > 0$ such that for every $\lambda \in \Sigma_\epsilon$, there hold
\begin{alignat}3
\label{est:2.1}
\lVert \CM_k (\lambda) f \rVert_{L_q (\HS)} & \le C \lVert f \rVert_{L_q (\HS)}, 
& \qquad k & = 1, 2,3, \\
\intertext{as well as}
\label{est:2.2}
\lVert \CN_k (\lambda) f \rVert_{L_q (\HS)} & \le C \lVert f \rVert_{L_q (\HS)},
& \qquad k & = 1, 2.
\end{alignat}
\end{prop}
\begin{proof}
Since we have $B = (\lambda^{1 \slash 2} \slash B) \lambda^{1 \slash 2} - (A \slash B) A$,
we see that \eqref{est:2.1} is a direct consequence of \cite[Lem.~5.4]{SS12}.
In fact, it follows that
$\lambda^{1 \slash 2} B^{- 1} \in \BM_0$ and $A B^{- 1} \in \BM_0$.
\par
To show \eqref{est:2.2} for the case $k = 1$, we first write 
\begin{equation}
\pd_\lambda \Big(B e^{- B (x_d + y_d)} \Big)
= \frac{1}{2 \mu B}e^{- B (x_d + y_d)} - \frac{x_d + y_d}{2 \mu} e^{- B (x_d + y_d)}.
\end{equation}
By Lemmas 5.2 and 5.3 in \cite{SS12} together with the Leibniz rule,
for any multi-index $\alpha' \in \BN_0^{d - 1}$ we obtain
\begin{align}
& \lvert \pd^{\alpha'}_{\xi'} (B^2 \pd_\lambda (B e^{- B (x_d + y_d)})) \rvert \\
& \quad \le \frac1{2 \mu} \lvert \pd^{\alpha'}_{\xi'} (B e^{- B (x_d + y_d)}) \rvert
+ \frac{x_d + y_d}{2 \mu} \lvert \pd^{\alpha'}_{\xi'}( B^2 e^{- B (x_d + y_d)}) \rvert \\
& \quad \le C (\lvert \lambda \rvert^{1 \slash 2} + A)^{1 - \lvert \alpha' \rvert} \Big( 
e^{- c (\lvert \lambda \rvert^{1 \slash 2} + A) (x_d + y_d)}
+ (\lvert \lambda \rvert^{1 \slash 2} + A) (x_d + y_d) 
e^{- c (\lvert \lambda \rvert^{1 \slash 2} + A) (x_d + y_d)} \Big) \\
& \quad \le C (\lvert \lambda \rvert^{1 \slash 2} + A)^{1 - \lvert \alpha' \rvert}
e^{- (c/2) (\lvert \lambda \rvert^{1 \slash 2} + A) (x_d + y_d)}.	
\end{align}	
Thus, in the same way as in \eqref{est:2.1}, 
we may show \eqref{est:2.2} for the case $k = 1$.
The other cases may be shown similarly.
The proof is complete.
\end{proof}

\section{Resolvent  estimates in the whole space}
\label{sec-3}
In this section, we establish the estimates for the operator 
$\bL_{0, \ell}(\lambda)$,
$\ell \in \{e, o\}$, defined by
\begin{equation}
\label{fop:1}
\bL_{0, \ell} (\lambda) f = \CF^{-1}_\xi \bigg[\frac{\CF[f^\ell](\xi)}
{\lambda + \mu|\xi|^2}\bigg], \quad \ell \in \{e, o\},
\end{equation}
where $f^e$ and $f^o$ mean the even and odd extensions, respectively,
whose definitions are given in \eqref{extension:1}. 
Clearly, $\bL_{0, e}$ and $\bL_{0, o}$ are the solution operators
for the equation $(\lambda - \mu \Delta) u = f$ in $\HS$ with 
homogeneous Dirichlet and Neumann boundary conditions on $\pd \HS$, respectively.
\par
In the sequel, we shall use symbols given in Convention \ref{convention} in Section 1.
In this section, first we shall prove the following theorem. 
\textcolor{red}{Here and in the following, recall that $\nabla_b$ is the symbol introduced in Convention \ref{convention}.}
\begin{thm}\label{thm:3.1}
 Let $1 < q < \infty$, $1 \leq r \leq \infty$, 
 $ - 1 + 1 \slash q < s < 1 \slash q$, 
$\epsilon \in (0, \pi \slash 2)$, and $\gamma>0$.  
Let $\sigma$ be a small positive number such that 
$- 1 + 1 \slash q < s \pm \sigma < 1 \slash q$.
For every $\lambda \in \Sigma_{\epsilon} + \gamma_b$ and $f \in C^\infty_0(\HS)$, 
there exists a constant $C_b$ depending on $\gamma_b$ such that
\begin{alignat}2
\label{3.1}
\lVert (\lambda, \lambda^{1/2} \nabla,   \nabla_b^2) \bL_{0, \ell} (\lambda) 
f \rVert_{ \CB^s_{q, r} (\BR^d)}
& \le C_b \lVert f \rVert_{ \CB^s_{q, r} (\HS)},
\\
\label{3.2}
\lVert (\lambda^{1/2} \nabla,    \nabla_b^2)\bL_{0, \ell} (\lambda) 
f \rVert_{ \CB^s_{q, r} (\BR^d)}
& \le C_b \lvert \lambda \rvert^{- \frac\sigma2}\lVert f \rVert_{ \CB^{s + \sigma}_{q, r} (\HS)},
\\
\|(1, \lambda^{-1/2}\nabla)\bL_{0, \ell} (\lambda) f\|_{\CB^s_{q, r}(\BR^d)} 
& \le C_b \lvert \lambda \rvert^{- (1 - \frac\sigma2)}
\lVert f \rVert_{ \CB^{s - \sigma}_{q, r} (\HS)},
\label{3.4}
\\
\label{3.3}
\lVert (\lambda, \lambda^{1/2} \nabla,
  \nabla_b^2) \pd_\lambda \bL_{0, \ell} (\lambda) 
f \rVert_{ \CB^s_{q, r}(\BR^d)}
& \le C_b \lvert \lambda \rvert^{- (1 - \frac\sigma2)}
\lVert f \rVert_{ \CB^{s - \sigma}_{q, r} (\HS)}.
\end{alignat}
\end{thm}
\begin{proof}
For the case $\CB^s_{q,r} = B^s_{q,r}$, we take $\gamma_b = \gamma > 0$, i.e., 
$\lambda \in \Sigma_\epsilon + \gamma$, and hence $|\lambda| \geq \gamma\sin\epsilon$.  
Thus, there hold
\begin{align*}
\|\bL_{0,\ell}(\lambda)f\|_{B^s_{q,r}(\BR^d)} &= \|\lambda \bL_{0,\ell}(\lambda)f\|_{B^s_{q,r}(\BR^d)} |\lambda|^{-1}
\leq (\gamma\sin\epsilon)^{-1} \|\lambda \bL_{0,\ell}(\lambda)f\|_{B^s_{q,r}(\BR^d)}, \\
\|\nabla \bL_{0,\ell}(\lambda)f\|_{B^s_{q,r}(\BR^d)} &= \|\lambda^{1/2}\nabla \bL_{0,\ell}(\lambda)f\|_{B^s_{q,r}(\BR^d)} 
|\lambda|^{-1/2}
\leq (\gamma\sin\epsilon)^{-1/2} \|\lambda^{1/2}\nabla \bL_{0,\ell}(\lambda)f\|_{B^s_{q,r}(\BR^d)}.
\end{align*}
Thus, for the case $\CB^s_{q,r} = B^s_{q,r}$ it suffices to prove 
\begin{align*}
\lVert (\lambda, \lambda^{1/2} \nabla, \nabla^2) \bL_{0,\ell} (\lambda) f \rVert_{\CB^s_{q, r} (\BR^d)} 
 &\le C_b \lVert f \rVert_{ \CB^{s}_{q, 1} (\HS)}, \\
\lVert (\lambda, \lambda^{1/2} \nabla, 
  \nabla^2) \bL_{0,\ell} (\lambda) f \rVert_{\CB^s_{q, r} (\BR^d)} 
& \le C_b \lvert \lambda \rvert^{- \frac\sigma2} 
\lVert f \rVert_{\CB^{s + \sigma}_{q, r} (\HS)}, \\ 	
\lVert (\lambda, \lambda^{1/2} \nabla, 
  \nabla^2) \pd_\lambda \bL_{0,\ell} (\lambda) f \rVert_{\CB^s_{q, r} (\BR^d)} 
& \le C_b \lvert \lambda \rvert^{- (1 - \frac\sigma2)} 
\lVert f \rVert_{\CB^{s - \sigma}_{q, r} (\HS)}.
\end{align*}
Recalling that \eqref{resol:3.1} holds, we see that \eqref{3.1} follows from
the Fourier multiplier theorem (cf. Proposition \ref{prop-Fourier-multiplier} in Appendix 
\ref{ap.B}). \par
To prove \eqref{3.2}--\eqref{3.3}, we introduce a symbol $c_b$ defined by 
$c_b=1$ if $\gamma_b=\gamma$ and $c_b=0$ if $\gamma_b=0$. 
To obtain \eqref{3.2}, we apply the Fourier multiplier theorem
to the representation
\begin{equation}
\lambda^{\frac{\sigma}{2}}\bL_{0, \ell}(\lambda)f= \CF^{-1}_{\xi}
\bigg[\frac{\lambda^{\frac{\sigma}{2}}(c_b+|\xi|^2)^{\frac{\sigma}{2}}\CF[f](\xi)}
{(\lambda+\mu|\xi|^2)^{\frac{\sigma}{2}}(\lambda + \mu|\xi|^2)^{1-\frac{\sigma}{2}}
(c_b+|\xi|^2)^{\frac{\sigma}{2}}}\bigg].
\end{equation}  
 Similarly,  writing 
\begin{align}
\lambda^{(1-\frac{\sigma}{2})}\bL_{0, \ell}(\lambda)f
& = \CF^{-1}_{\xi}
\bigg[\frac{\lambda^{1-\frac{\sigma}{2}}(c_b+|\xi|^2)^{-\frac{\sigma}{2}}\CF[f](\xi)}
{(\lambda+\mu|\xi|^2)^{1-\frac{\sigma}{2}}(\lambda + \mu|\xi|^2)^{\frac{\sigma}{2}}
(c_b+|\xi|^2)^{-\frac{\sigma}{2}}}\bigg], \\
\lambda^{(\frac12-\frac{\sigma}{2})}\nabla\bL_{0, \ell}(\lambda)f
& = \CF^{-1}_{\xi}
\bigg[\frac{\lambda^{\frac12-\frac{\sigma}{2}}(i\xi)(c_b+|\xi|^2)^{-\frac{\sigma}{2}}\CF[f](\xi)}
{(\lambda+\mu|\xi|^2)^{\frac12-\frac{\sigma}{2}}(\lambda + \mu|\xi|^2)^{\frac12+\frac{\sigma}{2}}
(c_b+|\xi|^2)^{-\frac{\sigma}{2}}}\bigg]
\end{align}
and applying the Fourier multiplier theorem 
yield \eqref{3.4}. 
To prove \eqref{3.3}, we write 
\begin{equation}
\pd_\lambda\bL_{0, \ell}(\lambda)f= -\CF^{-1}_{\xi}
\bigg[\frac{\CF[f](\xi)}
{(\lambda+\mu|\xi|^2)^2}\bigg].
\end{equation}
Hence, we have
\begin{equation}
\lambda^{1-\frac{\sigma}{2}}
\pd_\lambda\bL_{0, \ell}(\lambda)f= -\CF^{-1}_{\xi}
\bigg[\frac{\lambda^{1-\frac{\sigma}{2}}(c_b+|\xi|^2)^{-\frac{\sigma}{2}}\CF[f](\xi)}
{(\lambda+\mu|\xi|^2)^{1-\frac{\sigma}{2}}(\lambda+\mu|\xi|^2)^{1+\frac{\sigma}{2}}
(c_b+|\xi|^2)^{-\frac{\sigma}{2}}}\bigg],
\end{equation}
which together with the Fourier multiplier theorem yields \eqref{3.3}.
The proof of Theorem \ref{thm:3.1} is complete. 
\end{proof}
By virtue of Theorem  \ref{thm:3.1}, we may show the following Corollary.
\begin{cor}
\label{lem-CW1}
Let $1 < q < \infty$, $1 \leq  r \leq \infty-$, $ - 1 + 1 \slash q < s < 1 \slash q$, $\gamma > 0$, and
$\epsilon \in (0, \pi \slash 2)$.  
Let $\sigma$ be a positive number such that 
$- 1 + 1 \slash q < s \pm \sigma < 1 \slash q$.
Then, there exists an operator
\begin{equation}
\CW_1 (\lambda) \in \Hol (\Sigma_{\epsilon}, 
\CL (\CB^s_{q, r} (\HS)^d,  \sB^{s + 2}_{q, r} (\HS)^d))
\end{equation}
such that   
$\bu_1= \CW_1 (\lambda)\bff$  is a solution to \eqref{whole:1} possessing the estimate:
\begin{equation}\label{cor:3.2}\lVert (\lambda, \lambda^{1/2} \nabla,   \nabla_b^2) 
\bu_1\rVert_{\CB^s_{q, r} (\BR^d)}
\le C_b \lVert \bff \rVert_{\CB^s_{q, r} (\HS)}
\end{equation}
for any $\bff \in \CJ^s_{q,r}(\HS)$ and $\lambda \in \Sigma_\epsilon + \gamma_b$. \par
Moreover, for every $\lambda \in \Sigma_\epsilon + \gamma_b$ and 
$\bff \in C^\infty_0(\HS)^d$, there hold  
\allowdisplaybreaks 
\begin{alignat}2
\lVert(\lambda^{1/2} \nabla,    \nabla_b^2)\CW_1 (\lambda)\bff \rVert_{\CB^s_{q, r} (\BR^d)}
& \le C_b \lvert \lambda \rvert^{- \frac\sigma2}
\lVert \bff \rVert_{ \CB^{s + \sigma}_{q, r} (\HS)},
\\
\lVert (1, \lambda^{-1/2}\nabla)\CW_1 (\lambda)\bff\rVert_{\CB^s_{q, r} (\BR^d)}
& \le C_b \lvert \lambda \rvert^{-(1-\frac\sigma2)}
\lVert \bff \rVert_{ \CB^{s - \sigma}_{q, r} (\HS)},
\\
\lVert (\lambda,\lambda^{1/2} \nabla,    \nabla_b^2) \pd_\lambda 
\CW_1 (\lambda)\bff
 \rVert_{\CB^s_{q, r}(\BR^d)}
& \le C_b \lvert \lambda \rvert^{- (1 - \frac\sigma2)}
\lVert \bff \rVert_{ \CB^{s - \sigma}_{q, r} (\HS)}.
\end{alignat}
Here, all constants $C_b$ {\color{black} depend solely} on $\gamma_b$. 
\end{cor}	
\begin{proof}
Recall that the solution $\bw_1$ to \eqref{fund:1} is given by \eqref{sol-w1-w2}.
Hence, we define the solution operator $\CW_1 (\lambda) = (\CW_{1, 1} (\lambda), 
\ldots, \CW_{1, d} (\lambda))$ by
\begin{equation}
\CW_{1, j} (\lambda) \bff = \bL_{0, o} (\lambda) f_j, \qquad j = 1, \ldots, d - 1, 
\qquad \CW_{1, d} (\lambda) \bff = \bL_{0, e} (\lambda) f_d.
\end{equation}
Noting that $C^\infty_0(\HS)$ is dense in $\CB^s_{q,r}(\HS)$ for $1 < q < \infty$,
$1 \leq r \leq \infty-$ and $-1+1/q < s < 1/q$, applying the density argument 
to \eqref{3.1}, we have \eqref{cor:3.2}, and 
 other estimates  directly follow from \eqref{3.2}--\eqref{3.3} in Theorem \ref{thm:3.1}.
\end{proof}

\section{Resolvent  estimates in the half-space}
\label{sec-4}
Let $\epsilon \in (0, \pi \slash 2)$.
For $\lambda \in \Sigma_{\epsilon}$, define the integral operators by
\begin{equation}
\label{def-L-formula}
\begin{split}
\bL_{1} (\lambda) f & = \int^\infty_0 \CF^{- 1}_{\xi'}
[m_{2} (\lambda, \xi') B e^{- B (x_d + y_d)} 
\CF'[f] (\xi', y_d)] (x') \d y_d,  \\
\bL_{2} (\lambda) f & = \int^\infty_0 \CF^{- 1}_{\xi'}
[m_{3} (\lambda, \xi') A^2 B M_{x_d + y_d}
\CF'[ f] (\xi', y_d)] (x') \d y_d, \\
\bL_3 (\lambda) f & = \int^\infty_0 \CF^{- 1}_{\xi'}
[m_1 (\lambda, \xi') A e^{- A (x_d + y_d)} 
\CF'[ f] (\xi', y_d)] (x') \d y_d,
\end{split}
\end{equation}
respectively, where $m_l$, $l = 1, 2, 3$, are assumed to belong to $\BM_{- l}$
(cf. Definition \ref{def:2.15}).  
Notice that, as seen in Section \ref{sec-solution-formula},
we see that the solution $\bw_2$ to \eqref{fund:3} with the 
associated pressure $\fq_2$ may be given by a linear combination
of the above operators.
The aim of this section is to establish the estimates for the operators
defined in \eqref{def-L-formula}, which implies the estimates for the
solution operators for  \eqref{fund:3}. 
\subsection{Strategy of $\CB^s_{q,r}$ estimates in the half-space} \label{subsec.4.1}
In this subsection, we shall describe a method to evaluate integral operators defined in \eqref{def-L-formula}
in the $\CB^s_{q,r}(\HS)$-framework. In contrast to the whole space, we do not have 
the explicit representation of the $\CB^s_{q,r}(\HS)$-norm, and thus 
estimating the integral operators directly within $\CB^s_{q,r}(\HS)$ seems to fail.
Hence, we will use some combinations of complex and real interpolation methods.
The key idea described in this subsection is also summarized in the proceedings of 
the first author \cite{Shi23}, where one may find a comprehensive survey on 
a spectral analysis approach to initial boundary value problems with inhomogeneous boundary data,
including the free boundary problem of the Navier--Stokes equations within 
maximal $L_p$-regularity framework with $1 \le p < \infty$. \par
Roughly speaking, the idea is the following:
\begin{enumerate}
\item[\textbf{Step 1:}] If $0 < s < 1/q$, the starting evaluation is done in $\CH^1_q (\HS)$.  
Then, use complex interpolation to obtain the estimates in $\CH^\nu_q$ ($0 < \nu < 1/q$).   
Finally, by real interpolation, we arrive at the estimates in $\CB^s_{q,r} (\HS)$.
\item[\textbf{Step 2:}] If $-1+1/q < s < 0$, consider the dual operator and its estimate 
in $L_{q'} (\HS)$ and $\CH^1_{q'} (\HS)$. Then complex interpolation yields the estimates of dual 
operators in $\CH^\nu_{q'} (\HS)$ with $0 < \nu < 1/q' = 1-1/q$. Together with the duality argument, 
we obtain the estimates of the original operator in $\CH^{-\nu}_q (\HS)$. Finally, 
by real interpolation, we arrive at the estimates of the original operator in $\CB^s_{q,r} (\HS)$.
\item[\textbf{Step 3:}] The estimate in $\CB^0_{q,r} (\HS)$ follows from real interpolation 
between the estimates in $\CB^{-\omega_1}_{q,r} (\HS)$ and $\CB^{\omega_2}_{q,r} (\HS)$ for some small $\omega_1, \omega_2> 0$.
\end{enumerate}
We consider two operator-valued holomorphic functions $\CT_i(\lambda)$ $(i = 1,2)$ defined on
$\Sigma_{\epsilon}$ acting on $f \in C^\infty_0(\HS)$. We denote the dual operator
of $\CT_i(\lambda)$ by $\CT_i(\lambda)^*$, where $\CT_i(\lambda)^*$ satisfies the equality:
$$|(T_i(\lambda)f, \varphi)| = |(f, \CT_i(\lambda)^*\varphi)|$$
for any $f$, $\varphi \in C^\infty_0(\HS)$.  Here, recall that $(f, g) = \int_{\HS}f(x)g(x)\d x$,
in particular, we do not take the complex conjugate to define $(f,g)$. \par 
For convenience, we use the symbol introduced in Convention \ref{convention}.
We assume that the operators $\CT_i (\lambda)$ $(i = 1,2)$ satisfy the following assumption.
\begin{assump}\label{assump:4.1}
 Let $1 < q < \infty$, $\epsilon \in (0, \pi/2)$,
 $\gamma>0$, and $\alpha_1, \alpha_2, \beta_1, \beta_2 \in \BR$.  
For every $\lambda \in \Sigma_\epsilon + \gamma_b$ and $f \in C^\infty_0(\HS)$, 
 the following estimates hold for $i = 1,2$:
\begin{align}
\label{f1} \|\CT_i(\lambda)f\|_{\CH^1_q(\HS)} &\leq C|\lambda|^{\alpha_i} \|f\|_{\CH^1_q(\HS)},  \\
\label{f2} \|\CT_i(\lambda)f\|_{L_q(\HS)} &\leq C|\lambda|^{\alpha_i} \|f\|_{L_q(\HS)}, \\
\label{f3} \|\CT_i(\lambda)f\|_{\CH^{-1+i}_q(\HS)} & \leq C|\lambda|^{\beta_i}\|f\|_{\CH^{2-i}_q(\HS)}, \\
\label{d1*} \|\CT_i(\lambda)^*f\|_{L_{q'}(\HS)} &\leq C|\lambda|^{\alpha_i}\|f\|_{L_{q'}(\HS)},  \\
\label{d2*} \|\CT_i(\lambda)^*f\|_{\CH^1_{q'}(\HS)} &\leq C|\lambda|^{\alpha_i}\|f\|_{\CH^1_{q'}(\HS)}, \\
\label{d3*} \|\CT_i(\lambda)^*f\|_{\CH^{-1+i}_{q'}(\HS)} & \leq C|\lambda|^{\beta_i}\|f\|_{\CH^{2-i}_{q'}(\HS)}.
\end{align}
\end{assump}
Then, we have the following theorem.
\begin{thm}\label{spectralthm:1}
Let $1 < q < \infty$, $1 \leq r \leq \infty$,  $-1+1/q < s < 1/q$,  
$\epsilon \in (0, \pi/2)$,  $\gamma>0$, and $\alpha_1, \alpha_2, \beta_1, \beta_2 \in \BR$.
Let $\sigma>0$ be a  small number such that $-1+1/q < s-\sigma < s < s+\sigma < 1/q$.
Assume that Assumption \ref{assump:4.1} holds.  
Then, for every $\lambda \in \Sigma_\epsilon + \gamma_b$ and $f \in C^\infty_0(\HS)$, there hold
\begin{align*}
\|\CT_1(\lambda)f\|_{\CB^s_{q,r}(\HS)} &\leq C|\lambda|^{\alpha_1}\|f\|_{\CB^s_{q,r}(\HS)}, \\
\|\CT_1(\lambda)f\|_{\CB^s_{q,r}(\HS)} &\leq C|\lambda|^{(1-\sigma)\alpha_1 + \sigma\beta_1}
\|f\|_{\CB^{s+\sigma}_{q,r}(\HS)}, \\
\|\CT_2(\lambda)f\|_{\CB^s_{q,r}(\HS)} & \leq C|\lambda|^{\alpha_2}\|f\|_{\CB^s_{q,r}(\HS)}, \\
\|\CT_2(\lambda)f\|_{\CB^s_{q,r}(\HS)} &\leq C|\lambda|^{(1-\sigma)\alpha_2+\sigma\beta_2}
\|f\|_{\CB^{s-\sigma}_{q,r}(\HS)}.
\end{align*}
\end{thm}
 We divide the proof into the cases $0 < s < 1/q$ and 
$-1+1/q < s < 0$. The case $s = 0$ is obtained by interpolating estimates in  
$B^{-\omega_1}_{q,r} (\HS)$ and $B^{\omega_2}_{q,r} (\HS)$ for some small $\omega_1, \omega_2 >0$. 
\begin{lem}\label{lem.s.1} Assume that Assumption \ref{assump:4.1} holds. 
Let $q$, $r$, $\epsilon$, $\gamma$, $\alpha_1$, and $\beta_1$ be the same numbers as in 
Theorem \ref{spectralthm:1}. Let $1 \leq r \leq \infty$.  Assume 
$0 < s < 1/q$ and let $\sigma>0$ be a  number such that  $0 <  s+\sigma <1/q$. Then, 
for every $\lambda \in \Sigma_\epsilon + \gamma_b$ and $f \in C^\infty_0(\HS)$, there hold
\begin{align}
\|\CT_1(\lambda)f\|_{\CB^s_{q,r}(\HS)} & \leq C|\lambda|^{\alpha_1}\|f\|_{\CB^s_{q,r}(\HS)}, 
\label{A.1}\\
\|\CT_1(\lambda)f\|_{\CB^s_{q,r}(\HS)} & \leq C|\lambda|^{(1-\sigma)\alpha_1 + \sigma\beta_1}
\|f\|_{\CB^{s+\sigma}_{q,r}(\HS)}. \label{A.2}
\end{align}
\end{lem}
\begin{proof} 
Choose $\nu_1$ and  $\nu_2$ such that
$0 < s < s+\sigma < \nu_2 <\nu_1 < 1/q$.  Estimates \eqref{f1}, \eqref{f2}, and \eqref{f3} 
for $i=1$ are interpolated with the
complex interpolation method to obtain  
\begin{align}
\label{1}
\|\CT_1(\lambda)f\|_{L_q(\HS)} & \leq C|\lambda|^{\alpha_1}\|f\|_{L_q(\HS)}, \\
\label{2}
\|\CT_1(\lambda)f\|_{\CH^{\nu_1}_q(\HS)} & \leq C|\lambda|^{\alpha_1}\|f\|_{\CH^{\nu_1}_q(\HS)},
\\
\label{2'}
\|\CT_1(\lambda)f\|_{\CH^{\nu_2}_q(\HS)} 
& \leq C|\lambda|^{\alpha_1}\|f\|_{\CH^{\nu_2}_q(\HS)}, \\
\label{3}
\|\CT_1(\lambda)f\|_{L_q(\HS)} & \leq C|\lambda|^{\alpha_1(1-\nu_1) + \beta_1\nu_1}\|f\|_{\CH^{\nu_1}_q(\HS)}.
\end{align}
By interpolating \eqref{1} and  \eqref{2} with the real interpolation method, we have
\begin{equation}\label{4}
\|\CT_1(\lambda)f\|_{\CB^s_{q,r}(\HS)} \leq C|\lambda|^{\alpha_1}\|f\|_{\CB^s_{q,r}(\HS)}.
\end{equation}
Applying the real interpolation functor $(\,\cdot\,,\,\cdot\,)_{\theta_1, r}$ with $\theta_1=s/\nu_2$, 
we infer from \eqref{2'} and \eqref{3} that
\begin{equation}\label{5}
\|\CT_1(\lambda)f\|_{\CB^s_{q,r}(\HS)} \leq C|\lambda|^{\delta_1}
\|f\|_{\CB^{s+\sigma_1}_{q,r}(\HS)}, 
\end{equation}
where we have set 
\begin{align*}
\delta_1 &= \alpha_1\frac{s}{\nu_2} + \Big(\alpha_1(1-\nu_1)+\beta_1\nu_1 \Big) \bigg(1-\frac{s}{\nu_2}\bigg)
=\alpha_1 -(\alpha_1-\beta_1)\nu_1 \bigg(1-\frac{s}{\nu_2}\bigg), \\
\sigma_1 &=\nu_1\bigg(1-\frac{s}{\nu_2}\bigg).
\end{align*}
Now, we take $\nu_1$ and $\nu_2$ in such a way that $s < s + \sigma < s+\sigma_1$,
i.e., $\sigma/\nu_1 + s/\nu_2 < 1$ (recall that there holds $s+\sigma < \nu_2 <\nu_1 < 1/q$).
Thus, choosing  
$\theta_2 \in (0,1)$ such that $s+\sigma =\theta_2 s +(1- \theta_2)(s+\sigma_1)$,
i.e.,  $1-\theta_2 = \sigma/\sigma_1$, we apply the real interpolation
functor $(\,\cdot\,,\,\cdot\,)_{\theta_2, r}$ to \eqref{4} and \eqref{5} to obtain 
\begin{equation}\label{6}
\|\CT_1(\lambda)f\|_{\CB^{s}_{q,r}(\HS)} \leq C|\lambda|^{\delta_2}
\|f\|_{\CB^{s+\sigma}_{q,r}(\HS)}.
\end{equation}
where 
\begin{align*}
\delta_2 & =\alpha_1\theta_2 + \delta_1 (1-\theta_2) \\
& = \alpha_1\bigg(1-\frac{\sigma}{\sigma_1}\bigg)+ \frac{\delta_1}{\sigma_1}\sigma \\
& = \alpha_1\bigg(1-\frac{\sigma}{\sigma_1}\bigg) + \frac{\sigma}{\sigma_1}
\bigg\{\alpha_1-(\alpha_1-\beta_1)\nu_1\bigg(1-\frac{s}{\nu_2}\bigg)\bigg\} \\
& =\alpha_1 -\sigma(\alpha_1-\beta_1).
\end{align*}
as follows from $(\nu_1 \slash \sigma_1) (1-s \slash \nu_2) = 1$.
Hence, we have \eqref{A.1} and \eqref{A.2}. The proof of Lemma  \ref{lem.s.1} is complete.
\end{proof} 
\begin{lem}\label{lem.s.2} Assume that Assumption \ref{assump:4.1} holds. 
Let $q$, $r$, $\epsilon$, $\gamma$, $\alpha_1$, and $\beta_1$ be the same numbers as in 
Theorem \ref{spectralthm:1}.
Assume $-1+1/q < s < 0$ and let $\sigma>0$ be  a number such that  $-1+1/q < s+\sigma <0$. 
Then, for every $\lambda \in \Sigma_\epsilon + \gamma_b$ and $f \in C^\infty_0(\HS)$, there hold
\begin{align}
\|\CT_1(\lambda)f\|_{\CB^s_{q,r}(\HS)} & \leq C|\lambda|^{\alpha_1}\|f\|_{\CB^s_{q,r}(\HS)}, 
\label{A.1*}\\
\|\CT_1(\lambda)f\|_{\CB^s_{q,r}(\HS)} & \leq C|\lambda|^{\alpha_1(1-\sigma)+\beta_1\sigma}
\|f\|_{\CB^{s+\sigma}_{q,r}(\HS)}. \label{A.2*}
\end{align}
\end{lem}
\begin{proof} 
Since $-1+1/q < s < 0$, we have $0 < |s| < 1-1/q = 1/q'$. 
Let $\nu_3$, $\nu_4$ and $\sigma$ be positive numbers such that
\begin{equation}\label{number.1} 0 < \nu_4 < |s| -\sigma < |s|  < \nu_3 < 1/q'.\end{equation}
Using the complex interpolation method, 
by \eqref{d1*}, \eqref{d2*}, and \eqref{d3*} for $i=1$, we have
\begin{align*}
\|\CT_1(\lambda)^*\varphi\|_{L_{q'}(\HS)} & \leq C|\lambda|^{\alpha_1}\|\varphi\|_{L_{q'}(\HS)}, \\
\|\CT_1(\lambda)^*\varphi\|_{\CH^{\nu_3}_{q'}(\HS)} & \leq C|\lambda|^{\alpha_1}
\|\varphi\|_{\CH^{\nu_3}_{q'}(\HS)},
\\
\|\CT_1(\lambda)^*\varphi\|_{\CH^{\nu_4}_{q'}(\HS)} 
& \leq C|\lambda|^{\alpha_1}\|\varphi\|_{\CH^{\nu_4}_{q'}(\HS)}, \\
\|\CT_1(\lambda)^*\varphi\|_{L_{q'}(\HS)} & \leq C|\lambda|^{\alpha_1(1-\nu_3)+\beta_1\nu_3}
\|\varphi\|_{\CH^{\nu_3}_{q'}(\HS)}.
\end{align*}
By the duality argument, we obtain 
\begin{align}
\label{d4*}
\|\CT_1(\lambda)f\|_{L_q(\HS)} & \leq C|\lambda|^{\alpha_1}\|f\|_{L_q(\HS)}, \\
\label{d5*}
\|\CT_1(\lambda)f\|_{\CH^{-\nu_3}_q(\HS)} & \leq C|\lambda|^{\alpha_1}\|f\|_{\CH^{-\nu_3}_q(\HS)},
\\
\label{d10*}
\|\CT_1(\lambda)f\|_{\CH^{-\nu_4}_q(\HS)} 
& \leq C|\lambda|^{\alpha_1}\|f\|_{\CH^{-\nu_4}_q(\HS)}, \\
\label{d6*}
\|\CT_1(\lambda)f\|_{\CH^{-\nu_3}_q(\HS)} &
\leq C|\lambda|^{\alpha_1(1-\nu_3) + \beta_1\nu_3}\|f\|_{L_q(\HS)}.
\end{align}
In fact, note that $\CH^{-\nu_3}_q(\HS) = (\wt\CH^{\nu_3}_{q'}(\HS))^*$. 
For any $f, \varphi \in C^\infty_0(\HS)$, by the dual argument we observe 
\begin{align*}
|(\CT_1(\lambda)f, \varphi)| &= |(f, \CT_1(\lambda)^*\varphi)| \\
&\leq \|f\|_{\CH^{-\nu_3}_q(\HS)}\|\CT_1(\lambda)^*\varphi\|_{\CH^{\nu_3}_{q'}(\HS)}
\\
& \leq \|f\|_{\CH^{-\nu_3}_q(\HS)}C|\lambda|^{\alpha_1}\|\varphi\|_{\CH^{\nu_3}_{q'}(\HS)},
\end{align*}
which implies \eqref{d5*}.  Likewise, we have \eqref{d10*} and \eqref{d4*}. In addition, we see that
\begin{align*}
|(\CT_1(\lambda)f, \varphi)| &= |(f, \CT_1(\lambda)^*\varphi)| \\
&\leq \|f\|_{L_q(\HS)}\|\CT_1(\lambda)^*\varphi\|_{L_{q'}(\HS)}
\\
& \leq \|f\|_{L_q(\HS)}C|\lambda|^{\alpha_1(1-\nu_3)+\beta_1\nu_3}\|\varphi\|_{\CH^{\nu_3}_{q'}(\HS)},
\end{align*}
which implies \eqref{d6*}. \par 
Now, we shall prove \eqref{A.1*} and \eqref{A.2*}. 
Combining \eqref{d4*} and  \eqref{d5*} with the real interpolation method, we have
\begin{equation}\label{d9*}
\|\CT_1(\lambda)f\|_{\CB^{-|s|}_{q,r}(\HS)} \leq C|\lambda|^\alpha\|f\|_{\CB^{-|s|}_{q,r}(\HS)},
\end{equation}
which gives \eqref{A.1*}. 
To show \eqref{A.2*}, we first 
recall that $\nu_3$ and $\nu_4$ satisfy \eqref{number.1}.
Choose $\theta_3 \in (0, 1)$ such that
$-|s|= -\nu_3(1-\theta_3) - \nu_4\theta_3$, i.e., 
$\theta_3 = (\nu_3-|s|) \slash (\nu_3-\nu_4)$ is valid.
Applying the real interpolation functor $(\,\cdot\,,\,\cdot\,)_{\theta_3, r}$,
we infer from \eqref{d10*} and \eqref{d6*} that
$$\|\CT_1(\lambda)f\|_{\CB^{-|s|}_{q,r}(\HS)}\leq
 C|\lambda|^{\delta_3}\|f\|_{\CB^{(-\nu_4)\theta_3}_{q,r}}.
$$
with
\begin{align*}
\delta_3 & = \Big (\alpha_1(1-\nu_1) + \beta_1\nu_1 \Big)(1-\theta_3) + \alpha_1\theta_3 \\
& = \alpha_1-(\alpha_1-\beta_1)\nu_1(1-\theta_3) \\
&= \alpha_1-(\alpha_1-\beta_1)\nu_3 \bigg(1-\frac{\nu_3-|s|}{\nu_3-\nu_4}\bigg) \\
& = \alpha_1-(\alpha_1-\beta_1)\nu_3\frac{|s|-\nu_4}{\nu_3-\nu_4}.
\end{align*}
Therefore, we have
\begin{equation}\label{d11*}
\|\CT_1(\lambda)f\|_{\CB^{-|s|}_{q,r}(\HS)} \leq 
C|\lambda|^{\delta_3}
\|f\|_{\CB^{-\frac{\nu_4(\nu_3-|s|)}{\nu_3-\nu_4}}_{q,r}}.
\end{equation}
Since $0 < \nu_4 < |s|-\sigma$ and $0 < \nu_3-|s| < \nu_3-\nu_4$, there holds
$$-|s| < -|s| + \sigma < -\frac{\nu_4(\nu_3-|s|)}{\nu_3 - \nu_4}.$$
Choose $\theta_4 \in (0, 1)$ such that 
$$-|s| + \sigma =(1-\theta_4)(-|s|) + \theta_4 \bigg(-\frac{\nu_4(\nu_3-|s|)}{\nu_3-\nu_4}\bigg)
\qquad \Leftrightarrow \qquad
\theta_4 = \dfrac{(\nu_3-\nu_4)\sigma}{\nu_3(|s|-\nu_4)}.
$$
Applying the real interpolation functor $(\,\cdot\,,\,\cdot\,)_{\theta_4, r}$,
we infer from \eqref{d9*} and \eqref{d11*} that 
$$\|\CT_1(\lambda)f\|_{\CB^{-|s|}_{q,r}(\HS)} 
\leq C|\lambda|^{\delta_4}
\|f\|_{\CB^{-|s|+\sigma}_{q,r}(\HS)}
$$
with
\begin{align*}
\delta_4 & = \theta_4 \delta_3 + (1-\theta_4)\alpha_1 \\
& = \frac{(\nu_3-\nu_4)}{|s|-\nu_4}\frac{\sigma}{\nu_3}
\bigg(\alpha_1-(\alpha_1-\beta_1) \nu_3\frac{|s|-\nu_4}{\nu_3-\nu_4}\bigg)
+ \bigg(1-\frac{\nu_3-\nu_4}{|s|-\nu_4} \cdot \frac{\sigma}{\nu_3} \bigg)\alpha_1 \\
& = \alpha_1-(\alpha_1-\beta_1)\sigma,
\end{align*}
which shows \eqref{A.2*}.
\end{proof}
\begin{lem}\label{lem.s.3} Assume that Assumption \ref{assump:4.1} holds. 
Let $q$, $r$, $\epsilon$, $\gamma$, $\alpha_2$, and $\beta_2$ be the same numbers as in 
Theorem \ref{spectralthm:1}. 
Assume $0< s < 1/q $ and let $\sigma>0$ be a number such that  $0  < s -\sigma <1/q$. 
Then, for every $\lambda \in \Sigma_\epsilon + \gamma_b$ and $f \in C^\infty_0(\HS)$, 
there hold
\begin{align}
\|\CT_2(\lambda)f\|_{\CB^s_{q,r}(\HS)} & \leq C|\lambda|^{\alpha_2}\|f\|_{\CB^s_{q,r}(\HS)}, 
\label{A.3}\\
\|\CT_2(\lambda)f\|_{\CB^s_{q,r}(\HS)} & \leq C|\lambda|^{(1-\sigma)\alpha_2 + \beta_2\sigma}
\|f\|_{\CB^{s-\sigma}_{q,r}(\HS)}. \label{A.4}
\end{align}
\end{lem}
\begin{proof}
Let $\nu_5$ be a number such that $0 < s < s+\sigma < \nu_5 < 1/q$. 
Combining 
\eqref{f1} and \eqref{f2} as well as \eqref{f1} and \eqref{f3} for $i=2$ 
with the complex interpolation method imply 
\begin{align}
\label{d4} \|\CT_2(\lambda)f\|_{L_q(\HS)}& \leq C|\lambda|^{\alpha_2}
\|f\|_{L_q(\HS)}, \\
\label{d5} \|\CT_2(\lambda)f\|_{\CH^{\nu_5}_q(\HS)}
&
\leq C|\lambda|^{\alpha_2}\|f\|_{\CH^{\nu_5}_q(\HS)}, \\
\label{d6}
\|\CT_2(\lambda)f\|_{\CH^{\nu_5}_q(\HS)} & \leq 
C|\lambda|^{(1-\nu_5)\alpha_2 + \nu_5\beta_2}\|f\|_{L_q(\HS)}.
\end{align}
Combining \eqref{d4} and  \eqref{d5} with the real interpolation method yields 
\begin{equation}\label{d7}
\|\CT_2(\lambda)f\|_{\CB^s_{q,r}(\HS)}
\leq C|\lambda|^{\alpha_2}\|f\|_{\CB^s_{q,r}(\HS)},
\end{equation}
which shows \eqref{A.3}. \par
To prove \eqref{A.4}, we take $\nu_6$ and $\theta_5$ such that $0 < \nu_6 < \nu_5$ 
and $\theta_5 = \nu_6/ \nu_5 \in (0, 1)$ are valid, respectively.
Combining \eqref{d5} and \eqref{d6} with complex interpolation, we have 
\begin{equation}\label{d6'}
\|\CT_2(\lambda)f\|_{\CH^{\nu_5}_{q}(\HS)}
\leq C|\lambda|^{\delta_5}\|f\|_{\CH^{\nu_6}_{q}(\HS)}
\end{equation}
with
\begin{align*}
\delta_5 &= \alpha_2\theta_5+ \Big((1-\nu_5)\alpha_2 + \nu_5\beta_2 \Big)(1-\theta_5) \\
& = \alpha_2\theta_5+(1-\nu_5)\alpha_2 + \nu_5\beta_2-\alpha_2\theta_5+\nu_5\alpha_2\theta_5-\nu_5\beta_2\theta_5\\
& = (1-\nu_5)\alpha_2 + \nu_5\beta_2+\nu_5(\alpha_2-\beta_2)\theta_5 \\
& = (1-\nu_5)\alpha_2 + \nu_5\beta_2+\nu_5(\alpha_2-\beta_2)\frac{\nu_6}{\nu_5}\\
& = (1-\nu_5)\alpha_2 + \nu_5\beta_2 + \nu_6(\alpha_2-\beta_2) \\
& = \alpha_2-(\nu_5-\nu_6)(\alpha_2-\beta_2).
\end{align*}
Next, applying the real interpolation functor
$(\,\cdot\,,\,\cdot\,)_{\theta_6, r}$ with
$\theta_6 = s/\nu_5 \in (0, 1)$ to \eqref{d4} and \eqref{d6'} gives
\begin{equation}\label{d8}
\|\CT_2(\lambda)f\|_{\CB^s_{q,r}(\HS)}
\leq C |\lambda|^{\delta_6}
\|f\|_{\CB^{\frac{\nu_6}{\nu_5} s}_{q,r}(\HS)} 
\end{equation}
with
\begin{align*}
\delta_6 & =\delta_5 \theta_6 +\alpha_2(1-\theta_6) \\
& = \Big(\alpha_2-(\nu_5-\nu_6)(\alpha_2-\beta_2) \Big) \theta_6
+\alpha_2-\alpha_2\theta_6 \\
& = \alpha_2 - (\nu_5-\nu_6)(\alpha_2-\beta_2)\theta_6 \\
& = \alpha_2-(\nu_5-\nu_6) (\alpha_2-\beta_2) \frac{s}{\nu_5} .
\end{align*}
Finally, we combine \eqref{d7} and \eqref{d8} with real interpolation. To this end,
we choose $\nu_5$ and $\nu_6$ in such a way that  $(\nu_6/\nu_5)s < s-\sigma < s$, i.e.,
$0 < \nu_6 < (1-\sigma\slash s)\nu_5$. In addition, we take $\theta_7 \in (0, 1)$ such that
$s-\sigma = (1-\theta_7)s + \theta_7(\nu_6/\nu_5)s$, i.e.,  
$\theta_7 = \nu_5\sigma \slash s(\nu_5-\nu_6)$ is fulfilled. 
Then, applying the real interpolation functor
$(\,\cdot\,,\,\cdot\,)_{\theta_7, r}$ to \eqref{d7} and \eqref{d8} implies
$$\|\CT_2(\lambda)f\|_{\CB^s_{q,r}(\HS)} \leq C|\lambda|^{\delta_7}\|f\|_{B^{s-\sigma}_{q,r}(\HS)}$$
with
\begin{align*}
\delta_7& = (1-\theta_7)\alpha_2 + \theta_7 \delta_6 \\
& = \bigg(1-\frac{\sigma\nu_5}{s(\nu_5-\nu_6)}\bigg)\alpha_2
+ \frac{\sigma\nu_5}{s(\nu_5-\nu_6)}\bigg(\alpha_2-(\nu_5-\nu_6)(\alpha_2-\beta_2)\frac{s}{\nu_5} \bigg)\\
& = \alpha_2-(\alpha_2-\beta_2)\sigma \\
& = (1-\sigma)\alpha_2 + \beta_2\sigma.
\end{align*}
Thus, we have \eqref{A.4}, which completes the proof of Lemma \ref{lem.s.3}. 
\end{proof}
\begin{lem}\label{lem.s.4} Assume that Assumption \ref{assump:4.1} holds. 
Let $q$, $r$, $\epsilon$, $\gamma$, $\alpha_2$, and $\beta_2$ be 
the same numbers as in Theorem \ref{spectralthm:1}. 
Assume  $-1+1/q < s < 0$ and let $\sigma>0$ be  numbers such that  $-1+1/q <  s-\sigma  < 0$. 
Then, for every $\lambda \in \Sigma_\epsilon + \gamma_b$ and $f \in C^\infty_0(\HS)$, 
there hold
\begin{align}
\|\CT_2(\lambda)f\|_{\CB^s_{q,r}(\HS)} & \leq C|\lambda|^{\alpha_2}\|f\|_{\CB^s_{q,r}(\HS)},
\label{A.3*}\\
\|\CT_2(\lambda)f\|_{\CB^s_{q,r}(\HS)} & \leq C|\lambda|^{(1-\sigma)\alpha_2+\sigma\beta_2}
\|f\|_{\CB^{s-\sigma}_{q,r}(\HS)}. \label{A.4*}
\end{align}
\end{lem}
\begin{proof}
Combining \eqref{d1*}, \eqref{d2*}, and \eqref{d3*} for $i=2$ with the complex 
interpolation method for $|s| < \nu_8, \nu_9 < 1-1/q = 1/q'$, we obtain 
\begin{align}
\label{g1}
\| \CT_2(\lambda)^*\varphi\|_{L_{q'}(\HS)} 
& \leq C|\lambda|^{\alpha_2}\|\varphi\|_{L_{q'}(\HS)}, \\
\label{g2}
\|\CT_2(\lambda)^*\varphi\|_{\CH^{\nu_8}_{q'}(\HS)} 
& \leq C|\lambda|^{\alpha_2}\|\varphi\|_{\CH^{\nu_8}_{q'}(\HS)},
\\
\label{g2*}
\|\CT_2(\lambda)^*\varphi\|_{\CH^{\nu_9}_{q'}(\HS)} 
& \leq C|\lambda|^{\alpha_2}\|\varphi\|_{\CH^{\nu_9}_{q'}(\HS)}, \\
\label{g3}
\|\CT_2(\lambda)^*\varphi\|_{\CH^{\nu_8}_{q'}(\HS)} 
& \leq C|\lambda|^{(1-\nu_8)\alpha_2+\nu_8\beta_2}\|\varphi\|_{L_{q'}(\HS)}.
\end{align}
Thus, it follows from the duality argument that
\begin{align}
\label{h1}
\|\CT_2(\lambda)f\|_{L_{q}(\HS)} 
& \leq C|\lambda|^{\alpha_2}\|f\|_{L_{q}(\HS)}, \\
\label{h2}
\|\CT_2(\lambda)f\|_{\CH^{-\nu_8}_q(\HS)}
&\leq C|\lambda|^{\alpha_2}\|f\|_{\CH^{-\nu_8}_q(\HS)}, \\
\label{h2*} 
\|\CT_2(\lambda)f\|_{\CH^{-\nu_9}_{q}(\HS)} 
& \leq C|\lambda|^{\alpha_2}\|f\|_{\CH^{-\nu_9}_{q}(\HS)},
\\
\label{h3}
\| \CT_2(\lambda)f\|_{L_q(\HS)} 
& \leq C|\lambda|^{((1-\nu_8)\alpha_2 + \nu_8\beta_2)}\|f\|_{\CH^{-\nu_8}_{q}(\HS)}.
\end{align}
Noting that $-1 +1/q < -\nu_8 < -|s| < 0$ and combining \eqref{h1} and \eqref{h2} 
with real interpolation, we have 
\begin{equation}\label{h4} 
\|\CT_2(\lambda)f\|_{\CB^{-|s|}_{q,r}(\HS)}
\leq C|\lambda|^{\alpha_2}\|f\|_{\CB^{-|s|}_{q,r}(\HS)},
\end{equation}
which gives \eqref{A.3*}. \par
To show \eqref{A.4*}, we choose $\theta_8 \in (0, 1)$ in such a way that 
$|s| = \nu_9\theta_8$ and combining \eqref{h2*} and \eqref{h3} with real interpolation, 
we obtain
\begin{equation}\label{h5}
\|\CT_2(\lambda)f\|_{\CB^{-|s|}_{q,r}(\HS)}
\leq C|\lambda|^{\delta_8}
\|f\|_{\CB^{\sigma_2}_{q,r}(\HS)},
\end{equation}
where we have set
\begin{align}
\delta_8& = \alpha_2 \theta_8 + \Big((1-\nu_8)\alpha_2 + \nu_8\beta_2 \Big)(1-\theta_8) \\
& = \alpha_2-\nu_8(\alpha_2-\beta_2) + \nu_8(\alpha_2-\beta_2)\theta_8 \\
&= \alpha_2-\nu_8(\alpha_2-\beta_2) + \frac{\nu_8|s|}{\nu_9}(\alpha_2-\beta_2) \\
& =\alpha_2-\nu_8(\alpha_2-\beta_2) \bigg(1-\frac{|s|}{\nu_9} \bigg)
\end{align}
and
\begin{align}
\sigma_2 &=-\nu_9\theta_8 - \nu_8 (1-\theta_8)
= -\nu_9\frac{|s|}{\nu_9}-\nu_8 \bigg(1-\frac{|s|}{\nu_9} \bigg)
= -|s|-\nu_8 \bigg(1-\frac{|s|}{\nu_9} \bigg)
= -\bigg\{|s|+\nu_8 \bigg(1-\frac{|s|}{\nu_9} \bigg)\bigg\}.  
\end{align}
Now, we choose $\nu_9 \in (0, 1)$ in such a way that 
\begin{equation}
\label{const:1} 
-|s| > -|s|-\sigma > \sigma_2 = -|s|-\nu_8 \bigg(1-\frac{|s|}{\nu_9} \bigg)
\qquad \Leftrightarrow \qquad
\frac{\nu_8|s|}{\nu_8-\sigma} < \nu_9 < 1- \frac1q.
\end{equation}
Since $\sigma>0$ may be chosen so small that $\nu_8/(\nu_8-\sigma)$
is so close to $1$, it is possible choose $\nu_9$ in such a way that  
$|s| < \nu_9$ and \eqref{const:1} is fulfilled. 
Finally, we choose $\theta_9 \in (0, 1)$ such that 
\begin{equation}
-|s|-\sigma = -|s|\theta_9 -\bigg\{|s|+\nu_8\bigg(1-\frac{|s|}{\nu_9}\bigg)\bigg\}(1-\theta_9)
\qquad \Leftrightarrow \qquad
\theta_9 = 1 - \frac{\sigma\nu_9}{\nu_8(\nu_9-|s|)}. 
\end{equation}
Combining \eqref{h4} and \eqref{h5} with the real interpolation method implies that 
$$\|\CT_2(\lambda)f\|_{\CB^{-|s|}_{q,r}(\HS)}
\leq C|\lambda|^{\delta_9}\|f\|_{\CB^{-|s|-\sigma}_{q,r}(\HS)}
$$
with
\begin{align*} 
\delta_9&=\alpha_2\theta_9 + \delta_8(1-\theta_9) \\
& = \alpha_2 \bigg(1-\frac{\sigma\nu_9}{\nu_8(\nu_9-|s|)} \bigg) 
+ \bigg\{\alpha_2-\nu_8(\alpha_2-\beta_2) \bigg( 1-\frac{|s|}{\nu_9} \bigg) \bigg\}
\frac{\sigma\nu_9}{\nu_8(\nu_9-|s|)}\\
& = \alpha_2-\frac{\sigma\alpha_2\nu_9}{\nu_8(\nu_9-|s|)}
+ \frac{\alpha_2\sigma \nu_9}{\nu_8(\nu_9-|s|)}
-\frac{\nu_8(\alpha_2-\beta_2)(\nu_9-|s|)\sigma \nu_9}{\nu_9\nu_8(\nu_9-|s|)} \\
& = \alpha_2-(\alpha_2-\beta_2)\sigma \\ 
& = (1-\sigma)\alpha_2 + \beta_2\sigma.
\end{align*}
Hence, we have 
$$\|\CT_2(\lambda)f\|_{\CB^{-|s|}_{q,r}(\HS)}
\leq C|\lambda|^{(1-\sigma)\alpha_2 + \beta_2\sigma}\|f\|_{\CB^{-|s|-\sigma}_{q,r}(\HS)}.
$$
Namely, we have  \eqref{A.4*}. The proof of Lemma \ref{lem.s.4} is complete.
\end{proof}
\begin{proof}[Proof of Theorem \ref{spectralthm:1}.] 
By virtue of Lemmas \ref{lem.s.1}--\ref{lem.s.4}, 
it suffices to prove the case $s=0$. 
Let  $0 < \sigma < 1/q$ and let $\omega_1$ and $\omega_2$ be positive numbers such that  
$-1+1/q < -\omega_1 < -\omega_1 + \sigma < 0 <  \omega_2 < \sigma + \omega_2  < 1/q$. 
 By Lemmas \ref{lem.s.1}
and \ref{lem.s.2}, we have
$$\|\CT_1(\lambda)f\|_{B^{-\omega_1}_{q,r}(\HS)} \leq C|\lambda|^{\alpha_1}\|f\|_{B^{ -\omega_1}_{q,r}(\HS)},
\quad \|\CT_1(\lambda)f\|_{\CB^{\omega_2}_{q,r}(\HS)} \leq C|\lambda|^{\alpha_1}\|f\|_{\CB^{ \omega_2}_{q,r}(\HS)}.$$
Let $\theta_{10} \in (0, 1)$ such that $0 = (1-\theta_{10})(-\omega_1) + \theta_{10} \omega_2$, 
i.e., $\theta_{10} = \omega_1 \slash (\omega_1 + \omega_2)$.
Since $\CB^0_{q,r}(\HS) = (\CB^{-\omega_1}_{q,r}(\HS), \CB^{\omega_2}_{q,r}(\HS))_{\theta_{10}, r}$, 
by  real interpolation we have 
$$\|\CT_1(\lambda)f\|_{\CB^{0}_{q,r}(\HS)} \leq C|\lambda|^{\alpha_1}\|f\|_{\CB^{0}_{q,r}(\HS)}. $$
Moreover,  by Lemmas \ref{lem.s.1}
and \ref{lem.s.2}, we have
$$\|\CT_1(\lambda)f\|_{\CB^{-\omega_1}_{q,r}(\HS)} \leq C|\lambda|^{(1-\sigma)\alpha_1 + \sigma\beta_1}
\|f\|_{\CB^{-\omega_1+\sigma}_{q,r}(\HS)}, \quad
\|\CT_1(\lambda)f\|_{\CB^{\omega_2}_{q,r}(\HS)} \leq C|\lambda|^{(1-\sigma)\alpha_1 + \sigma\beta_1}
\|f\|_{\CB^{\omega_2+\sigma}_{q,r}(\HS)}.
$$
Since $\CB^\sigma_{q,r}(\HS) 
= (\CB^{-\omega_1+\sigma}_{q,r}(\HS), \CB^{\omega_2+\sigma}_{q,r}(\HS))_{\theta_{10}, r}$, 
by real interpolation, we have 
$$\|\CT_1(\lambda)f\|_{\CB^{0}_{q,r}(\HS)} 
\leq C|\lambda|^{(1-\sigma)\alpha_1 + \sigma\beta_1}\|f\|_{\CB^{\sigma}_{q,r}(\HS)}. $$
Analogously, we obtain 
\begin{align*}
\|\CT_2(\lambda)f\|_{\CB^{0}_{q,r}(\HS)} &\leq C|\lambda|^{\alpha_2}\|f\|_{\CB^{0}_{q,r}(\HS)}, \\
\|\CT_2(\lambda)f\|_{\CB^{0}_{q,r}(\HS)} &\leq C|\lambda|^{(1-\sigma)\alpha_2 + \sigma\beta_2}
\|f\|_{\CB^{-\sigma}_{q,r}(\HS)}. 
\end{align*}
The proof of Theorem \ref{spectralthm:1} is compelte.
\end{proof}

\subsection{Estimates of the integral operators}

We shall start with the estimates of $\bL_1(\lambda)$.
\begin{thm}
\label{thm:4.1} 
Let $1 < q < \infty$, $1 \le r \leq \infty$, 
$- 1 + 1 \slash q < s <  1 \slash q$, $0 < \epsilon < \pi \slash 2$, and $\gamma>0$. 
Let $\sigma$ be a positive  
number such that $- 1 + 1 \slash q < s \pm \sigma < 1 \slash q$. 
Assume that $m_{2} (\lambda, \xi')$ is  
a  symbol belonging to $\BM_{-2}$. 
Then, 
for every $\lambda \in \Sigma_{\epsilon} + \gamma_b$
 and $f \in C^\infty_0(\HS)$, 
there exists a constant $C_b$ depending on $\gamma_b$  such that	
\begin{align}
\label{b:4.1*}
\lVert (\lambda, \lambda^{1/2} \nabla, \nabla_b^2) \bL_{1} (\lambda) f \rVert_{\CB^s_{q, r} (\HS)} 
& \le C_b \lVert f \rVert_{ \CB^{s}_{q, 1} (\HS)},   \\
\label{b:4.2*}
\lVert (\lambda, \lambda^{1/2} \nabla, 
  \nabla_b^2) \bL_{1} (\lambda) f \rVert_{\CB^s_{q, r} (\HS)} 
& \le C_b \lvert \lambda \rvert^{- \frac\sigma2} 
\lVert f \rVert_{\CB^{s + \sigma}_{q, r} (\HS)}, \\ 
\label{b:4.4}
\lVert \bL_{1} (\lambda) f \rVert_{\CB^s_{q, r} (\HS)} 
& \le C_b \lvert \lambda \rvert^{- (1 - \frac\sigma2)}
\lVert f \rVert_{ \CB^{s - \sigma}_{q, r} (\HS)},  \\
\label{b:4.3*}		
\lVert (\lambda, \lambda^{1/2} \nabla, 
  \nabla_b^2) \pd_\lambda \bL_{1} (\lambda) f \rVert_{\CB^s_{q, r} (\HS)} 
& \le C_b \lvert \lambda \rvert^{- (1 - \frac\sigma2)} 
\lVert f \rVert_{\CB^{s - \sigma}_{q, r} (\HS)}.
\end{align}
\end{thm}

\begin{proof}
As has been seen in the proof of Theorem \ref{thm:3.1}, if $\CB^s_{q,r} = B^s_{q,r}$ and 
 $\lambda \in \Sigma_\epsilon + \gamma$,  it suffices to prove that 
\begin{align}
\label{b:4.1}
\lVert (\lambda, \lambda^{1/2} \nabla, \nabla^2) \bL_{1} (\lambda) f \rVert_{\CB^s_{q, r} (\HS)} 
 &\le C_b \lVert f \rVert_{ \CB^{s}_{q, 1} (\HS)},\\
\lVert (\lambda, \lambda^{1/2} \nabla, 
  \nabla^2) \bL_{1} (\lambda) f \rVert_{\CB^s_{q, r} (\HS)} 
& \le C_b \lvert \lambda \rvert^{- \frac\sigma2} 
\lVert f \rVert_{\CB^{s + \sigma}_{q, r} (\HS)}, \label{b:4.2}\\ 	
\lVert (\lambda, \lambda^{1/2} \nabla, 
  \nabla^2) \pd_\lambda \bL_{1} (\lambda) f \rVert_{\CB^s_{q, r} (\HS)} 
& \le C_b \lvert \lambda \rvert^{- (1 - \frac\sigma2)} 
\lVert f \rVert_{\CB^{s - \sigma}_{q, r} (\HS)}, \label{b:4.3}	
\end{align}
instead of  \eqref{b:4.1*}, \eqref{b:4.2*}, and \eqref{b:4.3*}, respectively.
 \par 
First, we consider the case $0 < s < 1/q$.  In view of Assumption \ref{assump:4.1} 
with $\alpha_1=0$ and $\beta_1=-1/2$,
to prove \eqref{b:4.1} and 
\eqref{b:4.2}, it suffices to prove that 
\begin{align}
\label{5.5.21}
\lVert (\lambda, \lambda^{1/2} \nabla, 
  \nabla^2) \bL_{1} (\lambda) f \rVert_{L_q (\HS)} 
& \le C \lVert f \rVert_{L_q (\HS)}, \\
\label{5.5.22}
\|(\lambda, \lambda^{1/2} \nabla, 
 \nabla^2) \bL_{1} (\lambda) f \|_{\CH^1_q (\HS)} 
& \le C \| f \|_{\CH^1_q (\HS)}. \\
\label{5.5.3} 
\lVert (\lambda, \lambda^{1/2} \nabla, 
  \nabla^2) \bL_{1} (\lambda) f \rVert_{L_q (\HS)} 
& \le C_b \lvert \lambda \rvert^{- 1 \slash 2} \lVert f \rVert_{\CH^1_q (\HS)}.
\end{align}
To prove \eqref{5.5.21}, \eqref{5.5.22}, and  \eqref{5.5.3}, by integration by parts
we write 
\begin{equation}\label{5.5.1}
\begin{aligned}
&\lambda^{n_0 \slash 2} \pd^{n_1}_d \pd_{x'}^{\alpha'} \bL_{1} (\lambda)f \\
&\quad = (- 1)^{n_1} \int^\infty_0 \CF^{- 1}_{\xi'}
[m_2 (\lambda, \xi') \lambda^{n_0 \slash 2} B^{n_1 - n_2 } (i \xi')^{\alpha'} 
B e^{- B (x_d + y_d)} \CF'[\pd_d^{n_2} f] (\xi', y_d)] (x') \d y_d	
\end{aligned}\end{equation}
for nonnegative integers $n_0$, $n_1$,  $n_2$,  and 
for multi-indices $\alpha'$. If it holds
\begin{equation}
\label{cond-L11}
n_0 + n_1 - n_2  + \lvert \alpha' \rvert = 2 
\end{equation}
we see that 
$
m_2 (\lambda, \xi') \lambda^{n_0 \slash 2} B^{n_1 - n_2} 
(i \xi')^{\alpha' } \in \BM_0,
$
which together with Proposition \ref{prop:3.1} implies that
\begin{equation}
\label{5.5.0}
\lVert \lambda^{n_0 \slash 2} \pd^{n_1}_d \pd_{x'}^{\alpha'} \bL_{1} (\lambda)f\rVert_{L_q (\HS)}
\le C \lVert \pd_d^{n_2} f \rVert_{L_q (\HS)}.
\end{equation}
Therefore, by virtue of suitable choices of 
$(n_0, n_1, n_2, \alpha')$ satisfying \eqref{cond-L11}, 
we obtain \eqref{5.5.21}, \eqref{5.5.22}, and \eqref{5.5.3}. 
\par 
To prove \eqref{b:4.4}, in view of Assumption \ref{assump:4.1} with $\alpha_2=-1$ and 
$\beta_2=-1/2$, it suffices to prove that
\begin{align}
\|\bL_1(\lambda)f\|_{L_q(\HS)} &\leq C|\lambda|^{-1}\|f\|_{L_q(\HS)}, \label{5.5.21*}\\
\|\bL_1(\lambda)f\|_{\CH^1_q(\HS)} &\leq C|\lambda|^{-1}\|f\|_{\CH^1_q(\HS)}, \label{5.5.22*} \\
\|\bL_1(\lambda)f\|_{\CH^1_q(\HS)}& \leq C_b|\lambda|^{-1/2}\|f\|_{L_q(\HS)}. \label{5.5.23*}
\end{align}
By \eqref{5.5.21}, we have \eqref{5.5.21*}. By \eqref{5.5.22}, we obtain \eqref{5.5.22*}.
Choosing $(n_0,n_1,n_2,\lvert \alpha' \rvert) = (1,1,0,0), (1,0,0,1)$
in \eqref{5.5.0}, we see that
$$\|\nabla \bL_1(\lambda)f\|_{L_q(\HS)} \leq C|\lambda|^{-1/2}\|f\|_{L_q(\HS)}. $$
If $\lambda \in \Sigma_\epsilon + \gamma$, then there holds $|\lambda| \geq \gamma\sin\epsilon$,
and hence it follows from \eqref{5.5.21*} that
$$\|\bL_1(\lambda)f\|_{L_q(\HS)} \leq C|\lambda|^{-1}\|f\|_{L_q(\HS)}
\leq C(\gamma\sin \epsilon)^{-1/2}|\lambda|^{-1/2}\|f\|_{L_q(\HS)}.$$
Thus, the estimate \eqref{5.5.23*} follows. \par
We now consider the dual operator 
$\bL_{1}(\lambda)^*$ acting on $\varphi \in C^\infty_0(\HS)$, which is defined by
\begin{equation}
\bL_{1}(\lambda)^*\varphi= \int^\infty_0\CF'[m_2(\lambda)Be^{-B(x_d+y_d)}
\CF^{-1}_{\xi'}[\varphi](\xi', x_d)](x')\d x_d.
\end{equation}
In fact, it follows from the Fubini theorem that
\begin{equation}
\label{L11-dual-representation}
\begin{split}
(\bL_{1} (\lambda) f, \varphi)
& = \int_0^\infty \bigg\{\int_{\HS} \CF^{- 1}_{\xi'}
[m_2 (\lambda, \xi') B e^{- B (x_d + y_d)} 
\CF'[ f] (\xi', y_d)] (x') \varphi(x) \d x\bigg\} \d y_d \\
& = \int_0^\infty\bigg\{  \int_0^\infty \int_{\BR^{d - 1}}
m_2 (\lambda, \xi') B e^{- B (x_d + y_d)} \CF'[ f] (\xi', y_d) \CF^{-1}_{\xi'}[\varphi] (\xi', x_d)
\d \xi' \d x_d\bigg\} \d y_d \\
& = \int_0^\infty\int_{\BR^{d-1}} 
f (y', y_d)\bigg\{\int^\infty_0
\CF'[m_2 (\lambda, \xi') B e^{- B (x_d + y_d)} \CF^{-1}_{\xi'}\varphi (\xi', x_d)](y')
\d x_d\bigg\}\d y' \d y_d  \\
& = (f, \bL_{1} ( \lambda)^* \varphi)	
\end{split}
\end{equation}
for $f, \varphi \in C^\infty_0 (\HS)$. 
Notice that there holds
$((\lambda, \lambda^{1/2} \nabla, 
\nabla^2) \bL_{1}(\lambda)f, \varphi)
= (f, (\lambda, \lambda^{1/2} \nabla, 
\nabla^2) \bL_{1}(\lambda)^*\varphi)$
for $f, \varphi \in C^\infty_0(\HS)$. Hence, if we set $T_1(\lambda)f
= (\lambda, \lambda^{1/2} \nabla, 
\nabla^2) \bL_{1}(\lambda)f$, then $T_1(\lambda)^*\varphi 
=(\lambda, \lambda^{1/2} \nabla, 
\nabla^2) \bL_{1}(\lambda)^*\varphi$. 
Since $\bL_1(\lambda)^*$ is obtained by exchanging $\CF^{-1}_{\xi'}$ and  $\CF'$, 
in the same manner as in the 
proof of  \eqref{5.5.22}--\eqref{5.5.3} and \eqref{5.5.21*}--\eqref{5.5.23*},  we see that 
\begin{equation}\label{addit:5.1}\begin{aligned}
\|(\lambda, \lambda^{1/2} \nabla,
\nabla^2)\bL_{1}(\lambda)^*\varphi\|_{L_{q'}(\HS)} &\le C\|\varphi\|_{L_{q'}(\HS)},
\\
\|(\lambda, \lambda^{1/2} \nabla,
 \nabla^2)\bL_{1}(\lambda)^*\varphi\|_{\CH^1_{q'}(\HS)} &\le C\|\varphi\|_{ \CH^1_{q'}(\HS)},
\\
\|(\lambda,\lambda^{1/2} \nabla,
 \nabla^2)\bL_{1}(\lambda)^*\varphi\|_{L_{q'}(\HS)} &\le C_b|\lambda|^{- \frac12}
\|\varphi\|_{\CH^1_{q'}(\HS)}, \\
\|\bL_{1}(\lambda)^*\varphi\|_{L_{q'}(\HS)} &\le C|\lambda|^{-1}\|\varphi\|_{L_{q'}(\HS)},
\\
\|\bL_{1}(\lambda)^*\varphi\|_{\CH^1_{q'}(\HS)} &\le C|\lambda|^{-1}\|\varphi\|_{ \CH^1_{q'}(\HS)},
\\
\|\bL_{1}(\lambda)^*\varphi\|_{\CH^1_{q'}(\HS)} &\le C_b|\lambda|^{- \frac12}
\|\varphi\|_{L_{q'}(\HS)},
\end{aligned}\end{equation}
which together with \eqref{5.5.21}, \eqref{5.5.22}, \eqref{5.5.3}, \eqref{5.5.21*},
\eqref{5.5.22*}, and \eqref{5.5.23*} implies that
$(\lambda, \lambda^{1/2}\nabla, \nabla^2)\bL_1(\lambda)$ satisfies Assumption
\ref{assump:4.1} with $i = 1$, $\alpha_1=0$, and $\beta_1=-1/2$, 
whereas $\bL_1(\lambda)$ satisfies Assumption \ref{assump:4.1}
with $i = 2$, $\alpha_2=-1$, and $\beta_2=-1/2$. 
Thus, by Theorem \ref{spectralthm:1}, we have \eqref{b:4.1}, \eqref{b:4.2} and \eqref{b:4.3}. 
\par

We now consider $\pd_\lambda \bL_{1}(\lambda)f$. To  prove \eqref{b:4.3},
in view of Assumption \ref{assump:4.1} with $i = 2$,
$\alpha_2=-1$, and $\beta_2=-1/2$, we shall prove that 
\begin{align}
\label{n5.6.10}
\| (\lambda, \lambda^{1/2} \nabla, 
  \nabla^2) \pd_\lambda \bL_{1} (\lambda) f \|_{L_q(\HS)}
& \le C|\lambda|^{-1}\|f\|_{L_q(\HS)}, \\
\label{n5.6.11}
\lVert (\lambda, \lambda^{1/2} \nabla, 
  \nabla^2) \pd_\lambda \bL_{1} (\lambda) f \rVert_{\CH^1_q (\HS)}
& \le C \lvert \lambda \rvert^{- 1} \lVert f \rVert_{\CH^1_q (\HS)}, \\
\label{n5.6.12}
\lVert (\lambda, \lambda^{1/2} \nabla, 
  \nabla^2) \pd_\lambda \bL_{1} (\lambda) f \rVert_{\CH^1_q (\HS)}
& \le C_b|\lambda|^{- \frac12} \lVert f \rVert_{L_q (\HS)}.
\end{align}
To this end, we note that
\begin{equation}\label{5.5.10}
\begin{aligned}
\pd_\lambda \bL_{1} (\lambda) f 
& = \int^\infty_0 \CF^{- 1}_{\xi'}
\Big[ \Big(\pd_\lambda m_2 (\lambda, \xi') \Big) 
B e^{- B (x_d + y_d)} \CF'[ f ](\xi', y_d) \Big] (x') \d y_d \\
& \quad + \int^\infty_0 \CF^{- 1}_{\xi'}
\Big[m_2 (\lambda, \xi') B^{- 2}
\Big(B^2 \pd_\lambda (B e^{- B (x_d + y_d)}) \Big)
\CF'[ f] (\xi', y_d) \Big] (x') \d y_d.		
\end{aligned}\end{equation}
By integration by parts, we may write  
\begin{align}
& \lambda^{n_0 \slash 2} \pd^{n_1}_d \pd_{x'}^{\alpha'} \pd_\lambda \bL_{1} (\lambda)f \\
& = (- 1)^{n_1} \bigg\{ \int^\infty_0 \CF^{- 1}_{\xi'}
\Big[ \big(\pd_\lambda m_2 (\lambda, \xi') \big) 
\lambda^{n_0 \slash 2} B^{n_1 - n_2} (i \xi')^{\alpha' } 
B e^{- B (x_d + y_d)} \CF' [\pd_d^{n_2} f] (\xi', y_d) \Big] (x') \d y_d \\
& \quad +  \int^\infty_0 \CF^{- 1}_{\xi'}
\Big[m_2 (\lambda, \xi') \lambda^{n_0 \slash 2} B^{n_1 - n_2  - 2} 
(i \xi')^{\alpha'} \Big(B^2 \pd_\lambda (B e^{- B (x_d + y_d)}) \Big)
\CF'[\pd_d^{n_2}f] (\xi', y_d) \Big] (x') \d y_d \\
& \quad + C_{n_1, n_2} \int_0^\infty \CF^{- 1}_{\xi'}
[ m_2 (\lambda, \xi') \lambda^{n_0 \slash 2} B^{n_1 - n_2 - 2} 
(i \xi')^{\alpha'} B e^{- B (x_d + y_d)}
\CF' [\pd_d^{n_2}f] (\xi', y_d) \Big] (x') \d y_d \bigg\}
\end{align}
for nonnegative integers $n_0$, $n_1$, and $n_2$,  and 
for multi-indices $\alpha'$. 
Here, $C_{n_1, n_2}$ is some constant that depends only on
$\mu$, $n_1$, and $n_2$. Indeed, it follows that
\begin{align}
\pd_d \pd_\lambda (B e^{- B (x_d + y_d)} )
& = \pd_d \bigg( \frac{1}{2 \mu B}e^{- B (x_d + y_d)} 
- \frac{x_d + y_d}{2 \mu} e^{- B (x_d + y_d)} \bigg) \\
& = - \frac{(- B) (x_d + y_d)}{2 \mu} e^{- B (x_d + y_d)} 
- \frac1\mu e^{- B (x_d + y_d)} \\
& = B \pd_\lambda (B e^{- B (x_d + y_d)} ) 
- \frac{1}{2 \mu} e^{- B (x_d + y_d)},
\end{align}
and then, by induction on  $n \in \BN_0$ it follows that
\begin{equation}
\pd_d^n \pd_\lambda (B e^{- B (x_d + y_d)} ) 
= B^n \pd_\lambda (B e^{- B (x_d + y_d)} ) 
- \frac{n}{2 \mu} (- B)^{n - 1} e^{- B (x_d + y_d)}.
\end{equation}
Since $m_2 (\lambda, \xi') \in \BM_{- 2}$ 
and  $\pd_\lambda m_2 (\lambda, \xi') \in \BM_{- 4}$, 
we have
\begin{align}
(\pd_\lambda m_2 (\lambda, \xi')) \lambda^{n_0 \slash 2} 
B^{n_1 - n_2} (i \xi')^{\alpha'} & \in \BM_0, \\
m_2 (\lambda, \xi') \lambda^{n_0 \slash 2} B^{n_1 - n_2 - 2} 
(i \xi')^{\alpha'} & \in \BM_0
\end{align} 
provided that 
\begin{equation}
\label{cond-L11-lambda}
n_0 + n_1 - n_2 + \lvert \alpha' \rvert = 4
\end{equation}
Together with Proposition \ref{prop:3.1}, it holds
\begin{equation}
\label{5.5.0-lambda}
\lVert \lambda^{n_0} \pd^{n_1}_d \pd_{x'}^{\alpha'} 
\pd_\lambda \bL_{1} (\lambda)f\rVert_{L_q (\HS)}
\le C \lVert \pd_d^{n_2} f \rVert_{L_q (\HS)}
\end{equation}
under the condition \eqref{cond-L11-lambda}.
Hence, by virtue of suitable choices of 
$(n_0, n_1, n_2,  \alpha')$ satisfying \eqref{cond-L11-lambda},
we have \eqref{n5.6.10}, \eqref{n5.6.11} with $\CH^1_q = \dot H^1_q$,
and \eqref{n5.6.12} with $\CH^1_q = \dot H^1_q$. The estimates
\eqref{n5.6.11} and \eqref{n5.6.12} for the case $\CH^1_q = H^1_q$ are obtained as follows:
The estimates \eqref{n5.6.10} and \eqref{n5.6.11} with $\CH^1_q = \dot H^1_q$ combine 
to give \eqref{n5.6.11}.
Moreover, we infer from \eqref{n5.6.10} that
$$\|(\lambda, \lambda^{1/2}\nabla, \nabla^2)\pd_\lambda \bL_1(\lambda)f\|_{L_q(\HS)}
\leq C|\lambda|^{-1}\|f\|_{L_q(\HS)} \leq C(\gamma \sin\epsilon)^{-1/2}|\lambda|^{-1/2}
\|f\|_{L_q(\HS)}
$$
for $\lambda \in \Sigma_\epsilon + \gamma$, 
which together with \eqref{n5.6.12} with $\CH^1_q=\dot H^1_q$ yields
\eqref{n5.6.12} with $\CH^1_q= H^1_q$. This completes the proof of
\eqref{n5.6.10}--\eqref{n5.6.12}. \par
Finally, we consider the dual operator $\pd_\lambda \bL_1(\lambda)^*$. 
 In view of  \eqref{5.5.10}, it is defined by 
\begin{align}
\pd_\lambda \bL_{1} (\lambda)^* \varphi
& = \int^\infty_0 \CF'
\Big[ \Big(\pd_\lambda m_2 (\lambda, \xi') \Big) 
B e^{- B (x_d + y_d)} \CF^{-1}_{\xi'}[ \varphi ](\xi', x_d) \Big] (y') \d x_d \\
& \quad + \int^\infty_0 \CF'
\Big[m_2 (\lambda, \xi') B^{- 2}
\Big(B^2 \pd_\lambda (B e^{- B (x_d + y_d)}) \Big)
\CF^{-1}_{\xi'}[ \varphi] (\xi', x_d) \Big] (y') \d x_d.	
\end{align}
Since $f$, $\varphi \in C^\infty_0(\HS)$, we see that 
$\pd_\lambda \bL_1(\lambda)^*(\lambda, \lambda^{1/2}\nabla, \nabla^2)\varphi
= (\lambda, \lambda^{1/2}\nabla, \nabla^2)\pd_\lambda \bL_1(\lambda)^*\varphi$.
Thus, if we set  $T_2(\lambda)f =  (\lambda, \lambda^{1/2}\nabla, \nabla^2)\pd_\lambda \bL_1(\lambda)f$, 
then $T_2(\lambda)^*\varphi 
= (\lambda, \lambda^{1/2}\nabla, \nabla^2)\pd_\lambda \bL_1(\lambda)^*\varphi$. 
Since $\pd_\lambda \bL_1(\lambda)^*\varphi$ is obtained by exchanging 
$\CF^{-1}_{\xi'}$ and  $\CF'$ in \eqref{5.5.10}, employing the 
same argument as in the proof of \eqref{n5.6.10} -- \eqref{n5.6.12}, 
we have 
\begin{align}
\|(\lambda, \lambda^{1/2} \nabla, 
\nabla^2)\pd_\lambda \bL_{1}(\lambda)^*\varphi\|_{L_{q'}(\HS)}
& \le C|\lambda|^{-1}\|\varphi\|_{L_{q'}(\HS)}, \\
\|(\lambda, \lambda^{1/2} \nabla, 
 \nabla^2)\pd_\lambda \bL_{1}(\lambda)^*\varphi\|_{\CH^1_{q'}(\HS)}
& \le C|\lambda|^{-1}\|\varphi\|_{\wt H^1_{q'}(\HS)}, \\
\|(\lambda, \lambda^{1/2} \nabla, 
 \nabla^2)\pd_\lambda \bL_{1}(\lambda)^*\varphi\|_{\CH^1_{q'}(\HS)}
& \le C_b|\lambda|^{- \frac12}\|\varphi\|_{L_{q'}(\HS)},
\end{align}
which together with \eqref{n5.6.10}, \eqref{n5.6.11}, and \eqref{n5.6.12} implies
that $(\lambda, \lambda^{1/2} \nabla, \nabla^2)\pd_\lambda \bL_{1}(\lambda)$ satisfies
Assumption \ref{assump:4.1} with $i = 2$, $\alpha_2=-1$, and $\beta_2=-1/2$.
Hence, by Theorem \ref{spectralthm:1}, we have \eqref{b:4.3}. The proof of
Theorem \ref{thm:4.1} is complete. 
\end{proof}
We next show the following theorem.
\begin{thm}
\label{thm:4.2}
Let $1 < q < \infty$, $1 \leq r \leq \infty$,  $- 1 + 1 \slash q < s <  1 \slash q$, 
$0 < \epsilon < \pi \slash 2$, and $\gamma > 0$.  
Let $\sigma$ be a positive 
number such that $- 1 + 1 \slash q < s \pm \sigma < 1 \slash q$. 
Let $\bL_2(\lambda)$ be the operator given in \eqref{def-L-formula}. 
Assume that $m_3 (\lambda, \xi')$  is a symbol belonging to $\BM_{- 3}$.  
Then, for every $\lambda \in \Sigma_{\epsilon} + \gamma_b$ and $f \in C^\infty_0(\HS)$, 
there exists a constant $C_b$ depending on $\gamma_b$ such that
\begin{align}
\label{b:l.1*} 
\lVert (\lambda, \lambda^{1/2} \nabla,   \nabla_b^2) \bL_2 (\lambda) f \rVert_{\CB^s_{q, r} (\HS)} 
& \le C_b \lVert f \rVert_{\CB^s_{q, r} (\HS)}, \\
\label{b:l.2*}
\lVert (\lambda, \lambda^{1/2} \nabla,
  \nabla_b^2) \bL_2 (\lambda) f \rVert_{\CB^s_{q, r} (\HS)} 
& \le C_b \lvert \lambda \rvert^{- \frac\sigma2} 
\lVert f \rVert_{\CB^{s + \sigma}_{q, r} (\HS)},\\
\label{b:l.4}		
\lVert \bL_2 (\lambda) f \rVert_{\CB^s_{q, r} (\HS)} 
& \le C_b \lvert \lambda \rvert^{- (1 - \frac\sigma2)} 
\lVert f \rVert_{\CB^{s - \sigma}_{q, r} (\HS)}, \\
\label{b:l.3*}
\lVert (\lambda, \lambda^{1/2} \nabla,
  \nabla_b^2) \pd_\lambda \bL_2 (\lambda) f \rVert_{\CB^s_{q, r} (\HS)} 
& \le C_b \lvert \lambda \rvert^{- (1 - \frac\sigma2)} 
\lVert f \rVert_{\CB^{s - \sigma}_{q, r} (\HS)}.
\end{align}
\end{thm}
\begin{proof}  
As has been seen in the proof of Theorem \ref{thm:3.1}, it suffices to prove that 
\begin{align}
\label{b:l.1} 
\lVert (\lambda, \lambda^{1/2} \nabla,   \nabla^2) \bL_2 (\lambda) f \rVert_{\CB^s_{q, r} (\HS)} 
& \le C_b \lVert f \rVert_{\CB^s_{q, r} (\HS)}, \\
\label{b:l.2}
\lVert (\lambda, \lambda^{1/2} \nabla,
  \nabla^2) \bL_2 (\lambda) f \rVert_{\CB^s_{q, r} (\HS)} 
& \le C_b \lvert \lambda \rvert^{- \frac\sigma2} 
\lVert f \rVert_{\CB^{s + \sigma}_{q, r} (\HS)},\\
\label{b:l.3}
\lVert (\lambda, \lambda^{1/2} \nabla,
  \nabla^2) \pd_\lambda \bL_2 (\lambda) f \rVert_{\CB^s_{q, r} (\HS)} 
& \le C_b \lvert \lambda \rvert^{- (1 - \frac\sigma2)} 
\lVert f \rVert_{\CB^{s - \sigma}_{q, r} (\HS)},
\end{align}
instead of \eqref{b:l.1*}, \eqref{b:l.2*}, and \eqref{b:l.3*}, respectively. \par
In the following, we shall verify the 
conditions in Assumption \ref{assump:4.1} with $i = 1$, $\alpha_1=0$, and $\beta_1=-1/2$
as well as with $i = 2$, $\alpha_2=-1$, and $\beta_2=-1/2$.
First, we shall show that 
\begin{align}
\|(\lambda, \lambda^{1/2} \nabla, 
 \nabla^2)\bL_2(\lambda)f\|_{\CH^1_q(\HS)}
& \le C\|f\|_{\CH^1_q(\HS)},\label{5.6.1} \\
\|(\lambda, \lambda^{1/2} \nabla, 
 \nabla^2)\bL_2(\lambda)f\|_{L_q(\HS)}
& \le C\|f\|_{L_q(\HS)}, \label{5.6.2}\\
\|(\lambda, \lambda^{1/2} \nabla, 
 \nabla^2)\bL_2(\lambda)f\|_{L_q(\HS)}
& \le C_b|\lambda|^{- \frac12}\|f\|_{\CH^1_q(\HS)}. \label{5.6.3}\\
\|\bL_2(\lambda)f\|_{\CH^1_q(\HS)}
& \le C |\lambda|^{-1}\|f\|_{\CH^1_q(\HS)},\label{5.6.4} \\
\|\bL_2(\lambda)f\|_{L_q(\HS)}
& \le C|\lambda|^{-1}\|f\|_{L_q(\HS)}, \label{5.6.5}\\
\|\bL_2(\lambda)f\|_{\CH^1_q(\HS)}
& \le C_b|\lambda|^{- \frac12}\|f\|_{L_q(\HS)}. \label{5.6.6}
\end{align}
To simplify the notation, we introduce symbols 
$E^{(n)}_{A, B}$, $n \in \BN_0$, defined by 
\begin{equation}
E^{(0)}_{A, B} = 0, \qquad E^{(n)}_{A, B} = (B - A)^{- 1} (B^n - A^n), \quad n \in \BN.
\end{equation}	
Then we may write 
\begin{equation}
\pd^n_d M_{x_d} = (- 1)^n (B - A)^{- 1} (B^n e^{- B x_d} - A^n e^{- A x_d})
= (- 1)^n (E^{(n)}_{A, B} e^{- B x_d} + A^n M_{x_d})
\end{equation}
for $n \in \BN_0$. Hence, we may have
\begin{align}
& \lambda^{n_0 \slash 2} \pd^{n_1}_d \pd_{x'}^{\alpha'} \bL_2 (\lambda)f \\
& = (- 1)^{n_1} \int_0^\infty \CF^{- 1}_{\xi'}
[m_3 (\lambda, \xi') \lambda^{n_0 \slash 2} A^2 B \pd_d^{n_1} M_{x_d + y_d} 
\CF'[f] (\xi', y_d)] (x') \d y_d \\
& = (- 1)^{n_1} \int_0^\infty \CF^{- 1}_{\xi'}
[Be^{- B (x_d + y_d)}  m_3 (\lambda, \xi') \lambda^{n_0 \slash 2} (i \xi')^{\alpha' } 
A^2  E^{(n_1)}_{A, B} 
\CF'[f] (\xi', y_d)] (x') \d y_d \\
& \quad + (- 1)^{n_1} \int_0^\infty \CF^{- 1}_{\xi'}[
ABM_{x_d + y_d}m_3 (\lambda, \xi') \lambda^{n_0 \slash 2} (i \xi')^{\alpha'} 
A^{n_1+1} B  \CF'[f] (\xi', y_d)] (x') \d y_d
\end{align}
for nonnegative integers $n_0$ and $n_1$ and
for multi-indices $\alpha'$. Since $m_3 (\lambda, \xi') \in \BM_{- 3}$,
if there holds $n_0 + n_1 + \lvert \alpha' \rvert = 2$, then
$m_3(\lambda, \xi')\lambda^{n_0/2}(i\xi')^{\alpha'}A^2E^{n_1}_{A,B}$
and $m_3(\lambda, \xi')\lambda^{n_0/2}(i\xi')^{\alpha'}A^{n_1+1}$
belong to $\BM_0$. 
Hence, Proposition \ref{prop:3.1} implies
\begin{equation}
\label{est-L2}
\lVert \lambda^{n_0 \slash 2} \pd^{n_1}_d \pd_{x'}^{\alpha'} \bL_2 (\lambda)f\rVert_{L_q (\HS)}
\le C\|f\|_{L_q(\HS)}, \qquad n_0+n_1+|\alpha'| = 2.
\end{equation}
In particular, by \eqref{est-L2}, we have \eqref{5.6.1},  \eqref{5.6.3}, and \eqref{5.6.5}. 
Moreover, if $n_0$, $n_1$, and $\alpha'$ satisfy $n_0+n_1+|\alpha'| = 3$, then
$m_3(\lambda, \xi')\lambda^{n_0/2}(i\xi')^{\alpha'}AE^{(n_1)}_{A,B}$
and $m_3(\lambda, \xi')\lambda^{n_0/2}(i\xi')^{\alpha'}A^{n_1}$
belong to $\BM_0$. Thus, it follows from Proposition \ref{prop:3.1} that
$$
\|\lambda^{n_0 \slash 2} \pd^{n_1}_d \pd_{x'}^{\alpha'} \bL_2 (\lambda)f\|_{L_q(\HS)}
\le C\|\CF^{-1}_{\xi'}[A \CF'[f]]\|_{L_q(\HS)},
\qquad n_0+n_1+|\alpha'| = 3.
$$
Writing 
\begin{equation}\label{nabla:5.0}
A= \frac{A^2}{A} = -\sum_{j=1}^{d-1} \frac{i\xi_j}{A} i\xi_j,
\end{equation}
we have $A\CF'[f] = -\sum_{j=1}^{d-1} (i\xi_j \slash A) \CF'[\pd_j f]$, and 
so we infer from the Fourier multiplier 
theorem in $\BR^{d-1}$ that
$$\|\CF^{-1}_{\xi'}[A \CF'[f]]\|_{L_q(\HS)} \leq \sum_{j=1}^{d-1}\|\pd_j f\|_{L_q(\HS)}
\leq \|\nabla f\|_{L_q(\HS)}.
$$
Hence, we have 
\begin{align}
\label{est-L2*}
\|\nabla \lambda^{n_0 \slash 2} \pd^{n_1}_d \pd_{x'}^{\alpha'} \bL_2 (\lambda)f\|_{L_q(\HS)}
&\leq C\|\nabla f\|_{L_q(\HS)}, \qquad n_0+n_1+|\alpha'| = 2, \\
\label{est-L3}
\|\lambda^{n_0 \slash 2} \pd^{n_1}_d \pd_{x'}^{\alpha'} \bL_2 (\lambda)f\|_{L_q(\HS)}
&\leq C\|\nabla f\|_{L_q(\HS)}, \qquad n_0+n_1+|\alpha'| = 3, \enskip 1 \leq n_0 \leq 3.
\end{align}
Noting that $\|\nabla f\|_{L_q(\HS)} \leq \|f\|_{H^1_q(\HS)}$ 
and $\|f\|_{L_q(\HS)} \leq \|f\|_{H^1_q(\HS)}$ if $\CB^s_{q,1} =
B^s_{q,1}$, by \eqref{est-L2}, \eqref{est-L2*}, and \eqref{est-L3} 
we have \eqref{5.6.1}, \eqref{5.6.2}, and \eqref{5.6.3}. Analogously, 
we have \eqref{5.6.4}. \par 

We now consider the dual operator $\bL_2(\lambda)^*$
 acting on $\varphi \in C^\infty_0(\HS)$, which is defined by 
\begin{equation}
\bL_2(\lambda)^*\varphi = \int^\infty_0\CF'[m_3(\lambda, \xi')
A^2BM_{x_d+y_d}\CF^{-1}_{\xi'}[\varphi](\xi', y_d)\d y_d.
\end{equation}
For $f$, $\varphi \in C^\infty_0(\HS)$, we see that 
$\bL_2(\lambda)^*((\lambda, \lambda^{1/2}\nabla, \nabla^2)\varphi) 
= (\lambda, \lambda^{1/2}\nabla, \nabla^2)\bL_2(\lambda)^*\varphi$, and so 
the dual operator of $(\lambda, \lambda^{1/2}\nabla, \nabla^2)\bL_2(\lambda)$
coincides with 
$ (\lambda, \lambda^{1/2}\nabla, \nabla^2)\bL_2(\lambda)^*$.  
Since $\bL_2(\lambda)^*$ is obtained by exchanging $\CF^{-1}_{\xi'}$ and 
$\CF'$, employing the same argument as proving \eqref{5.6.1}--\eqref{5.6.6}, 
we may prove the following estimates immediately:
\begin{align*}
\|(\lambda, \lambda^{1/2} \nabla, 
 \nabla^2)\bL_2(\lambda)^*\varphi\|_{\CH^1_q(\HS)}
& \le C\|\varphi\|_{\CH^1_q(\HS)}, \\
\|(\lambda, \lambda^{1/2} \nabla, 
 \nabla^2)\bL_2(\lambda)^*\varphi\|_{L_q(\HS)}
& \le C\|\varphi\|_{L_q(\HS)},\\
\|(\lambda, \lambda^{1/2} \nabla, 
 \nabla^2)\bL_2(\lambda)^*\varphi \|_{L_q(\HS)}
& \le C_b|\lambda|^{- \frac12}\|\varphi\|_{\CH^1_q(\HS)}, \\
\|\bL_2(\lambda)^*\varphi\|_{\CH^1_q(\HS)}
& \le C|\lambda|^{-1}\|\varphi\|_{\CH^1_q(\HS)}, \\
\|\bL_2(\lambda)^*\varphi\|_{L_q(\HS)}
& \le C|\lambda|^{-1}\|\varphi\|_{L_q(\HS)}, \\
\|\bL_2(\lambda)^*\varphi\|_{\CH^1_q(\HS)}
& \le C_b|\lambda|^{- \frac12}\|\varphi\|_{L_q(\HS)} 
\end{align*}
for any $\varphi \in C^\infty_0(\HS)$ and $\lambda \in \Sigma_\epsilon + \gamma_b$, 
which together with \eqref{5.6.1}--\eqref{5.6.6} implies that
$(\lambda, \lambda^{1/2}\nabla, \nabla^2)\bL_2(\lambda)$ satisfies Assumption \ref{assump:4.1} 
with $i = 1$, $\alpha_1=0$, and $\beta_1=1/2$ as well as   
$\bL_2(\lambda)$ satisfies Assumption \ref{assump:4.1} with 
$i = 2$, $\alpha_2=-1$, and $\beta_2=-1/2$. Hence, we arrive at
\eqref{b:l.1}, \eqref{b:l.2}, and \eqref{b:l.4}. \par
We next consider $\pd_\lambda \bL_2(\lambda)$. 
First we shall prove the following estimates:
\begin{align}
\|(\lambda, \lambda^{1/2} \nabla, 
 \nabla^2)\pd_\lambda\bL_2(\lambda)f\|_{\CH^1_q(\HS)}
& \le C|\lambda|^{-1}\|f\|_{\CH^1_q(\HS)},\label{5.6.7} \\
\|(\lambda, \lambda^{1/2} \nabla, 
 \nabla^2)\pd_\lambda\bL_2(\lambda)f\|_{L_q(\HS)}
& \le C|\lambda|^{-1}\|f\|_{L_q(\HS)}, \label{5.6.8}\\
\|(\lambda,\lambda^{1/2} \nabla, 
 \nabla^2)\pd_\lambda\bL_2(\lambda)f\|_{\CH^1_q(\HS)}
& \le C_b|\lambda|^{- \frac12}\|f\|_{L_q(\HS)}. \label{5.6.9}
\end{align}
To this end, we shall use the following representation:
\begin{equation}\label{adjoint:1}\begin{aligned}
& \lambda^{n_0 \slash 2} \pd^{n_1}_d \pd_{x'}^{\alpha'} \pd_\lambda 
\bL_2 (\lambda)f \\
& = (- 1)^{n_1} \int_0^\infty \CF^{- 1}_{\xi'}
\Big[B e^{- B (x_d + y_d)} \Big(\pd_\lambda (m_3 (\lambda, \xi') E^{n_1}_{A, B} )\Big) 
\lambda^{n_0 \slash 2} (i \xi')^{\alpha'} 
A^2  \CF'[f] (\xi', y_d) \Big] (x') \d y_d \\
& \quad + (- 1)^{n_1} \int_0^\infty \CF^{- 1}_{\xi'}
\bigg[\Big(B^2 \pd_\lambda (B e^{- B (x_d + y_d)}) \Big)
m_3 (\lambda, \xi') \lambda^{n_0 \slash 2} (i \xi')^{\alpha'} 
\frac{A^2 E^{n_1}_{A, B}}{B^2} 
\CF'[f] (\xi', y_d) \bigg] (x') \d y_d \\
& \quad + (- 1)^{n_1} \int_0^\infty \CF^{- 1}_{\xi'}
\Big[AB M_{x_d + y_d} 
\Big(\pd_\lambda m_3 (\lambda, \xi') \Big) \lambda^{n_0 \slash 2} 
(i \xi')^{\alpha' } A^{n_1+ 1} 
\CF'[f] (\xi', y_d) \Big] (x') \d y_d \\
& \quad + (- 1)^{n_1} \int_0^\infty \CF^{- 1}_{\xi'}
\bigg[\Big( B^2 \pd_\lambda (A B M_{x_d + y_d})
m_3 (\lambda, \xi') \lambda^{n_0 \slash 2} 
(i \xi')^{\alpha' } \frac{A^{n_1 + 1}}{B^2}
 \Big) \CF'[f] (\xi', y_d) \bigg] (x') \d y_d 
\end{aligned}
\end{equation}
for nonnegative integers $n_0$ and $n_1$ and 
for multi-indices $\alpha'$.
Since $m_3 (\lambda, \xi') \in \BM_{- 3}$ and $\pd_\lambda m_3(\lambda, \xi') \in \BM_{-5}$, 
we see that the following symbols: 
\begin{gather*}
\Big(\pd_\lambda (m_3 (\lambda, \xi') E^{n_1}_{A, B} )\Big) 
\lambda^{n_0 \slash 2} (i \xi')^{\alpha'} A^2, \quad
m_3 (\lambda, \xi') \lambda^{n_0 \slash 2} (i \xi')^{\alpha'} 
\frac{A^2 E^{n_1}_{A, B}}{B^2}, \\
m_3 (\lambda, \xi') \lambda^{n_0 \slash 2} (i \xi')^{\alpha'} 
\frac{A^2 E^{n_1}_{A, B}}{B^2}, \quad 
m_3 (\lambda, \xi') \lambda^{n_0 \slash 2} 
(i \xi')^{\alpha' } \frac{A^{n_1 + 1}}{B^2}
\end{gather*}
all belong to $\BM_0$ if $n_0+n_1+|\alpha'| = 5$ as well as 
the following symbols: 
\begin{gather*}
\Big(\pd_\lambda (m_3 (\lambda, \xi') E^{n_1}_{A, B} )\Big) 
\lambda^{n_0 \slash 2} (i \xi')^{\alpha'} A, \quad
m_3 (\lambda, \xi') \lambda^{n_0 \slash 2} (i \xi')^{\alpha'} 
\frac{A E^{n_1}_{A, B}}{B^2}, \\
m_3 (\lambda, \xi') \lambda^{n_0 \slash 2} (i \xi')^{\alpha'} 
\frac{A E^{n_1}_{A, B}}{B^2}, \quad 
m_3 (\lambda, \xi') \lambda^{n_0 \slash 2} 
(i \xi')^{\alpha' } \frac{A^{n_1}}{B^2}
\end{gather*}
all belong to $\BM_0$ if $n_0+n_1+|\alpha'| = 4$.
Hence, it follows from Proposition \ref{prop:3.1}  that
\begin{align*}
\lVert \lambda^{n_0 \slash 2} \pd^{n_1}_d \pd_{x'}^{\alpha'} 
\pd_\lambda \bL_2 (\lambda) f \rVert_{L_q (\HS)} 
& \le C\|f\|_{L_q(\HS)}, & &n_0+n_1+|\alpha'| = 4, \\
\lVert \lambda^{n_0 \slash 2} \pd^{n_1}_d \pd_{x'}^{\alpha'} 
\pd_\lambda \bL_2 (\lambda) f \rVert_{L_q (\HS)} 
& \le C\|\CF^{-1}[A\CF'[f]]\|_{L_q(\HS)} \
\leq C\|f\|_{\CH^1_q(\HS)}, & &n_0+n_1+|\alpha'| = 5.
\end{align*}
These estimates give \eqref{5.6.7}, \eqref{5.6.8}, and \eqref{5.6.9}. \par 
Let $\pd_\lambda \bL_2(\lambda)^*$ be the dual operator of  $\pd_\lambda \bL_2(\lambda)$ which 
is obtained by exchanging 
$\CF^{-1}_{\xi'}$ and $\CF'$ in the formula of $\pd_\lambda\bL_2(\lambda)$ given by setting 
$n_0=n_1=|\alpha'| = 0$ in \eqref{adjoint:1}. 
Hence, we immediately have 
\begin{align*}
\|(\lambda, \lambda^{1/2} \nabla,
\nabla^2)\pd_\lambda\bL_2(\lambda)^*\varphi\|_{L_{q'}(\HS)}
& \le C\|\varphi\|_{L_{q'}(\HS)}, \\
\|(\lambda, \lambda^{1/2} \nabla,
 \nabla^2)\pd_\lambda\bL_2(\lambda)^*\varphi\|_{\CH^1_{q'}(\HS)}
& \le C\|\varphi\|_{\CH^1_{q'}(\HS)}, \\
\|(\lambda, \lambda^{1/2} \nabla, 
 \nabla^2)\pd_\lambda\bL_2(\lambda)^*\varphi \|_{L_{q'}(\HS)}
& \le C_b|\lambda|^{- \frac12}\|\varphi\|_{\CH^1_{q'}(\HS)},
\end{align*}
which together with \eqref{5.6.7}, \eqref{5.6.8}, and \eqref{5.6.9}
implies that $(\lambda, \lambda^{1/2}\nabla, \nabla^2)\pd_\lambda \bL_2(\lambda)^*$
satisfies  Assumption \ref{assump:4.1} with $i = 2$, $\alpha_2=-1$, and $\beta_2=-1/2$. 
Thus, by Theorem \ref{spectralthm:1}, we have \eqref{b:l.3}.
The proof of Theorem \ref{thm:4.2} is complete. 
\end{proof}
Next, we prove the following theorem.
\begin{thm}
\label{thm:4.3} 
Let $1 < q < \infty$, $1 \leq r \leq \infty$,  $- 1 + 1 \slash q < s <  1 \slash q$, 
$0 < \epsilon < \pi \slash 2$, and $\gamma > 0$. 
Let $\sigma$ be a small positive  
number such that $- 1 + 1 \slash q < s \pm \sigma < 1 \slash q$. 
Assume that $m_1 (\lambda, \xi')$  is a symbol belonging to $\BM_{- 1}$.  
For every $\lambda \in \Sigma_{\epsilon} + \gamma_b$ and $f \in C^\infty_0(\HS)$, 
there exists a constant $C_b$ depending on $\gamma_b$ such that
\begin{align}
\label{a:1.1}
\lVert(\lambda^{1/2}, \nabla_b) \bL_3 (\lambda) f \rVert_{\CB^s_{q, r} (\HS)} 
& \le C_b \lVert f \rVert_{\CB^s_{q, r} (\HS)},\\
\label{a:1.2} 
\lVert \nabla_b\bL_3 (\lambda) f \rVert_{\CB^s_{q, r} (\HS)} 
& \le C_b \lvert \lambda \rvert^{- \frac\sigma2} 
\lVert f \rVert_{\CB^{s + \sigma}_{q, r} (\HS)}, \\
\label{a:1.3}
\lVert\nabla_b\pd_\lambda \bL_3 (\lambda) f \rVert_{\CB^s_{q, r} (\HS)} 
& \le C_b \lvert \lambda \rvert^{- (1 - \frac\sigma2)} 
\lVert f \rVert_{\CB^{s - \sigma}_{q, r} (\HS)}.
\end{align}
\end{thm}
\begin{proof}
Since $m_1 (\lambda, \xi') \in \BM_{- 1}$
and $\pd_\lambda m_1(\lambda, \xi') \in \BM_{-3}$,  it follows from 
Proposition~\ref{prop:3.1} that
\begin{align*}
\lVert(\lambda^{1/2}, \nabla_b)\bL_3 (\lambda) f \rVert_{\CH^1_q (\HS)} 
& \le C  \lVert f \rVert_{ \CH^1_q (\HS)}, \\
\lVert(\lambda^{1/2}, \nabla_b\bL_3 (\lambda) f \rVert_{ L_q (\HS)} 
& \le C \lVert f \rVert_{L_q (\HS)}, \\
\lVert  \nabla_b\bL_3 (\lambda) f \rVert_{ L_q (\HS)} 
& \le C_b|\lambda|^{- \frac12} \lVert f \rVert_{ \CH^1_q (\HS)}, \\
\lVert \nabla_b \pd_\lambda \bL_3 (\lambda) f \rVert_{\CH^1_q (\HS)} 
& \le C|\lambda|^{-1} 
\lVert f \rVert_{\CH^1_q (\HS)}, \\
\lVert \nabla_b \pd_\lambda \bL_3 (\lambda) f \rVert_{L_q (\HS)} 
& \le C|\lambda|^{-1} \lVert f \rVert_{L_q (\HS)},\\
\lVert \nabla_b\pd_\lambda \bL_3 (\lambda) f \rVert_{\CH^1_q (\HS)} 
& \le C_b|\lambda|^{- \frac12} \lVert f \rVert_{L_q (\HS)}.	
\end{align*}
 The dual operator $\bL_3^*(\lambda)^*$ is defined by exchanging $\CF^{-1}_{\xi'}$
and $\CF'$ as in the proof of Theorems  \ref{thm:4.1} and \ref{thm:4.2}. 
Then, by Proposition \ref{prop:3.1}, we have 
\begin{align*}
\lVert(\lambda^{1/2}, \nabla_b \bL_3 (\lambda)^* f \rVert_{\CH^1_{q'} (\HS)} 
& \le C  \lVert f \rVert_{\CH^1_{q'} (\HS)}, \\
\lVert(\lambda^{1/2}, \nabla_b)\bL_3 (\lambda)^* f \rVert_{L_{q'} (\HS)} 
& \le C \lVert f \rVert_{L_{q'} (\HS)}, \\
\lVert \nabla_b\bL_3 (\lambda)^* f \rVert_{\CH^1_{q'} (\HS)} 
& \le C_b|\lambda|^{-\frac12} \lVert f \rVert_{L_{q'} (\HS)}, \\
\lVert \nabla_b\pd_\lambda \bL_3 (\lambda)^* f \rVert_{\CH^1_{q'} (\HS)} 
& \le C|\lambda|^{-1} 
\lVert f \rVert_{\CH^{1}_{q'} (\HS)}, \\
\lVert \nabla_b\pd_\lambda \bL_3 (\lambda)^* f \rVert_{L_{q'} (\HS)} 
& \le C|\lambda|^{-1} \lVert f \rVert_{L_{q'} (\HS)},\\
\lVert \nabla_b\pd_\lambda \bL_3 (\lambda)^* f \rVert_{\CH^1_{q'} (\HS)} 
& \le C_b|\lambda|^{-\frac12} \lVert f \rVert_{L_{q'} (\HS)}.
\end{align*}	
Thus, by Theorem \ref{spectralthm:1}, \eqref{4}, and \eqref{d9*}, 
we have Theorem \ref{thm:4.3}. 
\end{proof}
\par

Combining Theorems \ref{thm:4.1}, \ref{thm:4.2}, and \ref{thm:4.3},
we have the following Corollary that gives the existence of solution operators for \eqref{fund:3}.
\begin{cor}
\label{lem-CW45Q} 
Let $1 < q < \infty$, $1 \leq r \leq \infty$, 
$- 1 + 1 \slash q < s <  1 \slash q$,  $0 < \epsilon < \pi \slash 2$,
and $\gamma > 0$.  
Let $\sigma$ be a positive   
number such that $- 1 + 1 \slash q < s \pm \sigma < 1 \slash q$. 
Let $J_m$ and $K_{\ell, m}$ be corresponding variables 
to $\lambda^{1/2}h_m$
and $\pd_\ell h_m$ for $\ell=1, \ldots, d$ and $m=1, \ldots, d-1$, respectively,
and set  $\bJ= (J_1, \ldots, J_{d-1})$ and $\bK = (K_{\ell, m})_{1 \le \ell \le d, 1 \le m \le d - 1}$.
Then, there exist operators
\begin{align*}
\CW_2 (\lambda) & \in \Hol (\Sigma_\epsilon, 
\CL(\CB^{s}_{q,r}(\HS)^{d(d-1)}, \sB^{s + 2}_{q, r} (\HS)^d)), \\
\CW_3 (\lambda) & \in \Hol (\Sigma_\epsilon, 
\CL(\CB^s_{q,r}(\HS)^{d-1}, \sB^{s + 2}_{q, r} (\HS)^d)), \\
\CQ (\lambda)
&\in \Hol (\Sigma_{\epsilon}, \CL(\CB^{s+1}_{q,r}(\HS)^{d(d-1)},  
\sB^{s + 1}_{q, r} (\HS)))
\end{align*}
such that the functions 
$\bw_2 = \CW_2(\lambda) \nabla \bh' 
+ \CW_3(\lambda)(\lambda^{1/2}\bh')$ and 
$\fq_2 = \CQ(\lambda)\nabla\bh'$ are solutions to \eqref{fund:3} for every
$\lambda \in \Sigma_\epsilon + \gamma_b$ and $\bh' \in \sB^{s+1}_{q,r}(\HS)^{d-1}$, 
and there hold
\begin{equation}\label{boundaryest:4.1}\begin{aligned}
\lVert (\lambda, \lambda^{1\slash2} \nabla, \nabla_b^2) \CW_2 (\lambda) 
\bK\rVert_{\CB^s_{q, r} (\HS)} & \le C_b\|\bK\|_{\CB^s_{q, r}(\HS)}, 
 \\
\lVert (\lambda, \lambda^{1\slash2} \nabla, \nabla_b^2) \CW_3 (\lambda) 
\bJ\rVert_{\CB^s_{q, r} (\HS)} & \le 
C_b\|\bJ\|_{\CB^s_{q, r}(\HS)}, \\
\|(\lambda^{1/2}, \nabla_b) \CQ(\lambda)\bK\|_{\CB^s_{q, r}(\HS)} 
& \le C_b \lVert \bK \rVert_{\CB^s_{q, r} (\HS)}
\end{aligned}\end{equation}
for every $\lambda \in \Sigma_\epsilon + \gamma_b$, $\bK \in \CB^s_{q,r}(\HS)^{d(d-1)}$, 
and $\bJ \in \CB^{s}_{q,r}(\HS)^{d-1}$.
\par 
Moreover, for every $\lambda \in \Sigma_\epsilon + \gamma_b$,
$\bJ\in C^\infty_0(\HS)^{d-1}$, and
$\bK \in C^\infty_0(\HS)^{d(d-1)}$, there exists a constant 
$C_b$ depending on $\gamma_b$ such that  there hold
\begin{equation}\label{boundaryest:4.2}\begin{aligned}
\lVert (\lambda, \lambda^{1/2} \nabla,
  \nabla_b^2) \CW_2 (\lambda)\bK \rVert_{\CB^s_{q, r} (\HS)} 
& \le C_b \lvert \lambda \rvert^{- \frac\sigma2} 
\lVert\bK \rVert_{\CB^{s + \sigma}_{q, r} (\HS)}, \\
\lVert \CW_2 (\lambda)\bK \rVert_{\CB^s_{q, r} (\HS)} 
& \le C_b \lvert \lambda \rvert^{- (1 - \frac\sigma2)}
\lVert\bK\rVert_{\CB^{s - \sigma}_{q, r} (\HS)}, \\
\lVert (\lambda, \lambda^{1/2} \nabla, 
  \nabla_b^2)( \pd_\lambda \CW_2 (\lambda))\bK\rVert_{\CB^s_{q, r} (\HS)} 
& \le C_b \lvert \lambda \rvert^{- (1 - \frac\sigma2)} 
\lVert \bK \rVert_{\CB^{s - \sigma}_{q, r} (\HS)}, \\
\lVert (\lambda, \lambda^{1/2} \nabla,   \nabla_b^2) 
\CW_3 (\lambda)\bJ \rVert_{\CB^s_{q, r} (\HS)} 
& \le C_b \lvert \lambda \rvert^{- \frac\sigma2} 
\lVert\bJ \rVert_{\CB^{s + \sigma}_{q, r} (\HS)}, \\
\lVert \CW_3 (\lambda)\bJ \rVert_{\CB^s_{q, r} (\HS)} 
& \le C_b \lvert \lambda \rvert^{- (1 - \frac\sigma2)}
\lVert\bJ\rVert_{\CB^{s - \sigma}_{q, r} (\HS)}, \\
\lVert (\lambda, \lambda^{1/2} \nabla, 
\nabla^2)( \pd_\lambda \CW_3 (\lambda))
\bJ\rVert_{\CB^s_{q, r} (\HS)} 
& \le C_b \lvert \lambda \rvert^{- (1 - \frac\sigma2)} 
\lVert \bJ \rVert_{\CB^{s - \sigma}_{q, r} (\HS)},\\
\lVert \nabla_b \CQ (\lambda)\bK \rVert_{\CB^s_{q, r} (\HS)} 
& \le C_b \lvert \lambda \rvert^{- \frac\sigma2} 
\lVert \bK \rVert_{\CB^{s + \sigma}_{q, r} (\HS)}, \\
\lVert \nabla_b \pd_\lambda \CQ (\lambda)\bK \rVert_{\CB^s_{q, r} (\HS)} 
& \le C \lvert \lambda \rvert^{- (1 - \frac\sigma2)} 
\lVert \bK \rVert_{\CB^{s - \sigma}_{q, r} (\HS)}.
\end{aligned}\end{equation}
\end{cor}
\begin{proof} 
Let $\bw_{2}(x)$ and $\fq_2(x)$ be a function given by \eqref{hsol:4}, which are solutions of
equations \eqref{fund:3}.  
To show the existence of 
 operators $\CW_{2}(\lambda)$,  $\CW_{3}(\lambda)$ and $\CQ(\lambda)$ we {\color{black} rely on} 
the so-called Volevich trick. 
For a symbol ${\color{black} m_2 (\lambda, \xi') \in \BM_{- 2}}$, we consider an operator
$\CF^{-1}_{\xi'}[Be^{-Bx_d} {\color{black} m_2} (\lambda, \xi')\CF'[f](\xi', 0)](x')$, where
$x'=(x_1, \ldots, x_{d-1})$. We write
\begin{align}
\CF^{-1}_{\xi'}&[Be^{-Bx_d} {\color{black} m_2} (\lambda, \xi')\CF'[f](\xi', 0)](x')\\
&=-\int^\infty_0\frac{\pd}{\pd y_d}
\CF^{-1}_{\xi'}[Be^{-B(x_d+y_d)} {\color{black} m_2} (\lambda, \xi')\CF'[f](\xi', y_d)](x')\d y_d \\
& = \int^\infty_0
\CF^{-1}_{\xi'}[Be^{-B(x_d+y_d)}\{B {\color{black} m_2} (\lambda, \xi')\CF'[f](\xi', y_d)
+  {\color{black} m_2} (\lambda, \xi')\CF'[\pd_df](\xi', y_d)\}](x')\d y_d. 
\end{align}
{\color{black} Together with} the identity: $1 = (\mu^{-1}\lambda + |\xi'|^2)B^{-2}$ {\color{black} we} obtain
\begin{equation}\label{volev:1}\begin{aligned}
\CF^{-1}_{\xi'}&[Be^{-Bx_d} {\color{black} m_2} (\lambda, \xi')\CF'[f](\xi', 0)](x')\\
& = \int^\infty_0\CF^{-1}_{\xi'}[Be^{-B(x_d+y_d)}\mu^{-1}\lambda^{1/2} B^{-1}
{\color{black} m_2} (\lambda, \xi')\CF'[\lambda^{1/2}f](\xi', y_d)](x')\d y_d \\
& \quad -\sum_{\ell=1}^{d-1} 
\int^\infty_0\CF^{-1}_{\xi'}[Be^{-B(x_d+y_d)}i\xi_\ell B^{-1}
{\color{black} m_2} (\lambda, \xi')\CF'[\pd_\ell f](\xi', y_d)](x')\d y_d \\
& \quad -\int^\infty_0\CF^{-1}_{\xi'}[Be^{-B(x_d+y_d)}
{\color{black} m_2} (\lambda, \xi')\CF'[\pd_df](\xi', y_d)](x')\d y_d.	
\end{aligned}\end{equation}
\par
Now, we consider another operator $\CF^{-1}_{\xi'}[A^2BM_{x_d}m_3(\lambda, \xi')\CF'[f](\xi', 0)](x')$ 
for a symbol $ m_3 (\lambda, \xi') \in \BM_{- 3}$.
By the Volevich trick, we write
\begin{equation}
\begin{aligned}
\CF^{-1}_{\xi'}&[{\color{black}A^2} BM_{x_d} {\color{black} m_3} (\lambda, \xi')\CF'[f](\xi', 0)](x')\\
&= -\int^\infty_0\frac{\pd}{\pd y_d}
\CF^{-1}_{\xi'}[{\color{black}A^2} BM_{x_d+y_d} {\color{black} m_3} (\lambda, \xi')\CF'[f](\xi', y_d)](x')\d y_d \\
& = 
\int^\infty_0 {\color{black} \CF^{-1}_{\xi'}[} (Be^{-B(x_d+y_d)} + ABM_{x_d+y_d}){\color{black}A^2 m_3} (\lambda, \xi')
\CF'[f](\xi', y_d)](x') \d y_d \\
& \quad - \int^\infty_0\CF^{-1}_{\xi'}[
{\color{black}A^2} BM_{x_d+y_d} {\color{black} m_3} (\lambda, \xi')\CF'[\pd_d f](\xi', y_d)](x')\d y_d.
\end{aligned}
\end{equation}
{\color{black}
Together with \eqref{nabla:5.0}, we may write
\begin{equation}
\label{volev:2}
\begin{aligned}
\CF^{-1}_{\xi'}&[A^2 BM_{x_d} m_3 (\lambda, \xi')\CF'[f](\xi', 0)](x')\\
& = 
\int^\infty_0 \CF^{-1}_{\xi'} \bigg[ (A Be^{-B(x_d+y_d)} 
+ A^2 BM_{x_d+y_d}) \frac{i \xi_\ell}{A} m_3 (\lambda, \xi')
\CF'[\pd_\ell f](\xi', y_d) \bigg](x') \d y_d \\
& \quad - \int^\infty_0\CF^{-1}_{\xi'}[
A^2 BM_{x_d+y_d} m_3 (\lambda, \xi')\CF'[\pd_d f](\xi', y_d)](x')\d y_d.	
\end{aligned}			
\end{equation}
}\noindent
In view of \eqref{volev:1} and \eqref{volev:2}, we {\color{black} define} 
$\CW_{2,j}(\lambda)\bK$ and $\CW_{3,k}\bJ$  by 
\allowdisplaybreaks
\begin{align}
\CW_{2, k} (\lambda)\bK  
&  = \sum_{\ell = 1}^{d - 1}
\int^\infty_0 \CF^{- 1}_{\xi'} \Big[
\mu^{- 1} B e^{- B (x_d + y_d)}
i \xi_\ell B^{- 3} \CF' [K_{\ell, k}] (\xi', y_d) \Big] (x')\d y_d \\
& \quad - \int^\infty_0 \CF^{- 1}_{\xi'} \Big[
\mu^{- 1} B e^{- B (x_d + y_d)}
B^{- 2} \CF'[K_{d, k}] (\xi', y_d) \Big] (x') \d y_d \\
& \quad - 2\sum_{\ell = 1}^{d - 1}
\int^\infty_0 \CF^{- 1}_{\xi'} \Big[\mu^{- 1} B e^{- B (x_d + y_d)}
(i\xi_k) D_{A, B}^{-1}  \CF' [K_{\ell, \ell}]] (\xi', y_d) \Big] (x') \d y_d \\
& \quad -2 \sum_{\ell = 1}^{d - 1}
\int^\infty_0 \CF^{- 1}_{\xi'} \Big[\mu^{- 1} A^2 B M_{x_d + y_d}
(i \xi_k/A) D_{A, B}^{- 1}
\CF' [K_{\ell, \ell}] (\xi', y_d) \Big] (x') \d y_d \\
& \quad + 2 \sum_{\ell = 1}^{d - 1}
\int^\infty_0 \CF^{- 1}_{\xi'} \Big[\mu^{-1} A^2 B M_{x_d + y_d}
(i\xi_k/A)(i\xi_\ell /A)D_{A, B}^{-1}
\CF' [K_{d, \ell}] (\xi', y_d) \Big] (x') \d y_d \\
& \quad + \sum_{\ell = 1}^{d - 1}
\int^\infty_0 \CF^{- 1}_{\xi'} \Big[\mu^{- 1} B e^{- B (x_d + y_d)}
i \xi_k (3 B - A) B^{- 1} D_{A, B}^{- 1}
\CF' [K_{\ell, \ell}] (\xi', y_d) \Big] (x') \d y_d \\
& \quad - \sum_{\ell = 1}^{d - 1}
\int^\infty_0 \CF^{- 1}_{\xi'} \Big[\mu^{- 1} B e^{- B (x_d + y_d)}
i\xi_k i \xi_\ell (3 B - A) B^{- 2} D_{A, B}^{- 1}
\CF' [K_{d, \ell}] (\xi', y_d) \Big] (x') \d y_d 
\end{align}
for $k = 1, \ldots, d - 1$,
\begin{equation}
\CW_{3, k}\bJ = \int^\infty_0 \CF^{- 1}_{\xi'} \Big[
\mu^{- 2} B e^{- B (x_d + y_d)}
\lambda^{1/2}B^{- 3} \CF'[J_k](\xi', y_d) \Big] (x') \d y_d
\end{equation} 
for $k = 1, \ldots, d - 1$, and 
\begin{align}
\CW_{2,d}(\lambda)\bK
&= 2 \sum_{\ell = 1}^{d - 1}
\int^\infty_0 \CF^{- 1}_{\xi'} \Big[\mu^{- 1} B e^{- B (x_d + y_d)}
A D_{A, B}^{- 1} \CF' [K_{\ell, \ell}] (\xi', y_d) \Big] (x') \d y_d \\
& \quad + 2 \sum_{\ell = 1}^{d - 1} \int^\infty_0 \CF^{- 1}_{\xi'} 
\Big[\mu^{-1} A^2 B M_{x_d + y_d} D_{A, B}^{-1}
\CF' [K_{\ell, \ell}] (\xi', y_d) \Big] (x') \d y_d \\
& \quad - 2 \sum_{\ell = 1}^{d - 1}
\int^\infty_0 \CF^{- 1}_{\xi'} \Big[
\mu^{-1} A^2 B M_{x_d + y_d} ( i \xi_\ell/A) D_{A, B}^{-1} 
\CF' [K_{d, \ell}] (\xi', y_d) \Big] (x') \d y_d \\
& \quad + \sum_{\ell = 1}^{d - 1}
\int^\infty_0 \CF^{- 1}_{\xi'} \Big[
\mu^{- 1} B e^{- B (x_d + y_d)} (B -A) D_{A, B}^{- 1}
\CF' [K_{\ell, \ell}] (\xi', y_d) \Big] (x') \d y_d, \\
& \quad - \sum_{\ell = 1}^{d - 1} \int^\infty_0 \CF^{- 1}_{\xi'} \Big[
\mu^{- 1} B e^{- B (x_d + y_d)} (B -A) i\xi_\ell
B^{- 1} D_{A, B}^{- 1} \CF'[K_{d, \ell}] (\xi', y_d) \Big] (x') \d y_d. 
\end{align}
Then, we have $\bw_2= \CW_2(\lambda) {\color{black} \nabla \bh'} 
+ \CW_3(\lambda) ({\color{black} \lambda^{1 \slash 2} \bh'})$,
where {\color{black} we have set}
\begin{equation}
\CW_2(\lambda)\bK
= (\CW_{2,1}(\lambda)\bK, \ldots, \CW_{2,d}(\lambda)\bK),
\quad \CW_3\bJ =   (\CW_{3,1}(\lambda)\bJ, \ldots, 
\CW_{3,d-1}(\lambda)\bJ, 0).
\end{equation}
{\color{black}Similarly,} by the Volevich trick, 
{\color{black} we may write}
$\fq_2 (x) = \CQ(\lambda) {\color{black} \nabla \bh'}$, 
where we have set
\begin{align}
\CQ(\lambda)\bK & = - 2 \sum_{\ell = 1}^{d - 1}
\int^\infty_0 \CF^{- 1}_{\xi'} \Big[
A e^{- A (x_d + y_d)} (A + B) B D_{A, B}^{- 1} 
\CF' [K_{\ell,\ell}] (\xi', y_d) \Big] (x') \d y_d \\ 
& \quad + 2 \sum_{\ell = 1}^{d - 1} \int^\infty_0 
\CF^{- 1}_{\xi'} \Big[ A e^{- A (x_d + y_d)}
(A + B)B( i \xi_\ell /A) D_{A, B}^{- 1}
\CF'[K_{d, \ell}] (\xi', y_d) \Big] (\xi') \d y_d.
\end{align}
The detailed derivation of  $\CW_2(\lambda)$, 
$\CW_3(\lambda)$ and $\CQ(\lambda)$ may  be 
found in \cite[Lem. 3.5.13]{Shi20}. \par 
According to the discussion in Section \ref{sec-solution-formula},
Problem \eqref{fund:3} admits solutions  $\bw_2$ and $\fq_2$. 
Recalling that $B^{-1} \in \BM_{-1}$ and $D_{A,B}^{-1} \in \BM_{-3}$, 
we see the following symbols:
\begin{gather}
\lambda^{1/2} B^{-3}, \enskip i\xi_\ell B^{-3},\enskip B^{-2}, \enskip
i\xi_\ell D_{A,B}^{-1}, \enskip i\xi_k(3B-A) B^{-1}D_{A,B}^{-1}, \\
i\xi_k i\xi_\ell(3B-A)B^{-2}D_{A,B}^{-1},\enskip 
AD_{A,B}^{-1}, \enskip (B-A)D_{A,B}^{-1}, \enskip
(B-A)i\xi_\ell B^{-1}D_{A,B}^{-1}
\end{gather}
all belong to $\BM_{-2}$,  and $(i\xi_k/A)D_{A,B}^{-1}$ and 
$(i\xi_k/A)(i\xi_\ell/A)D_{A,B}^{-1}$ 
belong to $\BM_{-3}$.  Thus, 
using Theorem \ref{thm:4.1} and Theorem \ref{thm:4.2}, 
we see that  $\CW_2(\lambda)$ and $\CW_3(\lambda)$ 
satisfy estimates in \eqref{boundaryest:4.1} and 
\eqref{boundaryest:4.2} for any $\lambda \in \Sigma_\epsilon + \gamma_b$, 
$\bK \in C^\infty_0(\HS)^{d(d-1)}$, and $\bJ \in C^\infty_0(\HS)^{d-1}$. 
Moreover, since $(A+B)BD_{A,B}^{-1}$ 
and $(A+B)B(i\xi_\ell A^{-1})D_{A,B}^{-1}$
belong to $\BM_{-1}$, using Theorem \ref{thm:4.3} we see that 
$\CQ (\lambda)$ satisfies the estimates in \eqref{boundaryest:4.1} and \eqref{boundaryest:4.2}. 
Since $C^\infty_0(\HS)$ is dense in $\CB^s_{q,r}(\HS)$ for $1 < q < \infty$, $1 \leq r \leq \infty-$,
and $-1+1/q < s < 1/q$, by the density argument we have the estimates in 
\eqref{boundaryest:4.1} for any $\bK \in \CB^s_{q,r}(\HS)^{d(d-1)}$ and $\bJ
\in \CB^s_{q,r}(\HS)^{d-1}$. Thus, for any $\bh \in \sB^{s+1}_{q,r}(\HS)$, we may 
define $\CW_2(\lambda)\nabla\bh'$, $\CW_3(\lambda)(\lambda^{1/2}\bh')$, and 
$\CP(\lambda)\nabla\bh'$, and hence from their definitions, we see that
$\bw_2= \CW_2(\lambda)\nabla\bh' + \CW_3(\lambda)(\lambda^{1/2}\bh')$ and 
$\fq_2= \CQ(\lambda)\nabla\bh'$ give solutions to  
\eqref{fund:3}. 
\end{proof}
\begin{rem}
We remark that, in contrast to the formula derived in \cite[Sec. 6]{SS12}, 
 there is no term of the form
\begin{equation}
\int_0^\infty \CF^{- 1}_{\xi'} [m (\lambda, \xi') 
 e^{- {\color{black} \lvert \xi' \vert} (x_d + y_d)} \CF'[\mathsf h] (\xi', y_d) ] (x') \d y_d
\end{equation}
in the representation of $Q_3$. If this term appears in the representation of the 
pressure term,  the pressure term itself may  not be estimated although
its gradient  may  be estimated, 
since this is a singular integral operator. This fact sometimes  
causes difficulty when we treat the pressure term. 
\end{rem}
\subsection{Proof of Theorems \ref{th-sth}.} 
We first construct solutions operators $\CS(\lambda)$ and $\CP(\lambda)$ of equations
\eqref{resol:1.1}. In the following, let $\bPhi = (\bPhi_1, \bPhi_2', \bPhi_3')$ 
be the corresponding variable to 
$\bsF= (\bff,  \lambda^{1/2}\bh', \nabla\bh')$ with $(\bff, \bh') \in \CD^s_{q,r}(\HS)$.
More precisely, $\bPhi_1$, $\bPhi_2'$, and $\bPhi_3'$ are the corresponding variable
to $\bff$, $\lambda^{1/2}\bh'$, and $\nabla\bh'$, respectively.  \par
Let $(\bff, \bh') \in \CD^s_{q,r}(\HS)$. Let $\CW_1(\lambda)$ be an operator given in Corollary \ref{lem-CW1} and set 
$\bw_1 = \CW_1(\lambda)\bff$. Then, we see that
$\bu_1=\bw_1$  satisfies \eqref{whole:1} due to $\bff \in \CJ^s_{q,r}(\HS)$.  Namely,   $\bw_1$
and $\fq_1=0$ satisfy
the following equations: 
$$\left\{\begin{aligned}
\lambda \bw_1 - \DV(\mu \BD(\bw_1) - \fq_1\BI)&= \bff 
\quad\text{in $\HS$}, \\
\dv \bw_1 &= 0 \quad\text{in $\HS$}, \\
w_{1,j}|_{\pd\HS} &= 0\quad j=1, \ldots, d-1, \\
2 \mu \pd_d w_{1,d} - \fq_1|_{\pd\HS} &= 0.
\end{aligned}\right.$$
Moreover, by  Corollary  \ref{lem-CW1}, there hold
\allowdisplaybreaks
\begin{equation}
\label{sest:1}
\begin{aligned}
\|(\lambda, \lambda^{1/2}\nabla, \nabla_b^2)\CW_1(\lambda)\bPhi_1\|_{\CB^s_{q,r}(\HS)}
 &\leq C\|\bPhi_1\|_{\CB^s_{q,r}(\HS)},  \\
\|(\lambda, \lambda^{1/2}\nabla, \nabla_b^2)\CW_1(\lambda)\bPhi_1\|_{\CB^s_{q,r}(\HS)}
&\leq C|\lambda|^{-\frac{\sigma}{2}}\|\bPhi_1\|_{\CB^s_{q,r}(\HS)}, \\
\|(1, \lambda^{-1/2}\nabla) \CW_1(\lambda)\bPhi_1\|_{\CB^s_{q,r}(\HS)} &\leq C|\lambda|^{-(1-\frac{\sigma}{2})}
\|\bPhi_1\|_{\CB^{s-\sigma}_{q,r}(\HS)}, \\
\|(\lambda, \lambda^{1/2}\nabla, \nabla_b^2)\pd_\lambda \CW_1(\lambda)\bPhi_1\|_{\CB^s_{q,r}(\HS)}
& \leq C|\lambda|^{-(1-\frac{\sigma}{2})}\|\bPhi_1\|_{\CB^{s-\sigma}_{q,r}(\HS)} 
\end{aligned}
\end{equation}
for any $\bPhi_1 \in C^\infty_0(\HS)^d$, $1 < q < \infty$, $1 \leq r \leq \infty$, and
$-1+1/q < s-\sigma <  s < s+\sigma < 1/q$.  
\par
We next consider the compensation of $\bu_1$ and $\fq_1$. 
Let $\CW_2(\lambda)$, $\CW_3(\lambda)$, and $\CQ(\lambda)$ be solution operators given in 
Corollary \ref{lem-CW45Q} and set 
\begin{equation}\label{cs3}\CS_3(\lambda) \bPhi_1 = 
\mu \BD(\CW_1(\lambda) \bPhi_1)\bn_0- \mu\langle \BD(\CW_1(\lambda) \bPhi_1)\bn_0, \bn_0 \rangle \bn_0.
\end{equation}
Let 
\begin{align*}\bu_2 &= \CW_2(\lambda)(\nabla\bh') + \CW_3(\lambda)(\lambda^{1/2}\bh')-
\Big(\CW_2(\lambda)(\nabla\CS_3(\lambda)\bff)
+ \CW_3(\lambda)(\lambda^{1/2}\CS_3(\lambda)\bff) \Big), \\
\fq_2 &= \CQ(\lambda)\nabla \bh' -\CQ(\lambda)\nabla \CS_3(\lambda)\bff.
\end{align*}
Then,  
as was discussed in Step 1 
in Section \ref{subsec.2.4} and Corollary \ref{lem-CW45Q},  noting that
$$\mu \BD(\bw_1)\bn_0
- \mu\langle \BD(\bw_1)\bn_0, \bn_0 \rangle \bn_0 = \CS_3(\lambda)\bff,$$
we see that 
  $\bu_2$ and $\fq_2$ are solutions to equations:
$$\left\{\begin{aligned}
\lambda \bu_2 - \DV(\mu \BD(\bu_2) - \fq_2\BI) & = 0 & \quad & \text{in $\HS$}, \\
\dv \bu_2 & =0 & \quad & \text{in $\HS$}, \\
(\mu \BD(\bu_2) - \fq_2\BI) \bn_0 & = \Big\{(\bh', 0) - \mu\Big(\BD(\bw_1)\bn_0 
- \langle \BD(\bw_1)\bn_0, \bn_0 \rangle \bn_0\Big) \Big\}\Big\vert_{\pd\HS} 
& \quad & \text{on $\pd \HS$}.
\end{aligned}\right.$$
In particular, $\bu = \bw_1+\bu_2$ and $\fq =\fq_2 $ are solutions of equations \eqref{resol:1.1}. \par
Set
\begin{align*}
\CS_2(\lambda) \bPhi &= \CW_2(\lambda) \bPhi_3' + \CW_2(\lambda) \bPhi_2'-
(\CW_2(\lambda)(\nabla\CS_3(\lambda)\bPhi_1)
+ \CW_3(\lambda)(\lambda^{1/2}\CS_3(\lambda)\bPhi_1)), \\
\CP(\lambda)\bPhi & = \CQ(\lambda)\bPhi_3' -\CQ(\lambda)\nabla \CS_3(\lambda)\bPhi_1.
\end{align*}
By Corollaries \ref{lem-CW1} and \ref{lem-CW45Q}, we see that
for any $\lambda \in \Sigma_\epsilon+\gamma_b$ and $\bPhi \in C^\infty_0(\HS)^{M_d}$,
there hold
\begin{align*}
&\|(\nabla, \lambda^{1/2}\nabla, \nabla_b^2)\CS_2(\lambda)\bPhi\|_{\CB^s_{q,r}(\HS)}
+ \|\nabla_b\CP(\lambda)\bPhi\|_{\CB^{s}_{q,r}(\HS)} \\
&\qquad 
\leq C(\|\bPhi\|_{B^s_{q,r}(\HS)} + \|(\lambda^{1/2}, \nabla)\nabla\CW_1(\lambda)\bPhi
\|_{B^s_{q,r}(\HS)} \leq C\|\bPhi\|_{\CB^s_{q,r}(\HS)}, \\
&\|(\nabla, \lambda^{1/2}\nabla, \nabla_b^2)\CS_2(\lambda)\bPhi\|_{\CB^s_{q,r}(\HS)}
+ \|\nabla_b\CP(\lambda)\bPhi\|_{\CB^{s}_{q,r}(\HS)} \\
&\qquad \leq C(|\lambda|^{-\frac{\sigma}{2}}\|\bPhi\|_{\CB^s_{q,r}(\HS)} + 
\|(\lambda^{1/2}, \nabla)\nabla \CW_1(\lambda)\bPhi\|_{\CB^s_{q,r}(\HS)})
\leq C|\lambda|^{-\frac{\sigma}{2}}\|\bPhi\|_{\CB^s_{q,r}(\HS)}.
\end{align*}
\par
We now estimate $\pd_\lambda \CS_2(\lambda)\bPhi$. 
It follows from the definition of $\CS_2(\lambda)$ that
\begin{align*}
\pd_\lambda \CS_2(\lambda)\bPhi &= \pd_\lambda \CW_2(\lambda)\bPhi'_3 +
\pd_\lambda \CW_3(\lambda)\bPhi'_2 - (\pd_\lambda\CW_2(\lambda))(\nabla \CS_3(\lambda)\bPhi_1)
- \CW_2(\lambda)(\nabla\pd_\lambda \CS_3(\lambda)\bPhi_1) \\
&\quad - (\pd_\lambda \CW_3(\lambda))(\lambda^{1/2}\CS_3(\lambda)\bPhi_1)
- \CW_3(\lambda)(\lambda^{1/2}\pd_\lambda \CS_3(\lambda)\bPhi_1)
-(1/2) \lambda^{-1/2}\CW_3(\lambda)(\CS_3(\lambda)\bPhi_1).
\end{align*}
By Corollary \ref{lem-CW45Q}, we have 
\begin{multline}
\|(\lambda, \lambda^{1/2}\nabla, \nabla_b^2)\CS_2(\lambda)\bPhi\|_{\CB^s_{q,r}(\HS)} \\
\leq C\Big(|\lambda|^{-(1-\frac{\sigma}{2})}\|(\bPhi'_2, \bPhi'_3)\|_{\CB^{s-\sigma}_{q,r}(\HS)} 
+|\lambda|^{-(1-\frac{\sigma}{2})}\|(\lambda^{1/2}, \nabla)\CS_3(\lambda)\bPhi_1\|_{\CB^{s-\sigma}_{q,r}(\HS)}
\\
+ \|(\lambda^{1/2}, \nabla)\pd_\lambda \CS_3(\lambda)\bPhi_1\|_{\CB^s_{q,r}(\HS)} 
+|\lambda|^{-1/2}\|\CS_3(\lambda)\bPhi_1\|_{\CB^s_{q,r}(\HS)}\Big).
\end{multline}
From the definition of $\CS_3(\lambda)$ in \eqref{cs3}, we obtain
\begin{align*}
\|(\lambda^{1/2}, \nabla)\pd_\lambda \CS_3(\lambda)\bPhi_1\|_{\CB^s_{q,r}(\HS)} \leq 
C\|(\lambda^{1/2}\nabla, \nabla^2)\pd_\lambda\CW_1(\lambda)\bPhi_1\|_{\CB^s_{q,r}(\HS)}
\leq C|\lambda|^{-(1-\frac{\sigma}{2})}\|\bPhi_1\|_{\CB^{s-\sigma}_{q,r}(\HS)}.
\end{align*}
In addition, we see that 
$$\|(\lambda^{1/2}, \nabla)\CS_3(\lambda)\bPhi_1\|_{\CB^{s-\sigma}_{q,r}(\HS)}
\leq C\|(\lambda^{1/2}\nabla, \nabla^2)\CW_1(\lambda)\bPhi_1\|_{\CB^{s-\sigma}_{q,r}(\HS)}
\leq C\|\bPhi_1\|_{\CB^{s-\sigma}_{q,r}(\HS)}.
$$
Here, since there holds $-1+1/q < s\pm \sigma < 1/q$, we may prove 
$$\|(\lambda, \lambda^{1/2}\nabla, \nabla_b^2)\CW_1(\lambda)\bPhi_1\|_{B^{s\pm\sigma}_{q,r}(\HS)}
\leq C\|\bPhi_1\|_{\CB^{s\pm\sigma}_{q,r}(\HS)}. 
$$
using the same argument as in the case $s$. 
Using Corollary \ref{lem-CW1}, we also have
\begin{align*}
|\lambda|^{-1/2}\|\CS_3(\lambda)\bPhi_1\|_{B^s_{q,r}(\HS)} &
\leq C\|\lambda^{-1/2}\nabla\CW_1(\lambda)\bPhi_1\|_{\CB^s_{q,r}(\HS)}
\leq C|\lambda|^{-(1-\frac{\sigma}{2})} \|\bPhi_1\|_{\CB^{s-\sigma}_{q,r}(\HS)}. 
\end{align*}
which together with the second estimate in \eqref{sest:1} implies that 
$$\|(\lambda, \lambda^{1/2}\nabla, \nabla_b^2)
\pd_\lambda \CS_2(\lambda)\bPhi\|_{\CB^s_{q,r}(\HS)} 
\leq C|\lambda|^{-(1-\frac{\sigma}{2})}
\|\bPhi\|_{\CB^{s-\sigma}_{q,r}(\HS)}.$$
\par
Finally, we estimate $\pd_\lambda \CP(\lambda)\bPhi$. To this end, we write
$$\pd_\lambda \CP(\lambda)\bPhi = \pd_\lambda \CQ(\lambda)\bPhi'_3-(\pd_\lambda \CQ(\lambda))
\nabla\CS_3(\lambda)\bPhi_1 - \CQ(\lambda)\nabla \pd_\lambda \CS_3(\lambda) \bPhi_1.
$$
By Corollary \ref{lem-CW45Q}, we have
\begin{align*}
\|\nabla_b\pd_\lambda \CP(\lambda)\bPhi\|_{\CB^s_{q,r}(\HS)}
& \leq C|\lambda|^{-(1-\frac{\sigma}{2})} \Big(\|\bPhi_3\|_{\CB^s_{q,r}(\HS)} 
+ \|\nabla^2\CW_1(\lambda)\bPhi_1\|_{\CB^{s-\sigma}_{q,r}(\HS)}\Big) \\
& \quad + C\|\nabla^2\pd_\lambda \CW_1(\lambda)\bPhi_1\|_{\CB^s_{q,r}(\HS)}.\\
& \leq C|\lambda|^{-(1-\frac{\sigma}{2})}\|\bPhi\|_{\CB^{s-\sigma}_{q,r}(\HS)}.
\end{align*} 
Set $\CS(\lambda)\bPhi = \CW_1(\lambda)\bPhi_1 + \CS_2(\lambda)\bPhi$.
Then, the estimates in the third and fourth assertions of 
Theorem \ref{th-sth} has been proved for any $\lambda \in \Sigma_\epsilon + \gamma_b$
and $\bPhi \in C^\infty_0(\HS)^{M_d}$.
Since $C^\infty_0(\HS)$ is dense in $\CB^s_{q,r}(\HS)$, we have the estimate in $(3)$ for 
any $\bPhi \in \CB^s_{q,r}(\HS)^{M_d}$ by the density argument.  
From the definition, we see that $\CS(\lambda)\bsF = \bw_1+\bu_2$ 
and $\CP(\lambda)\bsF = \fq_2=\fq$ are solutions of \eqref{resol:1.1}.  
The proof of Theorem \ref{th-sth} is now complete.
\subsection{Uniqueness of the generalized resolvent problem.}
The final task for the generalized resolvent problem is to prove 
the uniqueness of solutions. Namely, we consider the equations:
\begin{equation}
\label{uniqueeq:0} 
\left\{\begin{aligned}
\lambda \bu - \DV(\mu\BD (\bu) - \fp \BI) &= 0&\quad&\text{in $\HS$}, \\
\dv \bu&=0&\quad&\text{in $\HS$}, \\
(\mu\BD (\bu) - \fp \BI)\bn_0 &=0 &\quad&\text{on $\pd\HS$}
\end{aligned}\right.
\end{equation}
and shall show the following theorem. 
\begin{thm}
\label{thm:unique} 
Let $1 < q < \infty$, $1 \leq r \leq \infty-$, $-1+1/q < s < 1/q$, 
$\epsilon \in (0, \pi/2)$, and $\gamma>0$. 
Let $\lambda \in \Sigma_\epsilon + \gamma_b$.
If  $\bu \in \sB^{s+2}_{q,r}(\HS)^d$ and $\fp \in \sB^{s+1}_{q,r}(\HS) + \wh \CB^{s+1}_{q, r, 0}(\HS)$
satisfy equations \eqref{uniqueeq:0}, then $\bu=0$ and $\fq = 0$. 
\end{thm}
\begin{proof}
Let $\bff$ be any element in $C^\infty_0(\HS)^d$. 
Since  $-1+1/q' < -s < 1/q'$, $\bff \in \CB^{-s}_{q', r'}(\HS)^d$. 
Let $\bff = \bg + \nabla\fq_2$ be the second Helmholtz
decomposition of $\bff$ with $\bg \in \CJ^{-s}_{q', r'}$ and  
$\fq_2 \in \wh \CB^{-s+1}_{q', r', 0}(\HS)$. 
Applying Theorem \ref{th-sth} for $\bg = \bff - \nabla \fq$ gives  the existence of  
$\bv \in \sB^{-s+2}_{q', r'}(\HS)^d$ and $\fq_1 \in \CB^{-s+1}_{q', r'}(\HS)$,  which satisfies equations:
\begin{equation}
\label{uniqueeq:1} 
\left\{\begin{aligned}
\lambda \bv - \DV(\mu\BD (\bv) - \fq_1 \BI) &= \bff-\nabla\fq_2&\quad&\text{in $\HS$}, \\
\dv \bv&=0&\quad&\text{in $\HS$}, \\
(\mu\BD (\bv) - \fq_1 \BI)\bn_0 &=0 &\quad&\text{on $\pd\HS$}.
\end{aligned}\right.
\end{equation}
Let us write $\fp = \fp_1+\fp_2$ with $\fp_1 \in \sB^{s+1}_{q,r}(\HS)$, 
$\fp_2 \in \wh\CB^{s+1}_{q,r,0}(\HS)$. Since $\dv\bu=\dv\bv=0$, 
it follows from Theorem \ref{thm-solenoidal-characterization} that  $\bu \in \CJ^s_{q,r}(\HS)$ and 
$\bv \in \CJ^{-s}_{q', r'}$. 
In particular, 
we have $(\bu, \nabla \fq)=(\nabla \fp_2, \bv)=0$. Let $D_{ij}(\bu)$ stands for the
$(i, j)$-th component of $\BD(\bu)$ and we write $\bu=(u_1, \ldots, u_d)$, 
and $\bv=(v_1, \ldots, v_d)$. Moreover, let $\delta_{jk}$ be the Kronecker delta, i.e., 
$\delta_{jj}=1$ and $\delta_{jk}=0$ for $j\not=k$. 
 Thus, by the Gauss divergence theorem
we have
\begin{align*}
&(\bu, \bff) = (\bu, \lambda \bv - \DV(\mu\BD(\bv) - \fq_1))\\
&= \lambda (\bu, \bv) - (\bu|_{\pd\HS}, (\mu \BD(\bv) - \fq_1\BI)\bn_0|_{\pd\HS})_{\pd\HS} + 
\sum_{j.k=1}^d(\pd_k u_j, \mu D_{jk}(\bv))\\
&=\lambda(\bu, \bv) 
+ \frac\mu2\sum_{j,k=1}^d (D_{jk}(\bu),  D_{jk}(\bv)).
\end{align*}
Here, we have used the fact that $D_{jk}(\bv) = D_{kj}(\bv)$ 
and written $(f, g)_{\pd\HS} = \int_{\pd\HS} f(x')g(x')\d x'$ with $x'=
(x_1, \ldots, x_{d-1}) \in \BR^{d-1}$. Since  $(\mu\BD(\bv) - \fq\BI)\bn_0|_{\HS}=0$,
we arrive at the equality:
$$(\bu, \bff) = \lambda(\bu, \bv) + \frac{\mu}{2}(\BD(\bu), \BD(\bv)).$$
Likewise, recalling that $(\nabla\fp_2, \bv)=0$ and $\fp=\fp_1$ on $\pd\HS$,  we observe
\begin{align*}
0 &= (\lambda \bu - \DV (\mu\BD(\bu) - \fp \BI), \bv) \\
&= \lambda(\bu, \bv) + ((\mu\BD(\bu) - \fp)\bn_0|_{\pd\HS}, \bv|_{\pd\HS})_{\pd\HS}
+ \frac{\mu}{2}(\BD(\bu), \BD(\bv)) - (\fp_1, \dv\bv)\\
& = \lambda(\bu, \bv) + \frac{\mu}{2}(\BD(\bu), \BD(\bv)).
\end{align*}
Thus, we obtain $(\bu, \bff) = 0$.  Since $\bff$ is chosen arbitrarily in $C^\infty_0(\HS)^d$, 
we have $\bu=0$. In addition, it follows from \eqref{uniqueeq:1} that
$\nabla \fp=0$ and $\fp|_{\pd\HS} = 0$, which yields $\fp=0$.  
The proof of Theorem \ref{thm:unique} is complete.
\end{proof}

\section{Maximal $L_1$-regularity in the half-space}
\label{sec-5}
In this section, we consider the following equations:
\begin{equation}\label{eq:5.1}\left\{ \begin{aligned}
\pd_t \bU - \DV (\mu\BD (\bU) - P \BI) & = \bF 
 & \quad & \text{in $\HS \times \BR$}, \\
\dv \bU & = 0& \quad & \text{in $\HS \times \BR$}, \\
(\mu\BD (\bU) - P \BI) \bn_0 & = \bH \vert_{\pd \HS} & \quad & \text{on $\pd \HS \times \BR$}
\end{aligned}\right. \end{equation}
under the assumptions: 
\begin{equation}\label{assump:5.-1}
\dv\bF = 0\quad \text{in $\HS\times\BR$}, \qquad \bH = (H_1, \ldots, H_{d-1},0).
\end{equation}
 We shall prove the following theorem. 
\begin{thm}
\label{thm:5.1} 
Let $1 < q < \infty$, $-1+1/q < s < 1/q$, and $\gamma > 0$. 
If $\bF$ and $\bH$ satisfy the assumption \eqref{assump:5.-1} 
and the following regularity conditions:
\begin{equation}\label{assump:5.0}\begin{aligned}
e^{-\gamma_b t}\bF\in L_1(\BR, \CB^s_{q,1}(\HS)^d), \quad 
e^{-\gamma_b t} \bH
\in \CW^{1/2}_1(\BR, \CB^{s}_{q,1}(\HS)^d) \cap L_1(\BR, \CB^{s+1}_{q,1}(\HS)^d),
\end{aligned}\end{equation}
 then problem 
\eqref{eq:5.1} admits a solution $(\bU, P)$ with
\begin{equation}\label{regularity:1}
e^{-\gamma_bt}\bU \in \CW^1_1(\BR, \CB^s_{q,1}(\HS)^d) \cap 
  L_1 (\BR, \CB^{s+2}_{q,1}(\HS)^d), \quad
e^{-\gamma_b t}\nabla P \in L_1(\BR, \CB^s_{q,1}(\HS)^d)
\end{equation}
satisfying the estimate:
\begin{equation}\label{est:5.1}\begin{aligned}
&\lVert e^{-\gamma_bt}(\pd_t, \nabla_b^2)\bU \rVert_{L_{1} (\BR_+, \CB^s_{q,1}(\HS))}
+ \|e^{-\gamma_bt}\nabla P\|_{L_1(\BR, \CB^s_{q,1}(\HS)} \\
&\quad \le C \Big(\lVert e^{-\gamma_bt}(\bF, \nabla \bH)\rVert_{L_{1} (\BR, B^s_{q, 1} (\HS))}
+ \|e^{-\gamma_bt}\bH\|_{\CW^{1/2}_1(\BR, \CB^{s}_{q,1}(\HS))} \Big).
\end{aligned}\end{equation}
Here, $\CW^{1/2}_1(\BR, \CB^{s}_{q,1}(\HS)^d)$ is the symbol introduced in
Section \ref{sec-2.2}.
\end{thm}
\begin{proof} The proof is divided into the following two steps. 
Firstly, we assume that 
\begin{equation}\label{assump:5.1}
\bF\in C^\infty_0(\BR, \CB^s_{q,1}(\HS)^d), \quad \bH
\in C^\infty_0(\BR, \sB^{s+1}_{q,1}(\HS)^d).
\end{equation}
Secondly, if $\bF$ and $\bH$ satisfies \eqref{assump:5.0}, then there exist
sequences $\{\bF_j\}_{j=1}^\infty  \subset C^\infty_0(\BR, \CB^s_{q,1}(\HS)^d)$ and
$\{\bH_j\}_{j=1}^\infty \subset C^\infty_0(\BR, \CB^{s+1}_{q,1}(\HS)^d)$ such that 
\begin{align}
\label{limit:5.1}
\lim_{j\to\infty} \|e^{-\gamma_b}(\bF_j-\bF)\|_{L_1(\BR, \CB^s_{q,1}(\HS))} & = 0, \\
\label{limit:5.2}
\lim_{j\to\infty} \|e^{-\gamma_b}(\bH_j-\bH)\|_{L_1(\BR, \CB^{s+1}_{q,1}(\HS))} & = 0, \\
\label{limit:5.3}
\lim_{j\to\infty} \|e^{-\gamma_b}(\bH_j-\bH)\|_{\CW^{1/2}_1(\BR, \CB^{s}_{q,1}(\HS))} & = 0.
\end{align}
\textbf{Step 1:} 
Applying the Fourier--Laplace transform to \eqref{eq:5.1},
we have the following generalized resolvent problem. 
\begin{equation}\label{eq:5.2}\left\{ \begin{aligned}
\lambda \bu - \DV (\mu\BD (\bu) - \fq \BI) & = \CL[\bF] 
 & \quad & \text{in $\HS$}, \\
\dv \bu & = 0& \quad & \text{in $\HS$}, \\
(\mu\BD (\bu) - \fq \BI) \bn_0 & = \CL[\bH] \vert_{\pd \HS} & \quad & \text{on $\pd \HS$}.
\end{aligned}\right. \end{equation}
Since $\dv \CL[\bF]=0$ and $\CL[\bH] = (\CL[H_1], \ldots, \CL[H_{d-1}], 0)$,
 in view of Theorem \ref{th-sth}, by using the solution operators $\CS(\lambda)$ and
 $\CP(\lambda)$  given in Theorem \ref{th-sth},  we see that
$\bu= \CS(\lambda)\bCF$ and  $\fq =\CP(\lambda)\bCF$ are solutions to the generalized
resolvent problem \eqref{eq:5.1}, where we have set
$\bCF = (\CL[\bF], \lambda^{1/2}\CL[\bH], \nabla \CL[\bH)$. Let $\Lambda_{\gamma_b}^{1/2}$
be a symbol defined by 
\begin{equation}\label{half-derivative.1}
\Lambda_{\gamma_b}^{1/2}f = \CL^{-1}[\lambda^{1/2}\CL[f](\lambda)], 
\qquad \lambda=\gamma_b+i\tau, \enskip \tau\in\BR.
\end{equation}
With this symbol, we write $\bCF = \CL[(\bF, \Lambda_{\gamma_b}^{1/2}\bH, \nabla\bH)]$. 
Since we assume $\bF \in C^\infty_0(\BR, \CB^s_{q,1}(\HS)^d)$ and $\bH \in C^\infty_0(\BR, 
\CB^{s+1}_{q,1}(\HS)^d)$, we have $\CL^{-1}\CL[\bF] = \bF$ and $\CL^{-1}\CL[\bH]
= \bH$.  Thus, setting 
$$\bU = \CL^{-1}[\bu] = \CL^{-1}[\CS(\lambda)\bCF],
\qquad P = \CL^{-1}[\fq]= \CL^{-1}[\CP(\lambda)\bCF],
$$
we see that $\bU$ and $P$ satisfy equations \eqref{eq:5.1}. 
In fact, there hold
\begin{align}
\pd_t\bU - \DV(\mu \BD(\bU) - P \BI) 
= \CL^{-1}\Big[\lambda \CS(\lambda)\bCF 
- \DV\Big(\mu \BD(\CS(\lambda)\bCF) - \CP(\lambda)\bCF \BI \Big)\Big] 
= \CL^{-1}[\CL[\bF]] 
= \bF 
\end{align}
in $\HS\times\BR$ and
\begin{align}
\mathrm{tr}_0 \Big[(\mu \BD(\bU) - P \BI)\bn_0 \Big] 
& = \mathrm{tr}_0 \Big[\CL^{-1}[\mu \BD(\CS(\lambda)\bCF) - \CP(\lambda)\bCF \BI]\bn_0 \Big] \\
& = \CL^{-1}\Big[\mathrm{tr}_0 
\Big[\Big(\mu \BD(\CS(\lambda)\bCF) - \CP(\lambda)\bCF \BI\Big)\bn_0 \Big] \Big] \\
& = \CL^{-1} \Big[\mathrm{tr}_0 [\CL[\bH]] \Big] \\
& = \mathrm{tr}_0 \Big[\CL^{-1}[\CL[\bH]] \Big] \\ 
& = \mathrm{tr}_0 [\bH].
\end{align}
Here, we have used the fact that $\mathrm{tr}_0$ and $\CL^{-1}$ commute. Indeed,
$\bCF$ is a $C^\infty_0(\BR)$-function with respect to $t$, and hence the Fourier--Laplace
inverse transform $\CL^{-1}$ is defined by the Riemann integral. In addition,
$\mathrm{tr}_0$ is a bounded linear operator as follows from Proposition \ref{prop-trace}. 
Thus, $\mathrm{tr}_0$ and $\CL^{-1}$ commute.\par
Now, we shall discuss the $L_1$-in-time estimates of $\bU$ and $P$. By the Fubini theorem, we may write
\begin{equation}\label{repr:5.1}\begin{aligned}
\bU & = \CL^{-1}[\CS(\lambda)\bCF] \\
& =\frac{1}{2\pi}\int_{\BR} e^{(\gamma_b+i\tau)}\CS(\gamma_b+i\tau)
\bigg(\int_{\BR}e^{-(\gamma_b + i\tau) \ell} \CL^{-1}[\CF](\,\cdot\,, \ell)\d \ell\bigg)\d\tau
\\
&= \int_{\BR} \bigg(\int_{\BR} e^{(\gamma_b + i\tau)(t-s) }\CS(\gamma_b + i\tau)
(\bF, \Lambda_{\gamma_b}^{1/2}\bH, \nabla\bH)(\,\cdot\,, \ell)
\d\tau\bigg)\d\ell. 
\end{aligned}\end{equation}
Analogously, we may write
\begin{equation}\label{repr:5.2}\begin{aligned}
P & = \CL^{-1}[\CP(\lambda)\bCF] 
= \int_{\BR} \bigg(\int_{\BR} e^{(\gamma_b + i\tau)(t-s) }\CP(\gamma_b + i\tau)
(\bF, \Lambda_\gamma^{1/2}\bH,\nabla\bH)(\,\cdot\,, \ell)\d\tau\bigg)\d\ell.
\end{aligned}\end{equation}
To prove the desired estimates, 
for $\bPhi = (\bPhi_1, \bPhi_2', \bPhi_3') \in C^\infty_0(\HS)^{M_d}$, we consider  operators 
$$T(t)\bPhi = \frac{1}{2\pi}\int_{\BR}e^{(\gamma_b+i\tau)t} \CS(\gamma_b + i\tau)\bPhi\, d\tau, 
\quad 
P(t)\Phi = \frac{1}{2\pi}\int_{\BR} e^{(\gamma_b+i\tau) t}\CP(\gamma_b+i\tau)\bPhi\d \tau.
$$
From \eqref{repr:5.1} and \eqref{repr:5.2}  we may write 
\begin{equation}\label{repr:5.3}
\bU = \int_{\BR} T(t-\ell)(\bF, \Lambda_\gamma^{1/2}\bH, \nabla\bH)(\,\cdot\,, \ell)\d\ell,
\quad
P = \int_{\BR} P(t-\ell)(\bF, \Lambda_\gamma^{1/2}\bH, \nabla\bH)(\,\cdot\,, \ell)\d\ell.
\end{equation}
In the following, we follow the argument appearing in the classical theory of a $C_0$-analytic semigroup 
(cf. \cite[Ch. IX]{Yosida}). Let 
$\Gamma_\omega
= \Gamma_{\omega +} \cup \Gamma_{\omega -} \cup C_\omega$ and 
$\Gamma=\Gamma_+ \cup \Gamma_-$, where we have set
\begin{equation}\label{contour.1}\begin{aligned}
\Gamma_{\omega \pm}& = \{\lambda = re^{\pm i(\pi-\epsilon)} \mid \omega  < r < \infty\}, \\ 
C_\omega & =  \{\lambda = \omega e^{i\theta} \mid -(\pi-\epsilon) < \theta < \pi-\epsilon\}, \\
\Gamma_\pm &= \{\lambda \in \BC \mid \lambda = re^{\pm i(\pi-\epsilon)}, \enskip
r \in (0, \infty)\}.
\end{aligned}\end{equation}
We shall show that
\begin{alignat}{4}
T(t)\bPhi &= 0 &\qquad & \text{for $t < 0$}, \label{t.3.6.1} \\
T(t)\bPhi &= \frac{1}{2\pi i}\int_{\Gamma_\omega+\gamma_b} e^{\lambda t}\CS(\lambda)\bPhi \d\lambda
&\qquad &\text{for $t > 0$}. \label{t.3.6.2}
\end{alignat}
To this end, we take
\begin{align}
C^{(+)}_{R, \epsilon} & = \bigg\{\lambda \in \BC \,\left\vert\enskip \lambda = Re^{i\theta},
\enskip \theta \colon -\frac{\pi}{2} \to -(\pi-\epsilon) \bigg\}\right.
\cup
\bigg\{\lambda \in \BC \,\left\vert\enskip \lambda=Re^{i\theta},
\enskip \theta \colon  \pi-\epsilon \to  \frac{\pi}{2} \bigg\}\right. \\
C_R^{(-)} & = \bigg\{\lambda \in \BC \,\left\vert\, \lambda  = Re^{i\theta},
\enskip \theta : -\frac{\pi}{2} \to \frac{\pi}{2} \bigg\}\right., \\ 
\Gamma_{\omega, R} & = \{\lambda \in \Gamma_\omega  \mid |\lambda| \le R\}
\end{align}
for (possibly large) $R > 0$. First, we {\color{black} show \eqref{t.3.6.1}}. 
By the Cauchy theorem in theory
of one complex variable, we write 
\begin{equation}
\frac{1}{2\pi} \int^R_{-R}e^{(\gamma_b + i\tau) t}
\CS(\gamma_b + i\tau) \bPhi\d \tau
= \frac{1}{2\pi i}\int_{ C_R^{(-)} + \gamma_b} e^{\lambda t} \CS(\lambda)\bPhi\d\lambda.
\end{equation}
By Theorem  \ref{th-sth}, we know that
\begin{equation}\label{t.3.4.1} 
\|(\lambda, \nabla_b^2)\CS(\lambda)\bPhi\|_{\CB^s_{q,1}(\HS)} 
\le C\|\bPhi\|_{\CB^s_{q,1}(\HS)}\qquad\text{for $\lambda \in \Sigma_\epsilon + \gamma_b$}.
\end{equation}
Noting that there holds $|\lambda| \geq R$ for $\lambda \in  C_R^{(-)} + \gamma_b$
for large $R > 0$, it follows from \eqref{t.3.4.1} that
\begin{align}
\bigg\|\frac{1}{2\pi i}\int_{C_R^{(-)} + \gamma_b} e^{\lambda t} \CS(\lambda)\bPhi\d\lambda
\bigg\|_{\CB^s_{q,1}(\HS)}
& \le Ce^{\gamma_bt}\int^{\pi/2}_{-\pi/2}e^{-|t|R\cos\theta}\d\theta \;\|\bPhi\|_{\CB^s_{q,1}(\HS)}\\
& \le 2Ce^{\gamma_b t}\int^{\pi/2}_{0}e^{-|t|R\sin\theta}\d\theta \;\|\bPhi\|_{\CB^s_{q,1}(\HS)} \\
& \le 2Ce^{\gamma_b t}\int^{\pi/2}_0 e^{-(2|t|R/\pi)\theta}\d\theta \;\|\bPhi\|_{\CB^s_{q,1}(\HS)} \\
& \le \frac{Ce^{\gamma_b t}}{|t|R}\|\bPhi\|_{\CB^s_{q,1}(\HS)}.
\end{align}
{\color{black} Hence, it follows from the Fatou property (Proposition \ref{prop-Fatou}) that
\begin{align}
\lVert  T(t) \bPhi \rVert_{B^s_{q, 1} (\HS)} \le \liminf_{R \to \infty} 
\bigg\|\frac{1}{2\pi i}\int_{C_R^{(-)} + \gamma_b} e^{\lambda t} \CS(\lambda)\bPhi\d\lambda
\bigg\|_{B^s_{q,1}(\HS)} = 0
\end{align} 
for each $t < 0$.} This {\color{black} shows} \eqref{t.3.6.1}. \par 
{\color{black}The proof of \eqref{t.3.6.2} is similar to \eqref{t.3.6.1}. In fact}, 
for $t > 0$ it follows from the Cauchy theorem in theory
of one complex variable that
\begin{equation}
\frac{1}{2\pi} \int^R_{-R}e^{(\gamma_b + i\tau) t}
\CS(\gamma_b + i\tau) \bPhi\d \tau
= \frac{1}{2\pi i}\int_{\Gamma_{\omega, R}+\gamma_b } e^{\lambda t} 
\CS(\lambda)\bPhi\d\lambda
+\frac{1}{2\pi i}\int_{C_{R, \epsilon}^{(+)} + \gamma_b} e^{\lambda t} \CS(\lambda)\bPhi\d\lambda.
\end{equation}
From \eqref{t.3.4.1}, 
we have 
\begin{align}
\bigg\|\frac{1}{2\pi i}\int_{C_{R, \epsilon}^{(+)}+\gamma_b} 
e^{\lambda t} \CS(\lambda)\bPhi\d\lambda
\bigg\|_{\CB^s_{q,1}(\HS)}
& \le Ce^{\gamma_b t}\int^{\pi-\epsilon}_{\pi/2} e^{tR\cos\theta}\d\theta\;\|\bPhi\|_{\CB^s_{q,1}(\HS)}\\
& \le Ce^{\gamma_b t}\int^{\pi/2}_0 e^{-tR\sin\theta}\d\theta \; \|\bPhi\|_{\CB^s_{q,1}(\HS)} \\
& \le Ce^{\gamma_b t}\int^{\pi/2}_0 e^{-(2tR/\pi)\theta}\d\theta \;\|\bPhi\|_{\CB^s_{q,1}(\HS)}\\
& \le\frac{C e^{\gamma_b t} \pi}{2tR}\;\|\bPhi\|_{\CB^s_{q,1}(\HS)}.
\end{align}
Hence, it follows from the Fatou property (Proposition \ref{prop-Fatou}) that
\begin{align}
\bigg\lVert \liminf_{R \to \infty} \frac{1}{2\pi i}\int_{C_{R, \epsilon}^{(+)} + \gamma_b} 
e^{\lambda t} \CS(\lambda)\bPhi\d\lambda \bigg\rVert_{\CB^s_{q, 1} (\HS)} 
\le \liminf_{R \to \infty} 
\bigg\|\frac{1}{2\pi i}\int_{C_{R, \epsilon}^{(+)} + \gamma_b} e^{\lambda t} \CS(\lambda)\bPhi\d\lambda
\bigg\|_{\CB^s_{q,1}(\HS)} = 0
\end{align} 	
for each $t > 0$.
This shows \eqref{t.3.6.2}.  \par
In view of Proposition \ref{prop-real-interpolation}, 
for every $t > 0$ and $\bPhi \in C^\infty_0(\HS)^{M_d}$ we now prove
\begin{align}
\label{L1est.0-1}
\|\nabla_b^2T(t)\bPhi\|_{B^s_{q,1}(\HS)}
\le Ce^{\gamma_b t}t^{-1+\frac{\sigma}{2}}\|\bPhi\|_{B^{s+\sigma}_{q,1}(\HS)}, \\
\label{L1est.0-2}
\|\nabla_b^2 T(t)\bPhi_1\|_{B^s_{q,1}(\HS)}
\le Ce^{\gamma_b t}t^{-1-\frac{\sigma}{2}}\|\bPhi\|_{B^{s-\sigma}_{q,1}(\HS)}.
\end{align}
In fact, since there holds
$\|\nabla_b^2\CS(\lambda)\bPhi\|_{\CB^s_{q,r}(\HS)} \leq C\|\bPhi\|_{\CB^s_{q,r}(\HS)}$
with a constant independent of $\omega$, 
we may take the limit $\omega \to 0$ in \eqref{t.3.6.2}, and hence we have
$$
\nabla_b^2T(t)\bPhi = \frac{1}{2\pi i}\int_{\Gamma+\gamma_b } e^{\lambda t}\nabla_b^2\CS(\lambda)\bPhi \d\lambda
\qquad \text{for $t > 0$}.
$$
 Notice that $\lambda = \gamma_b + re^{\pm i(\pi-\epsilon)}$ for $\lambda
\in \Gamma_\pm + \gamma_b$, and thus $|e^{\lambda t}| = e^{\gamma_b t}
e^{\cos(\pi-\epsilon)rt} =e^{\gamma_b t} e^{-rt\cos\epsilon}$
for $\lambda \in \Gamma_\pm + \gamma_b$. Since 
$\| \nabla_b^2\CS(\lambda)\bPhi\|_{\CB^s_{q,1}(\HS)} \le C|\lambda|^{-\frac{\sigma}{2}}
\|\bPhi\|_{\CB^{s+\sigma}_{q,1}(\HS)}$ as follows from the fourth assertion in Theorem \ref{th-sth},
for $t>0$ we infer from \eqref{t.3.6.2} that
\begin{align}
\| \nabla_b^2T(t)\bPhi\|_{\CB^s_{q,1}(\HS)}
&\le Ce^{\gamma_b t}\int^\infty_0 e^{-rt\cos\epsilon }r^{-\sigma/2}\d r 
\|\bPhi\|_{\CB^{s+\sigma}_{q, 1}(\HS)}\\
&= Ce^{\gamma_b t}t^{-1+\frac\sigma2}\int^\infty_0e^{-\ell \cos\epsilon}\ell^{-\sigma/2}\d\ell\,
\|\bPhi\|_{\CB^{s+\sigma}_{q, 1}(\HS)},
\end{align}
which yields \eqref{L1est.0-1}. To prove \eqref{L1est.0-2}, 
by integration by parts, we write 
\begin{equation}
T(t)\bPhi = -\frac{1}{2\pi i t}\int_{\Gamma + \gamma_b} e^{\lambda t}\pd_\lambda  \CS(\lambda)
\bPhi \d\lambda.
\end{equation}
Since $\|\nabla_b^2\pd_\lambda\CS(\lambda)\bPhi\|_{\CB^s_{q,1}(\HS)} 
\le C|\lambda|^{-(1-\frac{\sigma}{2})}
\|\bPhi\|_{\CB^{s-\sigma}_{q,1}(\HS)}$ as follows from the fifth assertion in Theorem \ref{th-sth}, 
we have
\begin{align}
\|\nabla_b^2T(t)\bPhi\|_{\CB^s_{q,1}(\HS)}
&\le Ct^{-1}e^{\gamma_b t}\int^\infty_0 e^{-rt\cos\epsilon }r^{-(1-\frac\sigma2)} \d r \;
\|\bPhi\|_{\CB^{s-\sigma}_{q, 1}(\HS)}\\
&= Ce^{\gamma_b t}t^{-1-\frac\sigma2}\int^\infty_0e^{-\ell \cos\epsilon}  
\ell^{-1+\frac\sigma2}\d\ell \;
\|\bPhi\|_{\CB^{s-\sigma}_{q, 1}(\HS)},
\end{align}
which yields \eqref{L1est.0-2}. 
Thus, by Proposition \ref{prop-real-interpolation} we have
\begin{equation}\label{L1est.1}
\int_0^\infty e^{- \gamma_b t} \|\nabla_b^2T(t)\bPhi\|_{\CB^s_{q,1}(\HS)} \d t
\le C\|\bPhi\|_{(\CB^{s+\sigma}_{q,1}(\HS), \CB^{s-\sigma}_{q,1}(\HS))_{1/2,1}}
\le C\|\bPhi\|_{\CB^s_{q,1}(\HS)}.
\end{equation}
Since $C^\infty_0(\HS)$ is dense in $\CB^s_{q,1}(\HS)$, we have
\eqref{L1est.1} for {\color{black} any} $\bPhi \in \CB^s_{q,1}(\HS)^{M_d}$. \par
Applying \eqref{t.3.6.1} to the formula of $\bU$ in \eqref{repr:5.3} yields
$$\bU = \int^t_{-\infty}T(t-\ell)(\bF, \Lambda_{\gamma_b}^{1/2} \bH, \nabla \bH)(\,\cdot\,, \ell)\d\ell.
$$
Using \eqref{L1est.1} and the Fubini theorem, we have
\allowdisplaybreaks 
\begin{align*}
\int^\infty_0 e^{-\gamma_b t}\|\nabla_b^2 \bU(\,\cdot\,, t)\|_{\CB^s_{q,r}(\HS)}\d t
&\leq \int^\infty_0 e^{-\gamma_b t}\bigg(\int^t_{-\infty}\|\nabla^2T(t-\ell)(\bF, \Lambda_{\gamma_b}^{1/2}\bH, 
\nabla \bH)(\,\cdot\,, \ell)\|_{\CB^s_{q,1}(\HS)}\d\ell\bigg)\d t \\
& \leq \int^\infty_{-\infty}\bigg(\int^\infty_\ell 
e^{-\gamma_b t}\|\nabla^2T(t-\ell)(\bF, \Lambda_{\gamma_b}^{1/2}\bH, 
\nabla \bH)(\,\cdot\,, \ell)\|_{\CB^s_{q,1}(\HS)}\d t\bigg)\d\ell \\
& \leq C\int^\infty_{-\infty} e^{-\gamma_b \ell}
\|(\bF, \Lambda_{\gamma_b}^{1/2}\bH, 
\nabla \bH)(\,\cdot\,, \ell)\|_{\CB^s_{q,1}(\HS)}\d\ell.
\end{align*}
Since $\bH \in C^\infty_0(\BR, \CB^{s+1}_{q,1}(\HS)^d)$,
we have $\|\CL[\bH](\lambda)\|_{\CB^s_{q,1}(\HS)} \leq C|\lambda|^{-m}
\|e^{-\gamma_b t}\pd_t^m \bH\|_{L_1(\BR, \CB^s_{q,1}(\HS)}$ for any $m \in \BN_0$
and $\lambda \in \BC$ 
and $\CL[\bH](\lambda)$ is an entire function 
in $\BC$.
Thus, by the Cauchy theorem in theory of one complex variable, we have
$\Lambda_\gamma^{1/2}\bH = \Lambda_1^{1/2}\bH$ if $\gamma_b=\gamma>0$.  Thus, we have
\begin{equation}\label{L1est.2}
\int^\infty_0 e^{-\gamma_b t}\|\nabla_b^2\bU(\,\cdot\,, t)\|_{\CB^s_{q,1}(\HS)}\d t
\leq C\int^\infty_{-\infty} e^{-\gamma_b t}\|(\bF, \Lambda^{1/2}_b\bH, \nabla \bH)\|_{\CB^s_{q,1}(\HS)}\d t,
\end{equation}
where we have set $\Lambda^{1/2}_b\bH = \Lambda^{1/2}_1\bH$ if $\gamma_b=\gamma>0$ and 
$\Lambda^{1/2}_b \bH= \Lambda^{1/2}_0\bH$ if $\gamma_b=0$. 
\par
Finally, by Proposition \ref{half-derivative1} in Appendix B, there hold
\begin{align*}\int^\infty_{-\infty} e^{-\gamma t}\|\Lambda^{1/2}_1\bH\|_{B^s_{q,1}(\HS)}\d t
&\leq \|e^{-\gamma t}\bH\|_{W^{1/2}_1(\BR, B^s_{q,r}(\HS))}, 
\\
\int^\infty_{-\infty} \|\Lambda^{1/2}_0\bH\|_{\dot B^s_{q,1}(\HS)}\d t
&\leq \|\bH\|_{\dot W^{1/2}_1(\BR, \dot B^s_{q,r}(\HS))}.
\end{align*}
Summing up, we have obtained
\begin{equation}\label{lastest:5.1}
\int^\infty_0 e^{-\gamma_bt}\|\nabla_b^2\bU(\,\cdot\, , t)\|_{\CB^s_{q,1}(\HS)}
\leq C \Big(\|e^{-\gamma_bt}(\bF, \nabla\bH)\|_{L_1(\BR, \CB^s_{q,1}(\HS))}
+ \|e^{-\gamma_bt}\bH\|_{\CW^{1/2}_1(\BR, \CB^s_{q,1}(\HS))} \Big).
\end{equation} 
Analogously, we have 
\begin{equation}\label{L1est.3} 
\int^\infty_0 e^{-\gamma_b t}\|\nabla_b P(\,\cdot\, , t)\|_{\CB^s_{q,r}(\HS)}\d t
\leq C \Big(\|e^{-\gamma_bt}(\bF, \nabla\bH)\|_{L_1(\BR, \CB^s_{q,1}(\HS))}
+ \|e^{-\gamma_bt}\bH\|_{\CW^{1/2}_1(\BR, \CB^s_{q,1}(\HS))} \Big).
\end{equation}
Since $\bU$ and $P$ satisfy the first equation in \eqref{eq:5.1}, we have
$$\|\pd_t\bU(\,\cdot\, , t)\|_{\CB^s_{q,1}(\HS)} \leq \|\DV(\mu \BD(\bU) - P \BI)\|_{\CB^s_{q,1}(\HS)}
+ \|\bF(\,\cdot\, , t)\|_{\CB^s_{q,1}(\HS)}.
$$
Thus, using \eqref{lastest:5.1} and \eqref{L1est.3} 
we obtain
\begin{align}\label{L1est.4} 
&\int^\infty_0 e^{-\gamma_b t} \Big(\|(\pd_t,  \nabla_b^2)\bU(\,\cdot\, , t)\|_{\CB^s_{q,1}(\HS)}
+ \|\nabla_b P(\,\cdot\, , t)\|_{\CB^s_{q,r}(\HS)} \Big) \d t \\
&\qquad \leq C \Big(\|e^{-\gamma_bt}(\bF, \nabla\bH)\|_{L_1(\BR, \CB^s_{q,1}(\HS))}
+ \|e^{-\gamma_bt}\bH\|_{\CW^{1/2}_1(\BR, \CB^s_{q,1}(\HS))} \Big).
\end{align}
\vskip.4pc
\noindent \textbf{Step 2:} To verify that $(\bU, P)$ is a desired solution and
the estimate \eqref{est:5.1} holds under the assumption \eqref{assump:5.0}, 
we first approximate $\bF \in L_1(\BR, \CB^s_{q,1}(\HS)^d)$. 
Let 
$\varphi(t)$ be a smooth function defined on $\BR$ such that
$\varphi (t) = 1$ for $\lvert t \rvert \le 1$ and $\varphi (t) = 0$ for $\lvert t \rvert \ge 2$.
In addition, let $\psi_1 (t) \in C^\infty_0(\BR)$ be a nonnegative function such that
$\int_{- \infty}^\infty \psi_1 (t)\d t=1$. Set $\varphi_R(t) = \varphi(t/R)$ and 
$\psi_{1, \sigma}(t) = \sigma^{-1}\psi_1 (t/\sigma)$. For any $R>0$ and $\sigma>0$, we set 
$(e^{-\gamma_b t}\bF)_{\sigma, R} = \psi_{1, \sigma} * (\varphi_R e^{-\gamma_bt} \bF)$, 
where the symbol $*$ stands for
the convolution with respect to the time variable $t$. 
By the standard argument for the convolution, we see that 
$(e^{-\gamma_bt}\bF)_{R,\sigma} \in C^\infty_0 (\BR, \CB^s_{q, 1} (\HS)^d)$ and
\begin{equation}
\label{limitFReps}
\lim_{\sigma\to0} \lim_{R\to\infty} 
\|e^{-\gamma_b t}(e^{\gamma_bt}(e^{-\gamma_b t}\bF)_{R,\sigma} - \bF) \|_{L_{1} (\BR, B^s_{q,1}(\HS))} = 0
\end{equation}
as follows from $e^{-\gamma_bt}\bF \in L_1(\BR, \CB^s_{q,1}(\HS)^d)$. Since there holds
$(e^{-\gamma_bt}\bF)_{R,\sigma} \in C^\infty_0 (\BR, \CB^s_{q, 1} (\HS)^d)$, we see that
$e^{\gamma_bt}(e^{-\gamma_b t}\bF)_{R,\sigma}$ also belongs to $C^\infty_0 (\BR, \CB^s_{q, 1} (\HS)^d)$. \par
We next approximate $\bH \in \CW^{1/2}_1(\BR, \CB^s_{q,1}(\HS)^d) \cap L_1(\BR, \CB^{s+1}_{q,1}(\HS)^d)$.
For any $R>0$ and $\sigma>0$, we set 
$(e^{-\gamma_b t}\bH)_{\sigma, R} = \psi_{1, \sigma} * (\varphi_R e^{-\gamma_bt} \bH)$.
Then there holds
\begin{equation}
\label{limit:5.1.1}
\lim_{\sigma\to0} \lim_{R \to\infty}
\| e^{-\gamma_bt}(e^{\gamma_bt}(e^{-\gamma_bt}\bH)_{R, \sigma} - \bH)\|_{L_{1} (\BR, \CB^{s+1}_{q,1}(\HS))} = 0,
\end{equation}
since $e^{-\gamma_bt} \bH \in L_1(\BR, \CB^{s+1}_{q,r}(\HS)^d)$. Obviously, we have
$e^{\gamma_bt}(e^{-\gamma_bt}\bH)_{R, \sigma} \in C^\infty_0(\BR, \CB^{s+1}_{q,1}(\HS)^d)$.
Furthermore, using the assumption: 
 $e^{-\gamma_b t }\bH \in \CW^{1/2}_1(\BR, \CB^{s}_{q,r}(\HS)^d)$, we shall show that 
\begin{equation}
\label{limit:5.1.4}
\lim_{\sigma\to0} \lim_{R\to\infty}\|e^{-\gamma_bt}(e^{\gamma_b t}(e^{-\gamma_bt}\bH)_{R, \sigma} 
- \bH) \|_{\CW^{1/2}_1(\BR, \CB^s_{q,1}(\HS))}=0.
\end{equation}
To prove this, we first notice the following fact: 
For any $\bG \in  W^1_1(\BR, \CB^s_{q,1}(\HS)^d)$,
we see that 
\begin{equation}
\begin{aligned}
\|(\bG)_{R, \sigma} \|_{L_1(\BR, \CB^s_{q,1}(\HS))} 
& \le C\|\bG\|_{L_1(\BR, \CB^s_{q,1}(\HS))},  \\
\|(\bG)_{R,\sigma}\|_{W^1_1 (\BR, \CB^s_{q,1}(\HS))} 
& \le C\|\bG\|_{W^1_1 (\BR, \CB^s_{q,1}(\HS))},
\end{aligned}
\end{equation}
where we have set $(\bG)_{R, \sigma} = \psi_{1, \sigma} * (\varphi_R \bG)$.
Interpolating these inequalities yields
\begin{equation}
\label{limit:5.1.2}
\|(\bG)_{R, \sigma}\|_{\CW^{1/2}_1(\BR, \CB^s_{q,1}(\HS))}
\le C\|\bG \|_{\CW^{1/2}_1(\BR, \CB^s_{q,1}(\HS))}
\end{equation}
for $\bG\in W^1_1(\BR, \CB^s_{q,1}(\HS)^d)$. Since $W^1_1(\BR, \CB^s_{q,1}(\HS)^d)$ is 
dense in $\CW^{1/2}_1(\BR, \CB^s_{q,r}(\HS)^d)$, we see that \eqref{limit:5.1.2} 
also holds for any $\bH \in \CW^{1/2}_1(\BR, \CB^s_{q,1}(\HS)^d)$. 
For any $\omega > 0$, we choose $\bG \in W^1_1(\BR, \CB^s_{q,1}(\HS)^d)$ in such a way that
$\|e^{-\gamma_bt}\bH - \bG\|_{\CW^{1/2}_1(\BR, \CB^s_{q,1}(\HS))}< \omega$. We see easily that 
\begin{equation}\label{limit:5.1.5}
\|\bG - (\bG)_{R, \sigma}\|_{\CW^{1/2}_{q,1}(\BR, \CB^s_{q,1}(\HS))} \leq 
C\|\bG - (\bG)_{R, \sigma}\|_{W^1_1(\BR, \CB^s_{q,1}(\HS))} \to 0 \quad
\text{as $\sigma\to0$ and $R\to \infty$}.
\end{equation}
By the
triangle inequality and \eqref{limit:5.1.2}, we obtain
\begin{align*}
&\|e^{-\gamma_bt}(\bH -e^{\gamma_bt}(e^{-\gamma_bt}\bH)_{R, \sigma})\|_{\CW^{1/2}_1(\BR, \CB^s_{q,1}(\HS))}\\
&\leq \|e^{-\gamma_bt}\bH - \bG\|_{\CW^{1/2}_1(\BR, \CB^s_{q,1}(\HS))} 
+ \|\bG-(\bG)_{R, \sigma}\|_{\CW^{1/2}_1(\BR, \CB^s_{q,1}(\HS))} \\
&\quad + \|(\bG)_{R, \sigma} - (e^{-\gamma_bt}\bH)_{R, \sigma}\|_{\CW^{1/2}_1(\BR, \CB^s_{q,1}(\HS))}\\
&\leq (C+1)\omega + \|\bG-(\bG)_{R, \sigma}\|_{\CW^{1/2}_1(\BR, \CB^s_{q,1}(\HS))}.
\end{align*}
Thus, by \eqref{limit:5.1.5}, there holds
$$\liminf_{\substack{\sigma\to0 \\ R\to\infty}} 
\|e^{-\gamma_bt}(\bH -e^{\gamma_bt}(e^{-\gamma_bt}\bH)_{R, \sigma})\|_{\CW^{1/2}_1(\BR, \CB^s_{q,1}(\HS))}
\leq (C+1)\omega.$$
Since $\omega > 0$ is arbitrarily chosen, we have
$$\lim_{\sigma\to0}\lim_{R\to\infty}
\|e^{-\gamma_bt}(\bH -e^{\gamma_bt}(e^{-\gamma_bt}\bH)_{R, \sigma})\|_{\CW^{1/2}_1(\BR, \CB^s_{q,1}(\HS))}
=0.
$$
This completes the proof of \eqref{limit:5.1}, \eqref{limit:5.2}, and \eqref{limit:5.3}. 
\vskip.4pc
\noindent \textbf{Step 3:}
Finally, we shall complete the proof of Theorem \ref{thm:5.1}.
To this end, in view of Step 2, for any
$\bF$ and $\bH$ satisfying conditions \eqref{assump:5.1}, we choose sequences 
$\{\bF_j\}_{j=1}^\infty \subset C^\infty_0(\BR, \CB^s_{q,1}(\HS)^d)$
and $\{\bH_j\}_{j=1}^\infty \subset C^\infty_0(\BR, \sB^{s+1}_{q,1}(\HS)^d)$ such that 
\begin{equation}
\lim_{j\to\infty}\|e^{-\gamma_bt}(\bF-\bF_j)\|_{L_1(\BR, \CB^s_{q,1}(\HS))} =0
\end{equation}
and
\begin{equation}
\lim_{j\to\infty}(\|e^{-\gamma_bt}(\bH-\bH_j)\|_{L_1(\BR, \CB^{s+1}_{q,1}(\HS))} 
+\|e^{-\gamma_bt}(\bH-\bH_j)\|_{\CW^{1/2}_1(\BR, \CB^{s}_{q,1}(\HS))})=0
\end{equation}
are valid, respectively. From Step 1, there exist sequences 
$\{\bU_j\}_{j=1}^\infty$ and $\{P_j\}_{j=1}^\infty$,
which satisfy equations \eqref{eq:5.1}, the regularity conditions \eqref{regularity:1}, 
and the estimate \eqref{est:5.1} with 
 $\bU$, $P$, $\bF$, and $\bH$ replaced by $\bU_j$, $P_j$, $\bF_j$, and $\bH_j$,
respectively. Since  $\bU_j$, $P_j$, $\bF_j$, and $\bH_j$ satisfy the estimate \eqref{est:5.1},
by the standard argument, there exist $(\bU, P)$ such that
$(\bU, P)$ satisfies the regularity conditions \eqref{regularity:1} and 
$$\lim_{j \to \infty}(
\|e^{-\gamma_b t}(\pd_t(\bU_j-\bU), \nabla(P_j-P)\|_{L_1(\BR, \CB^s_{q,1}(\HS))}
+ \|e^{-\gamma_bt}\nabla_b^2(\bU_j-\bU))\|_{L_1(\BR, \CB^s_{q,1}(\HS))}) =0.$$
In particular, $(\bU, P)$ satisfies the estimate \eqref{est:5.1}. \par
Since $\bU_j$, $P_j$, $\bF_j$, and $\bH_j$ satisfy the first and second equations in \eqref{eq:5.1},
and hence $\bU$, $P$, $\bF$, and $\bH$ satisfy the first and second equations in \eqref{eq:5.1} as well. 
Concerning the boundary trace, noting that 
${\rm tr}_0(\mu \BD(\bU_j) - P_j \BI)\bn_0= {\rm tr_0}\bH_j$
and using  Proposition \ref{prop-trace}, we have 
\begin{align*}
& \|(e^{-\gamma_b t}{\rm tr}_0((\mu \BD(\bU)- \bH)\|_{L_1(\BR, \CB^{s+1-1/q}_{q,1}(\HS))} \\
& \leq \|e^{-\gamma_b t}{\rm tr}_0((\mu \BD(\bU) - P \BI)\bn_0 -
(\mu \BD(\bU_j) - P_j\BI)\bn_0)\|_{L_1(\BR, \CB^{s+1-1/q}_{q,1}(\HS))} \\
&\hskip7.55cm + \|e^{-\gamma_b t}{\rm tr}_0(\bH- \bH_j)\|_{L_1(\BR, \CB^{s+1-1/q}_{q,1}(\HS))}\\
& \leq C\Big(\|e^{-\gamma_bt}(\mu \BD (\bU-\bU_j) - (P-P_j) \BI)\|_{L_1(\BR, \CB^{s+1}_{q,1}(\HS))}
 + \|e^{-\gamma_b t}(\bH- \bH_j)\|_{L_1(\BR, \CB^{s+1}_{q,1}(\HS))}\Big)\\
& \leq C\Big(\|e^{-\gamma_bt}(\nabla^2_b(\bU-\bU_j), \nabla(P-P_j))\|_{L_1(\BR, \CB^{s}_{q,1}(\HS))}
 + \|e^{-\gamma_b t}(\bH- \bH_j)\|_{L_1(\BR, \CB^{s+1}_{q,1}(\HS))}\Big)
 \to 0
\end{align*}
as $j \to \infty$, which shows that 
$$\|e^{-\gamma_b t}{\rm tr}_0((\mu \BD(\bU) - P \BI)\bn_0 - \bH)\|_{L_1(\BR, \CB^{s+1-1/q}_{q,1}(\HS))}=0.$$ 
Namely, for almost all $t \in \BR$, we see that $(\bU, P)$ satisfies the boundary condition:
$$(\mu \BD(\bU) - P \BI)\bn_0 = \bH|_{\pd\HS} \quad\text{on $\pd\HS\times \BR$}.$$
The proof of Theorem \ref{thm:5.1} is complete.
\end{proof}

\section{Generation of a $C_0$-analytic semigroup which is $L_1$ in time}
\label{sec-6}
In this section, we consider  initial boundary {\color{black}value} problem
with homogeneous boundary conditions:
\begin{equation}\label{eq:6.1}
\left\{\begin{aligned}  
\pd_t \bV - \DV (\mu \BD(\bV) - P \BI) & = 0 &\quad&\text{in $\HS\times \BR_+$}, \\
\dv \bV & = 0 &\quad&\text{in $\HS\times \BR_+$}, \\
(\mu \BD(\bV) - P \BI)\bn_0 & = 0 &\quad&\text{in $\pd\HS\times \BR_+$}, \\
\quad \bV|_{t=0} & = \bV_0 &\quad&\text{in $\HS$}.
\end{aligned}\right.
\end{equation}
First, we shall show the generation of a $C_0$-analytic semigroup 
 associated with \eqref{eq:6.1}.  
{\color{black}To this end}, we consider the corresponding resolvent problem: 
\begin{equation}\label{req:6.1}
\left\{\begin{aligned}  
\lambda \bu - \DV(\mu\BD(\bu) - \fq \BI) & = \bff &\quad&\text{in $\HS$}, \\
\dv \bu & = 0 &\quad&\text{in $\HS$}, \\
(\mu\BD(\bu) - \fq \BI)\bn_0 & = 0 &\quad&\text{on $\pd\HS$}.
\end{aligned}\right.
\end{equation}
We shall formulate equations \eqref{req:6.1} in the semigroup setting. 
{\color{black}Following the idea due} to Grubb and Solonnikov \cite{GS91}, 
which was exactly formulated by {\color{black}the first author}~\cite{Shi14} 
in the present form below, we shall eliminate {\color{black} the pressure $\fq$
as well as the divergence-free condition} from System \eqref{req:6.1}. 
To this end, for any $\bu \in \CJ^s_{q,r}(\HS) \cap\CB^{s+2}_{q, r}(\HS)^{\color{black} d}$, 
we {\color{black}define a bounded linear} functional $K(\bu) \in \sB^{s+1}_{q, r}(\HS) +
\wh \CB^{s+1}_{q, r, 0}(\HS)$ {\color{black}as} a unique
solution to the weak Dirichlet problem:
\begin{equation}\label{wd:6.2}
\left\{\begin{aligned}
(\nabla K(\bu), \nabla\varphi) & = (\dv(\mu\BD(\bu)), \nabla \varphi)
& \quad & \text{for all $\varphi \in \wh \CB^{1 - s}_{q', r', 0}(\HS)$}, \\
K(\bu) & = 2\mu \pd_d u_d & \quad & \text{on $\pd \HS$}.
\end{aligned}\right.
\end{equation}
The unique existence of 
$K(\bu) $ is guaranteed by Corollary \ref{cor:2.1} and we have
\begin{equation}\label{est:K(u)}
\|\nabla K(\bu)\|_{\CB^s_{q, r}(\HS)} \le C\|\bu\|_{\CB^{s+2}_{q, r}(\HS)}.
\end{equation}
Let $\bff \in \CJ^s_{q,r}(\HS)$. If $\bu \in\CJ^{s}_{q,r}(\HS) \cap
\sB^{s+2}_{q, r}(\HS)^{\color{black} d}$ 
and $\fq \in \sB^{s+1}_{q,r}(\HS) + \wh \CB^{s+1}_{q, r, 0}(\HS)$ 
satisfy equations \eqref{req:6.1}, then for any 
$\varphi \in \wh \CB^{1 - s}_{q', r', 0}(\HS)$, we have
\begin{equation}
0 = (\bff, \nabla \varphi)
=(\lambda \bu, \nabla\varphi) -(\nabla(K(\bu) - \fq), \nabla\varphi).
\end{equation}
Since $\bu\in \CJ^s_{q,r}(\HS)$ implies 
$(\bu, \nabla\varphi)=0$ for any $\varphi \in \wh \CB^{1 - s}_{q', r, 0}(\HS)$
as follows from Theorem \ref{thm-solenoidal-characterization}, 
we have
\begin{equation}
(\nabla(K(\bu) - \fq), \nabla\varphi) = 0 \quad\text{for every $\varphi
\in \wh \CB^{1-s}_{q', r', 0}(\HS)$}.
\end{equation}
Moreover, from the boundary condition \eqref{req:6.1}$_3$, 
we see that $K(\bu) - \fq = 2\mu \pd_d u_d - 2\mu \pd_d u_d =0$ on $\pd\HS$.
Thus, the uniqueness of the weak Dirichlet problem 
(Theorem \ref{thm-weak-Dirichlet}) implies $\fq = K(\bu)$.
\par
Define a closed linear operator in $\CJ^s_{q, 1} (\HS)$ by means of
\begin{equation}
\label{def-A}
\sA^s_{q, r} \bu = - \dv (\mu \BD (\bu) - K (\bu) \BI) 
\end{equation}
with domain $\mathsf D (\sA^s_{q, r}) \subset \CJ^s_{q,r}(\HS) \cap 
\sB^{s + 2}_{q, r} (\HS)^d$ defined by
\begin{equation}
\mathsf D (\sA^s_{q, r}) = \{\bu \in \CJ^s_{q, r} (\HS) \cap \sB^{s + 2}_{q, r} (\HS)^d
\mid (\pd_d u_j + \pd_j u_d) \vert_{\pd \HS} = 0
\enskip \text{on $\pd\HS$ \enskip for $j = 1, \ldots, d - 1$}  \}.
\end{equation} 
Then, Problem \eqref{req:6.1} reads
\begin{equation}\label{semigourp:6.1}
\lambda \bu + \sA^s_{q, r} \bu =  \bff
\end{equation}
for $\bff \in \CJ^s_{q, r} (\HS)$. 
\par
The first aim of this section is to show the following theorem.
\begin{thm}
\label{th-semigroup-resolvent-est}
Let $1 < q < \infty$, $1 \leq r \leq \infty-$, $- 1 + 1 \slash q < s <  1 \slash q$, 
$0 < \epsilon < \pi \slash 2$, and $\gamma>0$. 
For every $\bff \in \CJ^s_{q, 1} (\HS)$ and 
$\lambda \in \Sigma_{\epsilon} + \gamma_b$, there exists a constant
$C > 0$ depending solely on $\gamma$, $\epsilon$, $q$, $r$, and $s$ such that
\begin{equation}
\label{rest:6.1}
\lVert (\lambda, \lambda^{1\slash2}\nabla, \nabla^2) (\lambda + \sA^s_{q, r})^{- 1} 
\bff \rVert_{\CB^s_{q, r} (\HS)} 
\le C \lVert \bff \rVert_{\CB^s_{q, r} (\HS)}
\end{equation}
\end{thm}
\begin{proof} 
Since $\fq$ may be replaced by $K(\bu)$ in \eqref{req:6.1}, by Theorem \ref{th-sth},
we have $(\lambda + \sA^s_{q, r})^{- 1}\bff = \CS(\lambda)(\bff, 0, 0)$ for any $\bff \in \CJ^s_{q,r}(\HS)$ and 
$\lambda \in \Sigma_\epsilon + \gamma_b$. The uniqueness of solutions follows
from Theorem \ref{thm:unique}. Thus, the desired assertion follows from Theorem \ref{th-sth}.
\end{proof}

In view of Theorem \ref{th-semigroup-resolvent-est}, the operator $- \sA^s_{q,r}$ generates a $C_0$-analytic
semigroup $\{e^{-\sA^s_{q,r}t}\}_{t \geq 0}$ associated with System \eqref{eq:6.1} with $\bV_0 \in 
\CJ^s_{q,r}$ and $P = K(\bV)$. From theory of a $C_0$-analytic semigroup (cf \cite[Ch. IX]{Yosida}), 
a $C_0$-analytic semigroup $\{e^{-\sA^s_{q,r}t}\}_{t \geq 0}$ is represented by
$$e^{-\sA^s_{q,r}t}\bff = \frac{1}{2\pi i}\int_{\Gamma_\omega+\gamma_b} e^{\lambda t}
(\lambda+\sA^s_{q,r})^{-1}\bff\, \d\lambda
= \frac{1}{2\pi i}\int_{\Gamma_\omega+\gamma_b} e^{\lambda t}
\CS(\lambda)(\bff, 0, 0)\d\lambda
$$
for any $\bff \in \CJ^s_{q,r}(\HS)$, where $\Gamma_{\omega} + \gamma_b$ is the same contour
as in \eqref{contour.1}.  In the sequel, we assume that $r=1$.  Then, from \eqref{t.3.6.2}, 
we see that $T(t)(\bff, 0, 0) = e^{-\sA^s_{q,1}t}\bff$ for any $\bff \in \CJ^s_{q,1}$
and $\bU = e^{-\sA^s_{q,1}t}\bff$ and $P = K(e^{-\sA^s_{q,1}t}\bff)$ satisfy equations \eqref{eq:6.1}.
By Theorem \ref{thm:5.1}, we have the following theorem.
\begin{thm}\label{initial-problem} Let $1 < q < \infty$, $-1+1/q < s < 1/q$, and $\gamma>0$.  Then, 
for every $\bV_0 \in \CJ^s_{q,1}(\HS)$, Problem \eqref{eq:6.1} admits a  solution
$\bV$ with 
$$e^{-\gamma_bt}\bV \in L_1(\BR_+, \CB^{s+2}_{q,r}(\HS)^d) \cap \CW^1_1(\BR_+, \CB^s_{q,1}(\HS)^d),
\quad P = K(\bV)
$$
satisfying the estimate:
\begin{align*}
\|e^{-\gamma_b t}\bV\|_{L_1(\BR_+, \CB^{s+2}_{q,1}(\HS))} 
+ \|e^{-\gamma_b t}(\pd_t\bV, \nabla P)\|_{L_1(\BR_+, \CB^s_{q,1}(\HS))}
 \leq C\|\bV_0\|_{\CB^s_{q,1}(\HS)}.
\end{align*}
\end{thm}
\begin{proof} The uniqueness part follows from the generation of a $C_0$-analytic semigroup.
In fact, by the Duhamel principle, $\bV$ satisfies equations \eqref{req:6.1}, and then
noting \eqref{def-A} we see that 
\begin{align*}
\frac{\pd}{\pd \tau} e^{-\sA^s_{q,1}(t-\tau)}\bV(\tau) 
&= e^{-\sA^s_{q,1}(t-\tau)}\sA^s_{q,1}\bV(\tau) + e^{-\sA^s_{q,1}(t-\tau)}\pd_\tau\bV(\tau) \\
&= e^{-\sA^s_{q,1}(t-\tau)}(\pd_\tau\bV(\tau) -\DV(\mu\BD(\bV(\tau))- P(\tau)\BI)) \\
&= 0.
\end{align*}
Integrating this identity from $0$ to $t$ with respect to $\tau$, we have
$\bV(t) = e^{-\sA^s_{q,1}t}\bV_0$, which implies the uniqueness of solutions.
\end{proof}

Finally, we shall give the proof of Theorems \ref{th-MR-inhomogeneous} and \ref{th-MR-homogeneous}.
\begin{proof}[Proof of Theorems \ref{th-MR-inhomogeneous} and \ref{th-MR-homogeneous}: Existence of solutions]
First, we consider the evolution equations:
\begin{equation}\label{maineq:6.1}\left\{ \begin{aligned}
\pd_t \wt\bV - \DV (\mu \BD(\wt\bV) - \wt\Pi\BI)  = \bF& 
 & \quad & \text{in $\HS \times \BR$}, \\
\dv \wt\bV & =G_\mathrm{div} = \dv\bG &\quad & \text{in $\HS \times \BR$}, \\
(\mu \BD(\wt\bV) - \wt\Pi\BI) \bn_0 & = \bH \vert_{\pd \HS} & \quad & \text{on $\pd \HS \times \BR$}.
\end{aligned}\right. \end{equation}
  Let $\bV_3$ be a function defined by
\eqref{div:1} and we know the estimate \eqref{div:4}. Setting $\wt\bV = \bV_3 + \bU$, we see that 
$\bU$ and $\wt \Pi$ should satisfy equations:
\begin{equation}\label{maineq:6.2}\left\{ \begin{aligned}
\pd_t \bU - \DV (\mu \BD(\bU) - \wt \Pi \BI) & = \bF-(\pd_t\bV_3-\mu\dv\BD(\bV_3))
 & \quad & \text{in $\HS \times \BR$}, \\
\dv \bU & =0 & \quad & \text{in $\HS \times \BR$}, \\
(\mu \BD(\bU) - \wt\Pi \BI) \bn_0 & = (\bH - \mu\BD(\bV_3)\bn_0)\vert_{\pd \HS} & \quad & \text{on $\pd \HS \times \BR$}.
\end{aligned}\right. \end{equation}
Here, we have set 
$\bV_3=(V_{3,1}, \ldots, V_{3,d})$ and $\bH = (H_1, \ldots, H_{d-1}, H_d)$. 
In view of \eqref{weakD:6.1}, we know that 
\begin{align*}&\bH - \mu\BD(\bV_3)\bn_0|_{\pd\HS} \\
&= \Big(H_1-\mu(\pd_d V_{3,1}-\pd_1V_{3,d}), \ldots, H_{d-1}-
\mu(\pd_d V_{3,d-1}-\pd_{d-1} V_{3,d}), H_d- 2\mu\pd_d V_{3,d}\Big)\Big|_{\pd\HS}\\
& = (H_1, \ldots, H_{d-1}, H_d-2\mu\pd_d V_{3,d})|_{\pd\HS}.
\end{align*}
 Let $Q_\mathrm{div}$ be a solution of weak Dirichlet
problem:
\begin{equation}\label{wd:6.1*}
(\nabla Q_\mathrm{div}, \nabla\varphi) = (\bF-(\pd_t\bV_3-\mu\dv\BD(\bV_3)), \nabla\varphi)
\quad\text{for all $\varphi \in \wh{\CB}{}^{1-s}_{q', \infty-, 0}(\HS)$},
\end{equation}
subject to $Q_\mathrm{div}|_{\pd\HS}=2\mu\pd_d V_{3d}-H_d|_{\pd\HS}$.  
By Corollary \ref{cor:2.1}, we know the unique existence
of $Q_\mathrm{div} \in \CB^{s+1}_{q,1}(\HS) + \wh{\CB}{}^{s+1}_{q,1,0}(\HS)$ possessing the estimate:
\begin{equation}\label{weakD:6.1}
\|\nabla Q_\mathrm{div}\|_{\CB^s_{q,1}(\HS)} \leq C \Big(\|\pd_t\bV_3-\mu\dv\BD(\bV_3)\|_{\CB^s_{q,1}(\HS)}
+ \|(\nabla H_d, \nabla^2V_{3,d})\|_{\CB^s_{q,1}(\HS)} \Big).
\end{equation}
Setting $\wt\Pi = Q_\mathrm{div} + P$, we see that $\bU$ and $P$ satisfy the equations: 
\begin{equation}\label{maineq:6.3}\left\{ \begin{aligned}
\pd_t \bU - \DV (\mu \BD(\bU) - P \BI) & = \bF' 
 & \quad & \text{in $\HS \times \BR$}, \\
\dv \bU & = 0& \quad & \text{in $\HS \times \BR$}, \\
(\mu \BD(\bU) - P \BI) \bn_0 & = \bH'\vert_{\pd \HS} 
& \quad & \text{on $\pd \HS \times \BR$}.
\end{aligned}\right. \end{equation}
Here, noting \eqref{g-vanish1}, we have set 
\begin{equation}\label{right:6.1}\begin{aligned}
\bF'& = \bF -  (\pd_t\bV_3-\mu\dv\BD(\bV_3))- \nabla Q_\mathrm{div}, \quad 
\bH'  = (H_{1}, \ldots, H_{d-1},0).
\end{aligned}\end{equation}
By \eqref{weakD:6.1}, we see that $\dv\bF' = 0$, and thus, $\bF'$ and $\bH'$ satisfy assumption \eqref{assump:5.-1}.
Therefore, Theorem \ref{thm:5.1} gives the existence of solutions $\bU$ and $P$ to 
equations \eqref{maineq:6.3}, which satisfy the estimate:
\begin{equation}\label{est:6.1}\begin{aligned}
&\|e^{-\gamma_bt}\pd_t\bU\|_{L_1(\BR, \CB^s_{q,1}(\HS))} 
+ \|e^{-\gamma_bt}\bU\|_{L_1(\BR, \CB^{s+2}_{q,1}(\HS))}
+ \|e^{-\gamma_bt}\nabla P\|_{L_1(\BR, \CB^s_{q,1}(\HS))}\\
&\quad \leq C \Big(\|e^{-\gamma_bt}\bF'\|_{L_1(\BR, \CB^s_{q,1}(\HS))}
+ \|e^{-\gamma_bt}\bH'\|_{\CW^{1/2}_1(\BR, \CB^s_{q,1}(\HS))} 
+ \|e^{-\gamma_bt}\nabla \bH'\|_{L_1(\BR, \CB^s_{q,1}(\HS))} \Big).
\end{aligned}\end{equation}
Thus, combining \eqref{div:4}, \eqref{weakD:6.1}, \eqref{right:6.1},
and \eqref{est:6.1}, we have the existence of solutions $\wt\bV$ and $\wt\Pi$ of equations \eqref{maineq:6.1},
which satisfies the estimate:
\begin{multline}\label{est:6.2}
\|e^{-\gamma_bt}\pd_t\wt\bV\|_{L_1(\BR, \CB^s_{q,1}(\HS))} 
+ \|e^{-\gamma_bt}\wt\bV\|_{L_1(\BR, \CB^{s+2}_{q,1}(\HS))} 
+ \|e^{-\gamma_bt}\nabla \wt\Pi\|_{L_1(\BR, \CB^{s}_{q,1}(\HS))}\\
\leq C \Big(\|e^{-\gamma_bt}(\bF, \pd_t\bG)\|_{L_1(\BR, \CB^{s+2}_{q,1}(\HS))} 
+ \|e^{-\gamma_b t} \bH'\|_{\CW^{1/2}_1(\BR, \CB^s_{q,1}(\HS))} \\
+ \|e^{-\gamma_bt}(\nabla G_\mathrm{div}, \nabla\bH)\|_{L_1(\BR, \CB^s_{q,1}(\HS))} \Big).
\end{multline}
\par
Finally, we consider equations \eqref{eq-Stokes}. Let $\bV = \wt\bV + \wt\bU$ and 
$\Pi = \wt\Pi + \wt P$.  Then, $\wt\bU$ and $\wt P$ should satisfy the initial value problem: 
\begin{equation}\label{evol:6.1}
\left\{\begin{aligned}  
\pd_t \wt\bU - \DV (\mu \BD(\wt\bU) - \wt P \BI) & = 0 &\quad&\text{in $\HS\times \BR_+$}, \\
\dv \wt \bU & = 0 &\quad&\text{in $\HS\times \BR_+$}, \\
(\mu \BD(\wt\bU) - \wt P \BI) \bn_0 & = 0 &\quad&\text{on $\pd\HS\times \BR_+$}, \\
\quad \wt\bU|_{t=0} & = \bV_0- \wt\bV|_{t=0} &\quad&\text{in $\HS$}.
\end{aligned}\right.
\end{equation}
We assume the compatibility condition: $\dv\bV_0 = G_\mathrm{div}|_{t=0}$ in $\HS$.  
Since $\dv\wt\bV 
=G_\mathrm{div}$, we see that $\dv (\bV_0- \wt\bV|_{t=0} ) =0$ in $\HS$.  Namely, $\bV_0-\wt\bV|_{t=0}
\in \CJ^s_{q,1}(\HS)$.  Thus, by Theorem \ref{initial-problem}, Problem \eqref{evol:6.1}
admits a solution $\wt\bU$ such that $\nabla \wt P = \nabla K(\wt\bU)$ and 
there holds
\begin{equation}\label{est:6.3}\begin{aligned}
&\|e^{-\gamma_b t}\wt\bU\|_{L_1(\BR_+, \CB^{s+2}_{q,1}(\HS))} 
+ \|e^{-\gamma_b t}\pd_t\wt\bU\|_{L_1(\BR_+, \CB^s_{q,1}(\HS))}
+\|e^{-\gamma_bt}\nabla \wt P\|_{L_1(\BR_+, \CB^s_{q,r}(\HS)} \\
&\quad \leq C \|\bV_0-\wt\bV|_{t=0}\|_{\CB^s_{q,1}(\HS)}.
\end{aligned}\end{equation}
Since $\wt\bV|_{t=0} = -\int^\infty_0\pd_t(e^{-\gamma_bt}\wt\bV(\,\cdot\, , t))\d t$, recalling
that $\gamma_b=\gamma>0$ if $\CB^s_{q,1}=B^s_{q,1}$
and $\gamma_b=0$ if $\CB^s_{q,1}=\dot B^s_{q,1}$, we have
\begin{equation}\label{est:6.4}\|\wt\bV|_{t=0}\|_{\CB^s_{q,1}(\HS)}
\leq C\Big(\|e^{-\gamma_b t}\pd_t\wt\bV\|_{L_1(\BR_+, \CB^s_{q,1}(\HS))}
+ \|e^{-\gamma_b t}\wt\bV\|_{L_1(\BR_+, \CB^{s+2}_{q,1}(\HS))}\Big).
\end{equation}
Combining \eqref{est:6.2}, \eqref{est:6.3}, and \eqref{est:6.4} yields
\begin{multline}
\|e^{-\gamma_b t}\wt\bU\|_{L_1(\BR_+, \CB^{s+2}_{q,1}(\HS))} 
+ \|e^{-\gamma_b t}\pd_t\wt\bU\|_{L_1(\BR_+, \CB^s_{q,1}(\HS))}
+\|e^{-\gamma_bt}\nabla P\|_{L_1(\BR_+, \CB^s_{q,r}(\HS)} \\
\leq C\Big(\|\bV_0\|_{\CB^s_{q,1}(\HS)} 
+ \|e^{-\gamma_bt}(\bF, \pd_t\bG)\|_{L_1(\BR, \CB^{s+2}_{q,1}(\HS))} \\
+ \|e^{-\gamma_b t}\bH'\|_{\CW^{1/2}_1(\BR, \CB^s_{q,1}(\HS))} 
+ \|e^{-\gamma_bt}(\nabla G_\mathrm{div}, \nabla\bH)\|_{L_1(\BR, \CB^s_{q,1}(\HS))} \Big).
\end{multline}
This completes the existence part of the proof of Theorems \ref{th-MR-inhomogeneous} and 
\ref{th-MR-homogeneous}. \end{proof}

\noindent\textit{Proof of Theorems \ref{th-MR-inhomogeneous} and 
\ref{th-MR-homogeneous}: Uniqueness of solutions.}
Finally, we prove the uniqueness of solutions to equations \eqref{eq-Stokes}. 
We consider the homogeneous equations:
\begin{equation}\label{homo.1}\left\{\begin{aligned}
\pd_t \bV - \DV(\mu \BD(\bV) - \Pi\BI) & = 0 &\quad&\text{in $\HS\times\BR_+$}, \\
\dv \bV & = 0&\quad &\text{in $\HS\times\BR_+$}, \\
(\mu \BD(\bV) - \Pi\BI) \bn_0 &=0&\quad &\text{on $\pd\HS\times\BR_+$}, \\
\bV \vert_{t = 0} & =0&\quad &\text{in $\HS$}.
\end{aligned}\right.\end{equation}
It remains to show the following theorem.
\begin{thm}
Let $d\geq 2$, $1 < q < \infty$, $-1+1/q < s < 1/q$, and $\gamma>0$. Let 
$\bV$ and $\Pi$ satisfy the regularity conditions:
$$
e^{-\gamma_b t}\pd_t \bV\in L_1(\BR_+, \CB^s_{q,1}(\HS), \quad
e^{-\gamma_b t}\nabla_b^2 \bV \in L_1(\BR_+, \CB^s_{q,1}(\HS)), \quad
e^{-\gamma_b t}\nabla \Pi \in L_1(\BR_+, \CB^s_{q,1}(\HS)),$$
and  homogeneous equations \eqref{homo.1}, then $\bV=0$ and $\Pi=0$.
\end{thm}
\begin{proof} First, we shall prove that $\Pi(\,\cdot\,, t) = K(\bV(\,\cdot\,, t))$ for almost all $t \in \BR_+$.  
Let $\varphi \in \wh {\CB}{}^{1-s}_{q', \infty-, 0}(\HS)$.  Then, for any $T > 0$, by the first equation
in \eqref{homo.1} we have 
$$0 = \int^T_0\Big((\pd_t \bV(\,\cdot\, , t), \nabla\varphi) - (\mu\DV\BD(\bV(\,\cdot\, , t)), \nabla\varphi)
+ (\nabla \Pi(\,\cdot\, , t), \nabla\varphi)\Big)\,\d t.$$
Since $\dv \bV(\,\cdot\,, t) = 0$ in $\HS$ for almost all $t > 0$ and $\bV(\,\cdot\,, 0)=0$,  we have 
\begin{align*}
\int^T_0(\pd_t\bV(\,\cdot\, , t), \nabla\varphi)\,\d t = (\bV(\,\cdot\, , T), \nabla\varphi) -(\bV(\,\cdot\, , 0),
\nabla\varphi) = 0
\end{align*}  
for almost all $T>0$.  Thus, from the definition of $K(\bV(\,\cdot\, , t))$ it follows that 
$$0 = \int^T_0\Big(\nabla \big(\Pi(\,\cdot\, , t) - K(\bV(\,\cdot\, , t)\big), \nabla\varphi\Big)\,\d t$$
for almost all $T > 0$. However, this holds for all $T>0$ due to the continuity of integration 
with respect to $T>0$. Thus,
for all $T > 0$, we have
$$\int^T_0 \Big(\nabla\big(\Pi(\,\cdot\, , t)-K(\bV(\,\cdot\, , t)\big), \nabla\varphi\Big)\,\d t=0.$$
Differentiating this formula with respect to $T$ yields
$$\Big(\nabla\big(\Pi(\,\cdot\, , t)- K(\bV(\,\cdot\, , t)) \big), \nabla\varphi \Big)=0 
\quad\text{
for almost all $t \in \BR_+$}. $$ 
From the boundary condition for $\Pi$ it follows that 
\begin{equation}
\Pi(\,\cdot\, , t) - K(\bV(\,\cdot\, , t))
= 2\mu\pd_d V_{d}(\,\cdot\, , t) - 2\mu \pd_d V_{d}(\,\cdot\, , t) = 0 \qquad \text{on $\pd\HS$
 for almost all $t>0$}.
\end{equation}
Hence, the uniqueness of the weak Dirichlet problem (Theorem \ref{thm-weak-Dirichlet})
yields $\Pi(\,\cdot\, , t) = K(\bV(\,\cdot\, , t))$ for almost all $t \in \BR_+$.  
\par
Since $\Pi(\,\cdot\, , t) = K(\bV(\,\cdot\, , t))$ for almost all $t \in \BR_+$, 
it follows from \eqref{def-A} that $\bV$ satisfies equation:
$$\pd_t \bV - \DV (\mu \BD(\bV) - K(\bV) \BI)=
\pd_t \bV + \sA^s_{q,1}\bV = 0\quad\text{in $\HS\times\BR_+$}. $$
From the regularity conditions for $\bV$ and the tangential component of the boundary 
conditions it follows that $\bV(\,\cdot\, , t) \in \mathsf D (\sA^s_{q, 1})$ for almost all
$t \in \BR_+$. 
Thus, for almost all $\tau \in \BR_+$, we have
$$\frac{\pd}{\pd \tau} e^{-\sA^s_{q,1}(t-\tau)}\bV(\,\cdot\, , \tau)
= e^{-\sA^s_{q,1}(t-\tau)} \Big(\sA^s_{q,1}\bV(\,\cdot\, , \tau) + \pd_\tau\bV(\,\cdot\, , \tau)\Big) = 0.$$
Integrating this formula from $0$ to $t > 0$ with respect to $\tau$
and using $\bV(\,\cdot\, , 0) = 0$ yield
$\bV(\,\cdot\, , t)=0$ for any $t \in \BR_+$.  
Moreover, $\Pi = K(\bV) = 0$.  
The proof is complete.
\end{proof}

\section{Nonlinear well-posedness}\label{sec-8}

\subsection{Nonlinearity}\label{sec-8.1}
In this subsection, we shall discuss the estimate of product $uv$ by using the 
Besov norms. Here, by virtue of our 
definition of Besov spaces by restriction, 
estimates on $\BR^d$ carry over to the space on $\HS$.
We shall use the following proposition due to Abidi and 
Paicu \cite[Cor.~2.5]{AP07}
as well as  Haspot \cite[Cor.~1]{H11}. 
\begin{prop}\label{prop:APH}
Let $1 \le q \le \infty$, $1 \leq r \leq \infty$, and 
$\Omega \in \{\BR^d, \HS\}$.  If the condition 
$|s| < d/q$ holds for $q\geq 2$ and {\color{black} the condition $-d/q' < s < d/q$ holds
for $1 \leq q <2$},  there holds
\begin{equation}\label{cor:prod.1}
\|uv\|_{\CB^s_{q,{\color{black}r}}(\Omega)} \leq C\|u\|_{\CB^s_{q,r}(\Omega)}
\|v\|_{\CB^{d/q}_{q, \infty}(\Omega) \cap L_\infty(\Omega)}
\end{equation}
for some constant $C > 0$ independent of $u$ and $v$. 
\end{prop}
We introduce propositions to estimate our nonlinear terms 
expressed by \eqref{def-nonlinear-terms}. 
\begin{prop}\label{prop:p1} Let {\color{black}$d-1 < q < 2d$}. If $s \in \BR$ satisfies 
\begin{equation}\label{cond-qs}
-1+d/q \leq s < 1/q   
\end{equation}
then for every $u \in \CB^s_{q,1}(\HS)$ and $v \in {\color{black} \CB^{d \slash q}_{q, 1}(\HS)}$, 
there holds
\begin{equation}
\label{est-prop:8.2}
\|uv\|_{\CB^s_{q, 1}(\HS)} \le C\|u\|_{\CB^s_{q,1}(\HS)}
\|v\|_{\CB^{d/q}_{q,1}(\HS)}.
\end{equation}
\end{prop}
\begin{rem}
To enclose the iteration scheme to solve the nonlinear problem \eqref{eq-fixed},
in view of \eqref{est-prop:8.2}, we need the additional condition
$d/q+1 \leq s+2$,
which yields the condition \eqref{cond-qs}.
Namely, it suffices to use Proposition \ref{prop:p1} in our situation.
\end{rem}
\begin{proof} 
{\color{black} If $q \geq2$, the condition: $q < 2d$ implies that $-d/q < -1+d/q$.  Since we know that 
\begin{equation}\label{fund.est} \|v\|_{\CB^{d/q}_{q,\infty}(\Omega) \cap L_\infty(\Omega)} \leq C\|v\|_{\CB^{d/q}_{q,1}(\Omega)}
\end{equation}
for $\Omega \in \{\BR^d, \HS\}$, by \eqref{cor:prod.1}, we have \eqref{est-prop:8.2}. 
If $1 \leq q < 2$, the condition: $d \geq 2$ implies that $-d/q' < -1+d/q$, and therefore,
by \eqref{cor:prod.1}, we have \eqref{est-prop:8.2}.  }
The proof is complete.
\end{proof}
\begin{prop}
\label{prop:p2}
Let {\color{black}$d-1 < q \leq d$ and let $-1 + d/q < s < 1/q$.}  
{\color{black}Then,} for every $u \in \CB^{s - 1}_{q, 1} (\HS)$ and 
$v \in \CB^{d \slash q}_{q, 1} (\HS)$
there holds
\begin{equation}\label{prod:e-2}
\lVert u v \rVert_{\CB^{s - 1}_{q, 1} (\HS)} 
\le C \lVert u \rVert_{\CB^{s - 1}_{q, 1} (\HS)} 
\lVert v \rVert_{\CB^{d \slash q}_{q, 1} (\HS)}.
\end{equation}	
\end{prop} 
\begin{proof}
{\color{black} If $q \geq 2$, the conditions: $q \leq d$ and $-1 + d/q < s$ imply 
$s-1 > -d/q$.  If $1 \leq q < 2$, the conditions: $d \geq 2$ and $-1+d/q < s$
imply $s-1 > -d/q'$. Thus, replacing $s$ with $s-1$ in \eqref{cor:prod.1} and using \eqref{fund.est}, 
we have \eqref{prod:e-2}.   }
The proof is complete. 
\end{proof}
\begin{prop}\label{lem:pro}
Let $1 \le q \le \infty$. For every $u$, $v \in \CB^{d \slash q}_{q,1}(\HS)$, there holds
\begin{equation}
\|uv\|_{\CB^{d \slash q}_{q,1}(\HS)} \le C\|u\|_{\CB^{d \slash q}_{q,1}(\HS)}
\|v\|_{\CB^{d \slash q}_{q,1}(\HS)}.
\end{equation}
Namely, $\CB^{d \slash q}_{q, 1}(\HS)$ is a Banach algebra. 
\end{prop}
\begin{proof} 
According to \cite[Prop. 2.3]{H11}, there holds
\begin{equation}
\|uv\|_{\CB^{{\color{black}d \slash q}}_{q, 1}(\HS)} 
\le C(\|u\|_{\CB^{{\color{black}d \slash q}}_{q,1}(\HS)}\|v\|_{L_\infty(\HS)}
+ \|u\|_{L_\infty(\HS)}\|v\|_{\CB^{{\color{black}d \slash q}}_{q,1}(\HS)}
\end{equation}
provided that $1 \le q \le \infty$.
By $\CB^{d \slash q}_{q, 1} (\HS) \hookrightarrow L_\infty (\HS)$, we have the desired estimate.
\end{proof}
The following result on composite functions is due to \cite[Thm. 2.87]{BCD} and \cite[Prop. 2.4]{H11}.
\begin{prop}
\label{prop-8.7}
Let $I \subset \BR$ be open. Let $s > 0$ and $\sigma$ be the smallest integer 
such that $\sigma \ge s$. Let $\sfF \colon I \to \BR$ satisfy $\sfF(0)=0$ and $\sfF' \in W^\sigma_\infty (I)$.
Assume that $v\in \CB^s_{q, r} (\HS)$ has values in $J \Subset I$.
Then it holds $\sfF (v) \in \CB^s_{q, r} (\HS)$ and there exists a constant $C$ 
depending only on $s$, $I$, $J$, and $d$ such that
\begin{equation}
\lVert \sfF (v) \rVert_{\CB^s_{q, r} (\HS)} 
\le C \Big(1 + \lVert v \rVert_{L_\infty (\HS)} \Big)^\sigma
\lVert \sfF' \rVert_{W^\sigma_\infty (I)} \lVert v \rVert_{\CB^s_{q, r} (\HS)}.
\end{equation}
\end{prop}
We also prepare the estimates of $\BA_{\bu}$ and $\BA_{\bu}^\top$. 
To this end, we use the following proposition.
\begin{prop}
\label{prop-Lagrangian}
Let $0 < T \le \infty$ and $q$ and $s$ satisfy conditions \eqref{cond-qs}.
Moreover, let $\fA_{\bu}$ stands for either $\BA_{\bu}$ or 
$\BA_{\bu}^\top$ defined by \eqref{def-Au}. 
The following assertions hold true:
\begin{enumerate}
\item Suppose that $\bu \in L_1 ((0, T), \CB^{s + 2}_{q, 1} (\HS)^d)$ 
satisfies $\lVert \nabla_y \bu 
\rVert_{L_1 ((0, T), \CB^{d/q}_{q, 1} (\HS))} \le c_0$
with some small constant $c_0 > 0$.
There exists a constant $C_{d, c_0} > 0$ such that
\begin{align}
\lVert \fA_{\bu} \rVert_{L_\infty ((0, T), \CB^{d/q}_{q, 1} (\HS))} & \le C_{d, c_0}, \\
\lVert (\fA_{\bu} - \BI, \fA_{\bu} \fA_{\bu} - \BI, \fA_{\bu} \fA_{\bu}^\top - \BI) 
\rVert_{L_\infty ((0, T),  \CB^{d/q}_{q, 1} (\HS))} 
& \le C_{d, c_0} \lVert \nabla_y \bu \rVert_{L_1 ((0, T),  \CB^{d/q}_{q, 1} (\HS))}, \\
\lVert \nabla_y (\fA_{\bu}, \fA_{\bu} \fA_{\bu}, \fA_{\bu} \fA_{\bu}^\top) 
\rVert_{L_\infty ((0, T),  \CB^s_{q, 1} (\HS))}
& \le C_{d, c_0} \lVert \nabla^2_y \bu \rVert_{L_1 ((0, T),  \CB^s_{q, 1} (\HS))}
\end{align}
and 
\begin{align}
\lVert \pd_t \nabla_y \fA_{\bu} \rVert_{L_1 ((0, T),  \CB^s_{q, 1} (\HS))}
& \le C_{d, c_0} \Big( 
\lVert \nabla^2_y \bu \rVert_{L_1 ((0, T), \CB^s_{q, 1} (\HS))} \\
& \qquad + \lVert \nabla_y \bu \rVert_{L_1 ((0, T), \CB^{d/q}_{q, 1} (\HS))}
\lVert \nabla^2_y \bu \rVert_{L_1 ((0, T), \CB^s_{q, 1} (\HS))}
\Big).
\end{align} 
\item Suppose that $\bu_1, \bu_2 \in L_1 ((0, T), \CB^{s+2}_{q, 1} (\HS))$ 
satisfies $\lVert (\nabla_y \bu_1, \nabla_y \bu_2)
\rVert_{L_1 ((0, T), \CB^{d/q}_{q, 1} (\HS))} \le c_0$
with some small constant $c_0>0$. For simplicity, the notations 
on difference are given by $(\delta \bu, \delta \fA_{\bu}) 
:= (\bu_2 - \bu_1, \fA_{\bu_2} - \fA_{\bu_1})$. 
There exists a constant $C_{d, c_0} > 0$ such that
\begin{align}
& \lVert (\delta \fA_{\bu}, 
\fA_{\bu_2} \fA_{\bu_2} - \fA_{\bu_1} \fA_{\bu_1}, 
\fA_{\bu_2} \fA_{\bu_2}^\top - \fA_{\bu_1} \fA_{\bu_1}^\top) 
\rVert_{L_\infty ((0, T), \CB^{d/q}_{q, 1} (\HS))} 
\le C_{d, c_0} \lVert \nabla_y \delta \bu \rVert_{L_1 ((0, T), \CB^{d/q}_{q, 1} (\HS))}, \\
& \lVert \nabla_y (\delta \fA_{\bu}, 
\fA_{\bu_2} \fA_{\bu_2} - \fA_{\bu_1} \fA_{\bu_1}, 
\fA_{\bu_2} \fA_{\bu_2}^\top - \fA_{\bu_1} \fA_{\bu_1}^\top) 
\rVert_{L_\infty ((0, T), \CB^s_{q, 1} (\HS))} \\
& \qquad \le C_{d, c_0} \Big(
\lVert \nabla^2_y \delta \bu \rVert_{L_1 ((0, T), \CB^s_{q, 1} (\HS))}
+ \lVert \nabla_y^2 (\bu_1, \bu_2) \rVert_{L_1 ((0, T), \CB^s_{q, 1} (\HS))}
\lVert \nabla_y \delta \bu \rVert_{L_1 ((0, T), \CB^{d/q}_{q, 1} (\HS))}
\Big), 
\end{align}
and 
\begin{align}
& \lVert \pd_t \nabla_y \delta \BA_{\bu} \rVert_{L_1 ((0, T), \CB^s_{q, 1} (\HS))} \\
& \qquad \le C_{d, c_0} \Big(
\lVert \nabla_y^2 \delta \bu \rVert_{L_1 ((0, T), \CB^s_{q, 1} (\HS))}
+ \lVert \nabla_y^2 (\bu_1, \bu_2) \rVert_{L_1 ((0, T), \CB^s_{q, 1} (\HS))}
\lVert \nabla_y \delta \bu \rVert_{L_1 (\CB^{d/q}_{q, 1} (\HS))} \\
& \qquad + \lVert \nabla_y (\bu_1, \bu_2) \rVert_{L_1 ((0, T), \CB^{d/q}_{q, 1} (\HS))} 
\lVert \nabla_y^2 \delta \bu \rVert_{L_1 ((0, T), \CB^s_{q, 1} (\HS))} \\
& \qquad + \lVert \nabla_y^2 (\bu_1, \bu_2) \rVert_{L_1 ((0, T), \CB^s_{q, 1} (\HS))}
\lVert \nabla_y \delta \bu \rVert_{L_1 ((0, T), \CB^{d/q}_{q, 1} (\HS))} \\
& \qquad + \lVert \nabla_y (\bu_1, \bu_2) \rVert_{L_1 ((0, T), \CB^{d/q}_{q, 1} (\HS))} 
\lVert \nabla_y^2 (\bu_1, \bu_2) \rVert_{L_1 ((0, T), \CB^s_{q, 1} (\HS))}
\lVert \nabla_y \delta \bu \rVert_{L_1 ((0, T), \CB^{d/q}_{q, 1} (\HS))}
\Big).
\end{align} 
\end{enumerate}
\end{prop}
\begin{proof}
Notice that $\fA_{\bu}$ is given by a Neumann series expansion of
$\int^t_0 \nabla\bu(\,\cdot\, , \tau)\d\tau$ and 
that  
\begin{equation}\label{neumann:8.1}
\bigg\|\int^t_0 \nabla \bu(\,\cdot\, , \tau)\d\tau\bigg\|_{L_\infty(\HS)}
\le \int^T_0 \|\nabla \bu(\,\cdot\, , \tau)\|_{\CB^{d/q}_{q,1}(\HS)}\d \tau
\le c_0.
\end{equation}
Here, we have used the fact that $\CB^{d \slash q}_{q, 1} (\HS) \hookrightarrow L_\infty (\HS)$.
From Proposition \ref{lem:pro} 
we know that $\CB^{d/q}_{q, 1} (\HS)$ is a Banach algebra 
and from Corollary \ref{prop:p1} 
we have $\|uv\|_{\CB^s_{q,1}(\HS)} \le C\|u\|_{\CB^s_{q,1}(\HS)}
\|v\|_{\CB^{d/q}_{q,1}(\HS)}$
{\color{black} for every $u \in \CB^s_{q, 1} (\HS)$ and $v \in \CB^{d \slash q}_{q, 1} (\HS)$.}
Hence, the desired estimates may be proved along the same argument as in 
the proof of Lemmas A.1 and A.2 in \cite{SSZ20}.
The proof is complete.
\end{proof}

\subsection{Proof of Theorem \ref{th-local-well-posedness-fixed}}
{\color{black} In this subsection, we assume that $q$ and $s$ satisfy the condition
described in Theorem \ref{th-local-well-posedness-fixed}.}
Let $\bb \in B^s_{q,1}(\HS)^d$ be a fixed initial data and 
$\ba \in B^s_{q,1}(\HS)$ is any initial data such that 
$\|\bb-\ba\|_{B^s_{q,1}(\HS)} < \sigma$,  Let $(\bu_{\ba}, \sfQ_{\ba})$ and 
$(\bu_{\bb}, \sfQ_{\bb})$ be solutions to the following linear system:
\begin{equation}
\label{eq-linear-L}
\left\{\begin{aligned}
\pd_t \bu_{\bd} - \DV (\mu\BD(\bu_{\bd}) -  \sfQ_{\bd}\BI) & = 0 
& \quad & \text{in $\BR^d_+ \times \BR_+$}, \\
\dv \bu_{\bd} & = 0 & \quad & \text{in $\BR^d_+ \times \BR_+$}, \\
(\mu\BD(\bu_{\bd}) - \sfQ_{\bd}\BI) \bn_0 & = 0 & 
\quad & \text{on $\pd \BR^d_+ \times \BR_+$}, \\
\bu_{\bd} \vert_{t = 0} & = \bd & \quad & \text{in $\BR^d_+$}.	
\end{aligned}\right.		
\end{equation}	
for $\bd \in \{\ba, \bb \}$. 
By virtue of Theorem \ref{th-MR-inhomogeneous}, we obtain a unique global 
solution $(\bu_{\bd}, \sfQ_{\bd})$ to \eqref{eq-linear-L} satisfying
\begin{equation}\label{small:8.0}
e^{-\gamma t}(\bu_{\bd}, \sfQ_{\bd}) \in \Big( L_{1} (\BR_+, B^{s + 1}_{q, 1} (\HS)^d) 
\cap W^1_{1} (\BR_+, B^s_{q, 1} (\HS)^d) \Big)
\times L_{1} (\BR_+, B^{s + 1}_{q, 1} (\HS) 
+ \wh B^{s + 1}_{q, 1} (\HS))
\end{equation}
as well as
\begin{equation}\label{initial:8.1}
\lVert (\pd_t \bu_{\bd}, \nabla \sfQ_{\bd}) 
\rVert_{L_1 ((0 ,T), B^s_{q, 1} (\HS))}
+ \lVert \bu_{\bd} \rVert_{L_1 ((0, T), B^{s + 2}_{q, 1} (\HS))}
\le C e^{\gamma T} \lVert \bd \rVert_{B^s_{q, 1} (\HS)}
\end{equation}
for any finite $T > 0$ and $\bd \in \{\ba, \bb \}$ (cf. \eqref{est-ET} {\color{black} below}).
Here, $\gamma > 0$ is the same constant as in Theorem \ref{th-MR-inhomogeneous}. \par
Let $T \in (0, 1)$ and $\omega \in (0, 1\slash2)$ be positive small numbers determined later
such that 
\begin{equation}\label{small:8.1}
\lVert (\pd_t \bu_{\bb}, \nabla \sfQ_{\bb}) 
\rVert_{L_1 ((0 ,T), B^s_{q, 1} (\HS))}
+ \lVert \bu_{\bb} \rVert_{L_1 ((0, T), B^{s + 2}_{q, 1} (\HS))}
\le \omega/4.
\end{equation}
Since $\bu_{\bb}$ and $\sfQ_{\bb}$ satisfy \eqref{small:8.0}, we may 
choose $T>0$ according to $\omega$ such that \eqref{small:8.1}
holds. 
Moreover, 
\begin{align}
&\lVert (\pd_t (\bu_{\ba}-\bu_{\bb}), \nabla (\sfQ_{\ba}-\sfQ_{\bb})) 
\rVert_{L_1 ((0 ,T), B^s_{q, 1} (\HS))}
+ \lVert \bu_{\ba}-\bu_{\bb} \rVert_{L_1 ((0, T), B^{s + 2}_{q, 1} (\HS))}
\\
&\quad \le C e^{\gamma T} \lVert \ba-\bb \rVert_{B^s_{q, 1} (\HS)}
< Ce^{\gamma T} \sigma.
\end{align}
Thus, choosing $\sigma > 0$ so small that $Ce^{\gamma T}\sigma < \omega/4$ 
is satisfied, we infer from \eqref{small:8.1} that
\begin{equation}\label{small:8.2}
\lVert (\pd_t \bu_{\ba}, \nabla \sfQ_{\ba}) 
\rVert_{L_1 ((0 ,T), B^s_{q, 1} (\HS))}
+ \lVert \bu_{\ba} \rVert_{L_1 ((0, T), B^{s + 2}_{q, 1} (\HS))}
\le \omega/2.
\end{equation}
Notice that we may choose $\sigma$ independent of $T$ since $T < 1$.
Define an underlying space $S_{T, \omega}$ by setting 
\begin{equation}
\label{def-STomega}
S_{T, \omega}
= \left\{ (\bu, \sfQ) \enskip \left| \enskip
\begin{aligned}
\bu & \in W^1_1 ((0, T), B^s_{q, 1} (\HS)^d) \cap L_1 ((0, T), B^{s + 2}_{q, 1} (\HS)^d), \\
 \sfQ & \in L_1((0, T), B^s_{q,1}(\HS) + \wh B^{s+1}_{q,1,0}(\HS)), \\
& {\color{black} \lVert (\bu, \sfQ) \rVert_{S_T}} \le \omega, \quad \bu|_{t=0} = \ba
\end{aligned}\right.\right\},	
\end{equation}
where we have set
\begin{align}
{\color{black} \lVert (\bu, \sfQ) \rVert_{S_T}} & := 
[\bu]_{q, s, T} + \lVert \nabla\sfQ \rVert_{L_1 ((0, T), B^s_{q, 1} (\HS))}, \\
{\color{black} [\bu]_{q, s, T}} & := \lVert \pd_t \bu \rVert_{L_1 ((0, T), B^s_{q, 1} (\HS))}
+ \lVert \bu \rVert_{L_1 ((0, T), B^{s + 2}_{q, 1} (\HS))}.
\end{align}
By \eqref{small:8.2}, it holds $(\bu_{\ba}, \sfQ_{\ba}) \in S_{T, \omega}$. 
Given $(\bw, \sfP) \in S_{T, \omega}$, 
we intend to find a fixed point in $S_{T, \omega}$ for the mapping 
$\Phi \colon (\bw, \sfP) \mapsto (\bu, \sfQ)$ with $(\bu,\sfQ)$ 
which is 
the solution to
\begin{align}
\label{iteration:eq1}
\left\{\begin{aligned}
\pd_t \bu - \DV (\mu \BD(\bu) - \sfQ \BI) & = \bF (\bw) & \quad & \text{in $\HS \times (0, T)$}, \\
\dv \bu & = G_\mathrm{div} (\bw) = \dv \bG (\bw) & \quad & \text{in $\HS \times (0, T)$}, \\
(\mu \BD(\bu) - \sfQ \BI) \bn_0 & = (\BH (\bw) \vert_{\pd \HS})
\bn_0 & \quad & \text{on $\pd \HS \times (0, T)$}, \\
\bu \vert_{t = 0} & = \ba & \quad & \text{in $\HS$},
\end{aligned}\right.
\end{align}	
Moreover, to simplify the notation,
we shall set $\bH (\bw) := \BH (\bw) \wt \bn_0$, 
where $\wt \bn_0$ means the extension of 
$\bn_0$ to $\HS$. Then we see that $(\bU, Q) :
= (\bu - \bu_{\ba}, \sfQ - \sfQ_{\ba})$ solves
\begin{align}
\label{eq-fixed-point}
\left\{\begin{aligned}
\pd_t \bU - \DV(\mu\BD(\bU) - Q\BI) & = \bF (\bw) & \quad & \text{in $\HS \times (0, T)$}, \\
\dv \bU & = G_\mathrm{div} (\bw) = \dv \bG (\bw) & \quad & \text{in $\HS \times (0, T)$}, \\
(\mu\BD(\bU) -  Q\BI) \bn_0 & 
= \bH (\bw) {\color{black} \vert_{\pd \HS}} & \quad & \text{on $\pd \HS \times (0, T)$}, \\
\bU \vert_{t = 0} & = 0 & \quad & \text{in $\HS$}.
\end{aligned}\right.
\end{align}
For $(\bw, \sfP) \in S_{T, \omega}$,
we infer from $B^{d \slash q}_{q, 1} (\HS) \hookrightarrow L_\infty (\HS)$ that
\begin{equation}
\bigg\|\int^t_0\nabla \bw(\,\cdot\, , \tau)\d\tau\bigg\|_{L_\infty(\HS)}
\le \int^T_0 \lVert\nabla\bw(\,\cdot\, , \tau)\|_{B^{d/q}_{q,1}(\HS)}\d\tau
\le C_1 \omega < \frac{1}{2} \qquad \text{for every $t>0$},
\end{equation}
where we have assumed $\omega < (2 C_1)^{- 1}$ if necessary,
and hence $\bF(\bw)$, $G_\mathrm{div}(\bw)$, $\bG(\bw)$ and $\bH(\bw)$ are well-defined.
To solve {\color{black}System \eqref{eq-fixed-point}}, we extend the 
right-hand members to $t \in \BR$ so that we may use Theorem \ref{th-MR-inhomogeneous}.
To this end, for a function $f$ defined on $(0, T)$ we define
\begin{equation}
\sfE_{(T)} f (\,\cdot\, , \tau) = \begin{cases}
{\color{black}0}, & {\color{black} t \le 0}, \\
f (\,\cdot\, , t), & 0 < t < T, \\
f (\,\cdot\, , 2 T - t), & T \le t < 2 T, \\
0, & t \ge 2 T.
\end{cases}
\end{equation} 
In addition, if $f (\,\cdot\, , 0) = 0$, there holds
\begin{equation}
\pd_t (\sfE_{(T)} f (\,\cdot\, , t)) = \begin{cases}
{\color{black}0}, & {\color{black} t \le 0}, \\
(\pd_t f) (\,\cdot\, , t), & 0 < t < T, \\
- (\pd_t f) (\,\cdot\, , 2 T - t), & T \le t < 2 T, \\
0, & t \ge 2 T.
\end{cases}
\end{equation} 	
It is easy to verify that
\begin{equation}
\label{est-ET}
\begin{split}
\lVert f (\,\cdot\, , t) \rVert_{L_1 ((0, T), X)} & \le
\lVert \sfE_{(T)} f (\,\cdot\, , t) \rVert_{L_1 (\BR, X)}
\le 2 \lVert f (\,\cdot\, , t) \rVert_{L_1 ((0, T), X)}, \\
\lVert (\pd_t f) (\,\cdot\, , t) \rVert_{L_1 ((0, T), X)} & \le
\lVert \pd_t (\sfE_{(T)} f (\,\cdot\, , t))
\rVert_{L_1 (\BR, X)}
\le 2 \lVert (\pd_t f) (\,\cdot\, , t) \rVert_{L_1 ((0, T), X)}
\end{split}
\end{equation} 	
for any Banach space $X$. 
Let $\bF_0 (\bw)$ be the zero extension of $\bF (\bw)$ with respect to
the time variable $t$. Consider the \textit{linear} system
\begin{align}
\label{eq-fixed-point-tilde}
\left\{\begin{aligned}
\pd_t \wt \bU - \DV(\mu(\wt \bU) - \wt Q\BI) & = \bF_0 (\bw)
& \quad & \text{in $\HS \times \BR$}, \\
\dv \wt \bU & = E_{(T)} G_\mathrm{div} (\bw) = \dv \Big(E_{(T)} \bG (\bw)\Big) 
& \quad & \text{in $\HS \times \BR$}, \\
(\mu(\wt \bU) - \wt Q\BI)\bn_0 & = E_{(T)} \bH (\bw) \vert_{\pd \HS}
& \quad & \text{on $\pd \HS \times \BR$}, \\
\wt \bU \vert_{t = 0} & = 0 & \quad & \text{in $\HS$}.	
\end{aligned}\right.	
\end{align}
Clearly, we find that $(\wt \bU, \wt  Q) \vert_{t \in (0, T)}$ is 
also a solution to \eqref{eq-fixed-point}.
Since the right-hand members of \eqref{eq-fixed-point-tilde} {\color{black} are}
defined on $\BR$, we may apply Theorem \ref{th-MR-inhomogeneous} to obtain
\begin{multline}\label{linear:8.1}
\lVert e^{-\gamma t}(\pd_t \wt \bU, \nabla \wt  Q) \rVert_{L_{1} (\BR_+, B^s_{q, 1} (\HS))} 
+ \lVert e^{-\gamma t}\wt \bU \rVert_{L_{1} (\BR_+, B^{s + 2}_{q, 1} (\HS))} 
+ \sup_{t \in [0, \infty)} e^{- \gamma t} \lVert \wt \bU (t) \rVert_{B^s_{q, 1} (\HS)} \\
\le C \bigg(
\Big\lVert e^{-\gamma t}\Big(\bF_0 (\bw) , \pd_t \Big( E_{(T)} \bG (\bw) \Big) \Big) 
\Big\rVert_{L_{1} (\BR, B^s_{q, 1} (\HS))} 
+ \lVert e^{-\gamma t}\nabla(E_{(T)} G_\mathrm{div}, E_{(T)} \bH) 
\rVert_{L_{1} (\BR, B^{s}_{q, 1} (\HS))} \\
+ \Big\lVert  e^{-\gamma t}E_{(T)} \bH 
\Big\rVert_{W^{1/2}_{1} (\BR, B^s_{q, 1} (\HS))} \bigg),
\end{multline}
where we have used the well-known embedding
$W^1_1 (\BR_+, X) \hookrightarrow \mathrm{BC} ([0, \infty), X)$
for a Banach space $X$.
By virtue of \eqref{est-ET} and Proposition \ref{prop-B7}
in Appendix B below, the above estimate may be read as
\begin{multline}\label{nonest:-1}
\lVert (\pd_t \bU, \nabla Q) \rVert_{L_1 ((0, T), B^s_{q, 1} (\HS))} 
+ \lVert \bU \rVert_{L_1 ((0, T), B^{s + 2}_{q, 1} (\HS))}
+ \sup_{0 \le t < T} \lVert \bU (t) \rVert_{B^s_{q, 1} (\HS)} \\
\le C e^{\gamma T} \Big(
\lVert (\bF (\bw), \pd_t \bG (\bw)) 
\rVert_{L_1 ((0, T), B^s_{q, 1} (\HS))} 
+ \lVert (G_\mathrm{div} (\bw), \bH (\bw)) 
\rVert_{L_1 ((0, T), B^{s + 1}_{q, 1} (\HS))} \\
+ \lVert {\color{black} \pd_t} \bH (\bw) 
\rVert_{L_{1} ((0, T), B^{s-1}_{q,1} (\HS))} \Big).
\end{multline}
The right-hand side of \eqref{nonest:-1} is easily 
bounded by means of Propositions \ref{prop:p1}
and \ref{prop-Lagrangian} except for the last term.
In fact, we have
$\|\nabla \bw\|_{L_1((0, T), B^{s+1}_{q,1}(\HS))}
\le C_1 \omega \le 1\slash2$ provided that $\omega \le (2 C_2)^{- 1}$.  
Thus, by Propositions \ref{prop:p1} and \ref{prop-Lagrangian}, there hold
\begin{equation}\label{nonest:8.0} \begin{aligned}
\lVert \bF (\bw) \rVert_{L_1 ((0, T), B^s_{q, 1} (\HS))} 
& \le C \omega^2, \\
\lVert G_\mathrm{div} (\bw) \rVert_{L_1 ((0, T), B^{s + 1}_{q, 1} (\HS))}
& \le C \omega^2, \\
\lVert \pd_t \bG (\bw) \rVert_{L_1 ((0, T), B^s_{q, 1} (\HS))} 
& \le C \omega^2, \\
\lVert \bH (\bw) \rVert_{L_1 ((0, T), B^{s + 1}_{q, 1} (\HS))}
& \le C \omega^2.
\end{aligned}\end{equation}
To estimate the last term of the right-hand side 
of \eqref{nonest:-1} it suffices to  consider the term of the form:
$\nabla \bw \, \sfF(\int^t_0\nabla \bw\d\tau)$, 
where $\sfF(\sfs)$ is given by a Neumann series expansion
with respect to $\sfs$ for $|\sfs| \le \omega < 1\slash2$ such that $\sfF(0) = 0$,
i.e., $\sfF(\sfs) = \sum_{j=1}^\infty \sfs^j$.
Let $\sfF'$ be the derivative of $\sfF$.
Then, it follows from Propositions \ref{prop:p2}, \ref{lem:pro}, and \ref{prop-8.7} that
\begin{align}
& \bigg\|\pd_t\bigg(\nabla \bw \, \sfF\bigg(\int^t_0\nabla \bw
\d\tau\bigg)\bigg)\bigg\|_{B^{s-1}_{q,1}(\HS)}\\
&\quad \le \bigg\|(\pd_t\nabla \bw) \, \sfF\bigg(\int^t_0\nabla \bw\d\tau\bigg)\bigg\|_{B^{s-1}_{q,1}(\HS)}
+\bigg\|\nabla \bw \nabla \bw \, \sfF'\bigg(\int^t_0\nabla \bw\d\tau\bigg)\bigg\|_{B^{s-1}_{q,1}(\HS)}
\\
&\quad\le C \left\{\|\pd_t \bw\|_{B^s_{q,1}(\HS)}
\bigg \|\, \sfF \bigg(\int^t_0\nabla \bw\d\tau \bigg) \bigg\|_{B^{d/q}_{q,1}(\HS)} 
+ \|\nabla \bw\|_{B^{s-1}_{q,1}(\HS)}
\bigg\|\nabla \bw  \, \sfF'
\bigg(\int^t_0\nabla \bw\d\tau\bigg)\bigg\|_{B^{d/q}_{q,1}(\HS)} \right\} \\
& \quad \le C \Big(\|\pd_t\bw\|_{B^s_{q,1}(\HS)} 
\|\nabla\bw\|_{L_1((0, T), B^{d/q}_{q,1}(\HS))}
+ \|\bw\|_{B^s_{q,1}(\HS)} \|\bw\|_{B^{d/q+1}_{q,1}(\HS)} \Big),
\end{align}
where the identity $\sfF' (s) = 1 + \sum_{j = 1}^\infty j s^{j - 1}$ has been used.
Since $d/q + 1 < s+2$ as follows from the condition 
\eqref{cond-qs-local1}, we may choose $\theta \in (0, 1)$ such that 
$d/q+1 = s(1-\theta) + (s+2)(1-\theta)$, and hence we 
infer from the interpolation inequality that 
$\|\bw\|_{B^{d/q+1}_{q,1}(\HS)} \le C\|\bw\|_{B^{s}_{q,1}(\HS)}^{1-\theta}
\|\bw\|_{B^{s+2}_{q,1}(\HS)}^\theta$. Thus, by the embedding
$B^{d\slash q +1}_{q, 1}(\HS) \hookleftarrow B^{s+1}_{q,1} (\HS)$
and the Young inequality, there holds
\begin{equation}
\label{boundary-estimate}
\begin{split}
\|\bw\|_{B^s_{q,1}(\HS)}\|\bw\|_{B^{d/q+1}_{q,1}(\HS)}
& \le C\|\bw\|_{B^s_{q,1}(\HS)}^{2-\theta}\|\bw\|_{B^{s+2}_{q,1}(\HS)}^{\theta} \\
& \le 
C_\theta \Big( \epsilon \|\bw\|_{B^{s+2}_{q,1}(\HS)} + \epsilon^{- \frac{\theta}{1-\theta}}
\|\bw\|_{B^s_{q,1}(\HS)}^{\frac{2-\theta}{1-\theta}} \Big)
\end{split}
\end{equation}
for every $\epsilon > 0$. Thus, we have 
\begin{equation}
\begin{split}
\|\pd_t\bH(\bw)\|_{L_1((0, T), B^{s-1}_{q,1}(\HS)}
&\le C \Big( \epsilon \|\bw\|_{L_1((0, T), B^{s+2}_{q,1}(\HS))}
+ \epsilon^{- \frac{\theta}{1-\theta}}
\|\bw\|_{L_\infty ((0, T), B^s_{q,1}(\HS))}^{\frac{2-\theta}{1-\theta}}T
\\
& \quad + \|\pd_t\bw\|_{L_1((0, T), B^s_{q,1}(\HS))}\|\bw\|_{
L_1 ((0, T), B^{s+2}_{q,1}(\HS))} \Big).
\end{split}
\end{equation}
Since it is assumed that $\|\bb-\ba\|_{B^s_{q,1}(\HS)} < \sigma$, there holds
\begin{equation}
\|\bw\|_{L_\infty ((0,T), B^s_{q,1}(\HS))} \le C\|\ba\|_{B^s_{q,1}(\HS)}
+ \|\pd_t\bw\|_{L_1((0, T), B^s_{q,1}(\HS))} 
\le C \Big(\|\bb\|_{B^s_{q,1}(\HS)} + \sigma + \omega \Big),
\end{equation}
where the estimate $\lVert \ba \rVert_{B^s_{q, 1} (\HS)} \le 
\lVert \bb - \ba \rVert_{B^s_{q, 1} (\HS)} + \|\bb\|_{B^s_{q,1}(\HS)}$
has been used. Hence, we end up with
\begin{equation}
\label{nonest:8.1}
\lVert \pd_t \bH (\bw) \rVert_{L_1 ((0, T), B^{s - 1}_{q, \infty} (\HS))}
\le C \bigg\{\omega^2 + \epsilon \omega + 
\epsilon^{- \frac{\theta}{1 - \theta}} T \Big(\|\bb\|_{B^s_{q,1}(\HS)} + \sigma
+ \omega \Big)^{\frac{2-\theta}{1-\theta}} \bigg\}.
\end{equation}
Combining \eqref{nonest:-1}, \eqref{nonest:8.0}, and \eqref{nonest:8.1}
implies that
\begin{equation}\label{nonest:8.2}\begin{aligned}
&\lVert {\color{red}e^{- \gamma t}} (\pd_t \wt \bU, \wt \nabla Q) \rVert_{L_{{\color{red}1}} (\BR_+, B^s_{q, 1} (\HS))} 
+ \lVert {\color{red}e^{- \gamma t}} \wt \bU \rVert_{L_{\color{red} 1} (\BR_+, B^{s + 2}_{q, 1} (\HS))} 
+ \sup_{t \in [0, \infty)} \lVert \wt \bU (t) \rVert_{B^s_{q, 1} (\HS)} \\
& \qquad \le C \bigg\{\omega^2 + \epsilon \omega 
+ \epsilon^{- \frac{\theta}{1 - \theta}} T \Big(\|\bb\|_{B^s_{q,1}(\HS)} + \sigma
+ \omega \Big)^{\frac{2-\theta}{1-\theta}} \bigg\}.
\end{aligned}\end{equation}
Since $0 < T < 1$ implies $Ce^{\gamma T} \le Ce^\gamma$, 
it follows from \eqref{nonest:8.2} that
\begin{equation}\label{nonest:8.2*}\begin{aligned}
&\lVert (\pd_t \bU, \nabla Q) \rVert_{L_1((0, T), B^s_{q, 1} (\HS))} 
+ \lVert \bU \rVert_{L_1((0, T), B^{s + 2}_{q, 1} (\HS))}
+ \sup_{t \in [0, T)} \lVert \bU (t) \rVert_{B^s_{q, 1} (\HS)} \\
& \qquad \le C_3 \bigg\{\omega^2 + \epsilon \omega 
+ \epsilon^{- \frac{\theta}{1 - \theta}} T \Big(\|\bb\|_{B^s_{q,1}(\HS)} + \sigma
 + \omega \Big)^{\frac{2-\theta}{1-\theta}} \bigg\}.
\end{aligned}\end{equation}
We first take $\epsilon$ such that $\epsilon = 1\slash(6 C_3)$,
and then
choosing $\omega>0$ and $T > 0$ such that $C_3 \omega < 1 \slash 6$ and
$C_3 \epsilon^{- \theta \slash (1 - \theta)} T (\|\bb\|_{B^s_{q,1}(\HS)}
+ \omega)^{(2-\theta)\slash(1-\theta)} < \omega/6$, we have 
$\lVert (\bU, Q) \rVert_{S_T} \le \omega/2.$
Thus, setting $\bu = \bu_{\ba} + \bU$ and $\sfQ = \sfQ_{\ba} +Q$ , by \eqref{small:8.2}
we have $\lVert (\bu, \sfQ) \rVert_{S_T} \le \omega$, 
which shows that $(\bu, \sfQ) \in  S_{T, \omega}$. 
Namely, we see that $\Phi$ maps $S_{T, \omega}$ into itself. \par
To verify that the map $\Phi$ is contractive, 
we consider two pairs of functions 
$(\bw^{(1)}, \sfQ^{(1)}), (\bw^{(2)}, \sfQ^{(2)}) 
\in S_{T, \omega}$
and set $(\bu^{(\ell)}, \sfP^{(\ell)}) 
= \Phi (\bw^{(\ell)}, \sfQ^{(\ell)})$, $\ell = 1, 2$.
Let $\delta \bu = \bu^{(2)} - \bu^{(1)}$ and
$\delta \sfP = \sfP^{(2)} - \sfP^{(1)}$. Then we have
\begin{align}
\label{eq-difference}
\left\{\begin{aligned}
\pd_t (\delta \bu) - \DV(\mu\BD(\delta \bu) -  \delta \sfP\BI) 
& = \bF (\bw^{(2)}) - \bF (\bw^{(1)}) 
& \quad & \text{in $\HS \times (0, T)$}, \\
\dv (\delta \bu) & = G_\mathrm{div} (\bw^{(2)}) - G_\mathrm{div} (\bw^{(1)}) 
& \quad & \text{in $\HS \times (0, T)$}, \\
(\mu\BD(\delta \bu) -  \delta \sfP\BI) \bn_0 
& = \bH (\bw^{(2)}) - \bH (\bw^{(1)})
& \quad & \text{on $\pd \HS \times (0, T)$}, \\
(\delta \bu) \vert_{t = 0} & = 0 & \quad & \text{in $\HS$}	
\end{aligned}\right.	
\end{align}
with
\begin{equation}
G_\mathrm{div} (\bw^{(2)}) - G_\mathrm{div} (\bw^{(1)}) 
= \dv \Big(\bG (\bw^{(2)}) - \bG (\bw^{(1)}) \Big) \qquad \text{in $\HS \times (0, T)$}.
\end{equation}
Then, we see that
\begin{align}
\bF (\bw^{(2)}) - \bF (\bw^{(1)}) 
& = \bigg(\int_0^t \nabla (\delta \bw) (y, \tau) \d \tau \bigg)^\top
\Big(\pd_t \bw^{(2)} - \mu \Delta \bw^{(2)} - \mu \nabla \dv \bw^{(2)} \Big) \\
& \quad + \bigg(\int_0^t \nabla \bw^{(1)} (y, \tau) \d \tau \bigg)^\top
\Big(\pd_t (\delta \bw) - \mu \Delta (\delta \bw) 
- \mu \nabla \dv (\delta \bw) \Big) \\
& \quad +\mu \bigg(\int_0^t \nabla (\delta \bw) (y, \tau) \d \tau \bigg)^\top \dv 
\Big( (\BA_{\bw^{(2)}} \BA_{\bw^{(2)}}^\top - \BI) \nabla \bw^{(2)} \Big)\\
&\quad + \mu\bigg\{\BI + \bigg(\int^t_0\nabla \bw^{(1)}(y, \tau)\d\tau\bigg)^\top \bigg\}
 \dv \bigg( (\delta (\BA_{\bw} \BA_{\bw}^\top)) \nabla \bw^{(2)} \bigg) \\
& \quad + \mu\bigg\{\BI + \bigg(\int^t_0\nabla \bw^{(1)}(y, \tau)\d\tau\bigg)^\top \bigg\} \dv 
\Big((\BA_{\bw^{(1)}} \BA_{\bw^{(1)}}^\top - \BI) \nabla (\delta \bw) \Big) \\
& \quad + \mu\big( \nabla_y(\delta\BA_{\bw}^\top):\nabla_y\bw^{(2)}\big)
+\mu\nabla_y\big((\BA^\top_{\bw^{(1)}}-\BI):\nabla_y(\delta\bw)\big), \\
G_\mathrm{div} (\bw^{(2)}) - G_\mathrm{div} (\bw^{(1)}) 
& = - (\delta \BA_{\bw}^\top) \colon \nabla \bw^{(2)}
+ (\BI - \BA_{\bw^{(1)}}^\top) \colon \nabla (\delta \bw), \\
\bG (\bw^{(2)}) - \bG (\bw^{(1)}) 
& = - (\delta \BA_{\bw}) \bw^{(2)} + (\BI - \BA_{\bw^{(1)}}) \delta \bw, \\
\bH (\bw^{(2)}) - \bH (\bw^{(1)})
& = \mu \bigg[ (\nabla_y (\delta \bw)) (\BI-\BA_{\bw^{(2)}}^\top) 
- (\nabla_y (\bw^{(1)})) \delta \BA_{\bw}^\top \\
& \quad + \bigg(\int_0^t \nabla_y (\delta \bw) (y, \tau) \d \tau \bigg)^\top         
[\nabla_y \bw^{(2)}]^\top \BA_{\bw^{(2)}} (\BI-\BA_{\bw^{(2)}}^\top) \\
& \quad + \bigg\{\BI + \bigg(\int_0^t \nabla_y \bw^{(2)} (y, \tau) \d \tau \bigg)^\top\bigg\}  
[\nabla_y (\delta \bw)]^\top \BA_{\bw^{(2)}} (\BI-\BA_{\bw^{(2)}}^\top) \\
& \quad + \bigg\{\BI + \bigg(\int_0^t \nabla_y \bw^{(1)} (y, \tau) \d \tau \bigg)^\top\bigg\}  
[\nabla_y \bw^{(1)}]^\top (\delta \BA_{\bw}) (\BI-\BA_{\bw^{(2)}}^\top) \\
& \quad - \bigg\{\BI + \bigg(\int_0^t \nabla_y \bw^{(1)} (y, \tau) \d \tau \bigg)^\top\bigg\}
[\nabla_y \bw^{(1)}]^\top \BA_{\bw^{(1)}} (\delta \BA_{\bw}^\top) \\
& \quad + \bigg(\int_0^t \nabla_y (\delta \bw) (y, \tau) \d\tau \bigg)^\top
[\nabla_y \bw^{(2)}]^\top \BA_{\bw^{(2)}} \\
& \quad + \bigg(\int_0^t \nabla_y \bw^{(1)} (y, \tau) \d\tau \bigg)^\top
[\nabla_y (\delta \bw)]^\top \BA_{\bw^{(2)}} \\
& \quad + \bigg(\int_0^t \nabla_y \bw^{(1)} (y, \tau) \d\tau \bigg)^\top
[\nabla_y \bw^{(1)}]^\top (\delta \BA_{\bw}) \\
& \quad + [\nabla_y (\delta \bw)]^\top (\BI-\BA_{\bw^{(2)}}) 
- [\nabla_y (\bw^{(1)})]^\top \delta \BA_{\bw} \bigg] \wt \bn_0.
\end{align}
Similarly to \eqref{nonest:8.0}, by Proposition \ref{prop-Lagrangian},
there hold
\begin{align}
\|\bF(\bw^{(2)}) - \bF(\bw^{(1)})\|_{L_1((0, T), B^s_{q,1}(\HS)}
& \le C\omega [\delta\bw]_{q, s, T}, \\
\|G_\mathrm{div}(\bw^{(2)}) - G_\mathrm{div}(\bw^{(1)})\|_{L_1((0, T), B^{s+1}_{q,1}(\HS)}
& \le C\omega [\delta\bw]_{q, s, T}, \\
\|\pd_t(\bG(\bw^{(2)}) - \bG(\bw^{(1)})\|_{L_1((0, T), B^{s}_{q,1}(\HS)}
& \le C\omega [\delta\bw]_{q, s, T}, \\
\|\bH(\bw^{(2)}) - \bH(\bw^{(1)})\|_{L_1((0, T), B^{s+1}_{q,1}(\HS)}
& \le C\omega [\delta\bw]_{q, s, T}.
\end{align}
To estimate $\|\pd_t(\bH(\bw^{(2)}) - \bH(\bw^{(1)}))\|_{L_1((0, T), 
B^{s-1}_{q,1}(\HS))}$, analogously to the proof of \eqref{nonest:8.1}, 
it suffices to estimate
\begin{equation}
\pd_t \bigg\{\nabla \bw^{(2)}\sfF\bigg(\int^t_0\nabla \bw^{(2)}\d\tau\bigg)
-\nabla \bw^{(1)}\sfF\bigg(\int^t_0\nabla \bw^{(1)}\d\tau \bigg) \bigg\}=: \sum_{m = 1}^5 I_m (t),
\end{equation}
where we have set
\begin{align}
I_1 (t) & = \nabla\pd_t(\delta\bw) \sfF\bigg(\int^t_0\nabla \bw^{(2)}\d\tau\bigg), \\
I_2 (t) & = \nabla\pd_t\bw^{(1)}\bigg\{\sfF\bigg(\int^t_0\nabla \bw^{(2)}\d\tau\bigg)
-\sfF\bigg(\int^t_0\nabla\bw^{(1)}\d\tau\bigg)\bigg\}, \\
I_3 (t) &=(\nabla\delta\bw)\sfF'\bigg(\int^t_0\nabla \bw^{(2)}\d\tau\bigg)\nabla\bw^{(2)},\\
I_4 (t) & = \nabla\bw^{(1)}\sfF'\bigg(\int^t_0\nabla \bw^{(2)}\d\tau\bigg)\nabla\delta\bw, \\
I_5 (t) & = \nabla\bw^{(1)}\bigg\{\sfF'\bigg(\int^t_0 \nabla\bw^{(2)}\d\tau\bigg)
- \sfF'\bigg(\int^t_0\nabla\bw^{(1)}\d\tau\bigg)\bigg\}\nabla\bw^{(1)}.
\end{align}
Recalling that $B^{d/q}_{q,1} (\HS)$ is a Banach algebra, $\sfF(0) = 0$, and 
$\int^T_0\|\nabla \bw^{(\ell)} (\,\cdot\, , \tau)\|_{B^{d/q}_{q,1}(\HS)}\d\tau 
\le \omega \le 1\slash2$, $\ell = 1, 2$,
we infer from Propositions \ref{prop:p2} and \ref{prop-8.7} that
\begin{align}
\|I_1 (t)\|_{B^{s-1}_{q,1}(\HS)}& \le C\|\pd_t(\delta\bw)(\,\cdot\, , t)\|_{B^s_{q,1}(\HS)}
\int^T_0\|\nabla\bw^{(2)}(\,\cdot\, , \tau)\|_{B^{d/q}_{q,1}(\HS)}\d\tau, \\
\|I_2 (t)\|_{B^{s-1}_{q,1}(\HS)}& \le C\|\pd_t\bw^{(1)}(\,\cdot\, , t)\|_{B^s_{q,1}(\HS)}
\int^T_0\|\nabla (\delta\bw)(\,\cdot\, , \tau)\|_{B^{d/q}_{q,1}(\HS)}\d\tau, \\
\|I_3 (t)\|_{B^{s-1}_{q,1}(\HS)} & \le \| (\delta\bw)(\,\cdot\, , t)\|_{B^s_{q,1}(\HS)}
\|\nabla\bw^{(2)}(\,\cdot\, , t)\|_{B^{d/q}_{q,1}(\HS)}, \\
\|I_4 (t)\|_{B^{s-1}_{q,1}(\HS)} & \le C\|\nabla\bw^{(1)}(\,\cdot\, , t)\|_{B^{d/q}_{q,1}(\HS)}
\|(\delta\bw)(\,\cdot\, , t)\|_{B^s_{q,1}(\HS)}, \\
\|I_5 (t)\|_{B^{s-1}_{q,1}(\HS)}& \le C\|\bw^{(1)}(\,\cdot\, , t)\|_{B^s_{q,1}(\HS)}
\|\nabla\bw^{(1)}(\,\cdot\, , t)\|_{B^{d/q}_{q,1}(\HS)}\int^T_0\|\nabla(\delta\bw)(\,\cdot\, ,
\tau)\|_{B^{d/q}_{q,1}(\HS)}\d\tau.
\end{align}
Recall that $s < d/q + 1 < s+2$ as follows from \eqref{cond-qs-local1}. 
For the estimate of $I_5 (t)$, choosing $\theta\in(0, 1)$ such that 
$d/q+ 1 = (1-\theta)s + \theta(s+2)$, we have
\begin{equation}
\|\bw^{(1)}(\,\cdot\, , t)\|_{B^s_{q,1}(\HS)}
\|\nabla\bw^{(1)}(\,\cdot\, , t)\|_{B^{d/q}_{q,1}(\HS)}
\le C_\theta \Big(\|\bw^{(1)}(\,\cdot\, , t)\|_{B^{s+2}_{q,1}(\HS)} 
+ \|\bw^{(1)}(\,\cdot\, , t)\|_{B^s_{q,1}(\HS)}^{\frac{2-\theta}{1-\theta}} \Big).
\end{equation}
Since $(\delta \bw)|_{t=0} = 0$, we observe
\begin{equation}
\|(\delta \bw)(\,\cdot\, , t)\|_{B^s_{q,1}(\HS)} \le C\|\pd_t(\delta \bw)
\|_{L_1((0, T),B^s_{q,1}(\HS))}.
\end{equation}
Moreover, there holds
\begin{equation}
\|\bw^{(1)}(\,\cdot\, , t)\|_{B^s_{q,1}(\HS)} \le C \bigg(\|\ba\|_{B^s_{q,1}(\HS)}
+ \int^T_0\|\pd_t\bw^{(1)}(\,\cdot\, , \tau)\|_{B^s_{q,1}(\HS)} \d \tau \bigg)
\le C \Big(\|\bb\|_{B^s_{q,1}(\HS)} + \omega \Big).
\end{equation}
Summing up, we obtain
\begin{align}
&\|\pd_t(\bH(\bw^{(2)}) - \bH(\bw^{(1)}))\|_{L_1((0, T), B^{s-1}_{q,1}(\HS))}
\\
&\quad
\le C \bigg\{\omega + T {\color{black} \Big(\|\bb\|_{B^s_{q,1}(\HS)}+\omega\Big)^{\frac{2-\theta}{1-\theta}}}
\bigg\} \Big(\|\delta\bw\|_{L_1((0, T), B^{s+2}_{q,1}(\HS))}
+ \|\pd_t(\delta\bw)\|_{L_1((0, T), B^s_{q,1}(\HS))} \Big).
\end{align}
Thus, it holds
\begin{equation}
{\color{black} \lVert (\delta\bu, \delta\sfP) \rVert_{S_T}} 
\le C_4 \bigg\{\omega + 
T \Big(\|\bb\|_{B^s_{q,1}(\HS)}+\omega \Big)^{\frac{2-\theta}{1-\theta}}\bigg\}
{\color{black} \lVert (\delta\bw, \delta\sfQ) \rVert_{S_T}}.
\end{equation}
Choosing $\omega$ and $T$ {\color{black}smaller such} that 
\begin{equation}
C_4 \bigg\{\omega + T \Big(\|\bb\|_{B^s_{q,1}(\HS)}+\omega 
\Big)^{\frac{2-\theta}{1-\theta}}\bigg\} \le \frac12
\end{equation}
{\color{black} if necessary}, we have
${\color{black} \lVert (\delta\bu, \delta\sfP) \rVert_{S_T}} 
\le (1/2) \lVert (\delta\bw, \delta\sfQ) \rVert_{S_T}$,
which shows that the map $\Phi$ is contractive. Hence, 
by the Banach fixed point theorem, there exists a unique $(\bu, Q)
\in S_{T, \omega}$ which {\color{black}satisfies} $\Phi(\bu, Q) = (\bu, Q)$.
Obviously, $(\bu, Q)$ is a required unique solution of equations
\eqref{eq-fixed}.
Therefore, the proof of Theorem \ref{th-local-well-posedness-fixed} is 
{\color{black} complete}. \par
\subsection{Proof of Theorem \ref{thm:lw.2}}
The proof of Theorem \ref{thm:lw.2} is similar to 
the procedure in the previous subsection.
{\color{black}The only difference here arises in the estimates for} the nonlinear term $\bH(\bw)$.
Let $S_{T, \omega}$ be the space defined in \eqref{def-STomega},
where $\omega \in (0, 1 \slash 2)$ is a small number determined later.
{\color{black} In contrast to the proof of Theorem \ref{th-local-well-posedness-fixed}}, 
the time interval $(0, T)$ is arbitrarily given. From \eqref{initial:8.1}, we choose $c_0 > 0$
so small that 
\begin{equation}\label{smallness:8.1}
\|(\pd_t\bu_{\ba}, \sfQ_{\ba})\|_{L_1((0, T), B^s_{q,1}(\HS))}
+ \|\bu_{\ba}\|_{L_1((0, T), B^{s+2}_{q,1}(\HS))}
\le Ce^{\gamma T}\|\ba\|_{B^s_{q,1}(\HS)}
\le Ce^{\gamma T}c_0 \le \omega/2
\end{equation}
is satisfied, where
$\omega$ will be chosen small enough later. 
Given $(\bw, \sfQ) \in S_{T, \omega}$, let 
${\color{black} (\wt \bU, \wt Q)} :=(\bu-\bu_a, \sfQ-Q_{\ba})$ solve
equations \eqref{eq-fixed-point-tilde},
where the boundary term $E_{(T)}\bH(\bw)$ is replaced with  
$\BH({\color{black} \wt \bw})$. Here, $\BH$ is a nonlinear term given in the fourth line 
of \eqref{def-nonlinear-terms} and ${\color{black} \wt \bw}$ is defined by
\begin{equation}
{\color{black} \wt \bw} = \wt  \bu_{\ba} + E_{(T)}(\bw-\bu_{\ba})
\end{equation}
{\color{black}with} $\wt \bu_{\ba}$ defined by 
\begin{equation}
\wt \bu_{\ba} = \begin{cases} \bu_{\ba}, &\quad t > 0, \\
\psi_2 (t)\bu_{\ba}(\,\cdot\, , -t), & \quad t < 0,
\end{cases}
\end{equation}
where $\psi_2 (t)$ is a smooth function defined on $\BR$ such that 
$\psi_2 (t) = 1$ for $t > - 1$ and $\psi_2 (t) = 0$ for $t < - 2$.
{\color{black} It is easy to find that $(\wt \bU, \wt  Q) \vert_{t \in (0, T)}$ 
is a solution to \eqref{eq-fixed-point}}. \par
To estimate $\BH({\color{black} \wt \bw})$, we {\color{black} use} the following estimate:
\begin{equation}\label{nonest:8.4}
 \|e^{-\gamma t}fg\|_{W^{1/2}_{1}(\BR, B^s_{q,1}(\HS))}
\le C\|e^{-\gamma t}f\|_{W^{1/2}_{1}(\BR, B^s_{q,1}(\HS))}
\Big(\|g\|_{L_\infty(\BR, B^{d/q}_{q,1}(\HS))}
+ \|\pd_tg\|_{L_1(\BR, B^{d/q}_{q,1}(\HS))} \Big).
\end{equation}
In fact, using Corollary \ref{prop:p1}, we have the following two estimates:
\begin{align}
\|e^{-\gamma t}fg\|_{L_1(\BR, B^s_{q,1}(\HS))} & 
\le C\|e^{-\gamma t}
f\|_{L_1(\BR, B^s_{q,1}(\HS))}\|g\|_{L_\infty(\BR, B^{d/q}_{q,1}(\HS))}
\\
\|e^{-\gamma t}fg\|_{W^1_1(\BR, B^s_{q,1}(\HS))} 
& \le C \Big(\|e^{-\gamma t}f\|_{W^1_1(\BR, B^s_{q,1}(\HS))}
\|g\|_{L_\infty(\BR, B^{d/q}_{q,1}(\HS))} \\
& \quad + \|e^{-\gamma t}f\|_{L_\infty(\BR, B^s_{q,1}(\HS))}
\|\pd_tg\|_{L_1(\BR, B^{d/q}_{q,1}
(\HS))} \Big) \\
&\le C\|e^{-\gamma t}f\|_{W^1_1(\BR, B^s_{q,1}(\HS))} \Big(
\|g\|_{L_\infty(\BR, B^{d/q}_{q,1}(\HS))}
+ \|\pd_tg\|_{L_1(\BR, B^{d/q}_{q,1}(\HS))} \Big),
\end{align}
where we have used the fact that 
$\|e^{-\gamma  t}f(\,\cdot\, , t)\|_{B^s_{q,1}(\HS)} 
\le \int_{- \infty}^\infty\|e^{-\gamma t}\pd_tf(\,\cdot\, , t)\|_{B^s_{q,1}(\HS)}\d t$.
Interpolating these two estimates with the real interpolation method 
and using 
\begin{equation}
W^{1/2}_1(\BR, B^s_{q, 1} (\HS)) = 
(L_1(\BR, B^s_{q, 1}(\HS)), W^1_1(\BR, B^s_{q, 1} (\HS)))_{1/2,1},
\end{equation}
see \eqref{inhomo-int.1}, we have \eqref{nonest:8.4}. 
Notice that the estimate of $\BH (\wt \bw)$ follows from the bound of the term
of the form $\nabla \wt \bw \sfF (\int_0^t \nabla \wt \bw \d \tau)$,
where $\sfF$ is the Neumann series introduced in the proof of
Theorem \ref{th-local-well-posedness-fixed}. 
It follows from \eqref{nonest:8.4} that
\begin{equation}
\bigg\|e^{-\gamma t} \nabla \wt \bw \sfF \bigg(\int_0^t \nabla \wt \bw \d \tau \bigg) 
\bigg\|_{W^{1/2}_{1}(\BR, B^s_{q,1}(\HS))}
\le C\|e^{-\gamma t}\nabla \wt \bw\|_{W^{1/2}_{1}(\BR, B^s_{q,1}(\HS))}
\| \wt \bw\|_{L_1(\BR, B^{s+2}_{q,1}(\HS))}.
\end{equation}
Since it follows from Proposition \ref{prop-B7} that
\begin{equation}
\| e^{-\gamma t}\nabla \wt \bw\|_{W^{1/2}_{1} (\BR, B^s_{q,1}(\HS))}
\le C \Big(\|e^{-\gamma t} \wt \bw\|_{W^1_{1} (\BR, B^s_{q,1}(\HS))}
+ \|e^{-\gamma t}\wt \bw\|_{L_{1} (\BR, B^{s+2}_{q,1}(\HS))} \Big),
\end{equation}
we deduce that
\begin{equation}
\bigg\|e^{-\gamma} \nabla \wt \bw \sfF \bigg(\int_0^t \nabla \wt \bw \d \tau \bigg) 
\bigg\|_{W^{1/2}_{1}(\BR, B^s_{q,1}(\HS))}
\le C\omega^2.	
\end{equation}
Namely, 
\begin{equation}
\|e^{-\gamma t}\BH(\wt \bw)\|_{W^{1/2}_{1}(\BR, B^s_{q,1}(\HS))}
\le C\omega^2.
\end{equation}
Since $B^{s+1}_{q,1}(\HS)$ is a Banach algebra (cf. Corollary \ref{prop:p1}), we have 
\begin{equation}
\|e^{-\gamma t} \BH( \wt \bw)\|_{L_{1} (\BR, B^{s+1}_{q,1}(\HS))}
\le C\omega^2.
\end{equation}
Combining these estimates with the estimates \eqref{nonest:8.0}$_{1,2,3}$,
by \eqref{linear:8.1} we have 
\begin{equation}
\|(\pd_t\bU, \nabla Q)\|_{L_1((0, T), B^s_{q,1}(\HS))}
+ \|\bU\|_{L_1((0, T), B^{s+2}_{q,1}(\HS))}
\le {\color{black} C_5} e^{\gamma T}\omega^2.
\end{equation}
Choosing $\omega>0$ so small that ${\color{black} C_5} e^{\gamma T}\omega<1/2$ 
is valid, we observe $\lVert (\bU, Q) \rVert_{S_T} < \omega/2$. 
Therefore, we see that $(\bu, \sfQ) =(\bu_{\ba}+\bU, \sfQ_{\ba} + Q)$ 
belongs to $S_{T, \omega}$, which shows that $\Phi$ maps $S_{T, \omega}$ into itself.
Proving contraction of $\Phi$ is similar to the proof of 
Theorem \ref{th-local-well-posedness-fixed} by using the aforementioned argument 
for the estimate of the boundary data $\BH$, and hence we may omit its proof.
{\color{black} Consequently}, $\Phi$ is a contraction mapping on $S_{T, \omega}$. 
Thus, there exists a unique fixed point 
$(\bu, \sfP) \in S_{T, \omega}$ of the mapping $\Phi$,
which is a unique solution to System \eqref{eq-fixed}.
The proof of Theorem \ref{thm:lw.2} is complete.
\subsection{Proof of Theorem \ref{th-global-well-posedness-fixed}}
In the following, we now assume that $q$ and $s$ satisfy
$d - 1 < q  < 2d$ and $s=-1+d/q$, respectively, which stems from
Proposition \ref{prop:p1}. More precisely, we see that the estimates
\begin{align}
\label{est-nonlinear-eg-homogeneous-0}
\bigg\lVert \pd_t u \cdot \int_0^\infty \nabla v (\,\cdot\, , \tau) 
\d \tau \bigg\rVert_{\dot B^s_{q, 1} (\HS)}
& \le C \lVert \pd_t u \rVert_{\dot B^s_{q, 1} (\HS)}
\int_0^\infty \|\nabla^2 v (\,\cdot\, , \tau)\|_{\dot B^s_{q, 1} (\HS)} \d \tau,\\
\label{est-nonlinear-eg-homogeneous-1}
\bigg\lVert \nabla^2 u \cdot \int_0^\infty \nabla v (\,\cdot\, , \tau) \d \tau 
\bigg\rVert_{\dot B^s_{q, 1} (\HS)}
& \le C \lVert \nabla^2 u \rVert_{\dot B^s_{q, 1} (\HS)}
 \int_0^\infty \|\nabla^2 v (\,\cdot\, , \tau)\|_{\dot B^s_{q, 1} (\HS)} \d \tau,
\end{align}
are valid whenever $d - 1 < q  < 2d$ and $s=-1+d/q$ are satisfied.
In fact, it follows from Proposition \ref{prop:p1} that
there holds
\begin{equation}
\bigg\lVert \pd_t u \cdot \int_0^\infty \nabla v (\,\cdot\, , \tau) 
\d \tau \bigg\rVert_{\dot B^s_{q, 1} (\HS)}
\le C \lVert \pd_t u \rVert_{\dot B^s_{q, 1} (\HS)}
\int_0^\infty \|\nabla v (\,\cdot\, , \tau)\|_{\dot B^{d \slash q}_{q, 1} (\HS)}
\d \tau 
\end{equation}
with all $(q, s)$ subject to \eqref{cond-qs}. Together with 
\cite[Cor. 3.20]{DHMTpre}, we deduce that
\begin{equation}
\bigg\lVert \pd_t u \cdot \int_0^\infty \nabla v (\,\cdot\, , \tau) 
\d \tau \bigg\rVert_{\dot B^s_{q, 1} (\HS)}
\le C \lVert \pd_t u \rVert_{\dot B^s_{q, 1} (\HS)}
\int_0^\infty \|\nabla^2 v (\,\cdot\, , \tau)\|_{\dot B^{- 1 + d \slash q}_{q, 1} (\HS)}
\d \tau 
\end{equation}
is valid provided that $(q, s)$ satisfies \eqref{cond-qs}.
Thus, we find that \eqref{est-nonlinear-eg-homogeneous-0} holds
with $s = - 1 + {\color{red}d} \slash q$
{\color{red}since the embedding $\dot B^s_{q, 1} (\HS) \hookrightarrow
\dot B^{- 1 + d\slash q}_{q, 1} (\HS)$ holds if and only if $s = - 1 + d \slash q$}. 
In addition, by Proposition \ref{prop:p1},
it is necessary to assume $d - 1 < q < 2 d$ to obtain 
\eqref{est-nonlinear-eg-homogeneous-0}. Likewise, we see that
\eqref{est-nonlinear-eg-homogeneous-1} holds for $d - 1 < q  < 2d$ and $s=-1+d/q$.
\par
Let $\ba \in \dot B^{- 1 + d \slash q}_{q,1}(\HS)^d$ 
be an initial data satisfying the smallness condition: 
$\|\ba\|_{\dot B^{- 1 + d \slash q}_{q,1}(\HS)} \le c_1\omega =: c_0$.  
Here, $c_1$ and $\omega$ are positive numbers determined later. 
Without loss of generality, we assume $0 < c_1, \omega < 1\slash2$.
Let $(\bu_{\ba}, \sfQ_{\ba})$  be solutions to System \eqref{eq-linear-L}.	
By virtue of Theorem \ref{th-MR-homogeneous}, we obtain a unique global 
solution $(\bu_{\ba}, \sfQ_{\ba})$ to \eqref{eq-linear-L} satisfying
\begin{equation}\label{small:8.0*}
\pd_t\bu_{\ba}, \nabla \sfQ_{\ba} \in L_1(\BR_+, \dot B^{- 1 + d \slash q}_{q,1}(\HS)^d),
\qquad \nabla^2\bu_{\ba} \in L_1(\BR_+, \dot B^{- 1 + d \slash q}_{q,1}(\HS)^{d^3})
\end{equation}
as well as
\begin{equation}\label{initial:8.1*}
\lVert (\pd_t \bu_{\ba}, \nabla^2\bu_{\ba}, \nabla \sfQ_{\ba}) 
\rVert_{L_1 (\BR_+, \dot B^s_{q, 1} (\HS))}
\le C_6 \lVert \ba \rVert_{B^s_{q, 1} (\HS)} \le C_6 c_1\omega.
\end{equation}
\par
Define an underlying space $\dot S_{\infty, \omega}$ by setting 
\begin{equation}
\dot S_{\infty, \omega}
= \left\{ (\bu, \sfQ) \enskip \left\vert \enskip
\begin{aligned}
\pd_t \bu, \nabla \sfQ & \in L_1(\BR_+, \dot B^{- 1 + 1 \slash q}_{q, 1} (\HS)^d), \\
\nabla^2 \bu & \in L_1(\BR_+, \dot B^{- 1 + d \slash q}_{q,1}(\HS)^{d^3}), \\
& \lVert (\bu, \sfQ) \rVert_{\dot S_\infty} \le \omega, \quad \bu|_{t=0} = \ba
\end{aligned}\right.\right\},	
\end{equation}	
where we have set
\begin{align}
\lVert (\bu, \sfQ) \rVert_{\dot S_\infty} & := 
\lVert (\pd_t \bu, \nabla^2 \bu, \nabla \sfQ) 
\rVert_{L_1 (\BR_+, \dot B^{- 1 + d \slash q}_{q, 1} (\HS))}.
\end{align}
Choosing $c_1 > 0$ so small that $C_6 c_1< 1/2$ is fulfilled, 
we see that $(\bu_{\ba}, \sfQ_{\ba})$ satisfies
\begin{equation}\label{initial:8.1**}
\lVert(\bu_{\ba}, \sfQ_{\ba})\rVert_{\dot S_\infty} \le \omega/2.
\end{equation}
Namely, there holds $(\bu_{\ba}, \sfQ_{\ba}) \in \dot S_{\infty, \omega}$. \par 
Given $(\bw, \sfP) \in \dot S_{\infty,\omega}$, 
we intend to find a fixed point in $\dot S_{\infty, \omega}$ for the mapping 
$\Phi \colon (\bw, \sfP) \mapsto (\bu, \sfQ)$ with $(\bu,\sfQ)$ 
which is the solution to \eqref{iteration:eq1}.
Then we see that $(\bU, Q) := (\bu - \bu_{\ba}, \sfQ - \sfQ_{\ba})$ solves
\eqref{eq-fixed-point}. 
For $(\bw, \sfP) \in \dot S_{\infty,\omega}$, we infer from the embedding
$L_\infty (\HS) \hookleftarrow \dot B^{d\slash q}_{q, 1} (\HS)$ and
\cite[Cor. 3.20]{DHMTpre} that
\begin{equation}
\label{8.24*}
\begin{split}
\bigg\|\int^t_0\nabla \bw(\,\cdot\, , \tau)\d\tau\bigg\|_{L_\infty(\HS)}
& \le \int^\infty_0 \lVert\nabla\bw(\,\cdot\, , \tau)\|_{\dot B^{d/q}_{q,1}(\HS)}\d\tau \\
& \le C_7 \int^\infty_0 \lVert\nabla^2\bw(\,\cdot\, , \tau)
\|_{\dot B^{- 1 + d/q}_{q,1}(\HS)}\d\tau \\
& \le C_7 \omega \\
& < \frac{1}{2}
\end{split}	
\end{equation}
for every $t>0$, provided that $\omega < (2 C_7)^{- 1}$ additionally if necessary.
Hence, we see that the right-hand members $\bF(\bw)$, $G_\mathrm{div}(\bw)$, $\bG(\bw)$, 
and $\bH(\bw)$ of \eqref{eq-fixed-point} are well-defined.
To solve System \eqref{eq-fixed-point}, 
we extend the right-hand members to $t \in \BR$
so that we may use Theorem \ref{th-MR-homogeneous}.
To this end, for a function $f$ defined on $\BR_+$ we set
\begin{equation}
\sfE_0 [f] (\,\cdot\, , t) = \begin{cases}
f(\,\cdot\,, t) & \quad\text{for $t>0$}, \\
0 &\quad \text{for $t < 0$}, 
\end{cases}		
\qquad
\sfE_1 [f] (\,\cdot\, , t) = \begin{cases}
f(\,\cdot\,, t) & \quad\text{for $t>0$}, \\
f(\,\cdot\,, - t) &\quad \text{for $t < 0$}.
\end{cases}		
\end{equation} 
Notice that, if $f (\,\cdot\, , 0) = 0$, there holds
\begin{equation}
\pd_t \sfE_0[f](\,\cdot\, , t)  = \begin{cases}
f(\,\cdot\,, t) & \quad\text{for $t>0$}, \\
0 &\quad  \text{for $t < 0$}.
\end{cases}	
\end{equation} 	
It is easy to verify that
\begin{equation}
\begin{aligned}
\lVert \sfE_\ell [f] (\,\cdot\, , t) \rVert_{L_1 (\BR, X)}
& \le 2\lVert f (\,\cdot\, , t) \rVert_{L_1 (\BR_+, X)}, & \quad & \ell = 0, 1, \\
\lVert \pd_t \sfE_\ell [f](\,\cdot\, , t)
\rVert_{L_1 (\BR, X)}
& \le 2 \lVert (\pd_t f) (\,\cdot\, , t) \rVert_{L_1 (\BR_+,  X)},
& \quad & \ell = 0, 1
\end{aligned}
\end{equation} 		
for a Banach space $X$.
Recalling  the definition of nonlinear terms \eqref{def-nonlinear-terms},
we define their extensions to $t\in\BR$ by 
\begin{equation}
\begin{split}
\wt \bF (\bw) & := 
\sfE_0\bigg[\bigg(\int^t_0\nabla \bw \d\tau \bigg)^\top \bigg]
\sfE_1 [\pd_t \bw - \mu \Delta_y \bw ] \\
& \quad + \mu \bigg(\BI+\sfE_0\bigg[\bigg(\int^t_0\nabla \bw \d\tau \bigg)^\top \bigg]\bigg)
\dv_y \Big( \sfE_0 [(\BA_{\bw} \BA_{\bw}^\top - \BI)]\sfE_1[ \nabla_y \bw] \Big) 
\\
&\quad + \mu \nabla_y \big(\sfE_0[\BA_{\bw}^\top - \BI] \colon 
\sfE_1[\nabla_y \bw] \big), \\
\wt G_\mathrm{div} (\bw) & := \sfE_0[\BI - \BA_{\bw}^\top] \colon 
\sfE_1[\nabla_y \bw,] \\
\wt \bG (\bw) & := \sfE_0[\BI - \BA_{\bw}]\sfE_1[\bw], 
\\
\wt \bH (\bw) & := 
\mu \bigg[ \sfE_1 \bigg[\nabla \bw + \bigg\{\BI+\bigg(\int^t_0\nabla \bw \d\tau\bigg)^\top \bigg\}
[\nabla \bw]^\top \BA_{\bw} \bigg] \sfE_0 (\BI-\BA_{\bw}^\top) \\
& \;\quad + \sfE_0 \bigg[\bigg(\int_0^t \nabla \bw \d \tau \bigg)^\top \bigg] 
\sfE_1 \Big([\nabla \bw]^\top \BA_{\bw} \Big) 
+ \sfE_1 [\nabla \bw]^\top \sfE_0 (\BI-\BA_{\bw}) \bigg] \wt \bn_0.
\end{split}	
\end{equation}
Since $G_\mathrm{div} (\bw)$ and $\bG (\bw)$ vanish for $t < 0$,
there holds $\wt G_\mathrm{div} (\bw) = \dv \wt \bG (\bw)$ for all $t \in \BR$.
Since we have $\wt \bF(\bw) = \bF(\bw)$, $\wt G_\mathrm{div} (\bw) 
= \wt G_\mathrm{div}(\bw)$, $\wt \bG(\bw) = \wt \bG(\bw)$, and 
$\wt \bH(\bw) = \wt \bH(\bw)$ for $t > 0$, we may consider the following
\textit{linear} system:
\begin{align}
\label{eq-fixed-point-tilde*}
\left\{\begin{aligned}
\pd_t \wt \bU - \DV(\mu\BD(\wt \bU) -  \wt Q\BI) & = \wt \bF (\bw)
& \quad & \text{in $\HS \times \BR$}, \\
\dv \bU & = \wt G_\mathrm{div} (\bw) = \dv \wt \bG (\bw) 
& \quad & \text{in $\HS \times \BR$}, \\
(\mu\BD(\wt \bU) -  \wt Q\BI) \bn_0 & = \wt \bH (\bw)
& \quad & \text{on $\pd \HS \times \BR$}, \\
\wt \bU \vert_{t = 0} & = 0 & \quad & \text{in $\HS$}.		
\end{aligned}\right.		
\end{align}
Clearly, we find that $(\wt \bU, \wt Q) \vert_{t \in \BR_+}$ is
also a solution to \eqref{eq-fixed-point}.
Since the right-hand members of \eqref{eq-fixed-point-tilde*} is
defined on $\BR$, we may apply Theorem \ref{th-MR-homogeneous} to obtain
\begin{multline}\label{linear:8.1*}
\lVert (\pd_t \wt \bU, \nabla^2 \wt \bU, \nabla \wt Q)
\rVert_{L_1(\BR_+, \dot B^{- 1 + d \slash q}_{q, 1} (\HS))} +
\sup_{t \in [0, \infty)} \|\wt \bU(\,\cdot\,, t)\|_{\dot B^{- 1 + d \slash q}_{q,1}(\HS)} \\
\le C \bigg(\lVert (\wt \bF(\bw), \nabla \wt G_\mathrm{div} (\bw), 
\pd_t \wt\bG(\bw), \nabla \wt \bH(\bw))
\rVert_{L_1 (\BR, \dot B^{- 1 + d \slash q}_{q, 1} (\HS))} 
+\|\wt \bH(\bw)\|_{\dot W^{1/2}_1(\BR, \dot B^{- 1 + d \slash q}_{q, 1} (\HS))} \bigg).			
\end{multline}
where we have used the well-known embedding
$\dot W^1_1 (\BR_+, X) \hookrightarrow \mathrm{BC} ([0, \infty), X)$
for a Banach space $X$. To bound the right-hand side of
\eqref{linear:8.1*}, we use the following inequality
\begin{equation}
\label{reverse-Benstein}
\lVert \nabla \bw \rVert_{L_1 (\BR_+, \dot B^{d \slash q}_{q, 1} (\HS))}
\le C \lVert \nabla^2 \bw \rVert_{L_1 (\BR_+, \dot B^{- 1 + d \slash q}_{q, 1} (\HS))},
\end{equation}
which follows from \cite[Cor. 3.20]{DHMTpre}.
Then, by virtue of \eqref{8.24*}, we see that
\begin{equation}\label{nonest:8.0*} 
\begin{aligned}
\lVert \wt \bF (\bw) \|_{L_1 (\BR, \dot B^{- 1 + d \slash q}_{q, 1} (\HS))} 
& \le C \omega^2, \\
\lVert \nabla \wt G_\mathrm{div} (\bw)
\rVert_{L_1 (\BR, \dot B^{- 1 + d \slash q}_{q, 1} (\HS))}
& \le C \omega^2, \\
\lVert \pd_t \wt \bG (\bw) \rVert_{L_1 (\BR, \dot B^{- 1 + d \slash q}_{q, 1} (\HS))} 
& \le C \omega^2, \\
\lVert \nabla \wt \bH (\bw) \rVert_{L_1 (\BR, \dot B^{- 1 + d \slash q}_{q, 1} (\HS))}
& \le C \omega^2, \\
\|\wt \bH(\bw) \|_{\dot W^{1/2}_1(\BR, \dot B^{- 1 + d \slash q}_{q,1}(\HS))}
& \le C\omega^2.		
\end{aligned}\end{equation}
Here, the estimates \eqref{nonest:8.0*}$_{1,2,3}$ are direct consequences
of Propositions \ref{prop:p1}, \ref{lem:pro}, and \ref{prop-Lagrangian}
and the embedding 
$\dot W^1_1 (\BR_+, \dot B^{- 1 + d \slash q}_{q, 1} (\HS)) 
\hookrightarrow \mathrm{BC} ([0, \infty), \dot B^{- 1 + d \slash q}_{q, 1} (\HS))$.
As it has been seen in the proof of Theorem \ref{th-local-well-posedness-fixed},
to show the estimates \eqref{nonest:8.0*}$_{4,5}$, it suffices to bound the term
of the form $\sfE_1 [\nabla \bw] \sfE_0 [\sfF (\int_0^t \nabla \bw \d \tau)]$.
Together with Propositions \ref{prop:p1} and \ref{lem:pro},
we infer from \eqref{reverse-Benstein} that
\begin{align}
\bigg\| \nabla \bigg(\sfE_1 [\nabla \bw] \sfE_0 
\bigg[\sfF \bigg(\int_0^t \nabla \bw \d \tau\bigg)\bigg] \bigg)
\bigg\|_{L_1(\BR, \dot B^{-1+d\slash q}_{q,1} (\HS))}
& \le C\|\nabla^2 \bw\|_{L_1(\BR_+, \dot B^{- 1 + d \slash q}_{q,1} (\HS))}
\|\nabla \bw \|_{L_1(\BR_+, \dot B^{d/q}_{q,1})} \\
& \le C\omega^2.
\end{align}
Namely, we have \eqref{nonest:8.0*}$_4$.
To show \eqref{nonest:8.0*}$_5$, we use the following estimate:
\begin{equation}\label{1/2-estimate}
\|fg\|_{\dot W^{1/2}_1(\BR, \dot B^{- 1 + d \slash q}_{q,1} (\HS))} 
\le C\|f\|_{\dot W^{1/2}_1(\BR, \dot B^{- 1 + d \slash q}_{q,1} (\HS))}
\Big(\|g\|_{L_\infty(\BR, \dot B^{d/q}_{q,1} (\HS))}
+ \|\pd_tg\|_{L_1(\BR, \dot B^{d/q}_{q,1} (\HS))} \Big),
\end{equation}
which may be proved along the same argument as in \eqref{nonest:8.4}
with the aid of \eqref{homo-int.1}. 
Applying the estimate \eqref{1/2-estimate}, we have
\begin{align}
&\bigg\| \nabla \bigg(\sfE_1 [\nabla \bw] \sfE_0 
\bigg[\sfF \bigg(\int_0^t \nabla \bw \d \tau\bigg)\bigg] \bigg)
\bigg\|_{\dot W^{1/2}_1(\BR, \dot B^s_{q,1})} \\
&\quad \le C\|\nabla \sfE_1[\bw]\|_{\dot W^{1/2}(\BR, \dot B^s_{q,1} (\HS))}
\int^\infty_0\|\nabla\bw (\,\cdot\, , \tau)\|_{\dot B^{d/q}_{q,1} (\HS)}\d\tau.
\end{align}
due to Propositions \ref{prop:p1} and \ref{lem:pro}.
Since Proposition \ref{prop-B7} implies
\begin{equation}
\|\nabla \sfE_1[\bw]\|_{\dot W^{1/2}(\BR, \dot B^{- 1 + d \slash q}_{q,1} (\HS))}
\le C \Big(\|\pd_t\bw\|_{L_1(\BR_+, \dot B^{- 1 + d \slash q}_{q,1} (\HS))}
+ \|\nabla^2\bw\|_{L_1(\BR_+, \dot B^{- 1 + d \slash q}_{q,1} (\HS))} \Big),
\end{equation}
we arrive at \eqref{nonest:8.0*}$_5$. Thus, from
\eqref{linear:8.1} and \eqref{nonest:8.0} there holds
$\lVert (\bU, Q) \rVert_{\dot S_\infty} \le C_8 \omega^2$.
Assuming that $0 < \omega < \min \{1\slash2, (2 C_6)^{- 1}, (2 C_7)^{- 1}, 
(2 C_8)^{- 1}\}$ is fulfilled, we have 
$\lVert (\bU, Q)\rVert_{\dot S_\infty} \le \omega/2$. 
Thus, $\bu= \bu_{\ba}+\bU$ and $\sfQ = \sfQ_{\ba} 
+ Q$ are solutions to \eqref{iteration:eq1}
and satisfy $\lVert (\bu, \sfQ) \rVert_{\dot S_\infty} \le \omega$.  
Hence, the mapping $\Phi$ maps $\dot S_{\infty, \omega}$ into itself. \par
To verify that the map $\Phi$ is contractive, 
we consider two pairs of functions 
$(\bw^{(1)}, \sfQ^{(1)}), (\bw^{(2)}, \sfQ^{(2)}) 
\in \dot S_{\infty, \omega}$
and set $(\bu^{(\ell)}, \sfP^{(\ell)}) 
= \Phi (\bw^{(\ell)}, \sfQ^{(\ell)})$, $\ell = 1, 2$.
Let $\delta \bu = \bu^{(2)} - \bu^{(1)}$ and
$\delta \sfP = \sfP^{(2)} - \sfP^{(1)}$. Then we see that
$(\delta \bu, \delta \sfP)$ solves \eqref{eq-difference} with $T = \infty$.
Analogously to \eqref{nonest:8.0}, by Propositions
\ref{prop:p1}, \ref{lem:pro}, and \ref{prop-Lagrangian},
there hold
\begin{align}
\|\bF(\bw^{(2)}) - \bF(\bw^{(1)})\|_{L_1(\BR_+, \dot B^{- 1 + d \slash q}_{q,1}(\HS))}
& \le C\omega \lVert (\delta\bw, \delta\sfQ)\rVert_{\dot S_\infty}, \\
\|\nabla( G_\mathrm{div} (\bw^{(2)}) - G_\mathrm{div} (\bw^{(1)}))
\|_{L_1(\BR_+, \dot B^{- 1 + d \slash q}_{q,1}(\HS))}
& \le C\omega \lVert (\delta\bw, \delta\sfQ)\rVert_{\dot S_\infty} \\
\|\pd_t(\bG(\bw^{(2)}) - \bG(\bw^{(1)}))
\|_{L_1(\BR_+, \dot B^{- 1 + d \slash q}_{q,1}(\HS))}
& \le C\omega \lVert (\delta\bw, \delta\sfQ)\rVert_{\dot S_\infty}, \\
\|\nabla(\bH(\bw^{(2)}) - \bH(\bw^{(1)}))
\|_{L_1(\BR_+, \dot B^{- 1 + d \slash q}_{q,1}(\HS))}
& \le C\omega\lVert (\delta\bw, \delta\sfQ)\rVert_{\dot S_\infty},\\
\|\bH(\bw^{(2)}) - \bH(\bw^{(1)})
\|_{\dot W^{1/2}_1(\BR_+, \dot B^{d \slash q}_{q,1}(\HS))}
& \le C\omega\lVert (\delta\bw, \delta\sfQ)\rVert_{\dot S_\infty}.	
\end{align}
Thus, we infer from Theorem \ref{th-MR-homogeneous} that
$\lVert (\delta\bu, \delta\sfP)\rVert_{\dot S_\infty}
\le C_9 \omega \lVert (\delta\bw, \delta\sfQ)\rVert_{\dot S_\infty}$.
Taking $\omega$ so small that $\omega <(2 C_9)^{- 1}$ if necessary,  
we see that the map $\Phi$ is contractive.  Hence, 
by the Banach fixed point theorem, there exists a unique $(\bu, \sfQ)
\in \dot S_{\infty, \omega}$ which satisfies $\Phi(\bu, \sfQ) = (\bu, \sfQ)$.
Clearly, $(\bu, \sfQ)$ is a required unique global strong solution to
System \eqref{eq-fixed}. The proof of Theorem 
\ref{th-global-well-posedness-fixed} is complete. 
\subsection{Proof of Corollaries \ref{cor-local-well-posedness-original},
\ref{cor:lw.2}, and \ref{cor-global-well-posedness-original}}
From \eqref{emb-BC1}, we see that
$\bu \in L_1 ((0, T), \mathrm{BC}^1 (\overline{\HS})^d)$. 
Thus, by virtue of the classical Picard-Lindel\"of theorem, 
there exists a unique $C^1$-flow $\bX_{\bu} (\,\cdot\, , t)$ satisfying \eqref{representation-Xu}. 
Furthermore, by a similar argument as in \cite[Sec. 8.3]{DHMTpre}, we see that
$\bX_{\bu} (\,\cdot\, , t)$ is a $C^1$-diffeomorphism from $\HS$ onto $\Omega (t)$ and measure 
preserving {\color{black}for every $t \in [0, T)$}. 
In fact, the divergence free condition $\dv \bv = 0$ yields $\det \BA_{\bu} = 1$,
see \cite[p. 382]{Sol88}. In addition, the inverse of $\bX_{\bu} (\,\cdot\,, t)$ exists
for each $t \in [0, T)$ due to the estimate on $\bu$ together with the condition 
\eqref{integral-smallness}. It follows from \ref{emb-BC1} that $\bX_{\bu}(\,\cdot\, , t)$ is a 
$C^1$-function from $\HS$ onto $\Omega (t)$ for each $t \in [0, T)$, since
$\bu \in L_1 ((0, T), \mathrm{BC}^1 (\overline{\HS})^d)$. 
In the following, let $\bX_{\bu}^{- 1}$ be the inverse of $\bX_{\bu}$.
\par
For any function $F \in \CB^s_{q, 1} (\HS)$, $1 < q < \infty$, 
$- \min (d \slash q, d \slash q') < s \le d \slash q$,
it follows from the chain rule (and the transformation 
rule for integrals) that
\begin{equation}
\label{norm-equivalence}
\lVert F \circ \bX_{\bu}^{- 1} \rVert_{\CB^s_{q, 1} (\Omega (t))}
\le C \lVert F \rVert_{\CB^s_{q, 1} (\HS)}
\end{equation}
with a constant $C > 0$. Indeed, \eqref{norm-equivalence} for the case 
$\CB^s_{q, 1} = \dot B^s_{q, 1}$ was proved in \cite[Prop. 8.7]{DHMTpre}
and for the case $\CB^s_{q, 1} = B^s_{q, 1}$ may be proved along the
same way as in the discussion given in Section 8.3 in \cite{DHMTpre}.
We now recall the relation 
$(\bv, P) = (\bu, {\color{black}\sfQ}) \circ \bX_{\bu}^{- 1}$ and the definition 
$\BA_{\bu} := (\nabla_y \bX_{\bu})^{- 1}$. 
Let $\BA_{\bu}^\top = (A_{j, k})$. There holds
\begin{equation}
\nabla_x (\bv, P) = (\BA_{\bu}^\top \nabla_y (\bu, Q)) \circ \bX_{\bu}^{- 1}, 
\quad \pd_{x_j} \pd_{x_k} \bv 
= \sum_{\ell, \ell' = 1}^d \Big( A_{j, \ell} \pd_{y_\ell} 
(A_{k, \ell'} \pd_{y_{\ell'}} \bu ) \Big) \circ \bX_{\bu}^{- 1},
\quad j, k = 1, \ldots, d,
\end{equation}
Concerning the time derivative of $\bv$, we rely on the relation
\begin{equation}
\pd_t \bv = (\pd_t \bu) \circ \bX_{\bu}^{- 1} 
- \Big((\bu \circ \bX_{\bu}^{- 1}) \cdot \nabla_x \Big) \bv
\end{equation}
By Theorem \ref{th-local-well-posedness-fixed} and
Proposition \ref{prop-Lagrangian}, we arrive at the
desired estimate stated in Corollary \ref{cor-local-well-posedness-original}. 
The proof of Corollary \ref{cor-local-well-posedness-original} 
is complete. Likewise, corollaries \ref{cor:lw.2} and 
\ref{cor-global-well-posedness-original} 
may be proved in the same manner by using Theorems~\ref{thm:lw.2} and
\ref{th-global-well-posedness-fixed}, respectively.
\appendix
\section{Recasting the system in Lagrangian coordinates}
\label{sec-A}
For the reader's convenience, we provide here how to derive
\eqref{def-nonlinear-terms}. To this end, we use the following 
well-known formulas:
\begin{alignat}2
\nabla_x & = \BA_{\bu}^\top \nabla_y, & \qquad 
\dv_x (\,\cdot\,) & = \BA_{\bu}^\top \colon \nabla_y (\,\cdot\,) 
= \dv_y (\BA_{\bu} \,\cdot \, ), \\
\bn & = \frac{\BA_{\bu}^\top \bn_0}{\lvert \BA_{\bu}^\top \bn_0 \rvert}, 
& \qquad
\nabla_x \dv_x (\, \cdot \,) & = \BA_{\bu}^\top \nabla_y \dv_y (\, \cdot \,)
+ \BA_{\bu}^\top \nabla_y ((\BA_{\bu}^\top - \BI) \colon \nabla_y \, \cdot \, ),
\end{alignat}
{\color{black}see, e.g., \cite[p. 383]{Sol88}.}
{\color{black} In fact,} as it was proved in \cite[{\color{black} p. 382}]{Sol88}, 
there holds $\det \BA_{\bu} = 1$ as follows from the divergence-free condition, 
which yields the first formula $\nabla_x = \BA_{\bu}^\top \nabla_y$.
By using these formulas, it is easy to verify the representations of
$G_\mathrm{div} (\bu)$ and $\bG (\bu)$. Hence, it suffices to derive
the representations of $\bF (\bu)$ and $\BH (\bu)$.  \par
By a direct calculation, we observe
\begin{equation}
\DV_x(\mu\BD(\bv) - P\BI) = \mu \Delta_x \bv 
+ \mu \nabla_x \dv_x \bv - \nabla_x P.
\end{equation}
We see that
\begin{align}
&\pd_t \bv + (\bv \cdot \nabla_x) \bv = \pd_t \bu, \\
&\Delta_x \bv = \dv_x\nabla_x \bv 
= \dv_y(\BA_{\bu} \BA_{\bu}^\top\nabla_y \bu)
= \dv_y((\BA_{\bu} \BA_{\bu}^\top-\BI)\nabla_y \bu) + \Delta_y \bu, \\
&\nabla_x\dv_x\bv = \BA_{\bu}^\top\nabla_y (\BA_{\bu}^\top:\nabla_y\bu)
= \BA_{\bu}^\top\nabla_y((\BA_{\bu}^\top-\BI):\nabla_y\bu)+
\BA_{\bu}^\top\nabla_y\dv_y\bu, \\
&\nabla_x P = \BA_{\bu}^\top\nabla_y \sfQ.
\end{align}
Since $\BA_{\bu}^\top$ is invertible and $(\BA_{\bu}^\top)^{-1}
= \BI + \left(\int^t_0\nabla \bu\d\tau \right)^\top$, 
the equation
\eqref{eq-original}$_1$ is transformed into
\begin{align}
& \pd_t \bu - \mu \Delta_y \bu - \mu \nabla_y \dv_y \bu
+ \nabla_y \sfQ \\
& = \bigg(\int^t_0\nabla\bu\d\tau\bigg)^\top
\Big(\pd_t \bu - \mu \Delta_y \bu\Big) 
+ \mu \bigg\{\BI + \bigg(\int^t_0\nabla\bu\d\tau\bigg)^\top\bigg\}
\dv_y \Big( (\BA_{\bu} \BA_{\bu}^\top - \BI) \nabla_y \bu \Big) \\
& \quad + \mu \nabla_y \Big((\BA_{\bu}^\top - \BI) \colon \nabla_y \bu \Big).	
\end{align}
Combined with 
\begin{equation}
\DV_y (\mu \BD(\bu) - \sfQ \BI) 
= \mu \Delta_y \bu + \mu \nabla_y \dv_y \bu - \nabla_y \sfQ,
\end{equation}
we have the representation of $\bF (\bu)$. Note that $\bF (\bu)$
does not contain the pressure $\sfQ$. \par
It remains to deal with $\BH (\bu)$. It is easy to find that
\begin{equation}
\mu\BD(\bv) -  P\BI = \mu \Big(\BA_{\bu}^\top \nabla_y \bu 
+ [\nabla \bu]^\top \BA_{\bu} \Big) - \sfQ \BI.
\end{equation}
Together with the boundary condition \eqref{eq-original}$_3$, it follows that
\begin{equation}
\mu \Big(\BA_{\bu}^\top \nabla_y \bu 
+ [\nabla \bu]^\top \BA_{\bu} \Big) 
\frac{\BA_{\bu}^\top \bn_0}{\lvert \BA_{\bu}^\top \bn_0 \rvert}
- \sfQ \frac{\BA_{\bu}^\top \bn_0}{\lvert \BA_{\bu}^\top \bn_0 \rvert} = 0
\qquad \text{on $\pd \HS$}.
\end{equation}
Multiplying this equation by $\lvert \BA_{\bu}^\top \bn_0 
\rvert (\BA_{\bu}^\top)^{- 1}$ yields
\begin{equation}
\mu \Big(\nabla_y \bu + (\BA_{\bu}^\top)^{- 1} 
[\nabla \bu]^\top \BA_{\bu} \Big) 
\BA_{\bu}^\top \bn_0 - \sfQ \bn_0 = 0
\qquad \text{on $\pd \HS$}.
\end{equation}
Namely, we have
\begin{align}
&(\mu \BD(\bu) - \sfQ \BI) \bn_0\\
&= \mu \Big[ \Big(\nabla_y \bu + (\BA_{\bu}^\top)^{- 1}
[\nabla_y \bu]^\top \BA_{\bu} \Big) (\BI-\BA_{\bu}^\top)
+ (\BI - (\BA_{\bu}^\top)^{- 1}) [\nabla_y \bu]^\top \BA_{\bu}
+ [\nabla_y \bu]^\top (\BI-\BA_{\bu}) \Big]\bn_0.
\end{align}
Since there holds
$(\BA_{\bu}^\top)^{- 1} = \BI + (\int_0^t \nabla_y \bu \d \tau )^\top$,
we obtain the representation of $\BH (\bu)$.

\section{Technical tools}\label{ap.B}

\begin{prop}
\label{prop-Fourier-multiplier}
Let $1 < q < \infty$, $1 \le r \le \infty$, $s \in \BR$, and $d \ge 1$.
Suppose that $m (\xi)$ is a complex-valued smooth function on 
$\BR^d \setminus \{0\}$ satisfying $\lvert \pd_\xi^\alpha m (\xi) \rvert 
\le C \lvert \xi \rvert^{- \lvert \alpha \rvert}$ for every 
$\alpha \in \BN^d$ such that $\lvert \alpha \rvert \le [d \slash 2] + 1$,
where $[d \slash 2]$ stands for the integer part of $d \slash 2$. 
There exists a constant $C$ such that
\begin{equation}
\lVert \CF^{- 1} [m (\xi) \CF [f]] \rVert_{\CB^s_{q, r} (\BR^d)} 
\le C \lVert f \rVert_{\CB^s_{q, r} (\BR^d)}.
\end{equation}	
\end{prop}
\begin{proof}
We first consider the homogeneous Besov space case.  From the definition of
$\dot B^s_{q,r}(\BR^d)$ it follows that 
$$\|\CF^{-1}[m(\xi)\CF[f]]\|_{\dot B^s_{q,r}(\BR^d)}
= \|2^{js}\|\dot \Delta_j \CF^{-1}[m(\xi)\CF[f]]\|_{L_q(\BR^d)}\|_{\ell^r(\BZ)}.
$$
Since $\dot \Delta_j \CF^{-1}[m(\xi)\CF[f]]= \CF^{-1}_\xi[m(\xi)\phi(2^{-j}\xi)\CF[f]]
= \CF^{-1}_\xi[m(\xi)\CF[\dot \Delta_j f]]$,
by the Fourier multiplier theorem of Mihlin-Ho\"rmander type, we have
$\|\dot \Delta_j \CF^{-1}[m(\xi)\CF[f]]\|_{L_q(\BR^d)} \leq C\|\dot \Delta_j f\|_{L_q(\BR^d)}$.
Thus, 
$$\lVert \CF^{- 1} [m (\xi) \CF [f]] \rVert_{\dot B^s_{q, r} (\BR^d)} 
\leq C\|2^{js}\|\dot \Delta_j f\|_{L_q(\BR^d)}\|_{\ell^r(\BZ)} = C\|f\|_{\dot B^s_{q,r}(\BR^d)}.
$$
Likewise, we have
\begin{align*}
\lVert \CF^{- 1} [m (\xi) \CF [f]] \rVert_{\dot B^s_{q, r} (\BR^d)} 
&= \|2^{js}\|\Delta_j \CF^{- 1} [m (\xi) \CF [f]] \|_{L_q(\BR^d)}\|_{\ell^r(\BN_0)} \\
&\leq C\|2^{js}\|\Delta_j f \|_{L_q(\BR^d)}\|_{\ell^r(\BN_0)}
= C\|f\|_{B^s_{q,r}(\BR^d)}.
\end{align*}
This completes the proof of Proposition \ref{prop-Fourier-multiplier}.
\end{proof}

\begin{prop}
\label{prop-real-interpolation}
Let $X_0$ and $X_1$ be Banach spaces which are an interpolation couple,
and let $Y$ be another Banach space. 
Assume that $0 < \sigma_0, \sigma_1, \theta < 1$ satisfy
$1 = (1 - \theta)(1- \sigma_0) + \theta(1+ \sigma_1)$. Let $\gamma \ge 0$.
For $t > 0$ let $T (t) \colon Y \to  X_0+X_1$ 
be a bounded linear operator such that
\begin{equation}
\label{assumption-decay}
\begin{split}
\lVert T (t) f \rVert_Y 
& \le C e^{\gamma t} t^{- 1 + \sigma_0} \lVert f \rVert_{X_0}, 
\qquad f \in X_0, \\
\lVert T (t) f \rVert_Y
& \le C e^{\gamma t} t^{- 1 - \sigma_1} \lVert f \rVert_{X_1}, 
\qquad f \in X_1.
\end{split}
\end{equation}
Then, there holds
\begin{equation}
\int^\infty_0 e^{-\gamma t}\lVert T (t) f \rVert_Y\d t
\le C \lVert f \rVert_{(X_0, X_1)_{\theta, 1}}
\end{equation}
with a constant $C > 0$ independent of $\gamma$.
\end{prop}
\begin{proof}
The proof is based on real interpolation. For $k \in \BZ$ set
\begin{equation}
b_k (f) := \sup_{t \in [2^k, 2^{k + 1}]} e^{- \gamma t} 
\lVert T (t) f \rVert_Y.
\end{equation}
We observe that
\begin{equation}
\label{est-B2}
\int^\infty_0 e^{-\gamma t}\lVert T (t) f \rVert_Y\d t
= \sum_{k \in \BZ} \int_{2^k}^{2^{k + 1}} e^{- \gamma t} 
\lVert T (t) f \rVert_Y\d t 
\le \sum_{k \in \BZ} 2^k b_k (f).
\end{equation}
Then we infer from the assumptions \eqref{assumption-decay} that
\begin{alignat}2
b_k (f) & \le C \sup_{t \in [2^k, 2^{k + 1}]} t^{- 1 + \sigma_0} 
\lVert f \rVert_{X_0} \le C 2^{- k (1 - \sigma_0)} \lVert f \rVert_{X_0}, 
& \qquad & f \in X_0, \\
b_k (f) & \le C \sup_{t \in [2^k, 2^{k + 1}]} t^{- 1 - \sigma_1} 
\lVert f \rVert_{X_1} \le C 2^{- k (1 + \sigma_1)} \lVert f \rVert_{X_1}, 
& \qquad & f \in X_1.
\end{alignat}
Namely, there hold
\begin{alignat}2
\lVert (b_k)_{k \in \BZ} \rVert_{\ell^{1 - \sigma_0}_\infty (\BZ)}
& \le C \lVert f \rVert_{X_0},
& \qquad & f \in X_0, \\
\lVert (b_k)_{k \in \BZ} \rVert_{\ell^{1 + \sigma_1}_\infty (\BZ)}
& \le C \lVert f \rVert_{X_1},
& \qquad & f \in X_1.
\end{alignat}
Since $(\ell^{1 - \sigma_0}_\infty (\BZ), \ell^{1 + \sigma_1}_\infty (\BZ))_{\theta, 1}
= \ell^1_1 (\BZ)$ due to \cite[Thm. 5.6.1]{BLbook}, it follows that
\begin{equation}
\label{est-B3}
\sum_{k \in \BZ} 2^k b_k (f) = \lVert b_k (f)_{k \in \BZ} \rVert_{\ell^1_1 (\BZ)}
\le C \lVert f \rVert_{(X_0, X_1)_{\theta, 1}}.
\end{equation}
Thus, the desired estimates follow from \eqref{est-B2} and \eqref{est-B3}.
\end{proof}
\begin{prop}
\label{prop-B4}
Let $1 < q < \infty$, $1 \leq r \leq \infty-$, and $- d \slash q' < s < d \slash q$  
$($or $- d \slash q' < s \le d \slash q$ if $r = 1$$)$. 
Then, $C^\infty_0 (\BR^d)$ is dense in $\dot B^s_{q,r} (\BR^d)$.
\end{prop}
\begin{proof}
We first note that $\CS_0 (\BR^d) := \{f \in \CS (\BR^d) \mid 0 \notin \supp (\CF [f])\}$
is dense in $\dot B^s_{q,r} (\BR^d)$, see \cite[Prop. 2.27]{BCD}. 
Let $\varphi \in C^\infty_0 (\BR^d)$ be a function such that $\varphi (x) = 1$ for 
$\lvert x \rvert \le 1$ and $\varphi (x) = 0$ for $\lvert x \rvert \ge 2$. In addition,
set $\varphi_R (x) := \varphi (x \slash R)$. It suffices to show
\begin{equation}
\lim_{R\to\infty} \lVert (1 - \varphi_R) f \rVert_{\dot B^s_{q,r} (\BR^d)} = 0 \qquad
\text{for any $f \in \CS_0 (\BR^d)$},
\end{equation}
since $\varphi_R f \in C^\infty_0 (\BR^d)$. For any $N \in \BN$, let 
$w_N (x) = (1 + \lvert x \rvert^2 )^N$ be a weight function. Then we see that
$w_N f \in \CS_0 (\BR^d)$ due to $f \in \CS_0 (\BR^d)$. Indeed, $f \in \CS_0 (\BR^d)$
implies $\supp (\CF[f]) \cap \{0\} = \emptyset$, which together with
$\CF[w_N f] (\xi) = (1 - \Delta_\xi)^N \CF[f] (\xi)$ yields
$\supp (\CF[w_N f]) \cap \{0\} = \emptyset$. Writing $(1 - \varphi_R ) f = (1 - \varphi_R) w_N^{- 1} w_N f$,
we infer from Proposition \ref{prop:APH} that
\begin{equation}
\lVert (1 - \varphi_R ) f \rVert_{\dot B^s_{q,r} (\BR^d)} 
= \lVert (1 - \varphi_R) w_N^{- 1} w_N f \rVert_{\dot B^s_{q,r} (\BR^d)} 
\le C\lVert (1 - \varphi_R) w_N^{- 1} 
\rVert_{\dot B^{d\slash q}_{q,\infty} (\BR^d) \cap L_\infty (\BR^d)} 
\lVert w_N f \rVert_{\dot B^s_{q,r} (\BR^d)}
\end{equation}
Since there hold
\begin{align*}
\lvert (1 - \varphi_R (x) ) \pd^\alpha w_N^{- 1} (x) \rvert &\le \begin{cases} C|x|^{-2N} 
&\quad |x| \geq R, \\ 0&\quad |x| \leq R, \end{cases} \\
|(\pd^\beta (1 - \varphi_R (x) )) \pd^\alpha w_N^{-1}(x)| &\le \begin{cases} CR^{-1}|x|^{-2N} 
&\quad |x| \geq R, \\ 0&\quad |x| \leq R, \end{cases}
\end{align*}
for $\beta \not=0$, 
we see that
\begin{align}
\lVert (1 - \varphi_R) w_N^{- 1} \rVert_{\dot B^{d\slash q}_{q,\infty} (\BR^d) \cap L_\infty (\BR^d)} 
& \le C \lVert (1 - \varphi_R) w_N^{- 1} \rVert_{H^{d+1}_q (\BR^d)} \\
& \le C \sum_{0 \le \lvert \alpha \rvert \le d+1} 
\lVert (1 - \varphi_R) \pd^\alpha w_N^{- 1} \rVert_{L_q (\BR^d)}
+ C_{\varphi, N, d} R^{-1} \\
& \le C R^{- \sigma}
\end{align}
with some positive constant $\sigma$, where the embedding 
\begin{equation}
\dot B^{d\slash q}_{q,\infty} (\BR^d) \cap L_\infty (\BR^d)
\hookleftarrow B^{d/q}_{q, 1} (\BR^d)
\hookleftarrow H^{d+1}_q (\BR^d), 
\end{equation}
has been used. 
Hence, we deduce that
\begin{equation}
\limsup_{R \to \infty} \lVert (1 - \varphi_R ) f \rVert_{\dot B^s_{q,r} (\BR^d)} 
\le C \limsup_{R \to \infty} R^{- \sigma} \lVert w_N f \rVert_{\dot B^s_{q,r} (\BR^d)} = 0.
\end{equation}
The proof for the case $s = d \slash q$ with $r = 1$ is similar, since we have
\begin{equation}
\lVert (1 - \varphi_R ) f \rVert_{\dot B^{d \slash q}_{q,1} (\BR^d)} 
= \lVert (1 - \varphi_R ) w_N^{-1} w_N f \rVert_{\dot B^{d \slash q}_{q,1} (\BR^d)} 
\le \lVert (1 - \varphi_R ) w_N^{-1} \rVert_{\dot B^{d \slash q}_{q,1} (\BR^d)} 
\lVert w_N f \rVert_{\dot B^{d \slash q}_{q,1} (\BR^d)},
\end{equation}
see Proposition \ref{lem:pro}. The proof is complete.
\end{proof}
\begin{prop}\label{prop:Bdense} Let $1 < q < \infty$, $1 \leq r \leq \infty-$, and $-1+1/q < s < 1/q$. 
Let $f \in \wh \CB^{s+1}_{q,r}(\BR^d)$.
Then, there exists a sequence $\{f_j\}_{j=1}^\infty \subset C^\infty_0(\BR^d)$ such that 
\begin{align}
\lim_{j\to\infty}\|\nabla(f_j-f)\|_{\CB^s_{q,r}(\BR^d)} = 0.
\end{align}
\end{prop}
\begin{proof} Let $\chi(\xi) \in C^\infty_0(\BR^d)$ which equals to $1$ for $|\xi| \leq 1$ and 
$0$ for $|\xi| \geq 2$.  Set $\phi_k(\xi) = \chi(2^{-k}\xi) - \chi(2^{-k+1}\xi)$ 
for $k\in\BZ$.
We see that there hold
\begin{equation}\label{A4.1}\begin{aligned}	
&0 \leq \phi_k(\xi) \leq 1, \quad 
{\rm supp}\, \phi_k(\xi) \subset \{\xi \in \BR^d \mid 2^{k-1} \leq |\xi| \leq 2^{k+1}\}, \\
&\sum_{k\in \BZ}\phi_k(\xi) = 1 \quad\text{for $\xi \in \BR^d\setminus\{0\}$}.
\end{aligned}\end{equation}
Since $f \in \CS'(\BR^d)$, the Fourier transform $\wh f$ of $f$ is well-defined and we see that
\begin{equation}
\left(\sum_{h=-k+1}^k \phi_h(\xi)\right)\wh f(\xi) = (\chi(2^{-k}\xi)- \chi(2^k\xi))\wh f(\xi)
\end{equation}
Set 
\begin{equation}\label{A4.2}
v_k = 
\CF^{-1}_{\xi}\left[ \left(\sum_{h=-k+1}^k \phi_h(\xi)\right)\wh f(\xi) \right]
=\CF^{-1}_{\xi}[ (\chi(2^{-k}\xi)- \chi(2^k\xi))\wh f(\xi)].
\end{equation}
Writing $1 = -\sum_{\ell=1}^d i\xi_\ell|\xi|^{-2}i\xi_\ell$, we have
\begin{equation}
v_k = 
\sum_{\ell=1}^d\sum_{h=-k+1}^k\CF^{-1}_{\xi}[ \phi_h(\xi)(-i\xi_\ell)|\xi|^{-2}]*(\pd_\ell f).    
\end{equation}
Then, we see that there exists a constant $M>0$ independent of $k$ such that
\begin{equation}\label{B.5}
\|v_k\|_{\CB^s_{q,r}(\BR^d)} \leq 2^k M \|\nabla f\|_{\CB^s_{q,r}(\BR^d)}.
\end{equation}
In fact, 
let $g_h^\ell(x) = \CF^{-1}_\xi[\phi_h(\xi)(-i\xi_\ell)|\xi|^{-2}]$. Setting 
$\psi_{0,\ell}(\xi) = \phi_0(\xi) |\xi|^{-2}(-i\xi_\ell)$, we have
\begin{equation}
g_h^\ell(x) = 2^{(d-1)h}\CF_{\xi}^{-1} [\psi_{0,\ell}(2^h\xi)] (x).
\end{equation}
Since $\psi_{0,\ell} \in C^\infty_0(D_{1/2,2})$ with
$D_{1/2,2} = \{\xi \in \BR^d \mid 2^{-1} \leq |\xi| \leq 2\}$, we observe
\begin{equation}
\|g^h_\ell\|_{L_1(\BR^d)} = 2^{-h}\|\CF_{\xi}^{-1}[\phi_{0,\ell}]\|_{L_1(\BR^d)}.    
\end{equation}
Thus, there holds
\begin{align}
\sum_{h=-k+1}^k\|g^\ell_h*(\pd_\ell f)\|_{\CB^s_{q,r}(\BR^d)} & \leq  
\left(\sum_{h=-k+1}^k2^{-h}\|\CF_{\xi}^{-1}[\psi_{0,\ell}]\|_{L_1(\BR^d)}\right)
\|\pd_\ell f\|_{\CB^s_{q,r}(\BR^d)} \\
& \leq 2^k \|\CF_{\xi}^{-1}[\psi_{0,\ell}]\|_{L_1(\BR^d)}\|\pd_\ell f\|_{\CB^s_{q,r}(\BR^d)},
\end{align}
which implies \eqref{B.5} with $M = \sum_{\ell=1}^d \|\CF_{\xi}^{-1}[\psi_{0,\ell}]\|_{L_1(\BR^d)}$. 
Here, we have used the Young inequality in Besov spaces:
\begin{equation}
\label{Young-Besov}
\lVert F*G \rVert_{\CB^s_{q,r} (\BR^d)} \le \lVert F \rVert_{L_1 (\BR^d)} \lVert G \rVert
_{\CB^s_{q,r} (\BR^d)}
\end{equation}
for functions $F \in \CB^s_{q,r} (\BR^d)$ and $G \in L_1 (\BR^d)$, 
which may be proved by the definition of Besov spaces via the Littlewood--Paley dyadic decomposition, 
the H\"older inequality, and the standard Young inequality for convolutions.
\par
Next, we shall show that 
\begin{equation}\label{A4.4} \lim_{k\to\infty} \|\pd_\ell v_k - \pd_\ell f\|_{\CB^s_{q,r}(\BR^d)} = 0.
\end{equation}
To this end, we set $\CF^{-1}_{\xi}[\chi(2^{\mp k}\xi)](x) = J^\pm_k(x)$.  From \eqref{A4.2} it follows that 
$\pd_\ell v_k = J^+_k*(\pd_\ell f) - J^-_k*(\pd_\ell f)$.
To prove \eqref{A4.4}, it suffices to show that 
\begin{align}
\lim_{k\to\infty}\|J^+_k*(\pd_\ell f) - \pd_\ell f\|_{\CB^s_{q,r}(\BR^d)} & = 0, \label{A4.5} \\
\lim_{k\to\infty}\|J^-_k*(\pd_\ell f)\|_{\CB^s_{q,r}(\BR^d)} & = 0. \label{A4.6}
\end{align}
First, we shall show \eqref{A4.5}.  
Recall that $C^\infty_0(\BR^d)$ is dense in $\CB^s_{q,r}(\BR^d)$
for $1 <q < \infty$, $1 \leq r \leq \infty-$ and $-1+1/q < s < 1/q$. Since $\pd_\ell f
\in \CB^s_{q,r}(\BR^d)$, for any
$\epsilon > 0$ there exists a function $g \in C^\infty_0(\BR^d)$ such that 
\begin{equation}\label{apro:A1}
\|\pd_\ell f- g\|_{\CB^s_{q,r}(\BR^d)}< \epsilon.
\end{equation} 
By using the triangle inequality, we may bound $\|J^+_k*(\pd_\ell f) - \pd_\ell f\|_{\CB^s_{q,r}(\BR^d)}$ by  
\begin{equation}
\|J^+_k*(\pd_\ell f) - \pd_\ell f \|_{\CB^s_{q,r}(\BR^d)} 
\leq \|J^+_k*(\pd_\ell f-g)\|_{\CB^s_{q,r}(\BR^d)} + \|J^+_k*g - g\|_{\CB^s_{q,r}(\BR^d)} 
+ \|g-\pd_\ell f\|_{\CB^s_{q,r}(\BR^d)}.
\end{equation}
Since $J^{\pm}_k(x) = \CF^{-1}[\chi(2^{\mp k}\xi)](x) = 2^{kd}\CF^{-1}[\chi](2^{\pm k}x)$, 
we have
\begin{equation}
\int_{\BR^d} J^\pm_k(x)\d x = \int_{\BR^d}\CF^{-1}[\chi](x)\d x = 
\int_{\BR^d} e^{-i\xi\cdot\ x}\CF^{-1}[\chi](x)\d x\bigg\vert_{\xi=0} = \chi(0) = 1
\end{equation}
as well as
\begin{equation}\label{A4.9}
\|J^\pm_k\|_{L_1(\BR^d)} \leq \|\CF^{-1}[\chi]\|_{L_1(\BR^d)} < \infty.
\end{equation}
Thus, it follows from \eqref{Young-Besov} that
\begin{align}
\|J^+_k*(\pd_\ell f-g)\|_{\CB^s_{q,r}(\BR^d)} 
&\leq \|J^+_k\|_{L_1(\BR^d)}\|\pd_\ell f- g\|_{\CB^s_{q,r}(\BR^d)} \\
&\leq \|\CF^{-1}[\chi]\|_{L_1(\BR^d)}
\|\pd_\ell f-g\|_{\CB^s_{q,r}(\BR^d)} \\
&\leq \|\CF^{-1}[\chi]\|_{L_1(\BR^d)} \epsilon.
\end{align}
On the other hand, there holds
\begin{equation}
J^+_k*g - g = \int_{\BR^d} \CF^{-1}[\chi](y)(g(x-2^{-k}y) - g(x))\d y
= \int_{\BR^d}\CF^{-1}[\chi](y) \bigg(\int^1_0 (\nabla g)(x- \theta 2^{-k}y)\d\theta\bigg)2^{-k}y \d y.
\end{equation}
Noting that $\|(\nabla g)(\,\cdot- \theta 2^{-k}y)\|_{\CB^s_{q,r}(\BR^d)} 
= \|\nabla g\|_{\CB^s_{q,r}(\BR^d)}$, we have
\begin{equation}
\|J^+_k*g-g\|_{\CB^s_{q,r}(\BR^d)} 
\leq 2^{-k}\int_{\BR^d} |\CF^{-1}[\chi](y) y|\d y\, \|\nabla g\|_{\CB^s_{q,r}(\BR^d)}
\to 0 \quad\text{as $k \to\infty$}.
\end{equation}
Summarizing, we observe that
\begin{equation}
\limsup_{k\to\infty} \|J^+_k*(\pd_\ell f) - \pd_\ell f\|_{\CB^s_{q,r}(\BR^d)}
\leq (\|\CF^{-1}[\chi]\|_{L_1(\BR^d)}+1)\epsilon.
\end{equation}
Since $\epsilon > 0$ is chosen arbitrarily, we have \eqref{A4.5}. \par
We next prove \eqref{A4.6}. To this end, setting $\chi_m(x) = \chi(2^{-m}x)$ for $m\in\BN$, we 
consider the decomposition 
\begin{equation}
J^-_k*\pd_\ell f = L^1_{k,m} + L^2_{k,m}
\end{equation}
with
$L^1_{k,m} = J^-_k*[(1-\chi_m)\pd_\ell f]$ and $L^2_{k,m} = J^-_k*[\chi_m \pd_\ell f]$.
Since $\pd_\ell f\in \CB^s_{q,r}(\BR^d)$, we see that 
\begin{equation}\label{A4.7} \lim_{m\to\infty}\|(1-\chi_m)\pd_\ell f\|_{\CB^s_{q,r}(\BR^d)} = 0. 
\end{equation}
In fact, for any $\epsilon > 0$ there exists a $g \in C^\infty_0(\BR^d)$ such that
\begin{equation}\label{A4.8}
\|\pd_\ell f- g\|_{\CB^s_{q,r}(\BR^d)} < \epsilon.
\end{equation}
By the triangle inequality and \eqref{cor:prod.1}, we have
\begin{align*}
&\|(1-\chi_m)\pd_\ell f\|_{\CB^s_{q,r}(\BR^d)} \\
&\quad \leq \|(1-\chi_m)(\pd_\ell f-g)\|_{\CB^s_{q,r}(\BR^d)} + \|(1-\chi_m)g\|_{\CB^s_{q,r}(\BR^d)} \\
&\quad \leq \|\chi_m(\pd_\ell f-g)\|_{\CB^s_{q,r}(\BR^d)} 
+ \|\pd_\ell f- g\|_{\CB^s_{q,r}(\BR^d)} + \|(1-\chi_m)g\|_{\CB^s_{q,r}(\BR^d)}\\
&\quad \leq C\Big(\|\chi_m\|_{\CB^{d/q}_{q,\infty}(\BR^d) \cap L_\infty(\BR^d)}
+ 1\Big)
\epsilon + \|(1-\chi_m)g\|_{\CB^s_{q,r}(\BR^d)} \\
&\quad \leq C\epsilon + \|(1-\chi_m)g\|_{\CB^s_{q,r}(\BR^d)}.
\end{align*}
Here $C$ is a constant independent of $m$.  Since the support of $g$ is compact, 
for a large $m \geq 1$, we see that $(1-\chi_m)g=0$, and hence it follows that
$$\limsup_{m\to \infty}\|(1-\chi_m)\pd_\ell f\|_{\CB^s_{q,r}(\BR^d)} \leq C\epsilon,$$
which together with the arbitrariness of the choice of $\epsilon>0$ yields \eqref{A4.7}. 
Thus, combining \eqref{A4.7} and \eqref{A4.9}, we see that for any $\epsilon > 0$
there exists a large $m_0 \in \BN$ such that for any $m \geq m_0$ there holds
\begin{equation}\label{A4.10}
\|L^1_{k,m}\|_{\CB^s_{q,r}(\BR^d)} \leq \epsilon. 
\end{equation}
\par
Next, we shall show that
\begin{equation}\label{A4.18}
\limsup_{k\to\infty} \|L^2_{k,m}\|_{\CB^s_{q,r}(\BR^d)} = 0
\quad \text{for any $m \in \BN$}.
\end{equation}
In fact, since $C^\infty_0(\BR^d)$ is dense in $\CB^s_{q,r}(\BR^d)$, for any $\epsilon > 0$, 
let  $g \in C^\infty_0(\BR^d)$ such that \eqref{A4.8} holds. 
By the triangle inequality, we obtain
\begin{equation}\label{A4.15}
\|J^-_k*[\chi_m\pd_\ell f]\|_{\CB^s_{q,r}(\BR^d)} \leq \|J^-_k*[\chi_m(\pd_\ell f -g)]\|_{\CB^s_{q,r}(\BR^d)}
+ \|J^-_k*[\chi_m g]\|_{\CB^s_{q,r}(\BR^d)}
\end{equation}
Using \eqref{A4.9} and \eqref{Young-Besov}, we have
\begin{equation}\label{A4.16}\begin{aligned}
&\|J^-_k*[\chi_m(\pd_\ell f- g)]\|_{\CB^s_{q,r}(\BR^d)} \\
&\quad \leq C\|\CF[\chi]\|_{L_1(\BR^d)}\|\chi_m\|_{\CB^{d/q}_{q,\infty}(\BR^d) \cap L_\infty(\BR^d)}
\|\pd_\ell f -g\|_{\CB^s_{q,r}(\BR^d)}\\
&\quad\leq C\epsilon
\end{aligned}\end{equation}
with some constant $C>0$ independent of $m$. On the other hand, 
it follows from \eqref{Young-Besov} that
\begin{equation}\label{A4.17}
\|J^-_k*[\chi_mg]\|_{\CB^s_{q,r}(\BR^d)} \leq \|J^-_k\|_{\CB^s_{q,r}(\BR^d)} \|\chi_m g\|_{L_1(\BR^d)}.
\end{equation}
In the case that $0 < s < 1/q$, we observe that
\begin{align*}
\|J^-_k\|_{L_q(\BR^d)} & = \|2^{-kd}\CF[\chi](2^{-k} \cdot\,)\|_{L_q(\BR^d)}
\leq 2^{-kd(1-1/q)}\|\CF[\chi]\|_{L_q(\BR^d)}, \\
\|\nabla J^-_k\|_{L_q(\BR^d)} & = 2^{-k}\|2^{-kd}(\nabla \CF[\chi])(2^{-k} \cdot\,)\|_{L_q(\BR^d)}
\leq 2^{-kd(1-1/q)}2^{-k}\|\nabla \CF[\chi]\|_{L_q(\BR^d)},
\end{align*}
Interpolating these two inequalities and noting that $2^{-k} < 1$ for $k \in \BN$, we have
\begin{equation}\label{A4.12}
\|J^-_k\|_{\CB^s_{q,r}(\BR^d)} \leq 2^{-(1-1/q)kd}\|\CF[\chi]\|_{\CB^s_{q,r}(\BR^d)}
\quad\text{for $k \in \BN$ if $0 < s < 1/q$}.
\end{equation}
On the other hand, in the case that $-1+1/q < s < 0$,
since $\CB^s_{q,r}(\BR^d) = \CB^{-s}_{q', r'}(\BR^d)^*$ and 
$C^\infty_0(\BR^d)$ is dense in $\CB^{-s}_{q', r'}(\BR^d)$ for $1 \leq r' \leq \infty-$, 
for any $\varphi \in C^\infty_0(\BR^d)$ we observe that
\begin{align*}
|(2^{-kd}\CF[\chi](2^{-k}x), \varphi)_{\BR^d}|
&= \bigg|\int_{\BR^d} 2^{-kd}\CF[\chi](2^{-k}x)\varphi(x)\d x\bigg| \\
& = \bigg|\int_{\BR^d} \CF[\chi](y)\varphi(2^k y)\d y\bigg| \\
&\leq \|\CF[\chi\|_{\CB^s_{q,r}(\BR^d)} \|\varphi(2^k\cdot\,)\|_{\CB^{-s}_{q', r'}(\BR^d)}.
\end{align*}
Since $0 < -s < 1/q'$, we see that 
\begin{align*}
\|\varphi(2^k\cdot\,)\|_{L_{q'}(\BR^d)} &= 2^{-kd/q'}\|\varphi\|_{L_{q'}(\BR^d)}, \\
\|\nabla \varphi(2^k\cdot\,)\|_{L_{q'}(\BR^d)}& \leq 2^k 2^{-kd/q'}\|\nabla\varphi\|_{L_{q'}(\BR^d)},
\end{align*}
which implies that 
$$\|\varphi(2^k\cdot\,)\|_{\CB^{-s}_{q',r'}(\BR^d)}
\leq 2^{k|s|-kd/q'}\|\varphi\|_{\CB^{-s}_{q',r'}(\BR^d)}.
$$
Here, to show this estimate, $1\le2^k$ has been used for the case 
$\CB^{-s}_{q',r'} (\BR^d) = B^{-s}_{q',r'} (\BR^d)$.
Thus, we deduce that
\begin{equation}\label{A4.13}
\|2^{-kd}\CF[\chi](2^{-k}\cdot\,)\|_{\CB^s_{q,r}(\BR^d)} \leq 2^{-kd(\frac{1}{q'}-\frac{|s|}{d})}\|\CF[(\chi)\|_{\CB^s_{q,r}(\BR^d)}
\end{equation}
for $-1/q' < s < 0$. Notice that there holds 
\begin{equation}
\frac{1}{q'}-\frac{|s|}{d} > \frac{1}{q'} -\frac{1}{dq'} = \frac{1}{q'}\bigg(1-\frac1d\bigg) > 0,   
\end{equation}
since $-1+1/q' < |s| < 1/q'$.
The remaining case $s=0$ may be achieved by interpolating \eqref{A4.12} and \eqref{A4.13}. 
Hence, there exists a positive number $\sigma_s$ such that for all $k \in \BN$ we have
\begin{equation}\label{A4.14}
\|J^-_k\|_{\CB^s_{q,r}(\BR^d)} = \|2^{-kd}\CF[\chi](2^{-k}\cdot\,)\|_{\CB^s_{q,r}(\BR^d)} \leq 2^{-\sigma_s k}\|\CF[\chi]\|_{\CB^s_{q,r}(\BR^d)}
\end{equation}		
with $1 < q < \infty$ and $-1+1/q<s<1/q$. 
Together with \eqref{A4.15}, \eqref{A4.16}, and \eqref{A4.17}, we infer that 
$$\|J^-_k*[\chi_m \pd_\ell f]\|_{\CB^s_{q,r}(\BR^d)}  \leq C\epsilon + 
2^{-\sigma_s k}\|\CF[\chi]\|_{\CB^s_{q,r}(\BR^d)}\|g\|_{L_1(\BR^d)}.
$$
Letting $k\to\infty$, we have
$$\limsup_{k\to\infty} 	\|J^-_k*[\chi_m \pd_\ell f]\|_{\CB^s_{q,r}(\BR^d)} \leq C\epsilon.
$$
Since the choice of $\epsilon > 0$ is arbitrary, we arrive at \eqref{A4.18}.
Recalling that  $J^-_k*\pd_\ell f= L^1_{k,m}+ L^2_{k,m}$ and using \eqref{A4.10} and \eqref{A4.18}, 
we see that for any $\epsilon > 0$, we have
$$\limsup_{k\to\infty} \|J^-_k*\pd_\ell f\|_{\CB^s_{q,r}(\BR^d)} \leq \epsilon.
$$
Together with the arbitrariness of the choice of $\epsilon > 0$, we have \eqref{A4.6}.
Finally, \eqref{A4.5} and \eqref{A4.6} gives \eqref{A4.4}.\par
Now, we shall complete the proof of Proposition \ref{prop:Bdense}.  Let $w_k(x) = \chi(2^{-2k}x)v_k(x)$ that is 
an element of $C^\infty_0(\BR^d)$. We shall show that 
\begin{equation}\label{A4.19}
\lim_{k\to \infty} \|\nabla w_k-\nabla f\|_{\CB^s_{q,r}(\BR^d)} = 0.
\end{equation}
To this end, we write
$$\nabla w_k - \nabla f = 2^{-2k}(\nabla\chi)(2^{-2k}x)v_k(x) + \chi(2^{-2k}x)\nabla(v_k-f)
-(1-\chi(2^{-2k}))\nabla f.$$
By \eqref{B.5}, we have
\begin{equation}
\label{B21}\begin{aligned}
&\|2^{-2k}(\nabla\chi)(2^{-2k}\cdot\,)v_k\|_{\CB^s_{q,r}(\BR^d)}\\
& \quad \leq C2^{-2k}\|(\nabla\chi)(2^{-2k}\cdot\,)
\|_{\CB^{d/q}_{q,\infty}(\BR^d)\cap L_\infty(\BR^d)} 
\|v_k\|_{\CB^s_{q,r}(\BR^d)} \\
&\quad  \leq CM2^{-k}\|\nabla f\|_{\CB^s_{q,r}(\BR^d)}
\end{aligned}
\end{equation}with some constant $C > 0$ independent of $k$.
In addition, we have
\begin{equation}
\label{B22}
\begin{aligned}
&\|\chi(2^{-2k}\cdot\,)\nabla(v_k-f)\|_{\CB^s_{q,r}(\BR^d)} \\
&\quad  \leq C\|\chi(2^{-2k}\cdot\,)\|_{\CB^{d/q}_{q,\infty}(\BR^d) \cap L_\infty(\BR^d)}
\|\nabla(v_k-f)\|_{\CB^s_{q,r}(\BR^d)} \\
&\quad  \leq C\|\nabla(v_k-f)\|_{\CB^s_{q,r}(\BR^d)}
\end{aligned}
\end{equation}
with come constant $C>0$ independent of $k$. 
Combining estimates \eqref{B21} and \eqref{B22} with \eqref{A4.4} and \eqref{A4.7}, 
we have \eqref{A4.19}. The proof of Proposition \ref{prop:Bdense} is complete.			
\end{proof}
In our approach, it is crucial to use the operator $\Lambda^{1/2}_\gamma f$, 
which is introduced in \eqref{half-derivative.1}, to control
the boundary terms. To this end, we 
shall prove the following propositions.
\begin{prop}\label{half-derivative1} Let $X$ be a Banach space and $\gamma \ge 0$. Define
\begin{equation}
\Lambda_\gamma^{1/2}f := \CF^{-1}[(\gamma + i\tau)^{1/2}\CF[f]](t)
= \frac{1}{2\pi}\int_{- \infty}^\infty e^{it\tau}(\gamma + i\tau)^{1/2}\CF[f](\tau)\d\tau.
\end{equation}
If $\gamma > 0$ there holds
\begin{equation}
\|\Lambda_\gamma^{1/2}f\|_{L_1(\BR, X)} \le C_\gamma \|f\|_{W^{1/2}_1(\BR, X)},
\end{equation}
whereas if $\gamma = 0$ there holds
\begin{equation}
\|\Lambda_0^{1/2}f\|_{L_1(\BR, X)} \le C \|f\|_{\dot W^{1/2}_1(\BR, X)}.
\end{equation}
\end{prop}
\begin{proof} We first deal with the case $\gamma > 0$. 
We use the dyadic decomposition given in \eqref{dyadic:1} with $d=1$. Recall that
$\phi(\tau) \in C^\infty(\BR)$ such that $\supp(\phi) \subset \{\tau \in \BR \mid
1/2 \le |\tau| \le 2\}$ such that $\sum_{j \in \BZ} \phi(2^{-j}\tau) = 1$ for 
$\tau \in \BR\setminus\{0\}$, and set $\phi_0(\tau) = 1- \sum_{j = 1}^\infty \phi(2^{-j}\tau)$.
Obviously, $\phi_0 \in C^\infty_0(\BR)$ which equals $1$ for $|t| \le 1/4$ and $0$ for 
$|\tau| \geq 8$. Let $\chi(\tau)$ and $\chi_0$ be $C^\infty_0(\BR)$ functions such that
$\chi(\tau) = 1$ for $1/2 \le |\tau| \le 2$ and $\chi(\tau) = 0$ for 
$|\tau| \le 1/8$ or $|\tau| \geq 4$, and $\chi_0(\tau) = 1$ for $|\tau| \le 8$ and 
$\chi_0(\tau) = 0$ for $|\tau| \geq 16$.  Notice that there hold
$\phi(2^{-j}\tau)\chi(2^{-j}\tau) = \phi(2^{-j}\tau)$ and $\phi_0(\tau)\chi_0(\tau) = \phi_0(\tau)$.
Let $\phi_j(\tau) = \phi(2^{-j}\tau)$ and $\chi_j (\tau) = \phi(2^{-j}\tau)$ for $j \in \BN$. 
Recall that the nonhomogeneous dyadic blocks $\{\Delta_j\}_{j \in \BZ}$ 
is defined by \eqref{dyadic:1}. We observe that
\begin{equation}
\label{id-B4}
\begin{split}
\Lambda_\gamma^{1/2}f & = \sum_{j=0}^\infty \CF^{-1}[(\gamma + i\tau)^{1/2}\phi_j(\tau)\CF[f](\tau)
\\
&= \sum_{j=0}^\infty \CF^{-1}[(\gamma + i\tau)^{1/2}\chi_j(\tau) \phi_j(\tau)\CF[f](\tau)]\\
& = \sum_{j=0}^\infty \Psi_j*\Delta_jf.
\end{split}
\end{equation}
Here, we have set
\begin{equation}
\Psi_j(t) = \frac{1}{2\pi}\int_{- \infty}^\infty e^{it\tau}(\gamma + i\tau)^{1/2}
\chi_j(2^{-j}\tau)\d\tau, \qquad j \in \BN_0.
\end{equation}
Using the change of variable: $2^{-j}\tau = \tau'$, we have
\begin{equation}
\Psi_j(t) = 2^{j/2}2^{j}\int_{- \infty}^\infty e^{i(2^jt)\tau'} (2^{-j}\gamma + i\tau')^{1/2}\phi(\tau')\d\tau'.
\end{equation}
Since $\chi(\tau') \in C^\infty_0 (\BR)$ and $\supp(\chi) \subset \{\tau' \in \BR \mid 
1/8 \le |\tau'| \le 8\}$, setting $\Xi_j(t) = 
\int_{- \infty}^\infty e^{it\tau'} (2^{-j}\gamma + i\tau')^{1/2}\phi(\tau')\d\tau'$, we have $\Psi_j(t) 
= 2^{j/2}2^j \Xi(2^jt)$ and $\|\Xi_j\|_{L_1(\BR)} \le C_\gamma$, where $C_\gamma$ is a constant 
independent of $j \in \BN$. Thus, by the change of variable: $2^j t= \tau'$, we see that 
$\|\Psi_j\|_{L_1(\BR)} \le C_\gamma 2^{j/2}$ for $j \in \BN$. Hence, it follows that 
\begin{equation}
\|\Lambda_\gamma^{1/2}f \|_{L_1(\BR, X)} 
\le C_\gamma \sum_{j = 0}^\infty 2^{j/2}\|\Delta_jf\|_{L_1(\BR, X)}
\le C_\gamma \|f\|_{W^{1/2}_1(\BR, X)}.
\end{equation}
The case $\gamma = 0$ may be proved similarly, if one replaces $\Delta_j$ by
$\dot \Delta_j$ in \eqref{id-B4}. The proof is complete.
\end{proof}
\begin{prop}\label{prop-B7}
Let $1 \le q \le \infty$, $s \in \BR$, and $0 < \theta < 1$.
Then the following two assertions are valid.
\begin{enumerate}
\item For every $f \in L_1 (\BR, B^{s + 2}_{q, 1} (\BR^d_+)) 
\cap W^1_1 (\BR, B^s_{q, 1} (\BR^d_+))$ it holds
\begin{equation}
\lVert f \rVert_{W^\theta_1 (\BR, B^{s + 2 (1 - \theta)}_{q, 1} (\BR^d_+))} 
\le C \lVert f \rVert_{L_1 (\BR, B^{s + 2}_{q, 1} (\BR^d_+))}^{1 - \theta}
\lVert f \rVert_{W^1_1 (\BR, B^s_{q, 1} (\BR^d_+))}^\theta.
\end{equation}
\item Assume additionally that $- 1 + 1 \slash q < s < 1 \slash q$. 
For every measurable function $f (x, t)$ in $\HS \times \BR$ such that 
$\pd_t f \in L_1 (\BR, \dot B^s_{q, 1} (\BR^d_+))$ and 
$\nabla^2 f \in L_1 (\BR, \dot B^s_{q, 1} (\BR^d_+))$ it holds
\begin{equation}
\lVert f \rVert_{\dot W^\theta_1 (\BR, B^{s + 2 (1 - \theta)}_{q, 1} (\BR^d_+))} 
\le C \lVert \nabla^2 f \rVert_{L_1 (\BR, \dot B^s_{q, 1} (\BR^d_+))}^{1 - \theta}
\lVert \pd_t f \rVert_{L_1 (\BR, \dot B^s_{q, 1} (\BR^d_+))}^\theta.
\end{equation}
\end{enumerate}
\end{prop}

\begin{proof}
The idea of the proof here is taken from the proof of
Proposition 8.3 on \cite{DHMTpre}. \\
\textbf{(1)}
Let $\widetilde f \in L_1 (\BR, B^{s + 2}_{q, 1} (\BR^d)) 
\cap W^1_1 (\BR, B^s_{q, 1} (\BR^d))$ be any extension of 
$f \in L_1 (\BR, B^{s + 2}_{q, 1} (\BR^d_+)) 
\cap W^1_1 (\BR, B^s_{q, 1} (\BR^d_+))$. If we establish
\begin{equation}
\label{est-interpolation-whole}
\lVert \widetilde f \rVert_{W^\theta_1 (\BR, B^{s + 2 (1 - \theta)}_{q, 1} (\BR^d))} 
\le C \lVert \widetilde f \rVert_{L_1 (\BR, B^{s + 2}_{q, 1} (\BR^d))}^{1 - \theta}
\lVert \widetilde f \rVert_{W^1_1 (\BR, B^{s + 2}_{q, 1} (\BR^d))}^\theta,
\end{equation}
then taking the infimum on all extensions gives the desired result
due to our definition of (inhomogeneous) Besov spaces by restriction.
Hence, it suffices to show \eqref{est-interpolation-whole}.
In the following, we shall write $f$ instead of $\wt f$ to simplify the notation.
To show \eqref{est-interpolation-whole}, we first note that 
\begin{equation}
\lVert f \rVert_{W^\theta_1 (\BR, B^{s + 2 (1 - \theta)}_{q, 1} (\BR^d))}
= \lVert f \rVert_{L_1 (\BR, B^{s + 2 (1 - \theta)}_{q, 1} (\BR^d))}
+ \int_\BR \lvert h \rvert^{- (\theta + 1)} \lVert f (\,\cdot\, , 
\,\cdot + h) - f (\,\cdot\, , \,\cdot\,) 
\rVert_{L_1 (\BR, B^{s + 2 (1 - \theta)}_{q, 1} (\BR^d))} \d h.
\end{equation}
Notice that we have
\begin{align}
\lVert f (\,\cdot\, , \,\cdot + h) - f (\,\cdot\, , \,\cdot\,) 
\rVert_{L_1 (\BR, B^{s + 2 (1 - \theta)}_{q, 1} (\BR^d))}
& = \int_\BR \sum_{j = 0}^\infty 2^{j (s + 2 (1 - \theta))} 
\lVert \Delta_j ( f (\,\cdot\, , t + h) - f (\,\cdot\, , t) )
\rVert_{L_q (\BR^d)} \d t \\
& = \sum_{j = 0}^\infty 2^{j (s + 2 (1 - \theta))} \int_\BR 
\lVert \Delta_j ( f (\,\cdot\, , t + h) - f (\,\cdot\, , t) )
\rVert_{L_q (\BR^d)} \d t
\end{align}
due to the monotone convergence theorem.
Thus, we may write
\begin{align}
\lVert f \rVert_{W^\theta_1 (\BR, B^{s + 2 (1 - \theta)}_{q, 1} (\BR^d))}
& = \lVert f \rVert_{L_1 (\BR, B^{s + 2 (1 - \theta)}_{q, 1} (\BR^d))} \\
& + \sum_{j = 0}^\infty 2^{j (s + 2 (1 - \theta))} \int_\BR \int_\BR
\lvert h \rvert^{- (\theta + 1)} 
\lVert \Delta_j ( f (\,\cdot\, , t + h) - f (\,\cdot\, , t) )
\rVert_{L_q (\BR^d)} \d t \d h.
\end{align}
For $j \in \BN_0$, we set
\begin{align}
I_{j, N}^{(\ell)} (f) & := \int_\BR \int_{\lvert h \rvert \le N} 
\lvert h \rvert^{- (\theta + 1)} 
\lVert \Delta_j ( f (\,\cdot\, , t + h) - f (\,\cdot\, , t) )
\rVert_{L_q (\BR^d)} \d t \d h, \\
I_{j, N}^{(h)} (f) & := \int_\BR \int_{\lvert h \rvert \ge N} 
\lvert h \rvert^{- (\theta + 1)}
\lVert \Delta_j ( f (\,\cdot\, , t + h) - f (\,\cdot\, , t) )
\rVert_{L_q (\BR^d)} \d t \d h.
\end{align}
Since it holds
\begin{equation}
\lvert h \rvert^{- 1} \lVert \Delta_j ( f (\,\cdot\, , t + h) 
- f (\,\cdot\, , t) ) \rVert_{L_q (\BR^d)}
\le \lVert \Delta_j \pd_t f (\,\cdot\, , t) \rVert_{L_q (\BR^d)},
\end{equation}
we see that
\begin{align}
I^{(\ell)}_{j, N} (f) & \le \int_{\lvert h \rvert \le N}
\lvert h \rvert^{- \theta} \d h \int_\BR 
\lvert h \rvert^{- 1} \lVert \Delta_j ( f (\,\cdot\, , t + h) - 
f (\,\cdot\, , t) ) \rVert_{L_q (\BR^d)} \d t \\
& \le C N^{1 - \theta} \int_\BR \lVert \Delta_j 
\pd_t f (\,\cdot\, , t) \rVert_{L_q (\BR^d)} \d t.
\end{align}
Likewise, we also have
\begin{align}
I^{(h)}_{j, N} (f) & =\int_{|h| \geq N}|h|^{-(\theta+1)}\d h
\int_{- \infty}^\infty \|\Delta_j(f(\,\cdot\, , t+h) - f(\,\cdot\, , t))\|_{L_q(\BR^d)}\d t\\
&\le 2\int_{|h| \geq N}|h|^{-(\theta+1)}\d h
\int_{- \infty}^\infty \|\Delta_jf(\,\cdot\, , t))\|_{L_q(\BR^d)}\d t \\
& \le C N^{- \theta} \int_\BR \lVert \Delta_j 
f (\,\cdot\, , t) \rVert_{L_q (\BR^d)} \d t.
\end{align}
We now take $N > 0$ such that
\begin{equation}
N = \lVert \Delta_j f \rVert_{L_1 (\BR, L_q (\BR^d))}
\lVert \Delta_j \pd_t f \rVert_{L_1 (\BR, L_q (\BR^d))}^{- 1}
\end{equation}
that yields
\begin{equation}
I^{(\ell)}_{j, N} (f) \le C \lVert \Delta_j
f \rVert_{L_1 (\BR, L_q (\BR^d))}^{1 - \theta} 
\lVert \Delta_j \pd_t f \rVert_{L_1 (\BR, L_q (\BR^d))}^\theta
\end{equation}
and
\begin{equation}
I^{(h)}_{j, N} (f) \le C \lVert \Delta_j 
f \rVert_{L_1 (\BR, L_q (\BR^d))}^{1 - \theta} 
\lVert \Delta_j \pd_t f \rVert_{L_1 (\BR, L_q (\BR^d))}^\theta.
\end{equation}
Thus, it follows that
\begin{equation}
\begin{split}
\lVert f \rVert_{W^\theta_1 (\BR, B^{s + 2 (1 - \theta)}_{q, 1} (\BR^d))}
& \le \lVert f \rVert_{L_1 (\BR, B^{s + 2 (1 - \theta)}_{q, 1} (\BR^d))} \\
& \quad + \sum_{j = 0}^\infty 2^{j (s + 2 (1 - \theta))} \lVert \Delta_j 
f \rVert_{L_1 (\BR, L_q (\BR^d))}^{1 - \theta} 
\lVert \Delta_j \pd_t f \rVert_{L_1 (\BR, L_q (\BR^d))}^\theta.
\end{split}
\end{equation}
We infer from the H\"older inequality for series that
\begin{equation}
\begin{split}
& \sum_{j = 0}^\infty 2^{j (s + 2 (1 - \theta))} \lVert \Delta_j 
f \rVert_{L_1 (\BR, L_q (\BR^d))}^{1 - \theta} 
\lVert \Delta_j \pd_t f \rVert_{L_1 (\BR, L_q (\BR^d))}^\theta \\
& \quad \le \bigg(\sum_{j = 0}^\infty 2^{j (s + 2)} \lVert \Delta_j
f \rVert_{L_1 (\BR, L_q (\BR^d))} \bigg)^{1 - \theta} 
\bigg(\sum_{j = 0}^\infty 2^{j s} \lVert \Delta_j \pd_t f 
\rVert_{L_1 (\BR, L_q (\BR^d))} \bigg)^\theta \\
& \quad = \lVert f \rVert_{L_1 (\BR, B^{s + 2}_{q, 1} (\BR^d))}^{1 - \theta}
\lVert \pd_t f \rVert_{L_1 (\BR, B^s_{q, 1} (\BR^d))}^\theta,
\end{split}
\end{equation}
where the monotone convergence theorem is used to obtain 
the last identity. Therefore, we arrive at
\begin{align}
\lVert f \rVert_{W^\theta_1 (\BR, B^{s + 2 (1 - \theta)}_{q, 1} (\BR^d))}
& \le C \Big(\lVert f \rVert_{L_1 (\BR, B^{s + 2 (1 - \theta)}_{q, 1} (\BR^d))} 
+ \lVert f \rVert_{L_1 (\BR, B^{s + 2}_{q, 1} (\BR^d))}^{1 - \theta}
\lVert \pd_t f \rVert_{L_1 (\BR, B^s_{q, 1} (\BR^d))}^\theta \Big) \\
& \le C \lVert f \rVert_{L_1 (\BR, B^{s + 2}_{q, 1} (\BR^d))}^{1 - \theta}
\lVert f \rVert_{W^1_1 (\BR, B^s_{q, 1} (\BR^d))}^\theta,
\end{align}
where the interpolation inequality
\begin{equation}
\lVert f \rVert_{L_1 (\BR, B^{s + 2 (1 - \theta)}_{q, 1} (\BR^d))}
\le \lVert f \rVert_{L_1 (\BR, B^{s + 2}_{q, 1} (\BR^d))}^{1 - \theta}
\lVert f \rVert_{L_1 (\BR, B^s_{q, 1} (\BR^d))}^\theta 
\end{equation}
is applied (cf. \cite[Thm.~2.2.10]{HNVW} and \cite[Thm. 1.18.4.]{Tbook78}). 
\vskip.4pc
\noindent \textbf{(2)} For $\alpha \in \BN_0^d$ satisfying $\lvert \alpha \rvert \le2$ 
let $E$ and $E_\alpha$ be extension operators such that
\begin{align}
E f (x) & = \begin{cases}
f (x), & \quad x \in \HS, \\
\displaystyle \sum_{j = 1}^3 a_j f (x', - j x_d), & \quad x \in \BR^d \setminus \HS, 
\end{cases} \\
E_\alpha f (x) & = \begin{cases}
f (x), & \quad x \in \HS, \\
\displaystyle \sum_{j = 1}^3 (- j)^{\alpha_d} a_j f (x', - j x_d), 
& \quad x \in \BR^d \setminus \HS, 
\end{cases} 
\end{align}
where the coefficients $a_1, a_2, a_3$ are the unique solutions to 
$\sum_{j = 1}^3 (- j)^k a_j = 1$, $k = 0, 1, 2$. It is well-known that
there holds $\pd^\alpha E f (x) = E_\alpha \pd^\alpha f (x)$, see, e.g.,
\cite[Thm. 5.19]{AFbook}. 
We infer from \cite[Prop. 8.3]{DHMTpre} and \cite[Prop.~2.3]{AP07} that
\begin{align}
\lVert f \rVert_{\dot W^\theta_1 (\BR, B^{s + 2 (1 - \theta)}_{q, 1} (\BR^d_+))} 
& \le \lVert E f \rVert_{\dot W^\theta_1 (\BR, B^{s + 2 (1 - \theta)}_{q, 1} (\BR^d))} \\
& \le C \lVert \nabla E f \rVert_{L_1 (\BR, \dot B^{s + 1}_{q, 1} (\BR^d))}^{1 - \theta}
\lVert \pd_t E f \rVert_{L_1 (\BR, \dot B^s_{q, 1} (\BR^d))}^\theta \\
& \le C \lVert \nabla^2 E f \rVert_{L_1 (\BR, \dot B^s_{q, 1} (\BR^d))}^{1 - \theta}
\lVert E \pd_t f \rVert_{L_1 (\BR, \dot B^s_{q, 1} (\BR^d))}^\theta \\
& \le C \lVert E_\alpha \nabla^2 f \rVert_{L_1 (\BR, \dot B^s_{q, 1} (\BR^d))}^{1 - \theta}
\lVert E \pd_t f \rVert_{L_1 (\BR, \dot B^s_{q, 1} (\BR^d))}^\theta.
\end{align}
Since it follows from \cite[Cor. 2.2.1]{DM15}
that $E \colon \dot B^s_{q, r} (\HS) \to \dot B^s_{q, r} (\BR^d)$ and 
$E_\alpha \colon \dot B^s_{q, r} (\HS) \to \dot B^s_{q, r} (\BR^d)$ are bounded due
to the restriction $- 1 + 1 \slash q < s < 1 \slash q$,
we arrive at the desired estimate. The proof is complete.
\end{proof}

\section*{Competing Interests}
On behalf of all authors, the corresponding author declares that there is no conflict of interest.


\end{document}